\documentclass[reqno]{amsart}
\usepackage{amssymb,amsmath,epsfig,graphics,mathrsfs}

\title{federico}
\date{November 2023}

\usepackage{fancyhdr}

\usepackage{hyperref}
\hypersetup{
	colorlinks   = true,
	urlcolor     = blue,
	linkcolor    = blue,
	citecolor   = red ,
	bookmarksopen=true
}

\usepackage{dsfont}

\usepackage{cancel} 
\usepackage{esvect}

\usepackage{multicol}

\usepackage{amssymb,amsmath,epsfig,graphics,mathrsfs}

\usepackage{graphicx}
\usepackage{caption}

\usepackage{bbm}
\usepackage[normalem]{ulem}

\usepackage[dvipsnames,table,xcdraw]{xcolor} 
\usepackage[a4paper,
left=1in,
right=1in,
top=1in,
bottom=1in,
footskip=.25in]{geometry}
\usepackage{hyperref}
\hypersetup{
	colorlinks   = true,
	urlcolor     = blue,
	linkcolor    = blue,
	citecolor   = red ,
	bookmarksopen=true
}

\def \R {\mathbb{R}}

\newcommand{\equ}[1]{(\ref{#1})}

\numberwithin{equation}{section}

\newcommand{\beq}{\begin{equation}}
	\newcommand{\bea}[1]{\begin{array}{#1} }
		\newcommand{\eeq}{ \end{equation}}
	\newcommand{\ea}{ \end{array}}

\newcommand{\inn}{{\quad\hbox{in } }}

\newtheorem{theorem}{Theorem}[section]
\newtheorem{lemma}[theorem]{Lemma}
\newtheorem{proposition}[theorem]{Proposition}
\newtheorem{corollary}[theorem]{Corollary}
\newtheorem{remark}[theorem]{Remark}

\makeatletter
\def\@settitle{\begin{center}%
		\baselineskip14\p@\relax
		\bfseries
		\uppercasenonmath\@title
		\@title
		\ifx\@subtitle\@empty\else
		\\[1ex]\uppercasenonmath\@subtitle
		\footnotesize\mdseries\@subtitle
		\fi
	\end{center}%
}
\def\subtitle#1{\gdef\@subtitle{#1}}
\def\@subtitle{}
\makeatother

\newcommand{\ch}[1]{#1}

\usepackage[section]{placeins}

\begin{document}
	\title[Existence of finite time blow-up in Keller-Segel system]{Existence of finite time blow-up in Keller-Segel system}

\author[ F.~Buseghin]{Federico Buseghin}
\address{\noindent F.~Buseghin: Department of Mathematical Sciences University of Bath, Bath BA2 7AY, United Kingdom}
\email{fb588@bath.ac.uk}

\author[J.~D\'avila]{Juan D\'avila}
\address{\noindent J.~D\'avila: Department of Mathematical Sciences University of Bath, Bath BA2 7AY, United Kingdom}
\email{jddb22@bath.ac.uk}

\author[M.~del Pino]{Manuel del Pino}
\address{\noindent M.~del Pino: Department of Mathematical Sciences University of Bath, Bath BA2 7AY, United Kingdom}
\email{m.delpino@bath.ac.uk}

\author[M.~Musso]{Monica Musso}
\address{\noindent M.~Musso: Department of Mathematical Sciences University of Bath, Bath BA2 7AY, United Kingdom.}
\email{m.musso@bath.ac.uk}

\begin{abstract}
	Perhaps the most classical diffusion model for chemotaxis is the Keller-Segel system
	\begin{equation}\tag{$\ast$}
		\label{ks0ab}
		\left\{ \begin{aligned}
			u_t =&\; \Delta u - \nabla \cdot(u \nabla v) \quad \inn \R^2\times(0,T),\\
			v =&\; (-\Delta_{\R^2})^{-1} u := \frac 1{2\pi} \int_{\R^2} \, \log \frac 1{|x-z|}\,u(z,t)\, dz,
			\\ & \qquad\ u(\cdot ,0) = u_0^{\star} \ge 0\quad\hbox{in } \R^2.
		\end{aligned}
		\right.
	\end{equation}
   We show that there exists $\varepsilon>0$ such that for any $m$ satisfying
    \begin{align*}
    	8\pi<m\le 8\pi+\varepsilon
    \end{align*}
    and any $k$ given points $q_{1},...,q_{k}$ in $\mathbb{R}^{2}$
    there is an initial data $u_0^*$ of \equ{ks0ab} for which the solution $u(x,t)$ blows-up in finite time as $t\to T$ with the approximate profile
	$$
		u(x,t)=\sum_{j=1}^{k}\frac{1}{\lambda_{j}^{2}(t)}U\left(\frac{x-\xi_{j}(t)}{\lambda_{j}(t)}\right)(1+o(1)),\ \ \ U(y)=\frac{8}{(1+|y|^{2})^{2}}
	,
	$$
	with $\lambda_{j}(t) \approx 2e^{-\frac{\gamma+2}{2}}\sqrt{T-t}e^{-\sqrt{\frac{|\ln(T-t)|}{2}}} $  where $\gamma=0.57721...$ is the Euler-Mascheroni constant, $\xi_{j}(t)\to q_{j}\in \mathbb{R}^{2}$ and such that 
	$$		\int_{\R^2}u(x,t)dx=km.
$$ 
This construction generalizes the existence result of the stable blow-up dynamics recently proved in \cite{CGMN1,CGMN2}.

\end{abstract}









\date{}





\maketitle

\section{Introduction}
In the present work we consider the two dimensional Keller-Segel model
\begin{equation}\label{ks0}
	\left\{ \begin{aligned}
		u_t =&\; \Delta u - \nabla \cdot(u \nabla v) \quad \inn \R^2\times(0,T),\\
		v =&\; (-\Delta_{\R^2})^{-1} u := \frac 1{2\pi} \int_{\R^2} \, \log \frac 1{|x-z|}\,u(z,t)\, dz,
		\\ & \qquad\ u(\cdot ,0) = u_0^{\star} \ge 0\quad\hbox{in } \R^2.
	\end{aligned}
	\right.
\end{equation}
The system \eqref{ks0} is used to model \emph{chemotaxis}, a biological phenomenon describing the motion of a population density, $u(x,t)$, in the presence of an external chemical stimulus, $v(x,t)$, called \emph{chemoattractant}. There is a huge literature on chemotaxis in biology and mathematics. This model was introduced by Keller-Segel \cite{KS} to describe the motion of the slime mold  amoebae Dictyostelium discoideum (see also \cite{Pat} for an earlier model proposed by Patlak and the later works \cite{KS1,KS2,KS0}). Under certain circumstances, these cells secret the chemoattractant themselves. This self-organization process is considered a fundamental mechanism in biology. \newline  The two dimensional case is clearly of special interest in biology but, as we will show, also from a mathematical point of view because of the scaling properties of \eqref{ks0}. \newline
We consider positive solutions which are well defined, unique and smooth up to a maximal time $0<T\le +\infty$. This problem formally preserves the mass
\begin{align*}
	M:=\int_{\R^2}u(x,t)dx=\int_{\R^2}u_{0}(x)dx \ \ \ \text{for any }t\ge0.
\end{align*}  
The study of positive steady states of \eqref{ks0} 
\begin{align}\label{steadyks0}
	\begin{cases}
		0=\Delta u -\nabla \cdot(u\nabla v), \ \ x\in \mathbb{R}^{2} \\
		0=\Delta v+u, \ \ x\in\mathbb{R}^{2},\\
		u>0
	\end{cases}
\end{align}
is equivalent to the study of ground states of the equation
\begin{align}\label{LiouvilleEqts}
	\Delta v+ \lambda e^{v}=0, \ \ x\in \mathbb{R}^{2} \ \text{and }\lambda>0.
\end{align}
In fact in \cite{dPW} they showed that a solution of \eqref{steadyks0} satisfies
\begin{align*}
	\int_{\R^2}u|\nabla (\ln u-v)|^{2}dx=0,
\end{align*}
so that $u=\lambda e^{v}$ for some positive constant $\lambda$, resulting in the equation \eqref{LiouvilleEqts} with $\lambda=1$.
Note that $v_{\lambda,\xi}(x)=\ln U_{\lambda,\xi}$ where
\begin{align}\label{bubbleF}
	U_{\lambda,\xi}(x)=\frac{1}{\lambda^{2}}U_{0}(\frac{x-\xi}{\lambda}), \ \ U_{0}(y)=\frac{8}{(1+|y|^{2})^{2}}, \ \ \lambda>0, \xi \in \mathbb{R}^{2}
\end{align}
is a solution of \eqref{steadyks0}. \newline
For a fast decaying solution one gets
\begin{align}\label{secMomForm}
	&\int_{\mathbb{R}^{2}}\Delta \omega|x|^{2}-\int_{\mathbb{R}^{2}}\nabla \cdot(\omega\nabla(-\Delta)^{-1}\omega)|x|^{2}dx=4M-\frac{M^{2}}{2\pi},
\end{align}
that leads to the following identity for the second moment of the solution
\begin{align}\label{secMomFormEquation}
	&\frac{d}{dt}\int_{\R^2}u(x,t)|x|^{2}dx=4M-\frac{M^{2}}{2\pi}.
\end{align}
In their seminal work \cite{JL}, J\"ager-Luckhaus derived the parabolic-elliptic system \eqref{ks0} and showed that solutions globally exist if the mass $M$ is small and blow-up in finite time if $M$ is large. More recently this first result has been improved. \newline
If $M<8\pi$ the second moment grows linearly and it can be proved that the solution exists for all time and diffuses to zero with a self-similar profile \cite{DP} \cite{BDP}. \newline When $M=8\pi$ there exist globally radially symmetric solutions to \eqref{ks0} for initial data with finite or infinite second moment \cite{BKLN} \cite{BKLN1}. If the initial second moment is finite, the infinite time blow-up takes place in the form of a bubble \eqref{bubbleF} with $\lambda(t)\to0$ \cite{BKLN1} \cite{BCM}. In \cite{GM} \cite{DdPDMW} they constructed explicit examples satisfying
\begin{align}\label{infinitetimeBlowUP}
	u(x,t)\approx \frac{1}{\lambda^{2}(t)}U_{0}(\frac{x}{\lambda(t)}), \ \ \ \lambda(t)\approx \frac{c}{\sqrt{\ln t}},
\end{align} 
for some constant $c>0$. The blow-up dynamics \eqref{infinitetimeBlowUP} was proved to be radially stable in \cite{GM}. In \cite{DdPDMW} they proved the nonradial stability. \newline
If $M>8\pi$, the negative rate of production of the second moment and the positivity of the solution implies finite time blow-up. Kozono-Sugiyama \cite{KoSu} showed that any blow-up solutions of \eqref{ks0} satisfies
\begin{align*}
	\|u(t)\|_{L^{\infty}(\mathbb{R}^{2})}\ge C(T-t)^{-1}.
\end{align*}
In general we say that a solution $u(t)$ of \eqref{ks0} exhibits type I blow-up at $t=T$ if there exists a constant $C>0$ such that
\begin{align*}
	\limsup_{t\to T}(T-t)\|u(t)\|_{L^{\infty}(\mathbb{R}^{2})}\le C
\end{align*}
otherwise the blow-up is of type II. Concrete examples of type II blow-up solutions were constructed in the radial setting by Herrero-Vel\'azquez in \cite{HV3} (see also \cite{Ve, Ve1, Ve2} for stability results) and by Rapha\"el-Schweyer \cite{RS}. More recently  Collot-Ghoul-Masmoudi-Nguyen \cite{CGMN1, CGMN2} improved the result in \cite{RS} constructing a stable dynamics in the nonradial setting where the solution of \eqref{ks0} satisfies
\begin{align}\label{finitetimeDyn}
	 	u(x,t)\approx \frac{1}{\lambda^{2}(t)}U_{0}(\frac{x}{\lambda(t)}), \ \ \, \lambda(t)\approx 2e^{-\frac{\gamma+2}{2}}e^{-\sqrt{\frac{|\ln(T-t)|}{2}}},
\end{align}
with $\gamma=0.57721...$ the Euler-Mascheroni constant. In \cite{CGMN2}  they additionally constructed other blow-up solutions with different rates corresponding to unstable blow-up dynamics. More recently Mizoguchi \cite{Mi} proved that \eqref{finitetimeDyn} is the only stable blow-up mechanism that occurs among radial nonnegative solutions. \newline

The purpose of this work is to construct finite-time blow-up solution by using a different method from \cite{HV3}, \cite{RS} and \cite{CGMN2}. Our proof is based on the gluing techniques developed in \cite{CdPM, DdPW,dPMW,DdPDMW}. The following is our main result.
\begin{theorem}\label{theorem1}
There exists $\varepsilon>0$ such that for any $m$ satisfying
\begin{align*}
	8\pi<m\le 8\pi+\varepsilon
\end{align*}
and any $k$ given points $q_{1},...,q_{k}$ in $\mathbb{R}^{2}$
there is an initial data $u_0^{\star}$ of \equ{ks0} for which the corresponding solution $u(x,t)$ has the form
$$
u(x,t)=\sum_{j=1}^{k}\frac{1}{\lambda_{j}^{2}(t)}U\left(\frac{x-\xi_{j}(t)}{\lambda_{j}(t)}\right)(1+o(1)),\ \ \ U(y)=\frac{8}{(1+|y|^{2})^{2}}
,
$$
uniformly on bounded sets of $\mathbb{R}^{2}$,
with $\lambda_{j}(t) \approx 2e^{-\frac{\gamma+2}{2}}\sqrt{T-t}e^{-\sqrt{\frac{|\ln(T-t)|}{2}}} $  where $\gamma=0.57721...$ is the Euler-Mascheroni constant , $\xi_{j}(t)\to q_{j}\in \mathbb{R}^{2}$ and such that 
$$		\int_{\R^2}u(x,t)dx=km.
$$ 
\end{theorem}
\noindent We remark that Theorem \ref{theorem1} generalizes the existence results in \cite{HV3}, \cite{RS} and \cite{CGMN1, CGMN2}. Firstly we construct a solution in the nonradial setting that concentrates at $k\ge1$ points in $\mathbb{R}^{2}$. 
Secondly our method allows us to find a concrete relation between the mass of $u$ and the blow up time. 
More precisely, in the case of a single blow up point, we show that there is $\varepsilon>0$ so that given $T>0$ small, for any $m$ in the range
\[
8\pi + \epsilon(T)\leq m\leq 8\pi+\varepsilon
\]
where 
\[
\epsilon(T)=2\sqrt{2}\pi e^{-(\gamma+2)}e^{-\sqrt{2|\ln T|}}\sqrt{2|\ln T|}(1+O(\frac{1}{\sqrt{|\ln T|}})),
\]
there is a solution $u(x,t)$ as described in Theorem~\ref{theorem1} with mass $m$.
\newline

The structure of the paper follows naturally from the proof presented in the infinite time construction based on the so-called \emph{inner-outer gluing scheme} \cite{DdPDMW}. We anticipate that the construction in finite time is different in some respects. In fact there are three main technical difficulties:
\begin{itemize}
	\item in the construction of the first approximation it is crucial to find a good approximation of the rate of the blow-up $\lambda(t)$. In finite time time we will show that $\lambda(t)$ must satisfy the following \emph{nonlocal} equation
		\begin{align}\label{nonlocalIntro}
			4\pi \int_{-\varepsilon(T)}^{T}\frac{\lambda\dot{\lambda}(s)}{T-s}ds-4\pi \int_{-\varepsilon(T)}^{t-\lambda^{2}}\frac{\lambda\dot{\lambda}(s)}{t-s}ds+4\pi(\gamma+1-\ln4)\lambda\dot{\lambda}(t)=0 \ \ \ \ t\in(0,T),
		\end{align}
	for some $T\ll\varepsilon(T)\ll1$ whose choice will be explained in Section \ref{MASSPHISECTION}. We observe that our construction allows us to find the following expansion for the mass of the solution
\begin{align*} \int_{\R^2} u(x,t)dx=8\pi+2\sqrt{2}\pi e^{-(\gamma+2)}e^{-\sqrt{2|\ln\varepsilon(T)|}}\sqrt{2|\ln\varepsilon(T)|}(1+O(\frac{1}{\sqrt{|\ln \varepsilon(T)|}})).
\end{align*} 
\noindent We remark that the nonlocal operator in \eqref{nonlocalIntro} has already been discovered in this context in \cite{DdPW}. Because of a cancellation of the main order terms, here finding the last term of \eqref{nonlocalIntro} is crucial and it requires a rather involved expansion of some integrals. As in \cite{DdPW} we will be only able to find an approximate solution of \eqref{nonlocalIntro} for a sufficiently small error but here the resolvent operator does not have the same smallness they used to implement a fixed point scheme. To find the desired approximate solution of \eqref{nonlocalIntro} in Section \ref{MASSPHISECTION} we will introduce a fixed-point scheme that is based on the fundamental observation that the resolvent operator is \emph{at main order} differentiable; \\
	\item a key ingredient of the inner-outer gluing scheme is finding estimates for the solution of an inner equation. In both finite and infinite time the prototype equation we get is 
	\begin{align}\label{prototypeinneqtINTRO}
		\lambda^{2}\partial_{t}\phi=L[\phi]+\lambda\dot{\lambda}(2\phi+y\cdot \nabla \phi)+h
	\end{align}
    where $L$ is the linearized operator around the correspondent approximate solution \eqref{finitetimeDyn}. After the time rescaling 
    \begin{align}\label{tauVARINTRO}
    		\tau:=\tau_{0}+\int_{0}^{t}\frac{1}{\lambda^{2}(s)}ds\implies\tau\to \infty \text{ as }t\to T
    \end{align}
    the equations in finite time and in infinite time look similar but there is a crucial difference:
    \begin{align*}
    	\lambda\dot{\lambda}(2\phi+y\cdot \nabla \phi)\approx\begin{cases}
    		-\frac{1}{2\tau \ln \tau}(2\phi+y\cdot \nabla \phi) \ \ \ \ \text{ in infinite time }\\[5pt]
    		-\frac{\ln \tau}{2\tau}(2\phi+y\cdot \nabla \phi) \ \ \ \ \ \text{in finite time}.
    	\end{cases}
    \end{align*}
    As we will explain in Sections \ref{ProofProp71}, \ref{Section11} and \ref{Section12} to take into account this difference three \emph{perturbed} inner theories will be required;\\
	\item in order to finally solve the resulting inner-outer gluing system in finite time a last improvement of the rate of the blow-up $\lambda(t)$ is required. In fact, in finite time the inner equation and the outer equation are coupled differently. This generates a loss in the time decay of the inner solution. 
\end{itemize}
\section{The approximations $u_{0}$ and $u_{1}$}
From now on we restrict ourselves to the case of a single point blow-up, namely $k=1$. In Section \ref{MultiSpikes} we will explain how we can easily extend the proof to the case $k>1$. \newline
In this section we will define a basic approximation to a solution of the Keller-Segel system \eqref{ks0}. Let us consider parameter functions
\begin{align*}
	0<\lambda(t)\to 0, \ \ \ \ \xi(t)\to q\in \mathbb{R}^{2}, \ \ \ \ \alpha(t)\to1 \ \ \  \text{as } t \to T
\end{align*}
that we will later specify. Let us introduce a smooth cut-off function $\chi$ such that $\chi(x,t)=\chi_{0}(\frac{x-\xi}{\sqrt{\delta(T-t)}})$ where $\chi_{0}$ is smooth, radial and satisfies
\begin{equation}\label{cutoff}
	\chi_{0}(s)=\begin{cases} 1 \ \ \ \ \ \text{if } s\le 1\\
		0 \ \ \ \ \  \text{if } s\ge2
	\end{cases}
\end{equation}
and $\delta$ is a small parameter that will be fixed in Section \ref{InnerTheoryIntro}. We consider the functions
\begin{align*}
	U(y)=\frac{8}{(1+|y|^{2})^{2}}, \ \ \ \ \Gamma_{0}=\log U(y)
\end{align*}
and we remark that this is a solution of the time independent equation
\begin{align}\label{timeindeq}
	\nabla \cdot(\nabla U-U\nabla\Gamma_{0})=\Delta_{y}U-\nabla_{y}U\cdot\nabla_{y}\Gamma_{0}+U^{2}=0, \ \ \ \ \ \ \  -\Delta\Gamma_{0}=U.
\end{align}
We define the approximate solution $u_{0}(x,t)$ as
\begin{align}\label{first approximation}
	&u_{0}(x,t)=\frac{\alpha}{\lambda^{2}}U(\frac{x-\xi}{\lambda})\chi(x,t) \\
	&v_{0}(x,t)=(-\Delta_{x})^{-1}u_{0}=\frac{1}{2\pi}\int_{\mathbb{R}^{2}}\log \frac{1}{|x-z|}u_{0}(z,t)dz. \nonumber
\end{align}
We consider the error operator 
\begin{align}\label{erroroperator}
	S(u)=-\partial_{t}u+\mathcal{E}(u)
\end{align}
where
\begin{align*}
	\mathcal{E}(u)=\Delta_{x}u-\nabla_{x}\cdot (u\nabla_{x}v), \ \ \ v=(-\Delta_{x})^{-1}u.
\end{align*}
We can compute the error of the first approximation \eqref{first approximation}. We have
\begin{align}\label{timederiU0}
	-\partial_{t}u_{0}(x,t)=&-\frac{\dot{\alpha}}{\lambda^{2}}U(y)\chi_{0}(w)+\alpha \frac{\dot{\lambda}}{\lambda^{3}}Z_{0}(y)\chi_{0}(w)+\frac{\alpha}{\lambda^{3}}\dot{\xi}\cdot \nabla_{y}U(y)\chi_{0}(w)+\nonumber \\ 
	&+\frac{\alpha}{\lambda^{2}\sqrt{\delta(T-t)}}U(y)\dot{\xi}\cdot \nabla_{w}\chi_{0}(w)-\frac{\alpha}{2\lambda^{2}(T-t)}U(y)\nabla_{w}\chi_{0}(w)\cdot w
\end{align}
where $y=\frac{x-\xi}{\lambda}$ and $w=\frac{x-\xi}{\sqrt{\delta (T-t)}}$ and 
\begin{align}\label{Z0}
	Z_{0}=2U(y)+y\cdot\nabla_{y}U(y).
\end{align}
We also have
\begin{align*}
	\mathcal{E}(u_{0})=&\Delta_{x}u_{0}-\nabla_{x} \cdot(u_{0}\nabla_{x}v_{0})=\\
	=&\frac{2\alpha}{\lambda^{3}\sqrt{\delta(T-t)}}\nabla_{w}\chi_{0}(w)\cdot \nabla_{y}U(y)+\frac{\alpha}{\delta(T-t)}\frac{1}{\lambda^{2}}\Delta_{w}\chi_{0}(w)U(y)-\\
	&-\frac{\alpha}{\lambda^{2}\sqrt{\delta(T-t)}}U(y)\nabla_{w}\chi_{0}(w)\cdot \nabla_{x}v_{0}+\frac{\alpha \chi_{0}(w)}{\lambda^{4}}\big[\Delta_{y}U-\nabla_{y}U\cdot\nabla_{y}v_{0}+\alpha \chi U^{2}\big]=
\end{align*}
\begin{align*}
	=&\frac{2\alpha}{\lambda^{3}\sqrt{\delta(T-t)}}\nabla_{z}\chi_{0}(w)\cdot \nabla_{y}U(y)+\frac{\alpha}{\delta(T-t)}\frac{1}{\lambda^{2}}\Delta_{w}\chi_{0}(w)U(y)-\\
	&-\frac{\alpha}{\lambda^{2}\sqrt{\delta(T-t)}}U(y)\nabla_{w}\chi_{0}(w)\cdot \nabla_{x}v_{0}+\\
	&+\frac{\alpha \chi_{0}(w)}{\lambda^{4}}\big[(\chi_{0}(w)\alpha-1)U^{2}(y)-\nabla_{y}U(y)\cdot (\nabla_{y}v_{0}-\nabla_{y}\Gamma_{0})\big]
\end{align*}
where in the last identity we used \eqref{timeindeq}. If we decompose
\begin{align}\label{decompV}
	v_{0}(y)=\alpha \Gamma_{0}+\mathcal{R}(y),
\end{align}
by writing $-\Delta_{x}(v_{0}-\alpha\Gamma_{0})=\frac{\alpha}{\lambda^{2}}U(y)(\chi-1)$ in radial coordinates we get
\begin{align}
	|\nabla_{y}\mathcal{R}(y)|\le \begin{cases}
		0  \ \ \ \ \text{if }|y|\le \frac{\sqrt{\delta(T-t)}}{\lambda}\\ 
		\frac{\alpha\lambda^{2}}{\delta(T-t)}\frac{1}{|y|} \ \ \ \ \text{if } |y|\ge\frac{\sqrt{\delta(T-t)}}{\lambda}.
	\end{cases}
\end{align} 
Then we have
\begin{align*}
	\mathcal{E}(u_{0})=&\frac{2\alpha}{\lambda^{3}\sqrt{\delta(T-t)}}\nabla_{w}\chi_{0}(w)\cdot \nabla_{y}U(y)+\frac{\alpha}{\delta(T-t)}\frac{1}{\lambda^{2}}\Delta_{w}\chi_{0}(w)U(y)-\\
	&-\frac{\alpha}{\lambda^{2}\sqrt{\delta(T-t)}}U(y)\nabla_{w}\chi_{0}(w)\cdot \nabla_{x}v_{0}+\\
	&+\frac{\alpha \chi_{0}(w)}{\lambda^{4}}\big[(\alpha-1)U^{2}-(\alpha-1)\nabla_{y}U\cdot \nabla_{y}\Gamma_{0}+\alpha(\chi_{0}(w)-1)U^{2}-\nabla_{y}U\cdot \nabla_{y}\mathcal{R}(y)\big]
\end{align*}
and this gives
\begin{align}\label{firstapprErr}
	S(u_{0})=&-\frac{\dot{\alpha}}{\lambda^{2}}U(y)\chi_{0}(w)+\alpha \frac{\dot{\lambda}}{\lambda^{3}}Z_{0}(y)\chi_{0}(w)+\frac{\alpha}{\lambda^{3}}\dot{\xi}\cdot \nabla_{y}U(y)\chi_{0}(w)+\nonumber \\ 
	&+\frac{\alpha}{\lambda^{2}\sqrt{\delta(T-t)}}U(y)\dot{\xi}\cdot \nabla_{w}\chi_{0}(w)-\frac{\alpha}{2\lambda^{2}(T-t)}U(y)\nabla_{w}\chi_{0}(w)\cdot w+ \nonumber\\
	&+\frac{2\alpha}{\lambda^{3}\sqrt{\delta(T-t)}}\nabla_{w}\chi_{0}(w)\cdot \nabla_{y}U(y)+\frac{\alpha}{\delta(T-t)}\frac{1}{\lambda^{2}}\Delta_{w}\chi_{0}(w)U(y)-\nonumber\\
	&-\frac{1}{\lambda^{2}\sqrt{\delta(T-t)}}U(y)\nabla_{w}\chi_{0}(w)\cdot \nabla_{x}{v_{0}}-\frac{\alpha(\alpha-1)\chi_{0}(w)}{\lambda^{4}}\nabla_{y}\cdot(U\nabla_{y}\Gamma_{0})+\nonumber\\
	&+\frac{\alpha \chi_{0}(w)}{\lambda^{4}}\big[\alpha(\chi_{0}(w)-1)U^{2}-\nabla_{y}U\cdot \nabla_{y}\mathcal{R}(y)\big]-\frac{\alpha-1}{\lambda^{2}\sqrt{\delta(T-t)}}U(y)\nabla_{w}\chi_{0}(w)\cdot \nabla_{x}v_{0}.
\end{align}
We proceed as in \cite{DdPDMW} and, in order to erase the main order term in \eqref{firstapprErr}, we introduce the operator 
\begin{align}\label{SixDimLap}
	\Delta_{6}v(x)=\Delta v(x)+4\frac{x}{|x|^{2}}\cdot\nabla v(x), \ \ \ \ x\in\mathbb{R}^{2}
\end{align}
where the notation is justified by the fact that when $v$ is radial this operator coincide with the Laplace operator in $\mathbb{R}^{6}$, that is $\Delta_{6}=\partial_{r}^{2}+\frac{5}{r}\partial_{r}$ if $r=|x|$. We introduce the radial correction $\varphi_{\lambda}(x,t)$ that solves
\begin{align}\label{eqphilambda}
	\begin{cases}
		\partial_{t}\varphi_{\lambda}=\Delta_{6}\varphi_{\lambda}+E(x,t) \ \ \ \text{in }\mathbb{R}^{2}\times (0,T)\\
		\varphi_{\lambda}(\cdot,-\varepsilon(T))=0,
	\end{cases}
\end{align}
by means of Duhamel's formula, where
\begin{align}\label{Ephilambda}
	E(x,t;\lambda)= \frac{\dot{\lambda}}{\lambda^{3}}Z_{0}(\bar{y})\chi_{0}(\bar{w})-\frac{1}{2\lambda^{2}(T-t)}U(\bar{y})\nabla_{\bar{w}}\chi_{0}(\bar{w})\cdot \bar{w}+\tilde{E}(x,t;\lambda),
\end{align}
\begin{align}\label{Etildephilambda}
	\tilde{E}(x,t;\lambda)=&\frac{2}{\lambda^{3}\sqrt{\delta(T-t)}}\nabla_{\bar{w}}\chi_{0}(\bar{w})\cdot \nabla_{\bar{y}}U(\bar{y})+\frac{1}{\delta(T-t)}\frac{1}{\lambda^{2}}\Delta_{\bar{w}}\chi_{0}(\bar{w})U(\bar{y})-\nonumber\\
	&-\frac{1}{\alpha}\frac{1}{\lambda^{2}\sqrt{\delta(T-t)}}U(\bar{y})\nabla_{\bar{w}}\chi_{0}(\bar{w})\cdot \nabla_{x}v_{0}
\end{align}
with
\begin{align}\label{radialWY}
	\bar{w}=\frac{x}{\sqrt{\delta(T-t)}}, \ \ \ \  \bar{y}=\frac{x}{\lambda}
\end{align}
and $\varepsilon(T)$ is an explicit \emph{small} parameter that, as we will be clear in Section \ref{MASSPHISECTION}, will satisfy $T\ll\varepsilon(T)\ll1$. This parameter will play a crucial role in the expansion of the mass of the solution. \newline We notice that \eqref{eqphilambda} differs from the correction introduced in \cite{DdPDMW} in some respects. They consider $\varphi_{\lambda}(x-\xi(t),t)$, introducing in the error the term $\nabla \varphi_{\lambda}\cdot \dot{\xi}$ that in finite time would be dangerously singular away from the the point where the singularity occurs. This modification will produce errors that are supported only close to the singularity. As we will explain in Section \ref{MASSPHISECTION} we need to define the problem for negative times to give meaning to a decomposition of the mass of $\varphi_{\lambda}$, see formula \eqref{ExpansionMassphilambda}. Consequently, when later we will write $\lambda=\lambda_{0}+\lambda_{1}$ both these functions will be defined in $(-\varepsilon(T),T)$. Another fundamental fact we want to underline is that the mass of $\varphi_{\lambda}$ tends to some constant as $t\to T$ as we will precisely state in Corollary \ref{Corollary EXPANSION MASS}. \newline
We introduce the approximate solution
\begin{align}\label{u1}
	u_{1}:=u_{0}+\varphi_{\lambda}
\end{align}
which depends on the parameter functions $\alpha(t)$, $\xi(t)$ and $\lambda(t)$. Correspondingly, we write
\begin{align*}
	v_{1}=(-\Delta_{x})^{-1}(u_{1}).
\end{align*}
We will show that by properly choosing the parameter functions we are able to find a solution of \eqref{ks0} as a lower order perturbation of $u_{1}$. 
\section{The first error of approximation}
In this section we will always assume that $\alpha(t)\to 1$, $\xi(t)\to 0$ are $C^{1}$ functions in the interval $(0,T)$ and, for reasons that will be clear in Section \ref{MASSPHISECTION}, $\lambda\dot{\lambda}(t)$ is a $C^{1}$ function in $(-\varepsilon(T),T)$. Moreover we assume that they satisfy the following estimates
\begin{align}\label{EstiLambda}
		c_{1}\sqrt{T-t}e^{-\sqrt{\frac{|\ln(T-t)|}{2}}}\le |\lambda(t)|\le c_{2}\sqrt{T-t}e^{-\sqrt{\frac{|\ln(T-t)|}{2}}} \text{for some } c_{1},c_{2}>0
\end{align}
\begin{align}\label{EstiPar}
	\begin{cases}
		|\lambda\dot{\lambda}(t)|+(T-t)\sqrt{|\ln(T-t)|}|\frac{d}{dt}(\lambda\dot{\lambda})(t)|\le C e^{-\sqrt{\frac{|\ln(T-t)|}{2}}},\\[7pt]
		|\dot{\xi}|\le C \frac{e^{-(\frac{3}{2}+\gamma_{1}) \sqrt{2|\ln(T-t)|}}}{\sqrt{T-t}} \ \ \ \ \ \text{for some }\gamma_{1}>0,\\[7pt]
		|\dot{\alpha}|\le C \frac{e^{-(\frac{3}{2}-\gamma_{2}) \sqrt{2|\ln(T-t)|}}}{T-t} \ \ \ \ \text{for any }\gamma_{2}>0.
	\end{cases}
\end{align}
In Section \ref{proofThm1} we will prove that \eqref{EstiLambda}, \eqref{EstiPar} are indeed satisfied.
By looking at \eqref{ExpansionMassphilambda}, it is also clear why it is so convenient to work with $\lambda\dot{\lambda}(t)$ instead of directly with $\lambda(t)$. \newline
We compute
\begin{align*}
	S(u_{1})=S(u_{0}+\varphi_{\lambda})=S(u_{0})-\partial_{t}\varphi_{\lambda}+\mathcal{L}_{u_{0}}[\varphi_{\lambda}]-\nabla \cdot(\varphi_{\lambda}\nabla \psi_{\lambda})
\end{align*}
where
\begin{align*}
	&\mathcal{L}_{u_{0}}[\varphi]=\Delta \varphi-\nabla \cdot(\varphi \nabla v_{0})-\nabla \cdot(u_{0}\nabla \psi), \ \ \ \ \psi_{\lambda}=(-\Delta)^{-1}\varphi_{\lambda}, \ \ v_{0}=(-\Delta)^{-1}u_{0}.
\end{align*} 
We observe that if $r=|x|$ we have
\begin{align}\label{secapprErr}
	S(u_{1})=&-\frac{\dot{\alpha}}{\lambda^{2}}U(y)\chi_{0}(w)+(\alpha-1) \frac{\dot{\lambda}}{\lambda^{3}}Z_{0}(y)\chi_{0}(w)+\frac{\alpha}{\lambda^{3}}\dot{\xi}\cdot \nabla_{y}U(y)\chi_{0}(w)+\nonumber \\ 
	&+\frac{\alpha}{\lambda^{2}\sqrt{\delta(T-t)}}U(y)\dot{\xi}\cdot \nabla_{w}\chi_{0}(w)-\frac{\alpha-1}{2\lambda^{2}(T-t)}U(y)\nabla_{w}\chi_{0}(w)\cdot w+ \nonumber\\
	&+\frac{2(\alpha-1)}{\lambda^{3}\sqrt{\delta(T-t)}}\nabla_{w}\chi_{0}(w)\cdot \nabla_{y}U(y)+\frac{\alpha-1}{\delta(T-t)}\frac{1}{\lambda^{2}}\Delta_{w}\chi_{0}(w)U(y)-\nonumber\\
	&-\frac{1-1/\alpha}{\lambda^{2}\sqrt{\delta(T-t)}}U(y)\nabla_{w}\chi_{0}(w)\cdot \nabla_{x}v_{0}-\frac{\alpha(\alpha-1)\chi_{0}(w)}{\lambda^{4}}\nabla_{y}\cdot(U\nabla_{y}\Gamma_{0})+\nonumber\\
	&+\frac{\alpha \chi_{0}(w)}{\lambda^{4}}\big[\alpha(\chi_{0}(w)-1)U^{2}-\nabla_{y}U\cdot \nabla_{y}\mathcal{R}(y)\big]-\frac{\alpha-1}{\lambda^{2}\sqrt{\delta(T-t)}}U(y)\nabla_{w}\chi_{0}(w)\cdot \nabla_{x}v_{0}\nonumber\\
	&-\frac{4}{r}\partial_{r}\varphi_{\lambda}-\nabla \cdot(\varphi_{\lambda}\nabla v_{0})-\nabla \cdot(u_{0}\nabla \psi_{\lambda})-\nabla \cdot(\varphi_{\lambda}\nabla \psi_{\lambda})+E(x-\xi,t)-E(x,t).
\end{align}
wher $\mathcal{R}$ is defined in the decomposition \eqref{decompV} and  $E$ in \eqref{Ephilambda}. \newline Now we find some preliminar estimates for $\varphi_{\lambda}$.
\begin{lemma}\label{estimatephilambda}
	Take $T$, $\varepsilon(T)$ sufficiently small and let $\varphi_{\lambda}$ be defined by \eqref{eqphilambda} with $\lambda$ satisfying \eqref{EstiLambda}, \eqref{EstiPar}. For any $(x,t)\in \mathbb{R}^{2}\times (-\varepsilon(T),T)$ we have
	\begin{align}\label{boundphilambda}
		|\varphi_{\lambda}(x,t)|+(|x|+\lambda)|\nabla \varphi_{\lambda}(x,t)|\le C \begin{cases}
			\frac{e^{-\sqrt{2|\ln(T-t)|}}}{\lambda^{2}+|x|^{2}} \ \ \ \ |x|\le \sqrt{T-t},\\[7pt]
		  \frac{e^{-\sqrt{2|\ln (|x|^{2})|}}}{|x|^{2}} \ \ \ \ \ \sqrt{T-t}\le |x|\le \sqrt{\varepsilon(T)},\\[7pt]
		   \frac{e^{-\sqrt{2|\ln (\varepsilon(T))|}}}{\varepsilon(T)}e^{-\frac{|x|^{2}}{4(t+2\varepsilon(T))}} \ \ \ \ \ |x|\ge\sqrt{\varepsilon(T)}
		\end{cases}
	\end{align}
and, additionally 
\begin{align}\label{betinnerphilambda}
	|\nabla \varphi_{\lambda}(x,t)|+(|x|+\lambda)|D^{2} \varphi_{\lambda}|\le C e^{-\sqrt{2|\ln(T-t)|}}\frac{|x|}{(\lambda+|x|^{2})^{2}}, \ \ \ |x|\le \sqrt{T-t}.
\end{align}
\end{lemma}
\begin{proof}
First we observe that the right-hand side of \eqref{eqphilambda} satisfies
\begin{align*}
	|E(x,t;\lambda)|\le C e^{-\sqrt{2|\ln(T-t)|}}\frac{1}{(\lambda^{2}+|x|^{2})^{2}}\tilde{\chi}
\end{align*}
where $\tilde{\chi}=\chi_{0}(\frac{|x|}{2\sqrt{\delta(T-t)}})$. To find an estimate for $\varphi_{\lambda}$ we use barriers. In what follows we construct a radial, positive and decaying in space $\varphi$ such that
\begin{align}\label{defBarphilambda}
	(\partial_{t}-\Delta_{6})\varphi\ge C e^{-\sqrt{2|\ln(T-t)|}}\frac{1}{(\lambda^{2}(t)+|x|^{2})^{2}}\tilde{\chi} \ \ \ \text{if }(x,t)\in \mathbb{R}^{2}\times(-\varepsilon(T),T).
\end{align} 
By comparison principle \eqref{defBarphilambda} implies $|\varphi_{\lambda}|\le \varphi$.
We observe that in order to get the following inequalities we will use multiple times that $T$ and $\varepsilon(T)$ are sufficiently small. \newline
Take $c_{0}>0$, the first piece of the barrier will be $\hat{\varphi}_{0}=c_{0}\frac{e^{-\sqrt{2|\ln(T-t)}}}{\lambda^{2}+|x|^{2}}$ that solves
\begin{align*}
	(\partial_{t}-\Delta_{6})\hat{\varphi}_{0}\ge C c_{0} e^{-\sqrt{2|\ln(T-t)|}}\frac{1}{(\lambda^{2}(t)+|x|^{2})^{2}} \ \ \ \ \text{if } |x|\le \sqrt{T-t}
\end{align*}
for some universal constant $C$.
If we take $\varphi_{0}=\hat{\varphi}_{0}\chi_{0}(\frac{|x|}{\sqrt{T-t}})$ we see that
\begin{align}\label{firstpiecBar}
 (\partial_{t}-\Delta_{6})\varphi_{0}\ge C c_{0} e^{-\sqrt{2|\ln(T-t)|}}\frac{1}{(\lambda^{2}+|x|^{2})^{2}}\chi_{0}(\frac{|x|}{\sqrt{T-t}})+ O(\frac{e^{-\sqrt{2|\ln(T-t)|}}}{(T-t)^{2}}\mathbbm{1}_{\{0\le |x|\le 2\sqrt{T-t}\}})
\end{align}
where $\mathbbm{1}$ is the indicator function and by $O(f(x))$ we mean that $O(f(x))\le C |f(x)|$. The second piece of the barrier will erase the remainders in \eqref{firstpiecBar}. The candidate function is $\hat{\varphi}_{1}=c_{1}\frac{e^{-\sqrt{2|\ln(|x|^{2}+(T-t))|}}}{(T-t)+|x|^{2}}$ and indeed we have
\begin{align*}
	(\partial_{t}-\Delta_{6})\hat{\varphi}_{1}\ge C c_{1} \frac{e^{-\sqrt{2|\ln(|x|^{2}+(T-t))|}}}{((T-t)+|x|^{2})^{2}} \ \ \ \ \text{if } |x|\le \sqrt{\varepsilon(T)}.
\end{align*} 
By properly taking $c_{1}$, if we define $\varphi_{1}=\hat{\varphi}_{1}\chi_{0}(\frac{|x|}{\sqrt{\varepsilon(T)}})$, we get 
\begin{align}\label{secpiecBar}
	(\partial_{t}-\Delta_{6})(\varphi_{0}+\varphi_{1})\ge C c_{0} e^{-\sqrt{2|\ln(T-t)|}}\frac{1}{(\lambda^{2}+|x|^{2})^{2}}\chi_{0}(\frac{|x|}{\sqrt{T-t}})+O(\frac{e^{-\sqrt{2|\ln \varepsilon(T)|}}}{\varepsilon(T)^{2}}\mathbbm{1}_{\{0\le |x|\le2\sqrt{\varepsilon(T)}\}}).
\end{align}
To erase the remainders in \eqref{secpiecBar} we introduce $\varphi_{2}=c_{2}\frac{e^{-\sqrt{2|\ln (\varepsilon(T))|}}}{\varepsilon(T)}e^{- \frac{|x|^{2}}{4(t+2\varepsilon(T))}}$, if $t>-\varepsilon(T)$ (here we use that $\varepsilon(T)\gg T$) we have
\begin{align*}
	(\partial_{t}-\Delta_{6})\varphi_{2}\ge Cc_{2} \frac{e^{-\sqrt{2|\ln (\varepsilon(T))|}}}{\varepsilon(T)^{2}} e^{-\frac{|x|^{2}}{4(t+2\varepsilon(T))}}.
\end{align*} 
By properly choosing $c_{0}$ and $c_{2}$ we see that $\varphi=\varphi_{0}+\varphi_{1}+\varphi_{3}$ satisfies \eqref{defBarphilambda}. Finally we observe that if $|x|\le \sqrt{T-t}$ we have $\varphi_{1},\varphi_{2}\le \varphi_{0}$, and that if $|x|\le \sqrt{\varepsilon(T)}$ we have $\varphi_{1}\le \varphi_{2}$. \newline Now we are interested in estimating the gradient. After writing $\bar{\varphi}_{\lambda}(|x|,t)=\frac{1}{\lambda^{2}}\hat{\phi}(\frac{|x|}{\lambda},t)$ and if $y=|x|/\lambda$ we observe that $\hat{\phi}$ must solve
\begin{align}\label{rescaledphilambda}
	\lambda^{2}\partial_{t}\hat{\phi}=\Delta_{6,y}\hat{\phi}+\lambda\dot{\lambda}(2\hat{\phi}+y\cdot \nabla \hat{\phi} )+\lambda^{4}E(\lambda y,t),
\end{align}
where $\Delta_{6,y}=\partial^{2}_{|y|}+\frac{5}{|y|}\partial_{|y|}$ is the rescaled operator. 
Then if we adequately change the time scale, by the previous estimate we found for $\varphi_{\lambda}$ and standard parabolic estimates we get the bound for the gradient in \eqref{boundphilambda}.
\newline In order to get \eqref{betinnerphilambda}, we formally differentiate in $y$ \eqref{rescaledphilambda}. Using the bound for the gradient we obtained before and standard parabolic estimates we get
\begin{align*}
	|D^{2}_{y}\hat{\phi}(y,t)|\le C \frac{e^{-\sqrt{2|\ln(T-t)|}}}{\lambda^{4}}\frac{1}{1+|y|^{4}}.
\end{align*}
Then, since by radial symmetry $\nabla_{y} \hat{\phi}(0,t)=0$, after integrating we finally get \eqref{betinnerphilambda}.
\end{proof}
Thanks to the additional regularity of $\lambda\dot{\lambda}(t)$ we have from \eqref{EstiPar} we can obtain some estimates for $\partial_{t}\varphi_{\lambda}$. These estimates will allow us to control second moment of the error away from the singularity and consequently to drastically simplify the proof of Corollary \ref{secmomconditionELLI}.
\begin{lemma}\label{Derestimatephilambda}
	Take $T$, $\varepsilon(T)$ sufficiently small and let $\varphi_{\lambda}$ be defined by \eqref{eqphilambda} with $\lambda$ satisfying \eqref{EstiPar}. Then for any $(x,t)\in \mathbb{R}^{2}\times (-\varepsilon(T),T)$ we have
	\begin{align}\label{Derboundphilambda}
		|\partial_{t}\varphi_{\lambda}(x,t)|+(|x|+\lambda)|\nabla \partial_{t}\varphi_{\lambda}(x,t)|\le C \begin{cases}
			e^{-\sqrt{2|\ln(T-t)|}}\frac{\ln(2+\frac{|x|^{2}}{\lambda^{2}})}{(\lambda^{2}+|x|^{2})^{2}} \ \ \ \ |x|\le \sqrt{T-t},\\[7pt]
			e^{-\sqrt{2|\ln |x|^{2}|}}\frac{\sqrt{|\ln |x|^{2}|}}{|x|^{4}} \ \ \ \ \ \sqrt{T-t}\le |x|\le\sqrt{\varepsilon(T)},\\[7pt]
			e^{-\sqrt{2|\ln \varepsilon(T)|}}\frac{\sqrt{|\ln \varepsilon(T)|}}{\varepsilon(T)^{2}}e^{-\frac{|x|^{2}}{4(t+2\varepsilon(T))}} \ \ \ \ \ |x|\ge\sqrt{\varepsilon(T)}.
		\end{cases}
	\end{align}
\end{lemma}
\begin{proof}
	We start by formally differentiating in time \eqref{eqphilambda} (the initial condition can be easily deduced from Duhamel's formula). We observe that, for example,
	\begin{align*}
		|\partial_{t}(\frac{\dot{\lambda}}{\lambda^{3}}Z_{0}\chi)|\le\frac{1}{\sqrt{|\ln(T-t)|}} \frac{e^{-\sqrt{2|\ln(T-t)|}}}{(\lambda^{2}+|x|^{2})^{3}}+\frac{e^{-2\sqrt{|\ln(T-t)|}}}{(\lambda^{2}+|x|^{2})^{3}}
	\end{align*} 
    and that we have similar estimates for the other terms in the right-hand side of \eqref{eqphilambda}. We need to take a constant $k$ sufficiently large.
    The first piece of the barrier we consider is $\hat{\varphi}_{0}=c_{0}e^{-\sqrt{2|\ln(T-t)|}}\frac{\ln(k+\frac{|x|}{\lambda})}{(\lambda^{2}+|x|^{2})^{2}}$, that satisfies
\begin{align}\label{choicek}
	(\partial_{t}-\Delta_{6})\hat{\varphi}_{0}\ge C c_{0} \frac{e^{-\sqrt{2|\ln(T-t)|}}}{(\lambda^{2}+|x|^{2})^{3}} \ \ \ \ \text{if } |x|\le \epsilon\sqrt{T-t}
\end{align} 
where $\epsilon>0$ is a small constant.  We observe that the choice of a sufficiently large $k$ allows us to control the left-hand side of \eqref{choicek} when $|x|$ is small. 
As in the proof of Lemma \ref{estimatephilambda} we take $\varphi_{0}=\hat{\varphi}_{0}\chi_{0}(\frac{|x|}{\epsilon\sqrt{T-t}})$ and we get
\begin{align}\label{firstpiecBarDER}
	(\partial_{t}-\Delta_{6})\varphi_{0}\ge& C c_{0} e^{-\sqrt{2|\ln(T-t)|}}\frac{1}{(\lambda^{2}+|x|^{2})^{3}}\chi_{0}(\frac{|x|}{\epsilon\sqrt{T-t}})+\nonumber\\
	&+O(\frac{e^{-\sqrt{2|\ln(T-t)|}}\sqrt{|\ln(T-t)|}}{(T-t)^{3}}\mathbbm{1}_{\{0\le |x|\le2 \sqrt{T-t}\}})
\end{align}
where we observe that to erase the remainders in \eqref{firstpiecBarDER} it is not necessary to keep track of the dependence on the parameters $\epsilon$ and $k$. The next term in the barrier is $\hat{\varphi}_{1}=c_{1}e^{-\sqrt{2|\ln(T-t+|x|^{2})|}}\frac{\sqrt{|\ln (T-t+|x|^{2})}}{(|x|^{2}+(T-t))^{2}}$ and we can take $\varphi_{1}=\hat{\phi}_{1}\chi_{0}(\frac{|x|}{\sqrt{\varepsilon(T)}})$.
We choose the remaining piece of the barrier as in the proof of Lemma \ref{estimatephilambda}. Thanks to the estimate for $\partial_{t}\varphi_{\lambda}$ we just proved, by standard parabolic estimates and absorbing the constant $k$ by taking $C$ sufficiently large we finally get \eqref{Derboundphilambda}.
\end{proof}
\begin{lemma}\label{lemmaesterr}
	Assuming \eqref{EstiLambda}, \eqref{EstiPar} for any $(x,t)\in \mathbb{R}^{2}\times(0,T) $ we have
	\begin{align}\label{innererr}
	     |\lambda^{4}S(u_{1})\chi(x,t)|\le C e^{-\sqrt{2|\ln(T-t)|}}\frac{\log(2+|y|)}{1+|y|^{6}}, \ \ \ y=\frac{x-\xi}{\lambda},
	\end{align}
if $|x-\xi|\le 2\sqrt{\delta(T-t)}$,
 \begin{align}\label{outererrro}
	|S(u_{1})(1-\chi)|\le C\begin{cases}
		\frac{e^{-2\sqrt{2|\ln |x|^{2}|}}}{|x|^{4}}\sqrt{|\ln |x|^{2} } \ \ \ \text{if } \sqrt{T-t}\le|x-\xi|\le \sqrt{\varepsilon(T)}\\[7pt]
		\frac{e^{-2\sqrt{2|\ln \varepsilon(T)|}}}{\varepsilon(T)^{2}}\sqrt{|\ln \varepsilon(T)|}e^{-\frac{|x|^{2}}{4(t+2\varepsilon(T))}} \ \ \ \ \ \text{if }|x-\xi|> \sqrt{\varepsilon(T)}.
	\end{cases}
\end{align}
\end{lemma}
\begin{proof}
	We will first prove \eqref{innererr}. By Lemma \ref{estimatephilambda} we have 
	\begin{align*}
		|\lambda^{2}U(y)\varphi_{\lambda}(\lambda y +\xi) |\le \frac{e^{-\sqrt{2|\ln(T-t)|}}}{1+|y|^{6}} \ \ \ \ \ \  \text{if }\frac{|x-\xi|}{\lambda}\le 2\frac{\sqrt{\delta(T-t)}}{\lambda}.
	\end{align*}
Recalling that $r=|x|$, we see by \eqref{decompV} that
\begin{align*}
	-\frac{4}{r}\partial_{r}\varphi_{\lambda}-\nabla \varphi_{\lambda}\cdot \nabla v_{0}&=-\frac{4}{r}\partial_{r}\varphi_{\lambda}-\nabla \varphi_{\lambda}\cdot \nabla \Gamma_{0}-(\alpha-1)\nabla \varphi_{\lambda}\cdot \nabla \Gamma_{0}-\nabla \varphi_{\lambda}\cdot \nabla \mathcal{R}=\\
	&=4(-\frac{1}{4}\nabla\Gamma_{0}\cdot\frac{x}{r}-\frac{1}{r})\partial_{r}\varphi_{\lambda}-(\alpha-1)\nabla \varphi_{\lambda}\cdot \nabla \Gamma_{0}-\nabla \varphi_{\lambda}\cdot \nabla \mathcal{R}=\\
	&=4(\frac{1}{|x|}\frac{(x-\xi)\cdot x}{\lambda^{2}+|x-\xi|^{2}}-\frac{1}{|x|})\partial_{r}\varphi_{\lambda}-(\alpha-1)\nabla \varphi_{\lambda}\cdot \nabla \Gamma_{0}- \nabla \varphi_{\lambda}\cdot \nabla \mathcal{R}=\\
	&=\frac{4}{|x|}\big(-\frac{\lambda^{2}}{\lambda^{2}+|x-\xi|^{2}}+\frac{(x-\xi)\cdot \xi}{\lambda^{2}+|x-\xi|^{2}}\big)\partial_{r}\varphi_{\lambda}-(\alpha-1)\nabla \varphi_{\lambda}\cdot \nabla \Gamma_{0}- \nabla \varphi_{\lambda}\cdot \nabla \mathcal{R}\le \\
	&\le \frac{1}{\lambda^{4}}\frac{e^{-\sqrt{2|\ln(T-t)|}}}{1+|y|^{6}} \ \ \ \text{if }\frac{|x-\xi|}{\lambda}\le2 \frac{\sqrt{\delta(T-t)}}{\lambda}
\end{align*}
where since $|x|$ is small we used \eqref{betinnerphilambda} and we used that the smallness assumption on $\xi$ that comes from \eqref{EstiPar}.
The logarithm in \eqref{innererr} comes from the terms involving $\nabla\psi_{\lambda}$. In fact they can be estimated recalling the estimate for $\varphi_{\lambda}$ we have in \eqref{boundphilambda} and observing that since everything is radial we have
\begin{align}\label{psilambda}
	\partial_{r}\psi_{\lambda}=-\frac{1}{r}\int_{0}^{r}\varphi_{\lambda}(s,t)sds.
\end{align}
To estimate $\lambda^{4}(E(x-\xi,t)-E(x,t))$ it is sufficient to observe that when $|y|\le2 \frac{|\xi|}{\lambda}\ll1$ (thanks to \eqref{EstiLambda},\eqref{EstiPar}) we can get any decay in space . If instead $|y|\ge 2\frac{|\xi|}{\lambda}$, recalling the notation \eqref{radialWY} and again \eqref{EstiLambda}, \eqref{EstiPar}, we can see that $\frac{x-\xi}{\lambda}=\bar{y}-\frac{\xi}{\lambda}\approx \bar{y}$ and finally obtain the desired decay. Noticing that the other terms in \eqref{secapprErr} can be estimated similarly, we get \eqref{innererr}.
\newline To get \eqref{outererrro} we only need to consider the following terms
\begin{align*}
	(1-\chi)[-\frac{4}{r}\partial_{r}\varphi_{\lambda}-\nabla \varphi_{\lambda}\cdot \nabla v_{0}-\nabla \cdot ( \varphi_{\lambda}\nabla \psi_{\lambda})].
\end{align*}
First we observe that by the same expansion we made above, we have that if $\sqrt{T-t}\le|x-\xi|\le\sqrt{\varepsilon(T)}$
\begin{align*}
	|-\frac{4}{r}\partial_{r} \varphi_{\lambda}-\nabla \varphi_{\lambda}\cdot \nabla v_{0}| \le \frac{e^{-4a\sqrt{|\ln(|x|^{2})}}}{|x|^{4}}.
\end{align*}
Among the remaining terms, the most delicate ones involve $ \nabla \psi_{\lambda}$. Formula \eqref{psilambda} and the estimate \eqref{boundphilambda} give 
\begin{align*}
	|\partial_{r}\psi_{\lambda}|\le& \frac{1}{r} \int_{0}^{r}|\varphi_{\lambda}(s,t)|sds\le \frac{1}{r}[\int_{0}^{\sqrt{T-t}}|\varphi_{\lambda}(s,t)|sds+\int_{\sqrt{T-t}}^{r}|\varphi_{\lambda}(s,t)|sds]\le \\
	\le &\frac{1}{r}[e^{-\sqrt{2|\ln(T-t)|}}|\ln (\frac{T-t}{\lambda^{2}})|+e^{-\sqrt{2|\ln r^{2}|}}\sqrt{|\ln r^{2}|}]\le \frac{e^{-\sqrt{2|\ln r^{2}|}}}{r}\sqrt{|\ln r^{2}|}.
\end{align*}
If $|x-\xi|\ge \sqrt{\varepsilon(T)}$ everything can be estimated analogously but $\nabla \varphi_{\lambda}\cdot \nabla \psi_{\lambda}$. Indeed by \eqref{boundphilambda}, \eqref{psilambda} we get
\begin{align*}
	|\partial_{r}\psi_{\lambda}|\le C \frac{1}{r}\int_{\mathbb{R}^{2}} |\varphi_{\lambda}|dx\le \frac{1}{r}e^{-\sqrt{2|\ln(\varepsilon(T))|}}\sqrt{|\ln \varepsilon(T)|}.
\end{align*}
\end{proof}
\section{The inner-outer gluing system}\label{InnerOuterSection}
We look for a solution of the Keller-Segel system \eqref{ks0} in the form of a small perturbation of \eqref{u1}:
\begin{align}\label{solutionexpansion}
     u(x,t)=u_{1}(x,t)+\Phi(x,t), \ \ \ \ (x,t)\in \mathbb{R}^{2}\times(0,T).
\end{align}
We write the perturbation $\Phi$ as a sum of ``inner" and ``outer" terms. The use of this terminology comes from the rescaling we need to define these corrections. More precisely, we write
\begin{align}\label{perturb}
	\Phi(x,t)=\frac{1}{\lambda^{2}}\phi^{i}(x,t)\chi(x,t)+\varphi^{o}(x,t), \ \ \ y=\frac{x-\xi}{\lambda}
\end{align}
where $\chi=\chi_{0}(\frac{|x-\xi|}{\sqrt{\delta(T-t)}})$ was introduced in \eqref{first approximation}. Recalling the definition of the error operator \eqref{erroroperator} we want to find $\Phi$ such that
\begin{align}\label{finaleqt}
	S(u_{1}+\Phi)=0.
\end{align}
As it is immediately clear by looking at \eqref{perturb} by expanding \eqref{finaleqt} we will obtain many terms multiplying $\chi$ and many other terms multiplying derivatives of $\chi$. This section is devoted to make this distinction clear and to define a parabolic system in the unknowns $\phi^{i}$, $\varphi^{o}$. Solving this system will require to choose the parameters $\alpha$, $\lambda$ and $\xi$. The solution of this system in particular will solve \eqref{finaleqt} and then will give us a solution of \eqref{ks0}.
\newline
Let us observe that
\begin{align*}
	S(u_{1}+\Phi)=&S(u_{1})-\partial_{t}\big(\frac{1}{\lambda^{2}}\partial_{t}\phi^{i}\chi\big)-\partial_{t}\varphi^{o}+\mathcal{L}_{u_{1}}\big[\frac{1}{\lambda^{2}}\phi^{i}\chi\big]+\mathcal{L}_{u_{1}}[\varphi^{o}]-\nabla \cdot ( \Phi \nabla (-\Delta)^{-1}\Phi)=\\
	=&S(u_{1})\chi+S(u_{1})(1-\chi)-\partial_{t}\big(\frac{1}{\lambda^{2}}\partial_{t}\phi^{i}\chi\big)-\partial_{t}\varphi^{o}+\mathcal{L}_{u_{1}}\big[\frac{1}{\lambda^{2}}\phi^{i}\chi\big]+\mathcal{L}_{u_{1}}[\varphi^{o}]-\\
	&-\nabla \cdot ( \Phi \nabla (-\Delta)^{-1}\Phi)\chi+\nabla \cdot ( \Phi \nabla (-\Delta)^{-1}\Phi)(1-\chi),
\end{align*}
where
\begin{align*}
	\mathcal{L}_{u_{1}}[\phi]=\Delta \phi-\nabla \cdot (\phi \nabla v_{1})-\nabla \cdot (u_{1}\nabla(-\Delta)^{-1}\varphi), \ \ \ \ v_{1}=(-\Delta)^{-1}u_{1}.
\end{align*}
We use the notation 
\begin{align*}
	\psi=\frac{1}{\lambda^{2}}(-\Delta)^{-1}\phi^{i}, \ \ \ \hat{\psi}=\frac{1}{\lambda^{2}}(-\Delta)^{-1}(\phi^{i}\chi).
\end{align*}
We expand
\begin{align*}
	\mathcal{L}_{u_{1}}[\frac{1}{\lambda^{2}}\phi \chi]=\chi\frac{1}{\lambda^{2}}\Delta \phi^{i}+\frac{2}{\lambda^{2}}\nabla \chi \cdot \nabla \phi^{i}+\frac{1}{\lambda^{2}}\phi \Delta \chi - \nabla \cdot ( \frac{1}{\lambda^{2}}\phi^{i}\chi \nabla v_{1})-\nabla \cdot (u_{1}\nabla \hat{\psi}).
\end{align*}
We have 
\begin{align*}
	\nabla \cdot(u_{1}\nabla \hat{\psi})=&\nabla \cdot ( \frac{\alpha}{\lambda^{2}}U\nabla \psi)\chi+\nabla \cdot(\frac{\alpha}{\lambda^{2}}U \nabla(\hat{\psi}-\psi))\chi+\frac{\alpha}{\lambda^{2}}U \nabla \chi \cdot \nabla \hat{\psi}+\\
	&+\nabla \cdot(\varphi_{\lambda}\nabla \psi)+\nabla \cdot ( \varphi_{\lambda}(\hat{\psi}-\psi))
\end{align*}
and 
\begin{align*}
	\nabla \cdot ( \frac{1}{\lambda^{2}}\phi^{i} \chi \nabla v_{1})=\nabla \cdot(\frac{1}{\lambda^{2}}\phi^{i} \nabla v_{1})\chi+\frac{1}{\lambda^{2}}\phi^{i}\nabla \chi \cdot \nabla v_{1}.
\end{align*}
We recall the notation 
\begin{align*}
	v_{1}=v_{0}+\psi_{\lambda}, \ \ \ \text{where }v_{0}=\frac{\alpha}{\lambda^{2}}(-\Delta)^{-1}(U\chi), \ \ \ \ \psi_{\lambda}=(-\Delta)^{-1}\varphi_{\lambda}
\end{align*}
and \eqref{decompV}
\begin{align*}
	v_{0}=\alpha \Gamma_{0}+\mathcal{R}, \ \ \ \mathcal{R}=\frac{\alpha}{\lambda^{2}}(-\Delta)^{-1}(U(\chi-1)).
\end{align*}
Then
\begin{align*}
	\nabla \cdot (\frac{1}{\lambda^{2}}\phi^{i}\chi \nabla v_{1})=&\nabla \cdot(\frac{1}{\lambda^{2}}\phi^{i}\nabla v_{0})\chi + \nabla \cdot ( \frac{1}{\lambda^{2}}\phi^{i}\nabla \psi_{\lambda})\chi+\frac{1}{\lambda^{2}}\phi^{i} \nabla \chi \cdot \nabla v_{0}+\frac{1}{\lambda^{2}}\phi^{i}\nabla \chi \cdot \nabla \psi_{\lambda}=\\
	=&\frac{\alpha}{\lambda^{2}}\nabla \cdot ( \phi^{i}\nabla \Gamma_{0})\chi+\nabla\cdot ( \frac{1}{\lambda^{2}}\phi^{i}\nabla \mathcal{R})\chi+\nabla \cdot ( \frac{1}{\lambda^{2}}\phi^{i}\nabla \psi_{\lambda})\chi +\\
	&+\frac{\alpha}{\lambda^{2}}\phi^{i}\nabla \chi \cdot \nabla \Gamma_{0}+\frac{1}{\lambda^{2}}\phi^{i}\nabla \chi \cdot \nabla \mathcal{R}+\frac{1}{\lambda^{2}}\phi^{i}\nabla \chi \cdot \nabla \psi_{\lambda}.
\end{align*}
Therefore
\begin{align*}
	\mathcal{L}_{u_{1}}\big[\frac{1}{\lambda^{2}}\phi^{i}\chi\big]=&\chi \frac{1}{\lambda^{2}}\Delta \phi^{i}+\frac{2}{\lambda^{2}}\nabla \chi \cdot \nabla \phi^{i}+\frac{1}{\lambda^{2}}\phi^{i} \Delta\chi -\\
	&-\big[\nabla \cdot ( \frac{\alpha}{\lambda^{2}}\phi^{i}\nabla \Gamma_{0})\chi+\nabla \cdot ( \frac{1}{\lambda^{2}}\phi^{i}\nabla \mathcal{R})\chi+\nabla \cdot(\frac{1}{\lambda^{2}}\phi^{i}\nabla \psi_{\lambda})\chi+\\
	&\hspace{0.43cm}+\frac{\alpha}{\lambda^{2}}\phi^{i}\nabla \chi \cdot \nabla \Gamma_{0}+\frac{1}{\lambda^{2}}\phi^{i}\nabla \chi \cdot \nabla \mathcal{R}+\frac{1}{\lambda^{2}}\phi^{i}\nabla \chi \cdot \nabla \psi_{\lambda}\big]-\\
	&-\big[\nabla \cdot ( \frac{\alpha}{\lambda^{2}}U\nabla \psi)\chi+\nabla \cdot(\frac{\alpha}{\lambda^{2}}U\nabla(\hat{\psi}-\psi))\chi+\frac{\alpha}{\lambda^{2}}U\nabla \chi \cdot \nabla \hat{\psi}+\\
	&\hspace{0.43cm}+ \nabla \cdot(\varphi_{\lambda}\nabla \psi)\chi+\nabla \cdot(\varphi_{\lambda}\nabla \psi)(1-\chi)+\\
	&\hspace{0.43cm}+\nabla \cdot(\varphi_{\lambda}\nabla(\hat{\psi}-\psi))\chi+\nabla \cdot(\varphi_{\lambda}\nabla(\hat{\psi}-\psi))(1-\chi)\big].
\end{align*}
Next we expand
\begin{align*}
	\mathcal{L}_{u_{1}}[\varphi^{o}]=\Delta \varphi^{o}-\nabla \cdot (\varphi^{o}\nabla v_{1})-\nabla \cdot ( u_{1}\nabla \psi^{o}), \ \ \ \ \ \psi^{o}=(-\Delta)^{-1}\varphi^{o}.
\end{align*}
We have 
\begin{align*}
	\nabla \cdot (u_{1}\nabla \psi^{o})=&\nabla \cdot ( \frac{\alpha}{\lambda^{2}}U\chi \nabla \psi^{o})+\nabla \cdot ( \varphi_{\lambda}\nabla \psi^{o})=\\
	=&\nabla \cdot(\frac{\alpha}{\lambda^{2}}U\nabla \psi^{o})\chi+\frac{\alpha}{\lambda^{2}}U\nabla \chi \cdot \nabla \psi^{o}+\nabla \cdot(\varphi_{\lambda}\nabla \psi^{o})\chi +\nabla \cdot(\varphi_{\lambda}\nabla \psi^{o})(1-\chi),
\end{align*}
and 
\begin{align*}
   \nabla \cdot(\varphi^{o}\nabla v_{1})=&\nabla \cdot(\varphi^{o}\nabla v_{0})+\nabla \cdot(\varphi^{o}\nabla \psi_{\lambda})=\\
   =&\alpha \nabla \cdot(\varphi^{o}\nabla \Gamma_{0})+\nabla \cdot(\varphi^{o}\nabla \mathcal{R})+\nabla \cdot(\varphi^{o}\nabla \psi_{\lambda})=\\
   =&\nabla\varphi^{o}\cdot \nabla \Gamma_{0}-\chi\frac{1}{\lambda^{2}}U\varphi_{0}-(1-\chi)\frac{1}{\lambda^{2}}U\varphi_{0}+(\alpha-1)\nabla \cdot(\varphi^{o}\nabla \Gamma_{0})\chi+\\
   &+(\alpha-1)\nabla \cdot(\varphi^{o}\nabla \Gamma_{0})\ch(1-\chi)+\nabla \cdot(\varphi^{o}\nabla \mathcal{R})+\nabla \cdot(\varphi^{o}\nabla \psi_{\lambda})\chi+\nabla \cdot(\varphi^{o}\nabla \psi_{\lambda})(1-\chi).
\end{align*}
Based on the previous formulas we get the inner equation
\begin{align*}
	\lambda^{4}\partial_{t}(\frac{1}{\lambda^{2}}\phi^{i}(y,t))\chi=&L[\phi^{i}]\chi+[ -(\alpha-1)\nabla_{y}\cdot (U\nabla_{y}\psi)-(\alpha-1)\nabla_{y}\cdot(\phi^{i}\nabla \Gamma_{0})]\chi+\lambda^{4}S(u_{1})\chi-\\
	&-\lambda^{2}\nabla_{y}\cdot(\varphi_{\lambda}\nabla_{y}\psi^{o})\chi-\lambda^{2}\nabla\cdot(\varphi^{o}\nabla_{y}\psi_{\lambda})\chi+\lambda^{2}U\varphi^{o}\chi-\alpha\nabla_{y}\cdot(U\nabla \psi^{o})\chi-\\
	&-\lambda^{2}\nabla_{y}\cdot(\varphi_{\lambda}\nabla_{y}\psi)\chi-\nabla_{y}\cdot(\phi^{i}\nabla_{y}\psi_{\lambda})\chi-(\alpha-1)\lambda^{2}\nabla \cdot(\varphi^{o}\nabla \Gamma_{0})\chi-\\
	&-\alpha \nabla_{y}\cdot(U\nabla_{y}(\hat{\psi}-\psi))\chi-\lambda^{2}\nabla_{y}\cdot(\varphi_{\lambda}\nabla_{y}(\hat{\psi}-\psi))\chi-\\
	&-\nabla_{y}\cdot((\phi^{i}\chi+\lambda^{2}\varphi_{0})\nabla_{y}(\hat{\psi}+\psi^{o}))\chi,
\end{align*}
where 
\begin{align}\label{LinOperator}
 L[\phi]=\Delta_{y}\phi-\nabla_{y}\cdot(U\nabla_{y}\psi)-\nabla_{y}\cdot(\phi \nabla \Gamma_{0}).
\end{align}
The inner equation can be slighly modified into the form
\begin{align}\label{cutoffinner}
	\lambda^{2}\partial_{t}\phi^{i}\chi=L[\phi^{i}]\chi+B[\phi^{i}]\chi+\hat{E}_{1}\widetilde{\chi}+\hat{F}[\phi^{i},\varphi^{o},\textbf{p}]\widetilde{\chi}
\end{align}
where we introduced the cut-off
\begin{align}\label{chitilde}
	\tilde{\chi}=\chi_{0}(\frac{|x-\xi|}{2\sqrt{\delta(T-t)}})
\end{align}
with $\chi_{0}$ defined in \eqref{cutoff}, the operator
\begin{align}\label{OperatorB}
	B[\phi^{i}]=\lambda\dot{\lambda}(2\phi^{i}+y\cdot \nabla\phi^{i})
\end{align}
that takes into account the radial part of the derivative in time of the rescaled function $\frac{1}{\lambda^{2}}\phi^{i}(\frac{x-\xi}{\lambda})=\frac{1}{\lambda^{2}}\phi(y,t)$, and we denoted
\begin{align}\label{VECp}
	\textbf{p}=(\lambda,\alpha,\xi), \ \ \ \ \ \ \  \ \ 	\hat{E}_{1}(y,t)=\lambda^{4}S(u_{1}(\textbf{p}))\chi(x,t)
\end{align}
\begin{align}\label{Fhat}
	\hat{F}[\phi^{i},\varphi^{o},\textbf{p}]=&\lambda\dot{\xi}\cdot \nabla \phi^{i} \chi-(\alpha-1)\nabla_{y}\cdot (U\nabla_{y}\psi)\chi-(\alpha-1)\nabla_{y}\cdot(\phi^{i}\nabla \Gamma_{0})\chi-\nonumber\\
	&-\lambda^{2}\nabla_{y}\cdot(\varphi_{\lambda}\nabla_{y}\psi^{o})\chi-\lambda^{2}\nabla\cdot(\varphi^{o}\nabla_{y}\psi_{\lambda})\chi+\lambda^{2}U\varphi^{o}\chi-\alpha\nabla_{y}\cdot(U\nabla \psi^{o})\chi-\nonumber\\
	&-\lambda^{2}\nabla_{y}\cdot(\varphi_{\lambda}\nabla_{y}\psi)\chi-\nabla_{y}\cdot(\phi^{i}\nabla_{y}\psi_{\lambda})\chi-(\alpha-1)\lambda^{2}\nabla \cdot(\varphi^{o}\nabla \Gamma_{0})\chi-\nonumber\\
	&-\alpha \nabla_{y}\cdot(U\nabla_{y}(\hat{\psi}-\psi))\chi-\lambda^{2}\nabla_{y}\cdot(\varphi_{\lambda}\nabla_{y}(\hat{\psi}-\psi))\chi-\nonumber\\
	&-\nabla_{y}\cdot((\phi^{i}\chi+\lambda^{2}\varphi^{o})\nabla_{y}(\hat{\psi}+\psi^{o}))\chi.
\end{align}
The reason why we need to introduce the cut-off \eqref{chitilde} will be clear in Section \ref{SectionChoice}. \newline
Similarly we formulate the outer equation as
\begin{align}\label{outereq}
	\partial_{t}\varphi^{o}=\Delta \varphi^{o}-\nabla \Gamma_{0}\cdot \nabla \varphi^{o}+G(\phi^{i},\varphi^{o},\textbf{p})
\end{align}
where
\begin{align*}
	G(\phi^{i},\varphi^{o},\textbf{p})=&S(u_{1},\textbf{p})(1-\chi)+\frac{2}{\lambda^{2}}\nabla \chi \cdot \nabla \phi^{i}+\frac{1}{\lambda^{2}}\phi^{i}\Delta \chi-\frac{1}{\lambda^{2}}\phi^{i}\partial_{t}\chi-\frac{\alpha}{\lambda^{2}}\phi^{i}\nabla \chi \cdot \nabla \Gamma_{0}+\nonumber\\
	&+\frac{1}{\lambda^{2}}U\varphi^{o}(1-\chi)-\alpha \lambda^{2}U\nabla \chi \cdot \nabla \psi^{o}-\nabla \cdot (\varphi_{\lambda}\nabla \psi^{o})(1-\chi)-\nonumber\\
\end{align*}
\begin{align}\label{G}
	&-\frac{1}{\lambda^{2}}\nabla \cdot ( \phi^{i}\nabla \mathcal{R})\chi -\frac{1}{\lambda^{2}}\phi^{i}\nabla \chi \cdot \nabla \mathcal{R}-\frac{1}{\lambda^{2}}\phi^{i}\nabla \chi \cdot \nabla \psi_{\lambda}-\nonumber\\
	&-\frac{\alpha}{\lambda^{2}}U \nabla \chi \cdot \nabla \hat{\psi}-\nabla \cdot(\varphi_{\lambda}\nabla(\hat{\psi}-\psi))(1-\chi)-\nonumber\\
	&-\nabla\cdot(\varphi_{\lambda}\nabla \psi)(1-\chi)-\nabla \cdot((\frac{1}{\lambda^{2}}\phi^{i}\chi+\varphi^{o})\nabla(\hat{\psi}+\psi^{o}))(1-\chi).
\end{align}
If $\phi^{i}$, $\varphi^{o}$ are solutions of \eqref{cutoffinner} and \eqref{outereq} respectively, then $u$ given by \eqref{solutionexpansion} satisfies \eqref{ks0}. We will need some conditions to solve this system and this is the role of the parameters $\textbf{p}$ defined in \eqref{VECp}. \newline The initial conditions and some further modifications of equations \eqref{cutoffinner}, \eqref{outereq} will be discussed in Section \ref{ProjectionIOsystem}.
\subsection{Choice of $\lambda_{0}$, $\alpha_{0}$ and $\xi_{0}$}\label{SectionChoice}
In this section we introduce the first approximation of the parameters, namely $\alpha_{0}$, $\lambda_{0}$ and $\xi_{0}$, in the context of the elliptic equation
\begin{align}\label{elliptic}
	L[\phi^{i}]=-\lambda^{4}S(u_{1}(\textbf{p}))\tilde{\chi} \ \ \ \ \text{in }\mathbb{R}^{2}
\end{align}
where we used the definition of $\textbf{p}$ in \eqref{VECp}. We observe that clearly $L$ in \eqref{elliptic} is invariant under dilations and translations. \newline
As we observed in Lemma \ref{lemmaesterr}, in contrast with \cite{DdPDMW}, the decay of $S(u_{1})$ drastically changes its behaviour if we move away from the singularity. This is why we need to keep the cut-off in the right-hand side of \eqref{elliptic}. As we will see in Section \ref{SecSecImprovement} the solution of this ``toy problem" will give us a last improvement for our ansatz $u_{1}$ that will erase the main order terms in the right-hand side of \eqref{cutoffinner}. We recall the statement of Lemma 5.1 in \cite{DdPDMW} where the interested reader can also find a proof.
\begin{lemma}\label{ellipticsol}
	Let h(y) be a radial function such that 
	\begin{align*}
		\|(1+|y|)^{\gamma}h(y)\|_{L^{\infty}(\mathbb{R}^{2})}<\infty,
	\end{align*}
for some $\gamma>4$ and satisfying 
\begin{align}
         \int_{\mathbb{R}^{2}}h(y)dy=0,\label{masscondh} \\
         \int_{\mathbb{R}^{2}}h(y)|y|^{2}dy=0.\label{secmomcondh}
\end{align}
Then there exists a radial solution $\phi(y)$ of 
\begin{align*}
	L[\phi]=h \ \ \ \ \text{in }\mathbb{R}^{2}
 \end{align*}
such that
\begin{align}\label{estimatephi0igammanot6}
	|\phi(y)|\le C \|(1+|y|)^{\gamma}h(y)\|_{L^{\infty}(\mathbb{R}^{2})}\frac{1}{(1+|y|)^{\gamma-2}} \ \ \ \,\text{if } \gamma\neq6,
\end{align}\label{estimatephi0igamma6}
\begin{align}
	|\phi(y)|\le C \|(1+|y|)^{\gamma}h(y)\|_{L^{\infty}(\mathbb{R}^{2})}\frac{\ln(2+|y|)}{1+|y|^{4}}, \ \ \ \text{if } \gamma=6,
\end{align}
and 
\begin{align}\label{zeromasselliptic}
	\int_{\mathbb{R}^{2}}\phi(y)dy=0.
\end{align}
\end{lemma}
\begin{proof}
	See the proof of Lemma 5.1 in \cite{DdPDMW}.
\end{proof}
We remark that condition \eqref{secmomcondh} is crucial to get \eqref{zeromasselliptic}. In fact it can be easily proved that the mass of $\phi$ is proportional to the second moment of the right-hand side. Recalling that we denoted $\textbf{p}=(\lambda,\alpha,\xi)$, later we want to assume that at main order  $\textbf{p}\approx\textbf{p}_{0}=(\lambda_{0},\alpha_{0},0)$. But then, since we want to solve \eqref{elliptic} by applying Lemma \ref{ellipticsol},  we make the right-hand side radial writing
\begin{align}\label{ApproximationFIrstTIme}
		L[\phi^{i}]=-\lambda^{4}_{0}S(u_{1}(\textbf{p}_{0}))\tilde{\chi}_{0} \ \ \ \ \text{in }\mathbb{R}^{2}.
\end{align} 
where we introduced the cut-off
\begin{align}\label{tildechi2}
	\tilde{\chi}_{0}=\chi_{0}(\frac{|x|}{2\sqrt{\delta(T-t)}}),
\end{align}
We remark that the factor $2$ in \eqref{tildechi2} (and similarly in \eqref{chitilde}) is important since, thanks to \eqref{EstiLambda}, \eqref{EstiPar}, it makes $S(u_{1}(\textbf{p}_{0}))(1-\tilde{\chi}_{0})$ differentiable in time. Otherwise you would need to control $\ddot{\alpha}$. This fact will be important in the proof of Lemma \ref{controlDerSecMom}.
By keeping in mind Lemma \ref{ellipticsol} it is clear why we will be interested in finding $\alpha_{0}(t)$, $\lambda_{0}(t)$ such that for any time $t\in(0,T)$
\begin{align}\label{MASSU1}
	\int_{\mathbb{R}^{2}}S(u_{1}(\textbf{p}_{0}))\tilde{\chi}_{0}dx=0
\end{align}
and for some $\sigma>0$
\begin{align}\label{SECMOMU1}
	\int_{\R^2}S(u_{1}(\textbf{p}_{0}))|x|^{2}\tilde{\chi}_{0}dx=O(e^{-(\frac{3}{2}+\sigma)\sqrt{2|\ln(T-t)|}}).
\end{align}
Notice that in \eqref{SECMOMU1} we are not imposing that the second moment is equal to zero since we are only able to solve the resulting equation up to a small error. For this reason in Section \ref{SecSecImprovement} we are going to project the right-hand side of \eqref{elliptic} introducing an error that thanks to the size of second moment \eqref{SECMOMU1} will be sufficiently small.
Achieving \eqref{MASSU1} and \eqref{SECMOMU1} for some $\textbf{p}_{0}$ will be the goal of the remaining part of this section, but first we need the following two results. In Lemma \ref{lemmaIntSecMom} we give an expansion of the main order terms of the second moment of $S(u_{1})$ without the cut-off. In Lemma \ref{controlDerSecMom} we present a formula that relates the second moment of $S(u_{1})\tilde{\chi}_{0}$ with the uncut second moment.
\begin{lemma}\label{lemmaIntSecMom}
   Let $u_{1}$ be defined in \eqref{u1}. Then for any $t$ in $(0,T)$
   \begin{align}\label{SecMomInt}
   	\int_{\mathbb{R}^{2}}S(u_{1}[\lambda,\alpha,0])|x|^{2}dx=&4\int_{\R^2}\varphi_{\lambda}dx-\alpha \int_{\R^2}\tilde{E}(x,t;\lambda)|x|^{2}dx+\nonumber\\
   	&-\frac{\dot{\alpha}}{\lambda^{2}}\int_{\mathbb{R}^{2}}U\chi |x|^{2}dx-(1-\alpha)\int_{\mathbb{R}^{2}}E(x,t;\lambda)|x|^{2}dx+\nonumber\\
   	&+4(\int_{\R^2}u_{0}+\int_{\R^2}\varphi_{\lambda})(1-\frac{1}{8\pi}\int_{\R^2}u_{0}-\frac{1}{8\pi}\int_{\R^2}\varphi_{\lambda})
   \end{align}
where $E$, $\tilde{E}$ are defined in \eqref{Ephilambda}, \eqref{Etildephilambda}.
\end{lemma}
\begin{proof}
	In what follows we will simply denote $u_{1}=u_{1}[\lambda,\alpha,0]$. Using \eqref{secMomForm} we see that
	\begin{align*}
		\int_{\R^2}S(u_{1})|x|^{2}dx=&-\int_{\R^2}\partial_{t}u_{0}|x|^{2}dx-\int_{\R^2}\partial_{t}\varphi_{\lambda}|x|^{2}dx+\\
		&+4(\int_{\mathbb{R}^{2}}u_{0}+\int_{\R^2}\varphi_{\lambda})(1-\frac{1}{8\pi}\int_{\mathbb{R}^{2}}u_{0}-\frac{1}{8\pi}\int_{\R^2}\varphi_{\lambda}).
	\end{align*}
   Multiplying \eqref{eqphilambda} by $|x|^{2}$ and integrating we get
   \begin{align*}
    \int_{\R^2}\partial_{t}\varphi_{\lambda}|x|^{2}dx=-4\int_{\mathbb{R}^{2}}\varphi_{\lambda}dx+\int_{\R^2}E(x,t;\lambda)|x|^{2}dx
   \end{align*}
    and then
    \begin{align*}
    	\int_{\mathbb{R}^{2}}S(u_{1})|x|^{2}dx=&-\int_{\mathbb{R}^{2}}\partial_{t}u_{0}|x|^{2}dx+4\int_{\mathbb{R}^{2}}\varphi_{\lambda}dx-\int_{\mathbb{R}^{2}}E(x,t;\lambda)|x|^{2}+\nonumber\\
    	&+4\big(\int_{\mathbb{R}^{2}}u_{0}+\int_{\mathbb{R}^{2}}\varphi_{\lambda}\big)\big(1-\frac{1}{8\pi}\int_{\R^2}u_{0}-\frac{1}{8\pi}\int_{\R^2}\varphi_{\lambda}\big).
    \end{align*}
    From \eqref{timederiU0} and from \eqref{Ephilambda}, \eqref{Etildephilambda}, recalling \eqref{radialWY}, we get
    \begin{align*}
    	    -\partial_{t}u_{0}(x,t)=-\frac{\dot{\alpha}}{\lambda^{2}}U(\bar{y})\chi_{0}(\bar{w})+\alpha E(x,t;\lambda)-\alpha \tilde{E}(x,t;\lambda).
    \end{align*}
    Finally we get
    \begin{align*}
    	\int_{\R^2}(\partial_{t}u_{0}+E(x,t;\lambda))|x|^{2}&=\int_{\R^2}(\partial_{t}u_{0}+\alpha E(x,t;\lambda))|x|^{2}dx+(1-\alpha)\int E(x,t;\lambda)|x|^{2}dx\\
    	&=\frac{\dot{\alpha}}{\lambda^{2}}\int_{\mathbb{R}^{2}}U\chi |x|^{2}dx+\alpha \int_{\mathbb{R}^{2}}\tilde{E}(x,t;\lambda)|x|^{2}dx+(1-\alpha)\int_{\R^2}E(x,t;\lambda)|x|^{2}dx.
    \end{align*}
\end{proof}
\begin{lemma}\label{controlDerSecMom}
	Let $u_{1}$ be defined in \eqref{u1}. Assume \eqref{EstiLambda}, \eqref{EstiPar}. Then for any $\rho>0$, and for any time t in $(0,T)$ we have
	\begin{align*}
		\int_{\mathbb{R}^{2}}S(u_{1}[\lambda,\alpha,0])\tilde{\chi}_{0}|x|^{2}dx=&	\int_{\mathbb{R}^{2}}S(u_{1}[\lambda,\alpha,0])|x|^{2}dx-\int_{\mathbb{R}^{2}}S(u_{1}[\lambda,\alpha,0])(x,T)|x|^{2}dx+\nonumber\\
		&-\frac{1}{\pi}(T-t)\big(\int_{\R^2}\varphi_{\lambda}\big)\partial_{t}\int_{\R^2}(u_{0}+\varphi_{\lambda})+O(e^{-(2-\rho)\sqrt{2|\ln(T-t)|}}).
	\end{align*}
\end{lemma}
\begin{proof}
	By \eqref{innererr} we know that as $t\to T$ we have $\int_{\mathbb{R}^{2}}S(u_{1})|x|^{2}\tilde{\chi}_{0}dx\to 0$ and then
	\begin{align*}
		\int_{\mathbb{R}^{2}}S(u_{1})|x|^{2}(1-\tilde{\chi}_{0})dx\to \int_{\mathbb{R}^{2}}S(u_{1})(x,T)|x|^{2}dx  \ \ \ \ \text{as }t\to T.
	\end{align*}
	 Then the statement is a consequence of the following Taylor expansion
	\begin{align}\label{TaylorSecMom}
		\int_{\mathbb{R}^{2}}S(u_{1})\tilde{\chi}_{0}|x|^{2}dx=&\int_{\mathbb{R}^{2}}S(u_{1})|x|^{2}dx-\int_{\mathbb{R}^{2}}S(u_{1})|x|^{2}(1-\tilde{\chi}_{0})dx=\nonumber\\
		=&\int_{\mathbb{R}^{2}}S(u_{1})|x|^{2}dx-\int_{\mathbb{R}^{2}}S(u_{1})(x,T)|x|^{2}dx-\nonumber\\
		&-(T-t)\int_{\mathbb{R}^{2}}\partial_{t}[S(u_{1})](x,\sigma)|x|^{2}(1-\tilde{\chi}_{0})dx+(T-t)\int_{\mathbb{R}^{2}}S(u_{1})|x|^{2}\partial_{t}\tilde{\chi}_{0}(x,\sigma)dx
	\end{align}
for some $\sigma\in(t,T)$.
To estimate the last term it is sufficient to recall \eqref{outererrro}. 
\newline The remaining term in \eqref{TaylorSecMom} is more delicate. The first basic observation is that, recalling that $\xi=0$, by \eqref{secapprErr} if $|x|\ge 2\sqrt{\delta(T-t)}$ we have
\begin{align*}
	&S(u_{1})=-\frac{4}{r}\partial_{r}\varphi_{\lambda}-\nabla \varphi_{\lambda}\cdot \nabla v_{0}-\nabla\varphi_{\lambda}\cdot \nabla \psi_{\lambda}+(\varphi_{\lambda})^{2}\\
	&\implies \partial_{t}S(u_{1})=-\frac{4}{r}\partial_{r}\partial_{t}\varphi_{\lambda}-\nabla \partial_{t}\varphi_{\lambda}\cdot \nabla v_{0}-\nabla \partial_{t}\varphi_{\lambda}\cdot \nabla \psi_{\lambda}+2\varphi_{\lambda}\partial_{t}\varphi_{\lambda}-\\
	&\hspace{2.55cm}-\nabla \varphi_{\lambda}\cdot  \partial_{t} \nabla\psi_{\lambda}-\nabla \varphi_{\lambda}\cdot \partial_{t}\nabla v_{0}=h_{1}(x,t)-\nabla \varphi_{\lambda}\cdot  \partial_{t} \nabla\psi_{\lambda}-\nabla \varphi_{\lambda}\cdot \partial_{t}\nabla v_{0}.
\end{align*}
We remark that we do not need any assumption on the second derivative of $\alpha(t)$ since we are in the region $|x|\ge 2\sqrt{\delta(T-t)}$.
Now, Lemma \ref{Derestimatephilambda} and the same observations we made to prove Lemma \ref{lemmaesterr} give that if $t>0$
\begin{align*}
    |h_{1}(x,t)|\le C\begin{cases}
    	\frac{e^{-2\sqrt{2|\ln|x|^{2}|}}}{|x|^{6}}|\ln|x|^{2}| \ \ \ \ \text{if }2\sqrt{\delta(T-t)}\le |x|\le \sqrt{\varepsilon(T)}\\
    	\frac{e^{-2\sqrt{2|\ln \varepsilon(T)|}}}{\varepsilon(T)^{3}}|\ln \varepsilon(T)| e^{-\frac{|x|^{2}}{4(t+2\varepsilon(T))}} \ \ \ \ \text{if }|x|\ge \sqrt{\varepsilon(T)}
    \end{cases} .
\end{align*}
and then if $t>0$ we have
\begin{align}\label{Estimateh1}
	|(T-t)\int_{\mathbb{R}^{2}}h_{1}(x,\sigma)|x|^{2}&(1-\chi_{0}(\frac{|x|}{2\sqrt{\delta(T-t)}}))dx|=O(e^{-(2-\rho)\sqrt{2|\ln(T-t)|}})
\end{align}
for any $\rho>0$ (here we also used that $f(x)=\frac{e^{-2\sqrt{2|\ln x|}}}{x}|\ln x|$ is a decreasing function if $x$ is small). Now we observe that since $\xi=0$ everything is radial and then we have
 \begin{align}\label{}
 	\partial_{t}\partial_{r}\psi_{\lambda}=-\frac{1}{r}\partial_{t}\int_{0}^{r}\varphi_{\lambda}(s,t)sds, \ \ \  \partial_{t}\partial_{r}v_{0}=-\frac{1}{r}\partial_{t}\int_{0}^{r}u_{0}(s,t)sds
 \end{align}
that implies
\begin{align*}
	-\nabla \varphi_{\lambda}\cdot  \partial_{t} \nabla\psi_{\lambda}-\nabla \varphi_{\lambda}\cdot \partial_{t}\nabla v_{0}=\frac{1}{r}\partial_{r}\varphi_{\lambda}[\frac{1}{2\pi}\partial_{t}\int_{\R^2}(u_{0}+\varphi_{\lambda})-\int_{r}^{\infty}\partial_{t}\varphi_{\lambda}(s,t)sds].
\end{align*}
The second term inside the parentheses can be treated as $h_{1}$, see Lemma \ref{Derestimatephilambda} and \eqref{Estimateh1}. The desired conclusion comes by observing that
\begin{align*}
	\frac{1}{2\pi}\big(\int_{\R^2}\frac{1}{r}\partial_{r}\varphi_{\lambda}|x|^{2}dx\big)\partial_{t}\int_{\R^2}(u_{0}+\varphi_{\lambda})=-\frac{1}{\pi}\big(\int_{\R^2}\varphi_{\lambda}\big)\partial_{t}\int_{\R^2}(u_{0}+\varphi_{\lambda}).
\end{align*}
\end{proof}
The following Remark is based on the fact that \eqref{EstiLambda}, \eqref{EstiPar} imply $\frac{|\xi|}{\lambda}\le e^{-(1+\gamma_{1})\sqrt{2|\ln(T-t)|}}$ and then for any $\rho>0$
\begin{align*}
	\int_{\R^2}\lambda^{4}S(u_{1}[\lambda,\alpha,\xi])\tilde{\chi}|y|^{2}dy=&\int_{|y|\le2 \frac{|\xi|}{\lambda}}\lambda^{4}S(u_{1}[\lambda,\alpha,\xi])\tilde{\chi}|y|^{2}dy+\int_{|y|\ge2\frac{|\xi|}{\lambda}}\lambda^{4}S(u_{1}[\lambda,\alpha,\xi])\tilde{\chi}|y|^{2}dy=\\
	=&O(e^{-(2-\rho)\sqrt{2|\ln(T-t)|}})+\int_{|y|\ge2\frac{|\xi|}{\lambda}}\lambda^{4}S(u_{1}[\lambda,\alpha,\xi])\tilde{\chi}|y|^{2}dy.
\end{align*}
Then we observe that if we use the notation \eqref{radialWY}, we have that $|y|\ge2\frac{|\xi|}{\lambda}$ implies $\frac{y}{\lambda}\approx \frac{y}{\lambda}+\frac{\xi}{\lambda}=\frac{\bar{y}}{\lambda}$ and consequently we get the following remark. 
At this point we are ready to fix $\lambda_{0}$ and $\alpha_{0}$ satisfying \eqref{MASSU1} and \eqref{SECMOMU1}. Recalling that $\xi_{0}=0$, by \eqref{SecMomInt} and Lemma \ref{controlDerSecMom} we get
   \begin{align}\label{SecMOmSimpli}
	\int_{\mathbb{R}^{2}}S(u_{1}(\textbf{p}_{0}))|x|^{2}\tilde{\chi}_{0}dx=&4\int_{\R^2}\varphi_{\lambda_{0}}dx-\alpha_{0} \int_{\R^2}\tilde{E}(x,t;\lambda_{0})|x|^{2}dx+\nonumber\\
	&-\frac{\dot{\alpha}_{0}}{\lambda_{0}^{2}}\int_{\mathbb{R}^{2}}U(\frac{x}{\lambda}_{0})\chi_{0}(\frac{|x|}{\lambda_{0}}) |x|^{2}dx-(1-\alpha_{0})\int_{\mathbb{R}^{2}}E(x,t;\lambda_{0})|x|^{2}dx+\nonumber\\
	&+4(\int_{\R^2}u_{0}+\int_{\R^2}\varphi_{\lambda_{0}})(1-\frac{1}{8\pi}\int_{\R^2}u_{0}-\frac{1}{8\pi}\int_{\R^2}\varphi_{\lambda_{0}})+\nonumber\\
	&-\frac{1}{\pi}(T-t)\big(\int_{\R^2}\varphi_{\lambda_{0}}\big)\partial_{t}\int_{\R^2}(u_{0}+\varphi_{\lambda_{0}})-\int_{\mathbb{R}^{2}}S(u_{1})(x,T)|x|^{2}dx+\nonumber\\
	&+O(e^{-(2-\rho)\sqrt{2|\ln(T-t)|}}).
\end{align}
At the same time we know that for any $t\in(0,T)$ we have
\begin{align*}
	\lambda_{0}^{4}\int_{\mathbb{R}^{2}}S(u_{1}(\textbf{p}_{0}))\tilde{\chi}_{0} d\bar{y}_{0}&=\lambda_{0}^{4}\int_{\mathbb{R}^{2}}S(u_{1}(\textbf{p}_{0}))d\bar{y}_{0}-\lambda_{0}^{4}\int_{\mathbb{R}^{2}}S(u_{1}(\textbf{p}_{0}))(1-\tilde{\chi}_{0})d\bar{y}_{0}=\\
	&=-\lambda_{0}^{2}\partial_{t}\big( \int_{\mathbb{R}^{2}}(u_{0}+\varphi_{\lambda_{0}})dx-\int_{t}^{T}\lambda_{0}^{2}(s)\int_{\mathbb{R}^{2}}S(u_{1}(\textbf{p}_{0}))(1-\tilde{\chi}_{0})d\bar{y}_{0}\big).
\end{align*}
Then we want to fix $\alpha_{0}(t)$ such that
\begin{align}\label{equationforalpha}
	\int_{\mathbb{R}^{2}}(u_{0}+\varphi_{\lambda_{0}})dx-\int_{t}^{T}\lambda_{0}^{2}(s)\int_{\mathbb{R}^{2}}S(u_{1}(\textbf{p}_{0}))(1-\tilde{\chi}_{0})d\bar{y}_{0}=8\pi +\int_{\R^2}\varphi_{\lambda_{0}}(\cdot,T)dx.
\end{align}
Observing that that thanks to \eqref{outererrro} the second term in the left-hand side has size $O(e^{-2\sqrt{2|\ln(T-t)|}}|\ln(T-t)|)$, we can rewrite \eqref{equationforalpha} and obtain for any $\rho>0$ and for any $t\in(0,T)$
\begin{align}\label{finaleqalpha0}
	\int_{\mathbb{R}^{2}}(u_{0}+\varphi_{\lambda_{0}})dx=8\pi+\int_{\R^2}\varphi_{\lambda_{0}}(x,T)dx+O(e^{-(2-\rho)\sqrt{2|\ln(T-t)|}}).
\end{align}
Thanks to \eqref{finaleqalpha0}, by \eqref{EstiPar}, and recalling \eqref{SECMOMU1}, the expansion \eqref{SecMOmSimpli} tells us we want to fix $\lambda_{0}$ such that for any time $t\in(0,T)$ we have
\begin{align}\label{equationlambda}
	4\int_{\mathbb{R}^{2}}\varphi_{\lambda_{0}}dx-\int_{\mathbb{R}^{2}}\tilde{E}(x,t;\lambda_{0})|x|^{2}dx=4\int_{\mathbb{R}^{2}}\varphi_{\lambda_{0}}(x,T)dx+O(e^{-(\frac{3}{2}+\sigma)\sqrt{2|\ln(T-t)|}})
\end{align}
for some $\sigma>0$. First we observe that an elementary computation gives
\begin{align}\label{ExpansionRHSeqlambda}
	\int_{\mathbb{R}^{2}}\tilde{E}|x|^{2}dx=-64\pi \beta \frac{\lambda^{2}}{\delta(T-t)}+O(\frac{\lambda^{4}}{(T-t)^{2}}), \ \ \ \ \beta=\int_{0}^{\infty}\frac{1-\chi_{0}(s)}{s^{3}}ds.
\end{align}
Moreover as it will be shown in Corollary \ref{ExpansionMass}, we have that if we let $\gamma$ be the Euler-Mascheroni constant, assuming \eqref{EstiLambda},\eqref{EstiPar}, for any $t\in (-\frac{T}{2},T)$ and for any $\kappa>0$ we have 
\begin{align}\label{ExpansionMassphilambda}
	\int_{\mathbb{R}^{2}}\varphi_{\lambda}dx-16\pi\beta \frac{\lambda^{2}}{\delta(T-t)}=&\int_{\mathbb{R}^{2}}\varphi_{\lambda}(x,T)dx+4\pi \int_{-\varepsilon(T)}^{T}\frac{\lambda\dot{\lambda}(s)}{T-s}ds-4\pi \int_{-\varepsilon(T)}^{t-\lambda^{2}}\frac{\lambda\dot{\lambda}(s)}{t-s}ds+\nonumber\\
	&+4\pi(\gamma+1-\ln4)\lambda\dot{\lambda}(t)+O(\frac{e^{-\sqrt{2|\ln(T-t)|}}}{|\ln(T-t)^{|\kappa}}).
\end{align}
By looking the right-hand side of \eqref{ExpansionMassphilambda}, if we take for instance $t=0$ it is immediately clear why we need $\varepsilon(T)>0$. It is also clear why it is helpful to have $t>-\frac{T}{2}$, in fact we see that we would have $-\varepsilon(T)-\lambda^{2}(-\varepsilon(T))< -\varepsilon(T)$ (see Section \ref{MASSPHISECTION} for more precise justification for this choice). At the same time we also recall that $\lambda$ must be defined in $(-\varepsilon(T),T)$. The idea is to extend the equation to the whole interval $(-\varepsilon(T),T)$ by cutting some remainders. This cut-off will be equal to $1$ in the interval $(0,T)$. We notice that in this way we are solving an equation in $(-\varepsilon(T),T)$ but that gives the desired conclusion only in $(0,T)$. We postpone the details to Section \ref{MASSPHISECTION}. We notice that in Lemma \ref{SECMASSprofile} we will also prove that
\begin{align}\label{approximatlambda}
	\lambda^{\star}(t)=e^{-\frac{\gamma+2}{2}}\sqrt{T-t}e^{-\sqrt{\frac{|\ln(T-t)|}{2}}}
\end{align}
is an approximate solution of
\begin{align*}
	4\pi \int_{-\varepsilon(T)}^{T}\frac{\lambda\dot{\lambda}(s)}{T-s}ds-4\pi \int_{-\varepsilon(T)}^{t-\lambda^{2}}\frac{\lambda\dot{\lambda}(s)}{t-s}ds+4\pi(\gamma+1-\ln4)\lambda\dot{\lambda}(t)=0 \ \ \ \ t\in(0,T),
\end{align*}
but the errors produced if we simply considered \eqref{approximatlambda} would be too large (see for example the errors in \eqref{ExpansionMassphilambda}). Consequently, we need to improve this first approximation and the following Proposition that we will prove in Section \ref{MASSPHISECTION} goes in this direction. We observe that the following result will give us a remainder that is much smaller than the one we required to proceed with the construction, see \eqref{SECMOMU1} and \eqref{equationlambda}.
	\begin{proposition}\label{prop-lambda0}
	For $T$, $\varepsilon(T)$ sufficiently small there exists $\lambda_0:[-\varepsilon(T),T) \to (0,\infty) $ such that if $t\in (0,T)$
	\begin{align}\label{EqPhiLam}
		4 \int_{\R^2} \varphi_{\lambda_0} dx
		- \int_{\R^2} \tilde E(\lambda_0) |x-\xi|^2dx
		& =4\int_{\mathbb{R}^{2}}\varphi_{\lambda_0}(x,T)dx+ \mathcal{E}[\lambda_{0}] ,
	\end{align}
    where
    \begin{align}\label{EstimatesErrorproplambda0}
    	&\mathcal{E}[\lambda_{0}](t)=O\Bigl(e^{-(2-\rho)\sqrt{2|\ln(T-t)|}}), \\ \ \ \ &\frac{d}{dt}\mathcal{E}[\lambda_{0}](t)=O\Bigl(\frac{e^{-(2-\rho)\sqrt{2|\ln(T-t)|}}}{T-t}), \\
    	&\frac{d^{2}}{dt^{2}}\mathcal{E}[\lambda_{0}](t)=O\Bigl(\frac{e^{-(\frac{3}{2}-\rho)\sqrt{2|\ln(T-t)|}}}{(T-t)^{2}}),
    \end{align}
	for any $\rho>0$. Moreover, as $  t \to T$, $\lambda_0$ satisfies
	\begin{align*}
		&\lambda_0(t) = 2e^{-\frac{\gamma+2}{2}}\sqrt{T-t}e^{-\sqrt{\frac{|\ln(T-t)|}{2}}}(1+\frac{1}{|\ln(T-t)|^{1/8}}),
		\\
		&\dot\lambda_0(t) = -e^{-\frac{\gamma+2}{2}}\frac{1}{\sqrt{T-t}}e^{-\sqrt{\frac{|\ln(T-t)|}{2}}} (1+\frac{1}{|\ln(T-t)|^{1/8}}),
		\\
		&|\ddot \lambda_0(t) |  \leq  \frac{1}{(T-t)^{3/2}}e^{-\sqrt{\frac{|\ln(T-t)|}{2}}}  ,
	\end{align*}
	where $\gamma=0.577215...$ is the Euler-Mascheroni constant.
\end{proposition}
\begin{proof}
	See Section \ref{MASSPHISECTION}.
\end{proof}
\noindent We anticipate that the estimate for the second derivative of the error in \eqref{EstimatesErrorproplambda0} could be improved at this level since we have sufficient control of the regularity of the right-hand side. This decay will be sufficient for our construction. \newline
Once $\lambda_{0}$ has been constructed as in Proposition \ref{prop-lambda0} we choose $\alpha_{0}$ such that  \eqref{MASSU1}, or equivalently \eqref{equationforalpha}, holds. We observe that for any $t\in(0,T)$ we have
\begin{align*}
	\alpha_{0}(t)\int_{\mathbb{R}^{2}}U(y)\chi_{0}\big(\frac{\lambda_{0} y}{\sqrt{\delta(T-t)}}\big)dy+\int_{\R^2}\varphi_{\lambda_{0}}dx=\alpha_{0}(t)\big(8\pi-16\pi\beta\frac{\lambda_{0}^{2}}{\delta(T-t)}+O(\frac{\lambda_{0}^{4}}{(T-t)^{2}})\big)+\int_{\R^2}\varphi_{\lambda_{0}}dx.
\end{align*}
But then, since $\lambda_{0}$ solves \eqref{EqPhiLam} and because of \eqref{finaleqalpha0},\eqref{ExpansionRHSeqlambda}, for any $t\in(0,T)$ we get
\begin{align}\label{alpha-1}
	|\alpha_{0}(t)-1|\le Ce^{-(2-\rho)\sqrt{2|\ln(T-t)|}}.
\end{align}
Thanks to the additional regularity we have from Proposition \ref{prop-lambda0}, we can also control the derivative of the error in \eqref{EqPhiLam} and prove that in the same interval for any $\rho>0$
\begin{align}\label{estimatealphadot}
	|\dot{\alpha}_{0}|\le \frac{Ce^{-(2-\rho)\sqrt{2|\ln(T-t)|}}}{T-t}, \ \ \ \ |\ddot{\alpha}_{0}|\le C \frac{e^{-(\frac{3}{2}-\rho)\sqrt{2|\ln(T-t)|}}}{(T-t)^{2}}.
\end{align}
As a corollary of Proposition \ref{prop-lambda0} we get
\begin{corollary}\label{secmomconditionELLI}
	Let $\textbf{p}_{0}=(\lambda_{0},\alpha_{0},0)$ with $\lambda_{0}$ given by Proposition \ref{prop-lambda0} and $\alpha_{0}$ defined in \eqref{MASSU1}. Then for any $t\in (0,T)$ and for any $\rho>0$ we have
	\begin{align*}
		\int_{\mathbb{R}^{2}}S(u_{1}(\textbf{p}_{0}))\tilde{\chi}_{0}|x|^{2}dx=O(e^{-(2-\rho)\sqrt{2|\ln(T-t)|}}).
	\end{align*}
\begin{proof}
	We observe that by \eqref{equationforalpha} we have
	\begin{align*}
		|\partial_{t}(\int_{\mathbb{R}^{2}}(u_{0}+\varphi_{\lambda}))|=|-\int_{\mathbb{R}^{2}}S(u_{1}(\textbf{p}_{0}))(1-\tilde{\chi}_{0})dx|=O(\frac{e^{-(2-\rho)\sqrt{|\ln(T-t)|}}}{T-t}).
	\end{align*}
	Then the proof is an immediate consequence of Lemma \ref{lemmaIntSecMom}, Lemma \ref{controlDerSecMom}, Proposition \ref{prop-lambda0} and \eqref{estimatealphadot}.
\end{proof}
\end{corollary}
\begin{remark}\label{Remarkchangeofcenter}
	Let us assume \eqref{EstiLambda}, \eqref{EstiPar}. A direct computation gives that the difference in the mass and the second moment after changing the center (i.e., assuming $\xi\neq0$) produces only small errors:
	\begin{align*}
		&\int_{\R^2}(\chi_{0}\lambda^{4}_{0}S(u_{1}))(y\frac{\lambda_{0}}{\lambda}+\frac{\xi}{\lambda}) \tilde{\chi}(y)dy=O(e^{-3\sqrt{2|\ln(T-t)|}}\sqrt{|\ln(T-t)|}), \\
		&\int_{\R^2}(\chi_{0}\lambda^{4}_{0}S(u_{1}))(y\frac{\lambda_{0}}{\lambda}+\frac{\xi}{\lambda}) \tilde{\chi}(y)|y|^{2}dy=O(e^{-(2-\rho)\sqrt{2|\ln(T-t)|}}) \ \ \ \text{for any }\rho>0.
	\end{align*}
\end{remark}
\begin{corollary}\label{DerTimErrorAnsatz}
		Let $\textbf{p}_{0}=(\lambda_{0},\alpha_{0},0)$ with $\lambda_{0}$ given by Proposition \ref{prop-lambda0} and $\alpha_{0}$ defined in \eqref{MASSU1}. Then for any $t\in (0,T)$ 
	\begin{align}\label{DerAnsEst}
		|\partial_{t}(\lambda_{0}^{4}S(u_{1}(\textbf{p}_{0}))\tilde{\chi}_{0})|\le C e^{-\sqrt{2|\ln(T-t)|}}\frac{\log^{2}(2+|y|)}{1+|y|^{6}},
	\end{align}
    and for any $\rho>0$
    \begin{align*}
		\partial_{t}\int_{\mathbb{R}^{2}}S(u_{1}(\textbf{p}_{0}))\tilde{\chi}_{0}|x|^{2}dx=O(\frac{e^{-(2-\rho)\sqrt{2|\ln(T-t)|}}}{T-t}).
	\end{align*}
\end{corollary}
\begin{proof}
	The proof is a consequence of formula \eqref{secapprErr}, Proposition \ref{prop-lambda0}, \eqref{estimatealphadot} and estimates \eqref{Derboundphilambda}. We notice that the the additional logarithm at the numerator we get in \eqref{DerAnsEst} is a consequence of estimate \eqref{Derboundphilambda}.
\end{proof}
\subsection{A further improvement of the approximation}\label{SecSecImprovement}  
By using again the variables \eqref{radialWY}, we can introduce $\phi_{0}^{i}$ that satisfies
\begin{align}\label{ellipticcorre}
	L[\phi_{0}^{i}]=-\lambda^{4}_{0}S(u_{1}(\textbf{p}_{0}))\tilde{\chi}_{0}+c_{0}(t)W_{2}(\bar{y}_{0}) \ \ \ \ x=\lambda_{0} \bar{y}_{0}\in \mathbb{R}^{2},
\end{align}
where $L$ is the operator \eqref{LinOperator} in \eqref{ellipticsol} and we used \eqref{tildechi2} that, consistently with the choice of $\textbf{p}_{0}$, will make the right-hand side radial. In \eqref{ellipticcorre} the time variable is regarded as a parameter, $W_{2}(\bar{y})$ is a fixed smooth radial function with compact support and such that
\begin{align}\label{W2}
	\int_{\mathbb{R}^{2}}W_{2}(y)dy=0, \ \ \ \int_{\mathbb{R}^{2}}W_{2}(y)|y|^{2}dy=1.
\end{align}
By Corollary \ref{secmomconditionELLI} we know that for any $t\in(0,T)$ and for any $\rho>0$
\begin{align}\label{c0}
	|c_{0}(t)|\le C e^{-(2-\rho)\sqrt{2|\ln(T-t)|}}.
\end{align}
 We also know, by the choice of $\alpha_{0}$ we made, that in $(0,T)$ we have $\int_{\mathbb{R}^{2}}S(u_{1}(\textbf{p}_{0}))\tilde{\chi}_{0}dx=0$. Thanks to Lemma \ref{ellipticsol} and estimate \eqref{innererr} we can take the zero mass solution of \eqref{ellipticcorre} that satisfies
\begin{align}\label{estEllCorr}
	|\phi_{0}^{i}(\bar{y}_{0},t)|\le C e^{-\sqrt{2|\ln(T-t)|}}\sqrt{|\ln(T-t)|}\frac{\ln(2+|\bar{y}_{0}|)}{1+|\bar{y}_{0}|^{4}}.
\end{align}

\subsection{Reformulation of the system}\label{ReformulationOf}
For any $(x,t)\in \mathbb{R}^{2}\times(0,T)$ we write
\begin{align}\label{decomPo}
	\begin{cases}
		\phi^{i}=\phi_{0}^{i}+\phi\\
		\textbf{p}=\textbf{p}_{0}+\textbf{p}_{1}
	\end{cases}
\end{align}
where $\textbf{p}_{0}=(\lambda_{0},\alpha_{0},0)$ and $\textbf{p}_{1}=(\lambda_{1},\alpha_{1},\xi_{1})$, with $\lambda_{0}$ constructed in Proposition \ref{prop-lambda0} and $\alpha_{0}$ such that \eqref{MASSU1} holds. Recalling \eqref{ellipticcorre}, by substituting \eqref{decomPo} into the equations \eqref{cutoffinner} and \eqref{outereq} we get the following equations for $\phi$ and $\varphi^{o}$:
\begin{align}\label{innereqfinal}
		\lambda^{2}\partial_{t}\phi\chi=L[\phi]\chi+B[\phi]\chi+\hat{E}_{2}\tilde{\chi}_{2}+\hat{F}_{2}[\phi,\varphi^{o},\textbf{p}]\widetilde{\chi}
\end{align}
\begin{align}\label{outereqfinal}
		\partial_{t}\varphi^{o}=\Delta \varphi-\nabla_{x}\Gamma_{0}(\frac{x-\xi}{\lambda})\cdot \nabla \varphi^{o} +G_{2}(\phi,\varphi^{o},\textbf{p}_{1})  
\end{align}
In \eqref{innereqfinal} we recall that $\tilde{\chi}$ is the same cut-off defined in \eqref{chitilde} and we introduced
\begin{align}\label{chitilde2}
\tilde{\chi}_{2}(x,t)=\chi_{0}\big(\frac{|x-\xi| }{\sqrt{T-t}}e^{\mu\sqrt{2|\ln(T-t)|}}\big)
\end{align}
with the $\chi_{0}$ as in \eqref{cutoff}, $\mu>0$ to be chosen small and 
\begin{align}\label{E2hat}
	\chi E_{2}=\hat{E}_{2}=(-\lambda^{2}\partial_{t}\phi_{0}^{i}+B[\phi_{0}^{i}]+c_{0}(t)W_{2}(y))\chi
\end{align}
\begin{align}\label{F2hat}
	\tilde{\chi}\hat{F}_{2}(\phi,\varphi,\textbf{p}_{1})=\tilde{\chi}\big[&\hat{F}(\phi_{0}^{i}+\phi,\varphi,\textbf{p}_{0}+\textbf{p}_{1})+\lambda^{4}S(u_{1}(\textbf{p}_{0}+\textbf{p}_{1}))\chi-\lambda_{0}^{4}S(u_{1}(\textbf{p}_{0}))\tilde{\chi}_{0}\chi+\nonumber\\
	&+[\hat{E}_{2}(\bar{y}_{0})-\hat{E}_{2}(y)]\tilde{\chi}_{2}\big]
\end{align}
\begin{align}
	G_{2}(\phi,\varphi^{o},\textbf{p}_{1})=G(\phi_{0}^{i}+\phi,\varphi^{o},\textbf{p}_{0}+\textbf{p}_{1})+\lambda^{-4}E_{2}(1-\tilde{\chi}_{2})\chi
\end{align}
where $\hat{F}$ and $G$ are \eqref{Fhat} and \eqref{G} respectively.
If we assume \eqref{EstiPar} we see that if $|y|\le \frac{2\sqrt{\delta(T-t)}}{\lambda(t)}$ we have
\begin{align}\label{E2pointw}
	|E_{2}(y,t)|\le C e^{-2\sqrt{2|\ln(T-t)\}}}\sqrt{|\ln(T-t)|} \frac{\ln(2+|y|)}{(1+|y|^{4})}+ C e^{-(2-\rho)\sqrt{2|\ln(T-t)|}}|W_{2}(y)|. 
\end{align}
We remark that the control of $\partial_{t}\phi_{0}^{i}$ is a consequence of Corollary \ref{DerTimErrorAnsatz}. \newline
The cut-off $\tilde{\chi}_{2}$ allows us to achieve
\begin{align}\label{INnnerESTE2}
	|E_{2}(y,t)\tilde{\chi}_{2}|\le C \frac{e^{-\nu \sqrt{2|\ln(T-t|)}}}{(1+|y|)^{6+\sigma}},
\end{align}
for any $\nu<1+2\mu-\frac{\sigma}{2}+\sigma\mu$. We will choose $\mu$ and $\sigma$ small positive numbers such that $2\mu-\frac{\sigma}{2}>0$ so that we have $1<\nu<1+2\mu-\frac{\sigma}{2}$. At the same time we notice that 
\begin{align*}
	|\lambda^{-4}E_{2}(1-\tilde{\chi}_{2})\chi|&\le\frac{e^{-2\sqrt{2|\ln(T-t)|}}}{\lambda^{4}} \sqrt{|\ln(T-t)|}\frac{\sqrt{|\ln(T-t)|}}{e^{2(1-2\mu)\sqrt{2|\ln(T-t)|}}}\chi\le\\
	&\le \frac{e^{-2(1-2\mu)\sqrt{2|\ln(T-t)|}}}{(T-t)^{2}}|\ln(T-t)| \chi.
\end{align*}
We anticipate that, as we will state more precisely in  Section \ref{OuterProblemSection} and Section \ref{proofThm1}, in order to solve the resulting system we will need to fix a $\nu$ that satisfies \eqref{INnnerESTE2} but assuring ourselves that $2(1-2\mu)>\nu+\frac{1}{2}$ in order to have a good control of this error in the outer region. This is the main reason why we will take $\nu>1$ but close to $1$ and $\mu$ sufficiently small.

\subsection{Splitting the inner solution $\phi$}\label{ProjectionIOsystem}In this section for any $(x,t)\in \mathbb{R}^{2}\times(0,T)$ we decompose $\phi$ in a radial part $[\phi]_{\text{rad}}$ defined by
\begin{align}\label{DecomPRAD}
	[\phi]_{\text{rad}}(|y|,t)=\frac{1}{2\pi}\int_{0}^{2\pi}\phi(|y|e^{i\theta},t)d\theta
\end{align}
and a term with no radial mode
\begin{align}\label{phi3def}
	\phi_{3}=\phi-[\phi]_{\text{rad}}
\end{align} . We further decompose $[\phi]_{\text{rad}}$ as
\begin{align*}
   [\phi]_{\text{rad}}=\phi_{1}+\phi_{2}.
\end{align*}
The function $\phi_{1}$ will solve an equation with the radial part of the right-hand side of \eqref{innereqfinal} which will be projected so that it has zero mass and zero second moment. The function $\phi_{2}$ will take care of the second moment that is not going to be identically equal to zero, but only sufficiently small.\newline
We introduce now another cut-off 
\begin{align}\label{chihat}
	\hat{\chi}(x,t)=\chi_{0}(\frac{|x-\xi|}{4\sqrt{\delta(T-t)}})
\end{align}
such that we can cut the operator $B$ observing that $B[\phi]\chi=B[\phi\hat{\chi}]\chi$. Cutting the operator $B$ we are paying the price of introducing the operator $\mathcal{E}[\phi_{1}]$ in the equation for $\phi_{1}$ and an operator $\mathcal{F}[\phi_{3}]$ in the equation for $\phi_{3}$. These two operators are explicit and will be introduced in Sections \ref{Section11} and \ref{Section12}. It will be helpful to consider \eqref{chihat} instead of \eqref{chitilde} since the pointwise estimate we will obtain for $\mathcal{E}[\phi]$ is sufficiently good only when $\hat{\chi}\equiv1$. As it will be explained in Section \ref{InnerTheoryIntro} we need to cut the operator $B$ to control the decay in space of the inner solutions. Recalling \eqref{E2hat}, \eqref{F2hat} and writing $\hat{F}_{3}=\hat{F}_{2}+\hat{E}_{2}\tilde{\chi}_{2}$ we can introduce the notation
\begin{align}\label{F4hat}
	\hat{F}_{4}=\hat{F}_{3}[\phi_{1}+\phi_{2}+\phi_{3},\varphi^{o},\textbf{p}_{1}]-\mathcal{E}[\phi_{1}]\chi-\mathcal{F}[\phi_{3}]\chi
\end{align}
and we can split the equation \eqref{innereqfinal} as
\begin{align*}
	\lambda^{2}\partial_{t}(\phi_{1}+\phi_{2})\chi=L[\phi_{1}+\phi_{2}]\chi+B[\phi_{1}+\phi_{2}\hat{\chi}]\chi+\mathcal{E}[\phi_{1}]\chi+([\hat{F}_{4}[\phi_{1}+\phi_{2}+\phi_{3},\varphi^{o},\textbf{p}_{1}]]_{\text{rad}})\tilde{\chi}
\end{align*}
\begin{align*}
	\lambda^{2}\partial_{t}\phi_{3}\chi=L[\phi_{3}]\chi+B[\phi_{3}\hat{\chi}]\chi+\mathcal{F}[\phi_{3}]\chi+(\hat{F}_{4}[\phi_{1}+\phi_{2}+\phi_{3},\varphi^{o},\textbf{p}_{1}]-[\hat{F}_{4}[\phi_{1}+\phi_{2}+\phi_{3},\varphi^{o},\textbf{p}_{1}]]_{\text{rad}})\tilde{\chi}
\end{align*}
where we decomposed $\hat{F}_{4}$ as we did for $\phi$ in \eqref{phi3def}. Finally, simplifying the cut-off $\chi$ and writing 
\begin{align}\label{F4Defin}
	F_{4}=\hat{F}_{4}/\chi,
\end{align} 
we obtain
\begin{align}\label{innereqfinalfinalMOD0}
			\lambda^{2}\partial_{t}(\phi_{1}+\phi_{2})=L[\phi_{1}+\phi_{2}]+B[\phi_{1}+\phi_{2}\hat{\chi}]+\mathcal{E}[\phi_{1}]+(F_{4}[\phi_{1}+\phi_{2}+\phi_{3},\varphi^{o},\textbf{p}_{1}]_{\text{rad}})\tilde{\chi},
\end{align}
\begin{align}\label{innereqfinalfinalMOD1}
	\lambda^{2}\partial_{t}\phi_{3}=L[\phi_{3}]+B[\phi_{3}\hat{\chi}]+\mathcal{F}[\phi_{3}]+(F_{4}[\phi_{1}+\phi_{2}+\phi_{3},\varphi^{o},\textbf{p}_{1}]-[F_{4}[\phi_{1}+\phi_{2}+\phi_{3},\varphi^{o},\textbf{p}_{1}]]_{\text{rad}})\tilde{\chi}.
\end{align}
We notice that, as we will prove in Section \ref{Section11}, we do not need to cut $\phi_{1}$.
In order to project the right-hand side of \eqref{innereqfinalfinalMOD0}, \eqref{innereqfinalfinalMOD1} we introduce some useful functions.
For any $h(y,t)$ with sufficient spatial decay we define
\begin{align}\label{masssecomomh}
	m_{0}[h](t)=\int_{\mathbb{R}^{2}}h(y,t)dy, \ \ \ \ m_{2}[h](t)=\int_{\mathbb{R}^{2}}h(y,t)|y|^{2}dy
\end{align}
and
\begin{align}\label{firstmomentsh}
	m_{1,j}[h](t)=\int_{\mathbb{R}^{2}}h(y,t)y_{j}dy, \ \ \ j=1,2,
\end{align}
which denote the mass, second moment and center of mass of $h$. Let $W_{0}\in C^{\infty}(\mathbb{R}^{2})$ be radial with compact support such that
\begin{align}\label{W0}
	\int_{\mathbb{R}^{2}}W_{0}dy=1, \ \ \ \ \int_{\mathbb{R}^{2}}W_{0}|y|^{2}dy=0.
\end{align}
Let $W_{1,j}$ , $j=1,2$ be smooth functions with compact support with form $W_{1,j}(y)=\tilde{W}(|y|)y_{j}$ so that
\begin{align}\label{W1j}
	\int_{\mathbb{R}^{2}}W_{1,j}(y)y_{j}=1.
\end{align}
We recall the definition of $W_{2}$ in \eqref{W2}. Then, $h-m_{0}[h]W_{0}$ has zero mass, $h-m_{2}[h]W_{2}$ has zero second moment, and $h-m_{1,1}[h]W_{1,1}-m_{1,2}[h]W_{1,2}$ has zero center of mass. We can write
\begin{align}\label{innereqfinalFINAL}
	\begin{cases}
		\lambda^{2}\partial_{t}\phi_{1}=&\hspace{-2.2cm}L[\phi_{1}]+B[\phi_{1}]+\mathcal{E}[\phi_{1}]+f-m_{0}[f]W_{0}-\\
		&-m_{2}[f]W_{2}, \ \ \ \text{in }\mathbb{R}^{2}\times(0,T)\\
		\phi_{1}(\cdot,0)=\phi_{1,0} \ \ \ \ \ \text{in }\mathbb{R}^{2},
	\end{cases}
\end{align}
\begin{align}\label{innereqfinalFINAL2}
	\begin{cases}
		\lambda^{2}\partial_{t}\phi_{2}=&\hspace{-2.2cm}L[\phi_{2}]+B[\phi_{2}\hat{\chi}]+m_{2}[f]W_{2}\ \ \ \ \text{in }\mathbb{R}^{2}\times(0,T)\\
		\phi_{2}(\cdot,0)=\phi_{2,0} \ \ \ \ \ \text{in }\mathbb{R}^{2},
	\end{cases}
\end{align}
\begin{align}\label{innereqfinalFINALMODE1}
	\begin{cases}
		\lambda^{2}\partial_{t}\phi_{3}=&\hspace{-2.2cm}L[\phi_{3}]+B[\phi_{3}\hat{\chi}]+\mathcal{F}[\phi_{3}]+f_{3}-m_{1,j}[f_{3}]W_{1,j}\ \ \ \ \text{in }\mathbb{R}^{2}\times(0,T)\\
		\phi_{3}(\cdot,0)=\phi_{3,0} \ \ \ \ \ \text{in }\mathbb{R}^{2},
	\end{cases}
\end{align}
where
\begin{align}\label{DEffGLUING}
	&f[\phi_{1}+\phi_{2}+\phi_{3},\varphi^{o},\textbf{p}_{1}]=([F_{4}[\phi_{1}+\phi_{2}+\phi_{3},\varphi^{o},\textbf{p}_{1}]]_{\text{radial}})\tilde{\chi},
\end{align}
\begin{align*}
	&f_{3}[\phi_{1}+\phi_{2}+\phi_{3},\varphi^{o},\textbf{p}_{1}]=([F_{4}[\phi_{1}+\phi_{2}+\phi_{3},\varphi^{o},\textbf{p}_{1}]]-[F_{4}[\phi_{1}+\phi_{2}+\phi_{3},\varphi^{o},\textbf{p}_{1}]]_{\text{radial}})\tilde{\chi}
\end{align*} 
with $F_{4}$ has been defined in \eqref{F4Defin} and \eqref{F4hat}.
We need to find parameters $\alpha_{1}(t)$, $\xi(t)(t)$ defined in $(0,T)$ and $\lambda_{1}(t)$ defined in $(-\varepsilon(T),T)$ such that
\begin{align}\label{orthogonalitycond}
	\begin{cases}
		m_{0}[f[\phi_{1}+\phi_{2}+\phi_{4},\varphi^{o},\textbf{p}_{1}]]=0\\
		m_{1,j}[f_{3}[\phi_{1}+\phi_{2}+\phi_{4},\varphi^{o},\textbf{p}_{1}]]=0, \ \ \ j=1,2,  \ \ \ \ \text{if } t\in (0,T).
	\end{cases}
\end{align}
The system above is coupled with the outer equation
\begin{align}\label{outereqfinalFINAL}
	\begin{cases}
		\partial_{t}\varphi^{o}=\Delta \varphi^{o}-\nabla_{x}\Gamma_{0}(\frac{x-\xi}{\lambda})\cdot \nabla \varphi^{o} +G_{2}(\phi_{1}+\phi_{2},\varphi^{o},\textbf{p}_{1})\\
		\varphi^{o}(\cdot,0)=0.
	\end{cases}
\end{align}
\begin{remark}\label{outersolutionRMKNegTim}
	We observe that the choice of $\lambda_{1}$ in the expansion \eqref{decomPo} will be crucial to obtain \eqref{orthogonalitycond}. For this purpose $\lambda_{1}$, that as $\lambda_{0}$ must be defined in $(-\varepsilon(T),T)$, will erase the mass introduced by the term $\lambda^{2}U\varphi^{o}$ in \eqref{F4hat}. This means that this term will  be the right-hand side of a nonlocal equation (similar to the one satisfied by $\lambda_{0}$ in Proposition \ref{prop-lambda0}) that must be defined for $t\in(-\varepsilon(T),T)$. The idea is to extend the outer solution $\varphi^{o}$ to the whole interval $(-\varepsilon(T),T)$ by simply assuming that $\varphi^{o}\equiv0$ when $-\varepsilon(T)<t\le 0$.
\end{remark}

\subsection{Mass, second moment and center of mass}\label{SecMassSecMom}
In this section we derive some formulas for the mass and the second moment that we have in the right-hand side of \eqref{innereqfinalFINAL} and \eqref{innereqfinalFINAL2}. We will also explain how the mass condition  \eqref{orthogonalitycond} will be satisfied and how we are going to make the right-hand side of \eqref{innereqfinalFINAL2}, the second moment of $f$,  sufficiently fast decaying. We remark that the coupling of these two equations is different to \cite{DdPDMW}. 
\begin{lemma}\label{lemmamassdiff}
	Let us assume \eqref{EstiLambda}, \eqref{EstiPar}. Let $\lambda=\lambda_{0}+\lambda_{1}$ and $\alpha=\alpha_{0}+\alpha_{1}$  with $\lambda_{0}$ given by Proposition \ref{prop-lambda0} and $\alpha_{0}$ defined in \eqref{MASSU1} then for any $t\in(0,T)$ and for any $\rho>0$ we have
	\begin{align*}
		m_{0}[\lambda^{4}S(u_{1}(\textbf{p}_{0}+\textbf{p}_{1}))\tilde{\chi}&-\lambda_{0}^{4}S(u_{1}(\textbf{p}_{0})\tilde{\chi}_{0}\tilde{\chi})]=\\
		&=-\lambda^{2}\partial_{t}\big[8\pi \alpha_{1}+\int_{\mathbb{R}^{2}}(\varphi_{\lambda}-\varphi_{\lambda_0})dx-16\pi \beta\frac{\lambda^{2}-\lambda_{0}^{2}}{\delta(T-t)}\big]+O(e^{-(3-\rho)\sqrt{2|\ln(T-t)|}})
	\end{align*}
where $\beta$ has been defined in \eqref{ExpansionRHSeqlambda}.
\end{lemma}
\begin{proof}
	Recalling \eqref{MASSU1} and Remark \ref{Remarkchangeofcenter}, the proof is a trivial adaptation of the computations we made in Section \ref{SectionChoice} to get \eqref{alpha-1}.
\end{proof}
Now we expand the main order term in the right-hand side of \eqref{innereqfinalFINAL2}. A fundamental difference with the analogous computation in \cite{DdPDMW} is that now the mass of $\varphi_{\lambda}$ tends to a constant different from zero (for further details see also \eqref{EXPANSIONMASS}).
\begin{lemma}\label{lemmasecmomdiff}
		Let us assume \eqref{EstiLambda}, \eqref{EstiPar}. Let $\lambda=\lambda_{0}+\lambda_{1}$ and $\alpha=\alpha_{0}+\alpha_{1}$  with $\lambda_{0}$ given by Proposition \ref{prop-lambda0} and $\alpha_{0}$ defined in \eqref{MASSU1} then for any $t\in(0,T)$ and for any $\rho>0$ we have
		\begin{align*}
			m_{2}[\lambda^{4}S(u_{1}(\textbf{p}_{0}+\textbf{p}_{1}))\tilde{\chi}-\lambda_{0}^{4}S(u_{1}(\textbf{p}_{0})\tilde{\chi}_{0}\tilde{\chi})]=\\
			&\hspace{-5cm}=-(32\pi+8\int_{\mathbb{R}^{2}}\varphi_{\lambda}(T))\alpha_{1}-\\
			&\hspace{-4.5cm}-\frac{1}{\pi}\int_{\mathbb{R}^{2}}\varphi_{\lambda}(T)\big(\int_{\mathbb{R}^{2}}(\varphi_{\lambda}-\varphi_{\lambda}(\cdot,T))-\int_{\mathbb{R}^{2}}(\varphi_{\lambda_{0}}-\varphi_{\lambda_{0}}(\cdot,T))-16\pi \beta \frac{\lambda^{2}-\lambda_{0}^{2}}{\delta(T-t)}\big)-\\
			&\hspace{-4.5cm}-\frac{1}{\pi}(T-t)\big(\int_{\R^2}\varphi_{\lambda}(T)\big)\partial_{t}\big(8\pi \alpha_{1}+\int_{\R^2}(\varphi_{\lambda}-\varphi_{\lambda_{0}})-16\pi \beta \frac{\lambda^{2}-\lambda_{0}^{2}}{\delta(T-t)}\big)+\\
			&\hspace{-4.5cm}+O(e^{-(2-\rho)\sqrt{2|\ln(T-t)|}}),
		\end{align*}
	where $\beta$ has been defined in \eqref{ExpansionRHSeqlambda}.
\end{lemma}
\begin{proof}
	After recalling Lemma \ref{lemmaIntSecMom}, Lemma \ref{controlDerSecMom}, Remark \ref{Remarkchangeofcenter} and assuming \eqref{EstiLambda}, \eqref{EstiPar} it is immediate to observe that for $t>0$ and for any $\rho>0$ we have
	\begin{align*}
		m_{2}[\lambda^{4}S(u_{1}(\textbf{p}))\tilde{\chi}]=&4\int_{\mathbb{R}^{2}}\varphi_{\lambda}-64\pi \beta \frac{\lambda^{2}}{\delta(T-t)}+4\big(\int_{\mathbb{R}^{2}}u_{0}(\textbf{p})+\int_{\mathbb{R}^{2}}\varphi_{\lambda}\big)\big(1-\frac{1}{8\pi}\int_{\mathbb{R}^{2}}u_{0}(\textbf{p})-\frac{1}{8\pi}\int_{\mathbb{R}^{2}}\varphi_{\lambda}\big)\\
		&-\frac{1}{\pi}(T-t)\big(\int_{\R^2}\varphi_{\lambda}\big)\partial_{t}\int_{\R^2}(u_{0}(\textbf{p})+\varphi_{\lambda})-\int_{\R^2}S(u_{1}(\lambda,\alpha,0))(x,T)|x|^{2}dx\\
		&+O(e^{-(2-\rho)\sqrt{2|\ln(T-t)|}})
	\end{align*}
and by Corollary \ref{secmomconditionELLI} and Remark \ref{Remarkchangeofcenter} we have that for any $\rho>0$
\begin{align*}
		m_{2}[\lambda_{0}^{4}S(u_{1}(\textbf{p}_{0}))\tilde{\chi}\tilde{\chi}_{0}]=O(e^{-(2-\rho)\sqrt{2|\ln(T-t)|}}).
\end{align*}
By expanding $\int_{\mathbb{R}^{2}}(u_{0}(\textbf{p})+\varphi_{\lambda})$ as in we did in Section \ref{SectionChoice} to get \eqref{alpha-1} we observe
\begin{align}\label{expansionMassu1}
   \int_{\mathbb{R}^{2}} (u_{0}(\textbf{p})+\varphi_{\lambda})dx=8\pi+8\pi(\alpha-1)+\int_{\mathbb{R}^{2}}\varphi_{\lambda}dx-16\pi \beta\frac{\lambda^{2}}{\delta(T-t)}+O(e^{-(2-\rho)\sqrt{2|\ln(T-t)|}}).
\end{align}
Hence
\begin{align*}
	4\big(\int_{\mathbb{R}^{2}}u_{0}(\textbf{p})&+\int_{\mathbb{R}^{2}}\varphi_{\lambda}\big)\big(1-\frac{1}{8\pi}\int_{\mathbb{R}^{2}}u_{0}(\textbf{p})-\frac{1}{8\pi}\int_{\mathbb{R}^{2}}\varphi_{\lambda}\big)=\\
	=&32\pi(-(\alpha-1)-\frac{1}{8\pi} \int_{\mathbb{R}^{2}}\varphi_{\lambda}+2\beta \frac{\lambda^{2}}{\delta(T-t)})-\frac{1}{2\pi} \big(8\pi(\alpha-1)+\int_{\mathbb{R}^{2}}\varphi_{\lambda}-16\pi \beta \frac{\lambda^{2}}{\delta(T-t)}\big)^{2}+\\
	&+O(e^{-(2-\rho)\sqrt{2|\ln(T-t)|}}).
\end{align*}
We want to expand the second term in the right-hand side. We see
\begin{align*}
	\big(8\pi(\alpha-1)+\int_{\mathbb{R}^{2}}\varphi_{\lambda}-16\pi \beta \frac{\lambda^{2}}{T-t}\big)^{2}=&\big(8\pi(\alpha-1)+\int_{\mathbb{R}^{2}}\varphi_{\lambda}(T)+\int_{\mathbb{R}^{2}}(\varphi_{\lambda}-\varphi_{\lambda}(\cdot,T))-16\pi \beta \frac{\lambda^{2}}{\delta(T-t)}\big)^{2}=\\
	=&(\int_{\mathbb{R}^{2}}\varphi_{\lambda}(T))^{2}+\\
	&+2\int_{\mathbb{R}^{2}}\varphi_{\lambda}(T)\big(8\pi(\alpha-1)+\int_{\mathbb{R}^{2}}(\varphi_{\lambda}-\varphi_{\lambda}(\cdot,T))-16\pi \beta \frac{\lambda^{2}}{\delta(T-t)}\big)+\\
	&+O(e^{-(2-\rho)\sqrt{2|\ln(T-t)|}}).
\end{align*}
After subtracting and adding $\frac{1}{\pi}\int_{\mathbb{R}^{2}}\varphi_{\lambda}(T)\big(8\pi(\alpha_{0}-1)+\int_{\mathbb{R}^{2}}(\varphi_{\lambda_{0}}-\varphi_{\lambda_{0}}(T))dx-16\pi\beta \frac{\lambda_{0}^{2}}{T-t}\big)$, $-32\pi(\alpha_{0}-1)$ and $-\frac{1}{\pi}(T-t)\big(\int_{\R^2}\varphi_{\lambda}(T)\big)\partial_{t}\int_{\R^2}(u_{0}(\textbf{p}_{0})+\varphi_{\lambda_{0}})$ we get the desired result.
\end{proof}
At this point we can explain how we achieve the zero mass condition and how we can make the second moment sufficiently fast decaying. We observe that thanks to Lemma \ref{lemmamassdiff} if $t>0$, the zero mass condition becomes 
\begin{align}\label{ZeromassSCondition}
     \int_{\R^2}(u_{0}(\textbf{p})-u_{0}(\textbf{p}_{0})+\varphi_{\lambda}-\varphi_{\lambda_{0}})=-\int_{t}^{T}K[\phi,\varphi^{o},\textbf{p}]+\int_{t}^{T}\int_{\mathbb{R}^{2}}(S(u_{1}(\textbf{p}))-S(u_{1}(\textbf{p}_{0})))(1-\tilde{\chi})dx
\end{align}
that at main order is
\begin{align}\label{ExpansionMassErr}
	8\pi \alpha_{1}+\int_{\mathbb{R}^{2}}(\varphi_{\lambda}-\varphi_{\lambda_0})dx-16\pi \beta\frac{\lambda^{2}-\lambda_{0}^{2}}{\delta(T-t)}+\int_{t}^{T}K[\phi,\varphi^{o},\textbf{p}](s)ds=\int_{\mathbb{R}^{2}}(\varphi_{\lambda}(T)-\varphi_{\lambda_{0}}(T))dx
\end{align}
where with $K$ we include all the terms (divided by $\lambda^{2}$) in the mass of \eqref{DEffGLUING} that are not in $	m_{0}[\lambda^{4}S(u_{1}(\textbf{p}_{0}+\textbf{p}_{1}))\tilde{\chi}-\lambda_{0}^{4}S(u_{1}(\textbf{p}_{0})\tilde{\chi}_{0}\tilde{\chi})]$. In Section \ref{proofThm1} we will show  that the main order term is $\int_{t}^{T}\frac{1}{\lambda^{2}(s)}\int_{\mathbb{R}^{2}}\lambda^{2}(s)U\varphi^{o}\tilde{\chi}dyds$. Equation \eqref{ExpansionMassErr} gives the value of $\alpha_{1}$. We remark that controlling $\alpha_{1}$ is not sufficient and that we need an estimate for $\dot{\alpha}_{1}$, but we will discuss how we can obtain sufficiently good estimate for $\dot{\alpha}_{1}$ later.\newline
So far we achieved the zero mass condition, but, as we will see in Section \ref{proofThm1}, to solve the gluing system we will need a sufficiently small right-hand side for \eqref{innereqfinalFINAL2}. Then Lemma \ref{lemmasecmomdiff} and \eqref{ExpansionMassErr} imply that, at main order, for any $t\in(0,T)$ we need
\begin{align}\label{ExpansionSecMomErr}
	m_{2}[f]\approx&4\int_{\mathbb{R}^{2}}(\varphi_{\lambda}-\varphi_{\lambda_{0}})dx-4\int_{\mathbb{R}^{2}}(\varphi_{\lambda}(T)-\varphi_{\lambda_{0}}(T))dx-64\pi \beta \frac{\lambda^{2}-\lambda_{0}^{2}}{\delta(T-t)}+\nonumber\\
		&(4+\frac{\int_{\R^2}\varphi_{\lambda}(T)}{\pi})\int_{t}^{T}K[\phi,\varphi^{o},\textbf{p}](s)ds-\nonumber\\
		&-\frac{1}{\pi}(T-t)\big(\int_{\R^2}\varphi_{\lambda}(T)\big)\partial_{t}\big(8\pi \alpha_{1}+\int_{\R^2}(\varphi_{\lambda}-\varphi_{\lambda_{0}})-16\pi \beta \frac{\lambda^{2}-\lambda_{0}^{2}}{\delta(T-t)}\big) \ \ \ \text{ sufficiently small.}
\end{align}
The idea is to fix $\lambda_{1}$ to erase the term in the second line of \eqref{ExpansionSecMomErr}. Notice that the last term cannot be included in the equation for $\lambda_{1}$ since we do not have sufficient information on its time derivative and it would be not possible to estimate the derivative of the remainders of the resulting equation and then obtain same estimate for $\dot{\alpha}_{1}$ in \eqref{ExpansionMassErr}. Then, we want to find $\lambda_{1}$ such that for any $t\in(0,T)$ and for some $\sigma>0$
\begin{align}\label{equationlambda1}
	4\int_{\mathbb{R}^{2}}(\varphi_{\lambda}-\varphi_{\lambda_{0}})dx&-4\int_{\mathbb{R}^{2}}(\varphi_{\lambda}(T)-\varphi_{\lambda_{0}}(T))dx-64\pi \beta \frac{\lambda^{2}-\lambda_{0}^{2}}{\delta(T-t)}+\nonumber\\
	&+(4+\frac{\int_{\R^2}\varphi_{\lambda}(T)}{\pi})\int_{t}^{T}K[\phi,\varphi^{o},\textbf{p}](s)ds=O(e^{-(\frac{3}{2}+\sigma)\sqrt{2|\ln(T-t)|}})
\end{align}
that can be rewritten as
\begin{align}\label{equationlambda1final}
	4\int_{\mathbb{R}^{2}}(\varphi_{\lambda}-\varphi_{\lambda_{0}})dx&-4\int_{\mathbb{R}^{2}}(\varphi_{\lambda}(T)-\varphi_{\lambda_{0}}(T))dx-64\pi \beta \frac{\lambda^{2}-\lambda_{0}^{2}}{\delta(T-t)}+\nonumber\\
	&+(4+\frac{\int_{\R^2}\varphi_{\lambda}(T)}{\pi})\int_{t}^{T}\int_{\mathbb{R}^{2}}U\varphi^{o}\tilde{\chi}dyds=O(e^{-(\frac{3}{2}+\sigma)\sqrt{2|\ln(T-t)|}}).
\end{align}
We notice that thanks to Remark \ref{outersolutionRMKNegTim} the right-hand side is well defined in $t\in(-\varepsilon(T),T)$. Another crucial observation is that $\int_{t}^{T}\int_{\mathbb{R}^{2}}U\varphi^{o}\tilde{\chi}dyds$ is differentiable (this will give us, similarly to Proposition \ref{prop-lambda0}, control of the remainders of the nonlocal equation). To get control of $\dot{\alpha}_{1}$ we should also differentiate \eqref{ExpansionMassErr} and then the right-hand side remainder in \eqref{equationlambda1final}. Similarly to Proposition \ref{prop-lambda0} we are going to find $\lambda_{1}$ such that 
\begin{align*}
		4\int_{\mathbb{R}^{2}}(\varphi_{\lambda}-\varphi_{\lambda_{0}})dx&-4\int_{\mathbb{R}^{2}}(\varphi_{\lambda}(T)-\varphi_{\lambda_{0}}(T))dx-64\pi \beta \frac{\lambda^{2}-\lambda_{0}^{2}}{\delta(T-t)}+\nonumber\\
		&+(4+\frac{\int_{\R^2}\varphi_{\lambda}(T)}{\pi})\int_{t}^{T}\int_{\mathbb{R}^{2}}U\varphi^{o}\tilde{\chi}dyds=\mathcal{E}(t)
\end{align*} 
and when we differentiate $\mathcal{E}(t)$ there will be a loss due to the lack of regularity of the right-hand side (we are not assuming any control of $\partial_{t}\varphi^{o}$). This will give a slightly worse decay for $\dot{\alpha}_{1}$, but we will have enough decay to get a contraction. \newline
Now we want to study the last and the most critical term in \eqref{ExpansionSecMomErr}. After solving the equation \eqref{equationlambda1} and thanks to \eqref{ZeromassSCondition} we have
\begin{align*}
	m_{2}[f]&\approx\frac{1}{\pi}\big(\int_{\R^2}\varphi_{\lambda}(T)\big)(T-t)\int_{\R^2}K[\phi,\varphi^{o},\textbf{p}]\approx\frac{1}{\pi}\big(\int_{\R^2}\varphi_{\lambda}(T)\big)(T-t)\int_{\R^2}U(y)\varphi^{o}(y\lambda,t)\tilde{\chi}dy
\end{align*}
and we will prove in Section \ref{proofThm1} that this will be sufficiently small to get a contraction. \newline
Finally we remark that in \cite{DdPDMW} introducing $\lambda_{1}$ was not necessary since $\int_{t}^{T}K(\varphi^{o},\phi,\textbf{p})(s)ds$ was already sufficiently small. In our case an analogous approach cannot be used since in the integral $\int_{t}^{T}\frac{1}{\lambda^{2}(s)}\int_{\mathbb{R}^{2}}\lambda^{2}(s)U\varphi^{o}\tilde{\chi}dyds$ an extra logarithm is produced. \newline
Finally we can write that at main order $\lambda_{1}$ will solve
\begin{align}\label{EquationLambda1}
	4\int_{\mathbb{R}^{2}}(\varphi_{\lambda}-\varphi_{\lambda_{0}})dx&-4\int_{\mathbb{R}^{2}}(\varphi_{\lambda}(T)-\varphi_{\lambda_{0}}(T))dx-64\pi \beta \frac{\lambda^{2}-\lambda_{0}^{2}}{\delta(T-t)}+\nonumber\\
	&+(4+\frac{\int_{\R^2}\varphi_{\lambda}(T)}{\pi})\int_{t}^{T}\int_{\mathbb{R}^{2}}U\varphi^{o}\tilde{\chi}dyds\approx 0
\end{align} 
and $\alpha_{1}$ will be given by
\begin{align}\label{alpha1Expansion}
	\alpha_{1}(t)\approx \frac{1}{32\pi^{2}}\int_{t}^{T}\int_{\R^2}U\varphi^{o}\tilde{\chi}dyds.
\end{align} 
\newline 

Now we want to find an useful formula for the center of mass. In the following we will assume 
\begin{align}\label{ExtraAssumptionSECMASSSECMOM}
	|\lambda(t)-\lambda_{0}(t)|\le C e^{-\sqrt{2|\ln(T-t)|}},  \ \ |\lambda \dot{\lambda}-\lambda_{0}\dot{\lambda}_{0}|\le C e^{-\frac{3}{2}\sqrt{2|\ln(T-t)|}}
\end{align}
where $\lambda_{0}$ is given by Proposition \ref{prop-lambda0}. This assumption will be completely justified in Section \ref{proofThm1}.
\begin{lemma}\label{LemmaFirstMomDIff}
		Let us assume \eqref{EstiLambda}, \eqref{EstiPar} and \eqref{ExtraAssumptionSECMASSSECMOM}. Let $\lambda=\lambda_{0}+\lambda_{1}$, $\alpha=\alpha_{0}+\alpha_{1}$ and $\xi=\xi_{1}(t)$  with $\lambda_{0}$ given by Proposition \ref{prop-lambda0} and $\alpha_{0}$ defined in \eqref{MASSU1}.
	    For any $t\in(0,T)$ we have
   \begin{align*}
			m_{i,j}[\lambda^{4}S(u_{1}(\textbf{p}_{0}+\textbf{p}_{1}))\tilde{\chi}-\lambda_{0}^{4}S(u_{1}(\textbf{p}_{0}))\tilde{\chi}_{0}\tilde{\chi}]=&\lambda \alpha \dot{\xi}_ {1,j}\int_{\R^2}\partial_{y_{j}} U(y)y_{j}\tilde {\chi} dy+\\
			&+\frac{\alpha \lambda^{2}}{\sqrt{T-t}}\dot{\xi}_{1,j}\int_{\R^2}U(y)\partial_{z_{j}}\chi_{0}(\frac{\lambda y}{\sqrt{\delta (T-t)}})y_{j}dy+\\
			&+O(e^{-(\frac{5}{2}+\gamma_{1})\sqrt{2|\ln(T-t)|}}).
	\end{align*}
\end{lemma}
\begin{proof}
	We can isolate the term that give some contribution in $\lambda^{4}S(u_{1}(\textbf{p}_{0}+\textbf{p}_{1}))\tilde{\chi}$
	\begin{align*}
		\lambda^{4}S(u_{1}(\textbf{p}_{0}+\textbf{p}_{1}))\tilde{\chi}=&\alpha \lambda\dot{\xi}\cdot \nabla_{y}U \chi\tilde{\chi}+\alpha \frac{\lambda^{2}}{\sqrt{T-t}}\dot{\xi}_{1,j}\int_{\R^2}U(y)\partial_{z_{j}}\chi_{0}(\frac{\lambda y }{\sqrt{\delta (T-t)}})y_{j}dy+\\
		&+\lambda^{4}\big(-\frac{4}{r}\partial_{r}\varphi_{\lambda}-\nabla \cdot(\varphi_{\lambda}\nabla v_{0})-\nabla \cdot(u_{0}\nabla \psi_{\lambda})-\nabla \cdot(\varphi_{\lambda}\nabla \psi_{\lambda})-E(x,t)\big)\tilde{\chi}+\\
		&+(...).
	\end{align*}
We stress a difference with \cite{DdPDMW}. Many terms here are centered at the origin and then they give some non-zero contribution, see also the definition of $\varphi_{\lambda}$ in \eqref{eqphilambda}. At the same time here we will show that their contribution is negligible. As a significative term we can isolate the term.
\begin{align}
	-E(x,t)\tilde{\chi}&=-\lambda\dot{\lambda}Z_{0}(\frac{x}{\lambda})\tilde{\chi}+(...)\label{MainTermINE}.
\end{align}
Now we can observe that thanks to \eqref{EstiPar}
\begin{align*}
		\lambda\dot{\lambda}\int_{\R^2}Z_{0}(y+\frac{\xi}{\lambda})\tilde{\chi}y_{j}dy=&\int_{\R^2}\big(\lambda\dot{\lambda}Z_{0}(y+\frac{\xi}{\lambda})-\lambda\dot{\lambda}Z_{0}(y)\big)\tilde{\chi}y_{j}dy=\\
		=&-\lambda\dot{\lambda}\int_{\R^2}\nabla_{y}Z_{0}(y) \cdot \frac{\xi}{\lambda}y_{j}\tilde{\chi}dy+O(e^{-(3+\gamma_{1})\sqrt{2|\ln(T-t)|}})=\\
		=&O(e^{-(3+\gamma_{1})\sqrt{2|\ln(T-t)|}}).
\end{align*}
The other terms in \eqref{MainTermINE} can be estimated similarly. We can then isolate the most important terms \begin{align*}
	\lambda^{4}S(u_{1}(\textbf{p}_{0}+\textbf{p}_{1}))\tilde{\chi}=&\alpha \lambda\dot{\xi}\cdot \nabla_{y}U \chi\tilde{\chi}+\alpha \frac{\lambda^{2}}{\sqrt{T-t}}\dot{\xi}_{1,j}\int_{\R^2}U(y)\partial_{z_{j}}\chi_{0}(\frac{\lambda y }{\sqrt{\delta (T-t)}})y_{j}dy+\nonumber\\
	&+\lambda^{4}\big(-\frac{4}{r}\partial_{r}\varphi_{\lambda}-\nabla \cdot(\varphi_{\lambda}\nabla v_{0})-\nabla \cdot(u_{0}\nabla \psi_{\lambda})-\nabla \cdot(\varphi_{\lambda}\nabla \psi_{\lambda})+(...).
\end{align*}
To study the remaining terms first we observe that for any sufficiently fast decaying functions $\phi$, $g$
\begin{align}
	&\int_{\R^2}\nabla \cdot( \phi ( \nabla (-\Delta)^ {-1}\phi))x=0, \ \ \ \int_{\R^2}\big(\nabla \cdot(\phi (\nabla(-\Delta)^{-1}g))-\nabla \cdot(g \nabla (-\Delta)^{-1}\phi)\big)xdx=0. \label{propertyCenterOfMass2}
\end{align}
Then for istance we can observe that thanks to \eqref{EstiPar} and the control of the derivatives we get from Lemma \ref{estimatephilambda}
\begin{align*}
   |\int_{\R^2}\lambda^{4}\nabla \cdot(\varphi_{\lambda}\nabla v_{0}+u_{0}\nabla \psi_{\lambda})\tilde{\chi}y_{j}dy|&\le \lambda|\xi| \int_{|x|\ge \sqrt{\delta(T-t)}}\frac{e^{-\sqrt{2|\ln(|x|^{2})}}}{|x|^{5}}|x|^{2}dx=O( e^{-(3+\gamma_{1})\sqrt{2|\ln(T-t)|}}).
\end{align*}
Then we can isolate the main order terms
\begin{align}\label{ISOLATEFirstMomLAMBDA}
	\lambda^{4}S(u_{1}(\textbf{p}_{0}+\textbf{p}_{1}))\tilde{\chi}=&\alpha \lambda\dot{\xi}\cdot \nabla_{y}U \chi\tilde{\chi}+\alpha \frac{\lambda^{2}}{\sqrt{T-t}}\dot{\xi}_{1,j}\int_{\R^2}U(y)\partial_{z_{j}}\chi_{0}(\frac{\lambda y }{\sqrt{\delta (T-t)}})y_{j}dy+\nonumber\\
	&-\lambda^{4}\frac{4}{r}\partial_{r}\varphi_{\lambda}\tilde{\chi}+(...).
\end{align}
To manage the remaining term we need to expand $\lambda_{0}^{4}S(u_{1}(\textbf{p}_{0}))\tilde{\chi}_{0}\tilde{\chi}$. Recalling that $\bar{y}_{0}=\frac{x}{\lambda_{0}}$ and $\bar{w}=\frac{x}{\sqrt{\delta(T-t)}}$, we get
\begin{align*}
	S(u_{1}(\textbf{p}_{0}))=&-\frac{\dot{\alpha}_{0}}{\lambda^{2}_{0}}U(\bar{y}_{0})\chi_{0}(\bar{w})+(\alpha_{0}-1) \frac{\dot{\lambda}_{0}}{\lambda^{3}_{0}}Z_{0}(\bar{y}_{0})\chi_{0}(\bar{w})+\\ 
	&-\frac{\alpha_{0}-1}{2\lambda^{2}_{0}(T-t)}U(\bar{y}_{0})\nabla_{\bar{w}}\chi_{0}(\bar{w})\cdot \bar{w}+\frac{2(\alpha_{0}-1)}{\lambda_{0}^{3}\sqrt{\delta(T-t)}}\nabla_{\bar{w}}\chi_{0}(\bar{w})\cdot \nabla_{\bar{y}}U(\bar{y}_{0})+ \nonumber\\
	&+\frac{\alpha_{0}-1}{\delta(T-t)}\frac{1}{\lambda_{0}^{2}}\Delta_{\bar{w}}\chi_{0}(\bar{w})U(\bar{y}_{0})-\frac{1-1/\alpha_{0}}{\lambda_{0}^{2}\sqrt{\delta(T-t)}}U(\bar{y}_{0})\nabla_{\bar{w}}\chi_{0}(\bar{w})\cdot \nabla_{x}v_{0}-\nonumber\\
	&-\frac{\alpha_{0}(\alpha_{0}-1)\chi_{0}(\bar{w})}{\lambda_{0}^{4}}\nabla_{\bar{y}}\cdot(U\nabla_{\bar{y}}\Gamma_{0})+\frac{\alpha_{0} \chi_{0}(\bar{w})}{\lambda_{0}^{4}}\big[\alpha_{0}(\chi_{0}(w)-1)U^{2}-\nabla_{\bar{y}_{0}}U\cdot \nabla_{\bar{y}_{0}}\mathcal{R}(\bar{y}_{0})\big]\nonumber\\
	&-\frac{\alpha_{0}-1}{\lambda_{0}^{2}\sqrt{\delta(T-t)}}U(\bar{y}_{0})\nabla_{\bar{w}}\chi_{0}(\bar{w})\cdot \nabla_{x}v_{0}-\frac{4}{r}\partial_{r}\varphi_{\lambda}-\nabla \cdot(\varphi_{\lambda}\nabla v_{0})-\nabla \cdot(u_{0}\nabla \psi_{\lambda})-\\
	&-\nabla \cdot(\varphi_{\lambda}\nabla \psi_{\lambda}).
\end{align*}
Almost all of these terms gives small contribution, for instance recalling \eqref{EstiPar} we have
\begin{align*}
	|\dot{\alpha}_{0}\lambda_{0}^{2}\int_{\R^2}U(\frac{\lambda_{0}}{\lambda}(y+\frac{\xi}{\lambda_{0}}))y_{j}\tilde{\chi_{0}}\tilde{\chi}dy|\le |e^{-2\sqrt{2|\ln(T-t)|}} | |\frac{\xi}{\lambda}|=O(e^{-(3+\gamma_{1})\sqrt{2|\ln(T-t)|}}).
\end{align*}
Managing the terms in divergence form involving $\varphi_{\lambda}$, recalling \eqref{propertyCenterOfMass2}, we isolate the main order term writing
\begin{align}\label{ISOLATEFirstMomLAMBDA0}
	S(u_{1}(\textbf{p}_{0}))=&-\frac{4}{r}\partial_{r}\varphi_{\lambda}+(...).
\end{align}
Keeping in mind \eqref{ISOLATEFirstMomLAMBDA}, \eqref{ISOLATEFirstMomLAMBDA0}, the desired conclusion is a consequence of the control of the second derivative of $\varphi_{\lambda}$ we have from Lemma \ref{estimatephilambda} and \eqref{ExtraAssumptionSECMASSSECMOM}. Directly estimating $\varphi_{\lambda}-\varphi_{\lambda_{0}}$ as in Lemma \ref{estimatephilambda} and observing that both $\varphi_{\lambda}$ and $\varphi_{\lambda_{0}}$ are centered at the origin we have
\begin{align*}
	|\lambda^{4}\int_{\R^2}(\frac{1}{r}\partial_{r}\varphi_{\lambda}-\frac{1}{r}\partial_{r}\varphi_{\lambda_{0}})y_{j}\tilde{\chi}dy|\le \frac{|\xi|}{\lambda} |\lambda\dot{\lambda}-\lambda_{0}\dot{\lambda_{0}}|\le C  e^{-(\frac{5}{2}+\gamma_{1})\sqrt{2|\ln(T-t)|}}.
\end{align*}
\end{proof}
\section{Proof of Theorem \ref{theorem1}}\label{proofThm1}
We define norms, adapted to the terms in the inner linear problems \eqref{innereqfinalFINAL}, \eqref{innereqfinalFINAL2} and \eqref{innereqfinalFINALMODE1}. First we introduce the variable $\tau$
\begin{align}\label{tauvar}
	\tau(t)=\tau_{0}+\int_{0}^{t}\frac{ds}{\lambda^{2}(s)}, \ \  t\in(0,T),
\end{align}
where $\tau_{0}$ is large and has been prescribed in Sections \ref{InnerTheoryIntro}, \ref{PreliminariesInnTh}, \ref{ProofProp71}, \ref{Section11} and \ref{Section12}.
Notice that $0<\tau \to \infty$ if $t\to T$, then we are studying now infinite time problems. The inner equations will be solved in the variables $y$ and $\tau$. Given positive numbers $\nu$, $p$, $\epsilon$ and $m\in \mathbb{R}$, we let
\begin{align}\label{RHSinnernorm}
	\|h\|_{0;\nu,m,p,\epsilon}=\inf \left\{ K \text{ s.t. } |h(y,\tau)|\le \frac{K}{\tau^{\nu}(\ln \tau)^{m}}\frac{1}{(1+|y|)^{p}}\begin{cases}
		1 \ \ \ \ |y|\le \sqrt{\tau}\\
		\frac{\tau^{\epsilon/2}}{|y|^{\epsilon}} \ \ \ |y|\ge \sqrt{\tau}
	\end{cases}\right\}.
\end{align}
We also define
\begin{align}\label{SolutionINNERNORM}
	\|\phi\|_{1;\nu,m,p,\epsilon}=\inf \left\{K \text{ s.t. } |\phi(y,t)|+(1+|y|)|\nabla_{y}\phi(y,t)|\le \frac{K}{\tau^{\nu}(\ln \tau)^{m}}\frac{1}{(1+|y|)^{p}}\begin{cases}
		1 \ \ \ \ |y|\le \sqrt{\tau}\\
		\frac{\tau^{\epsilon/2}}{|y|^{\epsilon}} \ \ \ \ |y|\ge \sqrt{\tau}
	\end{cases}\right\}.
\end{align}
We develop a solvability theory of problems \eqref{innereqfinalFINAL}, \eqref{innereqfinalFINAL2} and \eqref{innereqfinalFINALMODE1} that involves uniform space-time bounds in terms of the above norms. At this point we want to make explicit the initial conditions in the equations \eqref{innereqfinalFINAL} and \eqref{innereqfinalFINAL2}. We introduce the function $\tilde{Z}_{0}$ defined as
\begin{align}\label{Z0tilde}
	\tilde{Z}_{0}(|y|)=(Z_{0}(|y|)-m_{Z_{0}}U)\chi_{0}(\frac{|y|}{\sqrt{\tau_{0}}})
\end{align}
where $m_{Z_{0}}$ is such that $\int_{\mathbb{R}^{2}} \tilde{Z}_{0}=0$. In what follows we state some fundamental Propositions that will be proved in Section \ref{InnerTheoryIntro}, \ref{PreliminariesInnTh}, \ref{Section11} and \ref{Section12}. The only difference with the statement we will prove in those Sections is that here we are considering the cut-off $\hat{\chi}$ (it is sufficient to rename $\delta$) and we are also including in the norm of the solution the gradient (that can be estimated by standard parabolic estimates). We will also use the following notation
\begin{align}\label{MhatFunction}
	\hat{\chi}(x,t)=\chi_{0}(\frac{|x-\xi|}{4\sqrt{\delta(T-t)}})=\chi_{0}(\frac{|y|\lambda}{\hat{M}(\tau)}).
\end{align}
\begin{proposition}\label{PropMode0Mass0ANTICIP}
	Let $\lambda=\lambda_{0}+\lambda_{1}$ with $\lambda_{0}$ given by Proposition \ref{prop-lambda0} and let us assume \eqref{ExtraAssumptionSECMASSSECMOM}. Let $\sigma\in(0,1)$, $\epsilon>0$ with $\sigma+\epsilon<2$, $1<\nu<\frac{7}{4}$, $m\in \mathbb{R}$. Let $q\in(0,1)$. For $\tau_{0}$ sufficiently large and for all radially symmetric $h=h(|y|,\tau)$ with $\|h\|_{\nu,m,6+\sigma,\epsilon}<\infty$ and such that
	\begin{align*}
		\int_{\mathbb{R}^{2}}h(y,t)dy=0, \ \ \ \text{for all }(\tau_{0},\infty)
	\end{align*}
	there exists $c_{1}\in \mathbb{R}$ and a solution $\phi(y,\tau)=\mathcal{T}_{p}^{i,2}[h]$ of 
	\begin{align*}
		\begin{cases}
			\partial_{\tau}\phi = L[\phi]+B[\phi \hat{\chi}]+h \\
			\phi(\cdot, \tau_{0})=c_{1}\tilde{Z}_{0}
		\end{cases}
	\end{align*}
	that define a linear operators of $h$ and satisfy the estimate
	\begin{align*}
		\|\phi\|_{1;\nu-1,m+q+1,4,2+\sigma+\epsilon}\le \frac{C}{(\ln \tau_{0})^{1-q}}\|h\|_{0;\nu,m,6+\sigma,\epsilon}.
	\end{align*}
	\begin{align*}
		|c_{1}|\le C \frac{1}{\tau_{0}^{\nu-1}(\ln \tau_{0})^{m+2}}\|h\|_{\nu,m,6+\sigma,\epsilon}.
	\end{align*} 
\end{proposition}
	The proof of Proposition \ref{PropMode0Mass0ANTICIP} will be given in Section \ref{ProofProp71}.
\begin{proposition}\label{PropMode0MassSecMom0ANTICIP}
Let $\lambda=\lambda_{0}+\lambda_{1}$ with $\lambda_{0}$ given by Proposition \ref{prop-lambda0} and let us assume \eqref{ExtraAssumptionSECMASSSECMOM}. Let $\sigma\in(0,1)$, $\varepsilon>0$, $\sigma+\varepsilon<2$, $1<\nu<\min(1+\frac{\varepsilon}{2},3-\frac{\sigma}{2},\frac{3}{2})$, $m\in \mathbb{R}$. Let $0<q<1$. There exists a number $C>0$ such for $\tau_{0}$ sufficiently large and for all radially symmetric $h=h(|y|,\tau)$ with $\|h\|_{0;\nu,m,6+\sigma,\epsilon}<\infty$ and
	\begin{align*}
		\int_{\mathbb{R}^{2}}h(y,\tau)dy=0, \ \ \ \text{for all }(\tau_{0},\infty), \ \ \ \int_{\mathbb{R}^{2}}h(y,\tau)|y|^{2}dy=0 \ \ \ \ \text{for all}(\tau_{0},\infty)
	\end{align*}
	there exists $c_{1}\in \mathbb{R}$, an operator $\mathcal{E}$ and a solution $\phi(y,\tau)=\mathcal{T}_{p}^{i,1}[h]$ of 
	\begin{align*}
		\begin{cases}
			\partial_{\tau}\phi = L[\phi]+B[\phi]+\mathcal{E}[\phi]+h \\
			\phi(\cdot, \tau_{0})=c_{1}L[\tilde{Z}_{0}]
		\end{cases}
	\end{align*}
	that defines a linear operator of $h$ and satisfies the estimate
	\begin{align*}
		\|\phi\|_{1;\nu-\frac{1}{2},m+\frac{q+1}{2},4,2}\le C\|h\|_{0;\nu,m,6+\sigma,\epsilon}.
	\end{align*}
	Moreover we know that $\mathcal{E}[\phi]$ is radial, it satisfies $\int_{\R^2}\mathcal{E}[\phi](y)dy=0$, $ \int_{\R^2}\mathcal{E}[\phi]|y|^{2}dy=0$ and
	\begin{align}\label{EstimateEFORGLUING}
		\ \ |\mathcal{E}[\phi](y)|dy\le C \|h\|\frac{1}{\tau^{\nu}\ln^{m+q}\tau} \begin{cases}
			\frac{1}{\hat{M}^{2}(\tau)}\frac{1}{(1+\rho)^{6}} \ \ \ \ |y|\le \hat{M}(\tau) \\
			\frac{1}{1+|y|^{6}} \ \ \ |y|\ge \hat{M}(\tau).
		\end{cases}
	\end{align}
	Moreover $c_{1}$ is a linear operator of $h$ and
	\begin{align*}
		|c_{1}|\le C \frac{1}{\tau_{0}^{\nu-1}(\ln \tau_{0})^{m+2}}\|h\|_{0,\nu,m,6+\sigma,\epsilon}.
	\end{align*} 	
\end{proposition}
The proof of Proposition \ref{PropMode0MassSecMom0ANTICIP} will be given in Section \ref{Section11}.
\begin{proposition}\label{PropMode1ANTICIP}
Let $\lambda=\lambda_{0}+\lambda_{1}$ with $\lambda_{0}$ given by Proposition \ref{prop-lambda0} and let us assume \eqref{ExtraAssumptionSECMASSSECMOM}. Let $0<\varsigma<1$, $0<\epsilon<2$, $1<\nu<\min(1+\frac{\varepsilon}{2},\frac{3}{2}-\frac{\varsigma}{2})$ and $m\in \mathbb{R}$. There exists a number $C>0$ such for $\tau_{0}$ sufficiently large and for all $h=h(y,\tau)$ such that $[h]_{\text{rad}}=0$ with $\|h\|_{\nu,m,5+\varsigma,\epsilon}<\infty$ and
	\begin{align*}
		\int_{\mathbb{R}^{2}}h(y,\tau)y_{j}dy=0,  \ \ j=1,2\ \ \ \text{for all }(\tau_{0},\infty),
	\end{align*}
	there exists an operator $\mathcal{F}[\phi]$ and a solution $\phi(y,\tau)=\mathcal{T}_{p}^{i,3}[h]$ of 
	\begin{align*}
		\begin{cases}
			\partial_{\tau}\phi = L[\phi]+B[\phi \hat{\chi}]+\mathcal{F}[\phi]+h \\
			\phi(\cdot, \tau_{0})=0
		\end{cases}
	\end{align*}
	that defines a linear operator of $h$ and satisfies the estimate
	\begin{align*}
		\|\phi\|_{1;\nu,m,3+\varsigma,2+\epsilon}\le C\|h\|_{0;\nu,m,5+\varsigma,\epsilon}.
	\end{align*}
	Moreover $\mathcal{F}[\phi]$ does not have radial mode and 
	\begin{align}\label{EstimateFFORGLUING}
		|\mathcal{F}[\phi]|\le \frac{1}{\tau^{\nu+1}(\ln \tau)^{m-1}}\frac{1}{\hat{M}^{\varsigma}}(|W_{1}(y)|+|W_{2}(y)|).
	\end{align}
\end{proposition}
The proof of Proposition \ref{PropMode1ANTICIP} will be given in Section \ref{Section12}.\newline
At this point we want to consider the linear outer problem
\begin{align}\label{OuterEqfinSecANTICIP}
	\begin{cases}
		\partial_{t}\phi^{o}=L^{o}[\phi^{o}]+g(x,t), \ \ \ \text{in }\mathbb{R}^{2}\times(0,T)\\
		\phi^{o}(\cdot,0)=0, \ \ \text{in }\mathbb{R}^{2},
	\end{cases}
\end{align}
where 
\begin{align*} 
	L^{o}[\phi]:=\Delta_{x}\phi-\nabla_{x}\big[\Gamma_{0}\big(\frac{x-\xi(t)}{\lambda(t)}\big)\big]\cdot \nabla_{x}\phi=\Delta_{x}\phi+4\frac{x-\xi}{\lambda^{2}+|x-\xi|^{2}}\cdot \nabla_{x}\phi.
\end{align*}
For $g:\mathbb{R}^{2}\times(0,T)\to \mathbb{R}$ we consider the norm $\|g\|_{\star\star,0}$ defined as the least $K$ such that for all $(x,t)\in \mathbb{R}^{2}\times(0,T)$
\begin{align*}
	|g(x,t)|\le K \begin{cases}
		\frac{e^{-a\sqrt{2|\ln(|x|^{2}+(T-t))|}}}{(|x|^{2}+(T-t))^{2}|\ln((T-t)+|x|^{2})|^{b} }\ \ \ \ \text{if } 0\le |x|\le \sqrt{T}\\[7pt] 
		\frac{e^{-a\sqrt{2|\ln T|}}}{T^{2}|\ln T|^{b}}e^{-\frac{|x|^{2}}{4(t+T)}} \ \ \ \ \text{if }|x|\ge \sqrt{T}.
	\end{cases}
\end{align*}
where $a>0$, $b\in \mathbb{R}$ to be fixed. We also define the norm $\|\phi\|_{\star,o}$ as the least $K$ such that 
\begin{align}\label{OUTERNORMSGLU}
	|\phi^{o}(x,t)|+(\lambda+|x|)|\nabla_{x}\phi^{o}(x,t)|\le K\begin{cases}
		\frac{e^{-a\sqrt{2|\ln(|x|^{2}+(T-t))|}}}{(|x|^{2}+(T-t))|\ln((T-t)+|x|^{2})|^{b} }\ \ \ \ \text{if } 0\le |x|\le \sqrt{T}\\[7pt] 
		\frac{e^{-a\sqrt{2|\ln T|}}}{T|\ln T|^{b}}e^{-\frac{|x|^{2}}{4(t+T)}} \ \ \ \ \text{if }|x|\ge \sqrt{T}.
	\end{cases}
\end{align}
\begin{proposition}\label{Prop131ANTICIP}
	Let $a>0$, $b\in \mathbb{R}$ and $\lambda$, $\xi$, satisfy \eqref{EstiLambda}, \eqref{EstiPar}. Then there is a constant $C$ so that for $T$ sufficiently small and for $\|g\|_{\star\star,o}<\infty$ there exists a solution $\phi^{o}=\mathcal{T}^{o}_{p}[g]$ of \eqref{OuterEqfinSec}, which defines a linear operator of $g$ and satisfies
	\begin{align*}
		\|\phi^{o}\|_{\star,o}\le C \|g\|_{\star\star,o}.
	\end{align*}
\end{proposition}
The proof of Proposition \ref{Prop131ANTICIP} will be given in Section \ref{OuterProblemSection}. \newline
In what follows we work with $\textbf{p}=\textbf{p}_{0}+\textbf{p}_{1}=(\lambda_{0},\alpha_{0},0)+(\lambda_{1},\alpha_{1},\xi_{1})$. Recalling the definition of the operator $\mathcal{T}_{\textbf{p}}^{i,1}$, we let
\begin{align}\label{InnerOPER1}
	\mathcal{A}_{i,1}[\phi_{1},\phi_{2},\phi_{3},\varphi^{o},\textbf{p}_{1}]=\mathcal{T}^{i,1}_{\textbf{p}}[f-m_{0}[f]W_{0}-m_{2}[f]W_{2}],
\end{align}
\begin{align}\label{INNEROPER2}
	\mathcal{A}_{i,2}[\phi_{1},\phi_{2},\phi_{3},\varphi^{o},\textbf{p}_{1}]=\mathcal{T}^{i,2}_{\textbf{p}}[m_{2}[f]W_{2}],
\end{align}
\begin{align}\label{INNEROPERT3}
	\mathcal{A}_{i,3}[\phi_{1},\phi_{2},\phi_{3},\varphi^{o},\textbf{p}_{1}]=\mathcal{T}^{i,3}_{\textbf{p}}[f_{3}-m_{1,j}[f_{3}]W_{1,j}],
\end{align}
\begin{align*}
	\mathcal{A}_{o}[\phi_{1},\phi_{2},\phi_{3},\varphi^{o},\textbf{p}_{1}]=\mathcal{T}^{o}_{\textbf{p}}[G_{2}(\phi_{1}+\phi_{2}+\phi_{3},\varphi^{o},\textbf{p}_{1})].
\end{align*}
The the equations \eqref{innereqfinalFINAL}, \eqref{innereqfinalFINAL2}, \eqref{innereqfinalFINALMODE1} and \eqref{outereqfinalFINAL} can be written as 
\begin{align*}
   \phi_{1}=	\mathcal{A}_{i,1}[\phi_{1},\phi_{2},\phi_{3},\varphi^{o},\textbf{p}_{1}], \ \ \phi_{2}=\mathcal{A}_{i,2}[\phi_{1},\phi_{2},\phi_{3},\varphi^{o},\textbf{p}_{1}], \ \ \phi_{3}=\mathcal{A}_{i,3}[\phi_{1},\phi_{2},\phi_{3},\varphi^{o},\textbf{p}_{1}], 
\end{align*}
\begin{align*} \varphi^{o}=\mathcal{A}_{o}[\phi_{1},\phi_{2},\phi_{3},\varphi^{o},\textbf{p}_{1}]
\end{align*}
and we recall that we need to satisfy the orthogonality conditions
\begin{align}\label{OrthogonaliTYCONDITIONSGLUIGN} 
	\begin{cases}
		m_{0}[f[\phi_{1}+\phi_{2}+\phi_{3},\varphi^{o},\textbf{p}_{1}]]=0\\
		m_{1,j}[f_{3}[\phi_{1}+\phi_{2}+\phi_{3},\varphi^{o},\textbf{p}_{1}]]=0, \ \ \ j=1,2,  \ \ \ \ \text{if } t\in (0,T).
	\end{cases}
\end{align}
\newline

We consider first the zero mass condition, thanks to Lemma \ref{lemmamassdiff} we have
\begin{align}
	m_{0}[f]&=m_{0}[F\tilde{\chi}+\lambda^{4}S(u_{1}(\textbf{p}_{0}+\textbf{p}_{1}))\tilde{\chi}-\lambda_{0}^{4}S(u_{1}(\textbf{p}_{0}))\tilde{\chi}_{0}\tilde{\chi}+[E_{2}(\bar{y})-E_{2}(y)]\tilde{\chi}_{2}\tilde{\chi}+E_{2}\tilde{\chi}_{2}\tilde{\chi}-\nonumber\\
	&\hspace{0.9cm}-\mathcal{E}[\phi_{1}]\tilde{\chi}-\mathcal{F}[\phi_{3}\tilde{\chi}]]=\nonumber\\
	&=-\lambda^{2}\partial_{t}\big[8\pi \alpha_{1}+\int_{\mathbb{R}^{2}}(\varphi_{\lambda}-\varphi_{\lambda_0})dx-16\pi \beta\frac{\lambda^{2}-\lambda_{0}^{2}}{\delta(T-t)}\big]+O(e^{-(3-\rho)\sqrt{2|\ln(T-t)|}})+\nonumber\\
	&\hspace{1.5cm}+\int_{t}^{T}\frac{1}{\lambda^{2}(s)}m_{0}[F\tilde{\chi}+[E_{2}(\bar{y})-E_{2}(y)]\tilde{\chi}_{2}\tilde{\chi}+E_{2}\tilde{\chi}_{2}\tilde{\chi}-\mathcal{E}[\phi_{1}]\tilde{\chi}-\mathcal{F}[\phi_{3}\tilde{\chi}]](s)ds\big]=0\label{zeromassCondGlu}.
\end{align}
We call the $\alpha_{1}$ that vanishes at $T$ and that makes \eqref{zeromassCondGlu} true, the operator $\mathcal{A}_{p,\alpha_{1}}[\phi_{1},\phi_{2},\phi_{3},\varphi,\textbf{p}_{1}]$. \newline

For the conservation of mass, thanks to Lemma \ref{LemmaFirstMomDIff}, we get
\begin{align}
	m_{1,j}[f_{3}]=&\lambda \alpha \dot{\xi}_ {1,j}\int_{\R^2}\partial_{y_{j}} U(y)y_{j}\tilde {\chi} dy+\frac{\alpha \lambda^{2}}{\sqrt{T-t}}\dot{\xi}_{1,j}\int_{\R^2}U(y)\partial_{z_{j}}\chi_{0}(\frac{\lambda y}{\sqrt{\delta (T-t)}})y_{j}dy+\nonumber\\
	&+m_{i,j}[\big(F\tilde{\chi}+(E_{2}(\bar{y}_{0})-E_{2}(y))\tilde{\chi}_{2}\tilde{\chi}+{E}_{2}(t)\tilde{\chi}_{2}-\mathcal{E}[\phi_{1}]\tilde{\chi}-\mathcal{F}[\phi_{3}]\tilde{\chi}\big)]+\nonumber\\
	&+O(e^{-(\frac{5}{2}+\gamma_{1})\sqrt{2|\ln(T-t)|}})=0 \label{centMasszeroCondGLUE}
\end{align}
We call the $\xi_{j}$ that vanishes at $T$ and that makes \eqref{centMasszeroCondGLUE} true, the operator $\mathcal{A}_{p,\xi_{j}}[\phi_{1},\phi_{2},\phi_{3},\varphi,\textbf{p}_{1}]$.
\newline

Now we want to consider the second moment, thanks to Lemma \ref{lemmasecmomdiff}
\begin{align*}
	m_{2}[f]&=m_{2}[F\tilde{\chi}+\lambda^{4}S(u_{1}(\textbf{p}_{0}+\textbf{p}_{1}))\tilde{\chi}-\lambda_{0}^{4}S(u_{1}(\textbf{p}_{0}))\tilde{\chi}_{0}\tilde{\chi}+[E_{2}(\bar{y})-E_{2}(y)]\tilde{\chi}_{2}\tilde{\chi}+E_{2}\tilde{\chi}_{2}\tilde{\chi}-\nonumber\\
	&\hspace{0.9cm}-\mathcal{E}[\phi_{1}]\tilde{\chi}-\mathcal{F}[\phi_{3}\tilde{\chi}]]=\nonumber\\
	&=-(32\pi+8\int_{\mathbb{R}^{2}}\varphi_{\lambda}(T))\alpha_{1}\\
	&\hspace{0.35cm}-\frac{1}{\pi}\int_{\mathbb{R}^{2}}\varphi_{\lambda}(T)\big(\int_{\mathbb{R}^{2}}(\varphi_{\lambda}-\varphi_{\lambda}(\cdot,T))-\int_{\mathbb{R}^{2}}(\varphi_{\lambda_{0}}-\varphi_{\lambda_{0}}(\cdot,T))-16\pi \beta \frac{\lambda^{2}-\lambda_{0}^{2}}{\delta(T-t)}\big)-\\
	&\hspace{0.35cm}-\frac{1}{\pi}(T-t)\big(\int_{\R^2}\varphi_{\lambda}(T)\big)\partial_{t}\big(8\pi \alpha_{1}+\int_{\R^2}(\varphi_{\lambda}-\varphi_{\lambda_{0}})-16\pi \beta \frac{\lambda^{2}-\lambda_{0}^{2}}{\delta(T-t)}\big)+O(e^{-(2-\rho)\sqrt{2|\ln(T-t)|}})+\\
	&\hspace{0.35cm}+m_{2}[F\tilde{\chi}+(E_{2}(\bar{y})-E_{2}(y))\tilde{\chi}_{2}\tilde{\chi}+E_{2}\tilde{\chi}_{2}\tilde{\chi}-\mathcal{E}[\phi_{1}]\tilde{\chi}-\mathcal{F}[\phi_{3}\tilde{\chi}]].
\end{align*}
Similarly to Section \ref{SectionChoice} we cannot achieve the zero second moment condition and then our goal will be to make it as small as possible by fixing the parameter $\lambda_{1}$.
As we previously explained there is a delicate coupling with the zero mass condition \eqref{zeromassCondGlu} that justifies the choice
\begin{align}
		4\int_{\mathbb{R}^{2}}(\varphi_{\lambda_{0}+\lambda_{1}}-\varphi_{\lambda_{0}})dx&-4\int_{\mathbb{R}^{2}}(\varphi_{\lambda}(T)-\varphi_{\lambda_{0}}(T))dx-64\pi \gamma \frac{(\lambda_{0}+\lambda_{1})^{2}-\lambda_{0}^{2}}{\delta(T-t)}+\nonumber\\
		&+(4+\frac{\int_{\R^2}\varphi_{\lambda}(T)}{\pi})\int_{t}^{T}\int_{\mathbb{R}^{2}}U\varphi^{o}\tilde{\chi}dyds \ \ \ \ \text{ sufficiently small.}\label{EquationForLambda1}
\end{align}
The next result gives us a $\lambda_{1}$ that satisfies \eqref{EquationForLambda1}. We will need the following norm for $\kappa>0$, $m\in \mathbb{R}$
\begin{align}\label{norm1Lambda1}
	\|f\|_{1,\kappa,m}=\sup_{t\in[-\varepsilon(T),T]}e^{(1+\kappa)\sqrt{2|\ln(T-t)|}}|\ln(T-t)|^{m}(|f(t)|+(T-t)\sqrt{|\ln(T-t)|}|f'(t)|)|
\end{align}
\begin{proposition}\label{lambda1Lemma}
	Let $\kappa>0$, $m \in \mathbb{R}$ and $\lambda_{0}$ as in Proposition \ref{prop-lambda0}. For any $f$ such that $\|f\|_{1,\kappa,m}<\infty$
    there exists $\lambda_{1}=:\Lambda[f](t)$ that solves
    \begin{align}
    	4\int_{\mathbb{R}^{2}}(\varphi_{\lambda_{0}+\lambda_{1}}-\varphi_{\lambda_{0}})dx&-4\int_{\mathbb{R}^{2}}(\varphi_{\lambda_{0}+\lambda_{1}}(T)-\varphi_{\lambda_{0}}(T))dx-64\pi \beta \frac{(\lambda_{0}+\lambda_{1})^{2}-\lambda_{0}^{2}}{\delta(T-t)}+f(t)=\mathcal{R}[\lambda_{1}],
    \end{align}	
    and
    \begin{align*}
          &|\lambda_{1}(t)|\le C\|f\|_{1,\kappa,m}\sqrt{T-t} \frac{e^{-(\frac{1}{2}+\kappa)\sqrt{2|\ln(T-t)|}}}{|\ln(T-t)|^{m+1/2}}, \\ &|\dot{\lambda}_{1}(t)|\le C  \|f\|_{1,\kappa,m}\frac{e^{-(\frac{1}{2}+\kappa)\sqrt{2|\ln(T-t)|}}}{\sqrt{T-t}|\ln(T-t)|^{m+1/2}}\\
          &|\ddot{\lambda}_{1}(t)|\le C\|f\|_{1,\kappa,m} \frac{e^{-(1+\kappa)\sqrt{2|\ln(T-t)|}}}{(T-t)^{3/2}|\ln(T-t)|^{m+1/2}}.
    \end{align*}
    Moreover for any $0<\sigma<\frac{1}{2}$ we have
    \begin{align}\label{ERROREQLAMBDA1}
    	|\mathcal{R}[\lambda_{1}]|\le C\|f\|_{1,\kappa,m} \frac{e^{-(1+\kappa+\sigma)\sqrt{2|\ln(T-t)|}}}{|\ln(T-t)|^{m+1/2}}, \ \ \ \ |\frac{d}{dt}\mathcal{R}[\lambda_{1}]|\le C\|f\|_{1,\kappa,m}\frac{e^{-(1+\kappa)\sqrt{2|\ln(T-t)|}}}{(T-t)|\ln(T-t)|^{m}}.
    \end{align}
\end{proposition}
\begin{proof}
	See Section \ref{lambda1Section}.
\end{proof}

By Lemma \ref{lambda1Lemma}, if $\lambda_{1}=\Lambda[(4+\frac{\int_{\R^2}\varphi_{\lambda}(T)}{\pi})\int_{t}^{T}\int_{\mathbb{R}^{2}}U\varphi^{o}\tilde{\chi}dyds]$, using \eqref{zeromassCondGlu}, we get
\begin{align}
	m_{2}[f]&=\mathcal{R}[\lambda_{1}]-\nonumber\\
	&\hspace{0.35cm}-\frac{1}{\pi}\frac{T-t}{\lambda^{2}(t)}\big(\int_{\R^2}\varphi_{\lambda}(T)\big)m_{0}[F\tilde{\chi}+[E_{2}(\bar{y})-E_{2}(y)]\tilde{\chi}_{2}\tilde{\chi}+E_{2}\tilde{\chi}_{2}\tilde{\chi}-\mathcal{E}[\phi_{1}]\tilde{\chi}-\mathcal{F}[\phi_{3}\tilde{\chi}]]+\nonumber\\
	&\hspace{0.35cm}+m_{2}[F\tilde{\chi}+[E_{2}(\bar{y})-E_{2}(y)]\tilde{\chi}_{2}\tilde{\chi}+E_{2}\tilde{\chi}_{2}\tilde{\chi}-\mathcal{E}[\phi_{1}]\tilde{\chi}-\mathcal{F}[\phi_{3}\tilde{\chi}]]+O(e^{-(2-\rho)\sqrt{2|\ln(T-t)|}}).\label{zerosecmomCondGlu}
\end{align}
It is natural then to introduce the operator $\mathcal{A}_{p,\lambda_{1}}[\phi_{1},\phi_{2},\phi_{3},\varphi,\textbf{p}_{1}]=\Lambda[(4+\frac{\int_{\R^2}\varphi_{\lambda}(T)}{\pi})\int_{t}^{T}\int_{\mathbb{R}^{2}}U\varphi^{o}\tilde{\chi}dyds]$.
\newline

Then we define $\mathcal{A}_{p}$ by
\begin{align}\label{Ap}
	\mathcal{A}_{p}[\phi_{1},\phi_{2},\phi_{3},\varphi,\textbf{p}_{1}]=(\mathcal{A}_{p,\lambda_{1}}[\phi_{1},\phi_{2},\phi_{3},\varphi,\textbf{p}_{1}],\mathcal{A}_{p,\alpha_{1}}[\phi_{1},\phi_{2},\phi_{3},\varphi,\textbf{p}_{1}],\mathcal{A}_{p,\xi_{1}}[\phi_{1},\phi_{2},\phi_{3},\varphi,\textbf{p}_{1}]).
\end{align}
Then 
\begin{align*}
	\textbf{p}_{1}=\mathcal{A}_{p}[\phi_{1},\phi_{2},\phi_{3},\varphi,\textbf{p}_{1}]
\end{align*}
is equivalent to the equations \eqref{OrthogonaliTYCONDITIONSGLUIGN} with $\lambda_{1}$ solving \eqref{EquationForLambda1}. We write
\begin{align*}
	\vec{\phi}=(\phi_{1},\phi_{2},\phi_{3},\varphi,\textbf{p})
\end{align*}
and 
\begin{align*}
	\mathcal{A}[\vec{\phi}]=(\mathcal{A}_{i,1}[\vec{\phi}],\mathcal{A}_{i,2}[\vec{\phi}],\mathcal{A}_{i,3}[\vec{\phi}],\mathcal{A}_{p}[\vec{\phi}]).
\end{align*}
The goal is to find $\vec{\phi}$ such that $\vec{\phi}=\mathcal{A}[\vec{\phi}]$. We define the spaces on which we will consider the operator $\mathcal{A}$ to set up the fixed point problem. For certain choices of constants $\nu, q, \sigma,\varsigma, \epsilon,a,b,\varrho,m_{1},m_{2},m_{3},\mu$ that we will make precise later (we remark that the parameter $\epsilon$ is fundamental in the definition of the operators \eqref{InnerOPER1}, \eqref{INNEROPER2})
\begin{align*}
	X_{i,1}=\{\phi_{1} \in L^{\infty}(\mathbb{R}^{2}\times(0,T))\ |& \ \nabla_{y}\phi_{1}\in L^{\infty}(\mathbb{R}^{2}\times(0,T)), \ \|\phi_{1} \|_{1;\nu-\frac{1}{2},\frac{q+1}{2},4,2}<\infty, \\
	& \int_{\R^2}\phi_{1}(y,t) dy=0,  \ \int_{\R^2}\phi_{1}(y,t) ydy=0, \ \int_{\R^2}\phi_{1}(y,t)|y|^{2}dy=0, \  t\ge0 \},
\end{align*}
\begin{align*}
	X_{i,2}=\{\phi_{2} \in L^{\infty}(\mathbb{R}^{2}\times(0,T))\ |& \ \nabla_{y}\phi_{2}\in L^{\infty}(\mathbb{R}^{2}\times(0,T)), \ \|\phi_{2} \|_{1;\nu-\frac{1}{2},\frac{q+1}{2},4,2+\sigma+\epsilon}<\infty, \\
	& \int_{\R^2}\phi_{2}(y,t) dy=0,  \ \int_{\R^2}\phi_{2}(y,t) ydy=0, \  t\ge0 \},
\end{align*}
\begin{align*}
	X_{i,3}=\{\phi_{3} \in L^{\infty}(\mathbb{R}^{2}\times(0,T))\ |& \ \nabla_{y}\phi_{3}\in L^{\infty}(\mathbb{R}^{2}\times(0,T)), \ \|\phi_{3} \|_{1;\nu,0,3+\varsigma,{2+\epsilon}}<\infty, \\
	& \int_{\R^2}\phi_{3}(y,t) dy=0,  \ \int_{\R^2}\phi_{3}(y,t) ydy=0, \ \int_{\R^2}\phi_{3}(y,t)|y|^{2}dy=0, \  t\ge0 \},
\end{align*}
\begin{align*}
	&X_{o}=\{\varphi^{o} \in L^{\infty}(\mathbb{R}^{2}\times[0,T)) \ | \ \nabla_{y}\varphi^{o} \in L^{\infty}(\mathbb{R}^{2}\times(0,T)) \ \ \|\varphi^{o}\|_{\star,0}<\infty\},\\
	&X_{p}=\{(\lambda_{1},\alpha_{1},\xi_{1})\in C^{1}([0,T)) \ | \ \|\alpha_{1}\|^{(1)}_{C^{1};\nu+\frac{1}{2},m_{1}}<\infty, \ \|\xi_{1}\|^{(2)}_{C^{1};\varrho,m_{2}}<\infty, \  \|\lambda_{1}\|_{C^{2};\nu+\frac{1}{2},m_{3}}<\infty \}
\end{align*}
where the norms $\|\phi \|_{1;\nu-\frac{1}{2},\frac{q+1}{2},4,2}$ and $\|\varphi\|_{\star,o}$ are defined in \eqref{SolutionINNERNORM}, \eqref{OUTERNORMSGLU} and $\|\alpha_{1}\|^{(1)}_{C^{1};\nu+\frac{1}{2},m_{1}}$, $\|\xi_{1}\|^{(2)}_{C^{1};\varrho,m_{2}}$,  $\|\lambda_{1}\|_{C^{2};\nu+\frac{1}{2},m_{3}}$, are defined as
\begin{align}\label{normforalpha1}
	\|\alpha_{1}\|^{(1)}_{C^{1};\nu+\frac{1}{2},m_{1}}=\sup_{t\in[0,T)}\big(e^{(\nu+\frac{1}{2})\sqrt{2|\ln(T-t)|}}|\ln(T-t)|^{m_{1}}(|g(t)|+(T-t)|\dot{g}(t)|)\big),
\end{align}
\begin{align}\label{normforxi1}
	\|\xi_{1}\|^{(2)}_{C^{1};\varrho,m_{2}}=\sup_{t\in[0,T)}\big(\frac{e^{\varrho\sqrt{2|\ln(T-t)|}}|\ln(T-t)|^{m_{2}}}{\sqrt{T-t}}(|\xi_{1}(t)|+(T-t)|\ln(T-t)|^{1/2}|\dot{\xi}_{1}(t)|)\big),
\end{align} 
\begin{align}\label{normforlambda1}
	\|\lambda_{1}\|_{C^{2};\nu-\frac{1}{2},m_{3}}=\sup_{t\in[0,T)}\frac{e^{(\nu-\frac{1}{2})\sqrt{2|\ln(T-t)|}}|\ln(T-t)|^{m_{3}}}{\sqrt{T-t}}(|\lambda_{1}(t)|+(T-t)|\dot{\lambda}_{1}(t)|+(T-t)^{2}|\ddot{\lambda}_{1}(t)|).
\end{align}
We notice that in the norm \eqref{normforalpha1} there is a loss in the control of derivative (one could expect to see $\sqrt{|\ln(T-t)|}$ multiplying $\dot{g}$). This is due to the poor control we have of the derivative of the remainder of the equation for $\lambda_{1}$ that is present also in the equation of $\alpha_{1}$, see \eqref{zeromassCondGlu}, \eqref{EquationForLambda1} and \eqref{ERROREQLAMBDA1}. 
\newline We use the following notation: for $\textbf{p}_{1}=(\lambda_{1},\alpha_{1},\xi_{1})$ 
\begin{align*}
	\|\textbf{p}_{1}\|_{X_{p}}=\|\alpha_{1}\|^{(1)}_{C^{1};\nu+\frac{1}{2},m_{1}}+\|\xi\|^{(2)}_{C^{1};\varrho,m_{2}}+\|\lambda_{1}\|_{C^{2};\nu+\frac{1}{2},m_{3}}
\end{align*}
and for $\vec{\phi}=(\phi_{1},\phi_{2},\phi_{3},\varphi^{o},\textbf{p}_{1})$ 
\begin{align*}
	\|\vec{\phi}\|_{X}=\|\phi_{1}\|_{1;\nu-\frac{1}{2},\frac{q+1}{2},4,2}+\|\phi_{2}\|_{1;\nu-\frac{1}{2},\frac{q+1}{2},4,2+\sigma+\epsilon}+\|\phi_{3} \|_{1;\nu,0,3+\varsigma,{2+\epsilon}}+\|\varphi^{o}\|_{\star,o}+\|\textbf{p}_{1}\|_{X_{p}}.
\end{align*}
\newline To properly define the operators $\mathcal{A}_{i,j}$ in these normed spaces we need the following conditions
\begin{align}\label{preliminarCONDITIONTODEFINOPER}
	1<\nu<\min(1+\frac{\epsilon}{2},3-\frac{\sigma}{2},\frac{3}{2}-\frac{\varsigma}{2}), \ \ \  q\in(0,1), \ \ \ \sigma+\epsilon<2, \ \ \  \sigma\in(0,1), \ \ \ \  0<\epsilon<2,  \ \ \ a>0
\end{align}
with the above notation we consider the fixed point problem
\begin{align}\label{FixedPointTHM1}
	\vec{\phi}=\mathcal{A}[\vec{\phi}],
\end{align}
with $\vec{\phi}$ in the unit closed ball of $X$. A solution of this fixed point problem gives a solution of the system of equations \eqref{innereqfinalFINAL}, \eqref{innereqfinalFINAL2}, \eqref{innereqfinalFINALMODE1}, \eqref{orthogonalitycond} and \eqref{outereqfinalFINAL}. 
\begin{proof}[Proof of  Theorem \ref{theorem1}]
	Thanks to notation we introduced it is clear that our goal is to find a fixed point for \eqref{FixedPointTHM1}. \newline
	Since we want \eqref{EstiLambda}, \eqref{EstiPar} to hold, we immediately have a set of conditions for the parameters we introduced in the norms. Namely
	\begin{align}\label{EstiParEstiLambdaCONDITIONSGLU}
		\varrho>\frac{3}{2}, \ \ \ \nu>1.
	\end{align}
	First we claim that for some $\vartheta_{1}>0$
	\begin{align}\label{Claim1Gluing}
		\|\mathcal{A}_{i,1}[\phi_{1},\phi_{2},\phi_{3},\varphi^{o},\textbf{p}_{1}]\|_{1;\nu-\frac{1}{2},\frac{q+1}{2},4,2}\le C e^{-\vartheta_{1}\sqrt{|\ln T|}} \|\vec{\phi}\|_{X} 
	\end{align}
 and 
 \begin{align}\label{Claim1GluingPart2}
 		\|\mathcal{A}_{i,3}[\phi_{1},\phi_{2},\phi_{3},\varphi^{o},\textbf{p}_{1}]\|_{1;\nu,0,,3+\varsigma,2+\varepsilon}\le C e^{-\vartheta_{1}\sqrt{|\ln T|}} \|\vec{\phi}\|_{X}.
 \end{align}
Recalling \eqref{F4hat}, Propositions \ref{PropMode0MassSecMom0ANTICIP} and \ref{PropMode1ANTICIP}, it is enough to show that for some $\vartheta_{1}>0$
\begin{align}\label{Claim1GluingBIS}
	\|\tilde{\chi}F_{4}\|_{0,\nu,0,6+\sigma,\varepsilon}\le C e^{-\vartheta_{1}\sqrt{2|\ln T|}}\|\vec{\phi}\|_{X}.
\end{align}
We start with the forces we introduced to obtain the inner theories. Thanks to \eqref{EstimateEFORGLUING}, \eqref{EstimateFFORGLUING}, recalling the definitions \eqref{InnerOPER1}, \eqref{INNEROPERT3} we have
\begin{align*}
	\|\mathcal{E}[\phi_{1}]\tilde{\chi}+\mathcal{F}[\phi_{3}]\tilde{\chi}\|_{0,\nu,0,6+\sigma,\varepsilon}\le C e^{-\vartheta_{1}\sqrt{|\ln T|}}\|\vec{\phi}\|_{X}
\end{align*}
where we used that $0<\sigma<1$. Now that we addressed the contribution of the forces another important difference with the infinite time construction \cite{DdPDMW} is $\dot{\alpha}_{1}$ since as we know there is a loss of decay due to the introduction of $\lambda_{1}$. We observe that among the terms in $\lambda^{4}S(u_{1}(\textbf{p}_{0}+\textbf{p}_{1}))\tilde{\chi}-\lambda_{0}^{4}S(u_{1}(\textbf{p}_{0}))\tilde{\chi}_{0}\tilde{\chi}$ we have
\begin{align*}
	&|\lambda_{0}^{2}\dot{\alpha}_{1}U \tilde{\chi}|\le C\|\alpha_{1}\|^{(1)}_{C^{1},\nu+\frac{1}{2},m_{1}}\lambda_{0}^{2}\frac{e^{-(\nu+\frac{1}{2})\sqrt{2|\ln(T-t)|}}}{(T-t)|\ln(T-t)|^{m_{1}}} \frac{1}{(1+|y|)^{4}}\tilde{\chi}\\
	&\implies \|\lambda_{0}^{2}\dot{\alpha}_{1}U\tilde{\chi}\|_{0,\nu,0,6+\sigma,\varepsilon}\le C e^{-\vartheta_{1}\sqrt{2|\ln T|}}\|\vec{\phi}\|_{X}
\end{align*}
for some $\vartheta_{1}>0$. The remaining terms in $\lambda^{4}S(u_{1}(\textbf{p}_{0}+\textbf{p}_{1}))\tilde{\chi}-\lambda_{0}^{4}S(u_{1}(\textbf{p}_{0}))\tilde{\chi}_{0}\tilde{\chi}$ can be controlled by \eqref{EstiPar} (or  \eqref{EstiParEstiLambdaCONDITIONSGLU}) and the second derivatives of $\varphi_{\lambda}$ as in Lemma \ref{estimatephilambda}. Then we have 
\begin{align*}
	\|\lambda^{4}S(u_{1}(\textbf{p}_{0}+\textbf{p}_{1}))\tilde{\chi}-\lambda_{0}^{4}S(u_{1}(\textbf{p}_{0}))\tilde{\chi}_{0}\tilde{\chi}\|_{0;\nu,6+\sigma,\varepsilon}\le C e^{-\vartheta_{1} \sqrt{2|\ln T|}}\|\vec{\phi}\|_{X}.
\end{align*}
Another interesting term has already been discussed in Section \ref{ReformulationOf}, by \eqref{INnnerESTE2} for any 
\begin{align}\label{ParamGLU1}
	1<\nu<1+2\mu-\frac{\sigma}{2}
\end{align}
with 
\begin{align}\label{ParamGLU1sec}
	2\mu-\frac{\sigma}{2}>0
\end{align}
we have
\begin{align*}
    \|E_{2}\tilde{\chi}\|_{0;\nu,0,6+\sigma,\varepsilon}\le e^{-\vartheta_{1} \sqrt{2|\ln T|}}\|\vec{\phi}\|_{X}
\end{align*}
(as we anticipated in \ref{ReformulationOf} we will fix $\nu$ close to 1 and $\mu$, $\sigma$ small to control some terms in the outer problem, see \eqref{alpha1tilde}). The same estimates can be obtained for $[E_{2}(\bar{y}_{0})-E_{2}(y)]\tilde{\chi}_{2}\tilde{\chi}$ where we recall the definition \eqref{ellipticcorre}. \newline
Now we estimate two fundamental terms that govern the coupling of the inner-outer system as we described in Section \ref{SecMassSecMom}. First we see that
\begin{align*}
	&|\lambda^{2}U\varphi^{o}\tilde{\chi}|\le C \|\varphi^{o} \|_{\star,o} \frac{\lambda^{2}}{T-t}\frac{e^{-a\sqrt{2|\ln(T-t)|}}}{|\ln(T-t)|^{b}} \frac{1}{(1+|y|)^{4}}\tilde{\chi} \\
	& \implies  \|\lambda^{2}U\varphi^{o}\tilde{\chi}\|_{0;\nu,0,6+\sigma,\varepsilon}\le Ce^{-\vartheta_{1} \sqrt{2|\ln T|}}\|\vec{\phi}\|_{X}
\end{align*}
where we used the definition of the norm \eqref{OUTERNORMSGLU}, $0<\sigma<1$ and we fixed
\begin{align}\label{paramAFixedGlu}
	a=\nu+\frac{1}{2}.
\end{align}
The second term is $\alpha \nabla _{y}\cdot(U\nabla_{y}\psi^{o})$ and can be estimated similarly. Two other interesting terms are $\nabla_{y}\cdot(U\nabla_{y}(\hat{\psi}-\psi))\tilde{\chi}$ and $\lambda^{2}\nabla_{y}\cdot(\varphi_{\lambda}\nabla_{y}(\hat{\psi}-\psi))\tilde{\chi}$. The radial part of $\phi$ can be estimated easily simply observing that $\psi'(\rho)=- \frac{1}{\rho} \int_{\rho}^{\infty}\phi(s)sds$. When we consider the nonradial part of $\phi$ we can simply observe that if $h:\mathbb{R}^{2}\to \mathbb{R}$ satisfying $\|(1+|\cdot|)^{2+w}h(\cdot)\|_{L^{\infty}(\mathbb{R}^{2})}<\infty$ for $0<w<1$ and such that
\begin{align*}
	\int_{\R^2}h(y)dy=0
\end{align*}
then 
\begin{align*}
	|(-\Delta)^{-1}h(y)|\le \frac{C}{(1+|y|)^{m}}\|(1+|\cdot|)^{2+m}h(\cdot)\|_{L^{\infty}(\mathbb{R}^{2})}.
\end{align*}  \newline
All the remaining terms can be estimated as in \cite{DdPDMW} so that we can finally prove \eqref{Claim1Gluing} without introducing any additional restriction on the parameters. \newline

The next claim we make is that for some $\vartheta_{2}>0$
\begin{align}\label{claim1GLUING}
	\|\mathcal{A}_{i,2}[\phi_{1},\phi_{2},\phi_{3},\varphi^{o},\textbf{p}_{1}]\|_{1;\nu-\frac{1}{2},\frac{q+1}{2},4,2}\le C \frac{1}{|\ln T|^{\vartheta_{2}}}\|\vec{\phi}\|_{X}
\end{align}
In order to prove \eqref{claim1GLUING}, recalling Proposition \ref{PropMode0Mass0ANTICIP}, we need to show that
\begin{align}\label{RHSToproveCLAIM1GLUING}
	|m_{2}[F_{4}\tilde{\chi}]|\le C \|\vec{\phi}\|_{X}\frac{1}{\tau^{\nu+\frac{1}{2}}}\frac{1}{(\ln \tau)^{-\frac{1+q}{2}}}  \le C \|\vec{\phi}\|_{X} \frac{ e^{-(\nu+\frac{1}{2})\sqrt{2|\ln(T-t)|}}}{|\ln(T-t)|^{\frac{1}{2}(\nu+\frac{1}{2})-\frac{q+1}{4}}}.
\end{align}
Let us start with the forces we introduced to obtain the inner theories, recalling \eqref{EstimateEFORGLUING} and \eqref{EstimateFFORGLUING} we have
\begin{align*}
	m_{2}[\mathcal{E}[\phi_{1}]\tilde{\chi}]\le C\|\vec{\phi}\|_{X}\frac{e^{-(\nu+1)\sqrt{2|\ln(T-t)|}}}{|\ln(T-t)|^{\frac{\nu+q}{2}}}, \ \ \ |m_{2}[\mathcal{F}[\phi_{3}]\tilde{\chi}]|\le C \|\vec{\phi}\|_{X}\frac{e^{-(\nu+1+\frac{\varsigma}{2})\sqrt{2|\ln(T-t)|}}}{|\ln(T-t)|^{\frac{\nu}{2}}}.
\end{align*}
Now we see that, thanks to \eqref{E2pointw}, for any $\rho>0$
\begin{align*}
	&|m_{2}[E_{2}\tilde{\chi}_{2}]|\le C e^{-(2-\rho)\sqrt{2|\ln(T-t)|}}
 \end{align*}
and the same holds for $m_{2}[(E_{2}(\bar{y}_{0})-E_{2}(y))\tilde{\chi}_{2}]$. A fundamental term is $m_{2}[\lambda^{4}S(u_{1}(\textbf{p}_{0}+\textbf{p}_{1}))\tilde{\chi}-\lambda_{0}^{4}S(u_{1}(\textbf{p}_{0}))\tilde{\chi}_{0}\tilde{\chi}]$. We recall that by definition $\lambda_{1}=\Lambda[ (4+\frac{\int_{\R^2}\varphi_{\lambda}(T)}{\pi})\int_{t}^{T}\int_{\mathbb{R}^{2}}U\varphi^{o}\tilde{\chi}dyds]$, where the operator $\Lambda$ has been defined in Lemma \ref{lambda1Lemma}. Since we have \eqref{OUTERNORMSGLU} with $a=\nu+\frac{1}{2}$, by Lemma \ref{lemmasecmomdiff} and the same considerations we made in Section \ref{SecMassSecMom}, or by \eqref{zeromassCondGlu}, we have
\begin{align*}
	&m_{2}[\lambda^{4}S(u_{1}(\textbf{p}_{0}+\textbf{p}_{1}))\tilde{\chi}-\lambda_{0}^{4}S(u_{1}(\textbf{p}_{0}))\tilde{\chi}_{0}\tilde{\chi}]\approx\frac{1}{\pi}\big(\int_{\R^2}\varphi_{\lambda}(T)\big)(T-t)\int_{\R^2}U(y)\varphi^{o}(y\lambda,t)\tilde{\chi}dy\\
	&\implies |m_{2}[\lambda^{4}S(u_{1}(\textbf{p}_{0}+\textbf{p}_{1}))\tilde{\chi}-\lambda_{0}^{4}S(u_{1}(\textbf{p}_{0}))\tilde{\chi}_{0}\tilde{\chi}]| \le C \|\vec{\phi}\|_{X} \frac{e^{-(\nu+\frac{1}{2})\sqrt{2|\ln(T-t)|}}}{|\ln(T-t)|^{b}}
\end{align*}
then we need 
\begin{align}\label{COnDPAR3}
	b>\frac{1}{2}(\nu+\frac{1}{2})-\frac{q+1}{4} \iff b>\frac{\nu}{2}-\frac{q}{4 }.
\end{align}
Of the remaining terms the most delicate ones are 
\begin{align}\label{SecMomDCritiCAGLU}
	m_{2}[\nabla_{y}\cdot(U\nabla_{y}(\hat{\psi}-\psi))\tilde{\chi}] , \ \  m_{2}[\lambda^{2}\nabla_{y}\cdot(\varphi_{\lambda}\nabla_{y}(\hat{\psi}-\psi))\tilde{\chi}]
\end{align}
and
\begin{align}\label{SecMomDCritiCAGLU2}
	&m_{2}[\lambda^{2}\nabla_{y}\cdot(\varphi_{\lambda}\nabla_{y}\psi^{i})+\nabla _{y}\cdot(\phi_{i} \nabla_{y}\psi_{\lambda})\tilde{\chi}], \ \  m_{2}[\nabla_{y}\cdot((\phi^{i}\chi+\lambda^{2}\varphi^{o})\nabla_{y}(\hat{\psi}+\psi^{o}))\tilde{\chi}].
\end{align}
To estimate \eqref{SecMomDCritiCAGLU} it is enough to use the pointwise estimates we previously found. To estimate \eqref{SecMomDCritiCAGLU2} it is enough to observe that if $\psi=(-\Delta_{y})^{-1}\phi$, by Pohozaev's identity we have
\begin{align}\label{IdentityFORSECmomGLU}
	&\int_{\mathbb{R}^{2}}\nabla \cdot(\phi \nabla_{y}\psi)|y|^{2}\tilde{\chi}dy=\nonumber\\
	&=\int_{\R^2}\nabla_{y} \big[\nabla_{y}\psi (y\cdot \nabla \psi)-y\frac{|\nabla_{y}\psi|^{2}}{2}\big]\big[\frac{\lambda |y|}{2\sqrt{\delta(T-t)}}\chi_{0}'(\frac{\lambda y}{2 \sqrt{\delta(T-t)}})+2\chi_{0}(\frac{\lambda y}{2\sqrt{\delta(T-t)}})\big]\le\nonumber\\
	&\le  C\int_{2\frac{\sqrt{\delta(T-t)}}{\lambda(t)}\le |y|\le 4\frac{\sqrt{\delta(T-t)}}{\lambda(t)} }|\nabla_{y}\psi|^{2}dy
\end{align}
and that (if $\phi_{i}:=\phi_{1}+\phi_{2}+\phi_{3}$), we have
\begin{align*}
	\int_{\R^2}&\nabla_{y}\cdot(\lambda^{2}\varphi_{\lambda}\nabla_{y}(-\Delta_{y})^{-1}\phi_{i})\tilde{\chi}|y|^{2}dy+\int_{\R^2}\nabla_{y}\cdot(\phi_{i} \nabla_{y}\psi_{\lambda})\tilde{\chi}|y|^{2}dy=\\
	=&\int_{\R^2}\nabla_{y}\cdot[(\lambda^{2}\varphi_{\lambda}+\phi_{i})\nabla_{y}(\psi_{\lambda}+(-\Delta_{y})^{-1}\phi_{i})]\tilde{\chi}|y|^{2}dy-\\
	&-\int_{\R^2}\nabla_{y}\cdot[\lambda^{2}\varphi_{\lambda}\nabla_{y}\psi_{\lambda}]\tilde{\chi}|y|^{2}dy-\int_{\R^2}\nabla_{y}\cdot[\phi_{i} \nabla_{y}((-\Delta_{y})^{-1}\phi_{i})]\tilde{\chi}|y|^{2}dy.
\end{align*}
Another interesting term in the second moment is 
\begin{align*}
	m_{2}[\lambda^{2}U \varphi^{o}]\le C \|\vec{\phi}\|_{X}\frac{e^{-(\nu+\frac{3}{2}-\rho)\sqrt{2|\ln(T-t)|}}}{|\ln(T-t)|^{b}}, \ \ \ \ \text{for any }\rho>0.
\end{align*}
Since all the remaining terms can be estimated similarly, we can finally obtain \eqref{RHSToproveCLAIM1GLUING} and then \eqref{claim1GLUING}. \newline

The next claim we make that there is $\vartheta_{3}>0$
\begin{align}\label{claim3GLUING}
	&\|\mathcal{A}_{o}[\phi_{1},\phi_{2},\phi_{3},\varphi^{o},\textbf{p}_{1}]\|_{\star,o}\le C \frac{1}{|\ln T|^{\vartheta_{3}}}\|\vec{\phi}\|_{X}.
\end{align}
Thanks to Proposition \ref{Prop131ANTICIP} we just need to show
\begin{align*}
	\|G_{2}(\phi_{1}+\phi_{2}+\phi_{3},\varphi^{o},\textbf{p}_{1})\|_{\star\star,o}\le C \frac{1}{|\ln T|^{\vartheta_{3}}}\|\vec{\phi}\|_{X}.
\end{align*}
We start with $\lambda^{-4}E_{2}(1-\tilde{\chi}_{2})\chi$. Thanks to Section \ref{ReformulationOf} we now that for some $\theta>0$
\begin{align*}
	\|\lambda^{-4}E_{2}(1-\tilde{\chi}_{2})\chi \|_{\star\star,o}\le C e^{-\theta\sqrt{2|\ln T|}}\|\vec{\phi}\|_{X}
\end{align*}
if 
\begin{align}\label{CONDITIONPARMATERGLUIN3}
	2(1-2\mu)>\nu+\frac{1}{2}.
\end{align}
If we assume that 
\begin{align}\label{nuUPPERGLu}
   \nu+\frac{1}{2}<2 \iff \nu<\frac{3}{2}
\end{align} 
by Lemma \ref{lemmaesterr} we also know that for some $\theta>0$
\begin{align*}
	\|S(u_{1})(1-\chi)\|_{\star\star,o}\le C e^{-\theta\sqrt{2|\ln T|}}\|\vec{\phi}\|_{X}.
\end{align*}
A fundamental term, recalling the notation $\phi_{i}=\phi_{1}+\phi_{2}+\phi_{3}$, gives
\begin{align*}
	|\frac{1}{\lambda^{2}}\phi_{i}\Delta \chi|\le& C \|\vec{\phi}\|_{X} \frac{1}{\lambda^{2}}\frac{e^{-(\nu-\frac{1}{2})\sqrt{2|\ln(T-t)|}}}{|\ln(T-t)|^{\frac{1}{2}(\nu-\frac{1}{2})+\frac{q+1}{4}}}\frac{1}{\delta(T-t)} \frac{\lambda^{4}}{|x-\xi|^{4}}|\Delta_{z}\chi_{0}(z)|\le \\
	\le&C \|\vec{\phi}\|_{X}\frac{e^{-(\nu+\frac{1}{2})\sqrt{2|\ln(T-t)|}}}{(T-t)^{2}|\ln(T-t)|^{\frac{1}{2}(\nu-\frac{1}{2})+\frac{q+1}{4}}} |\Delta_{z}\chi_{0}(z)|
\end{align*}
and then
\begin{align*}
	\| \frac{1}{\lambda^{2}}\phi_{i}\Delta \chi \|_{\star\star,o}\le \frac{C}{|\ln T|^{\frac{1}{2}(\nu-\frac{1}{2})+\frac{q+1}{4}-b}}\|\vec{\phi}\|_{X}.
\end{align*}
It is clear that to get \eqref{claim3GLUING} we will need 
\begin{align}\label{CondtionPARAMERB}
	b<\frac{1}{2}(\nu-\frac{1}{2})+\frac{q+1}{4} \iff b<\frac{\nu}{2}+\frac{q}{4}.
\end{align}
The terms involving $\varphi^{o}$ can be estimated easily, take for instance
\begin{align*}
	&|\frac{1}{\lambda^{2}}U \varphi^{o}(1-\chi)|\le C \frac{\lambda^{2}}{|x-\xi|^{4}}|\varphi^{o}|(1-\chi)\le C \frac{e^{-\sqrt{2|\ln(T-t)|}}}{|x|^{2}+(T-t)}|\varphi^{o}|(1-\chi)\\
	&\implies \|\frac{1}{\lambda^{2}}U\varphi^{o}(1-\chi)\|_{\star\star,o}\le C e^{-\sqrt{2|\ln T|}}\|\vec{\phi}\|_{X}.
\end{align*}
 The remaining claim we make is that there is $\vartheta_{4}>0$
 \begin{align}\label{LASTCLAIMGLUI}
 	\|\mathcal{A}_{p}[\phi_{1},\phi_{2},\phi_{3},\varphi^{o},\textbf{p}_{1}]\|_{p}\le C \frac{1}{|\ln T|^{\vartheta_{4}}}\|\vec{\phi}\|_{X}.
 \end{align}
To prove \eqref{LASTCLAIMGLUI} we claim that if 
\begin{align*}
	(\tilde{\lambda}_{1},\tilde{\alpha}_{1},\tilde{\xi}_{1})=\mathcal{A}_{p}[\phi_{1},\phi_{2},\phi_{3},\varphi^{o},\textbf{p}_{1}]
\end{align*}
then
\begin{align}\label{lambda1tilde}
	\|\tilde{\lambda}_{1}\| \le C \frac{1}{|\ln T|^{\vartheta_{4}}}\|\vec{\phi}\|_{X},
\end{align}
\begin{align}\label{alpha1tilde}
	\|\tilde{\alpha}_{1}\|\le C \frac{1}{|\ln T|^{\vartheta_{4}}}\|\vec{\phi}\|_{X},
\end{align}
\begin{align}\label{xi1tilde}
	\|\tilde{\xi}_{1}\|\le C \frac{1}{|\ln T|^{\vartheta_{4}}}\|\vec{\phi}\|_{X}.
\end{align}
The proof of \eqref{lambda1tilde} is an immediate consequence of Proposition \ref{lambda1Lemma} after observing that 
\begin{align}\label{RHSalpha1lambda1}
	|\int_{t}^{T}\int_{\R^2}U\varphi^{o}\tilde{\chi}dyds|\le C \|\vec{\phi}\|_{X} \frac{e^{-(\nu+\frac{1}{2})\sqrt{2|\ln(T-t)|}}}{|\ln(T-t)|^{b-\frac{1}{2}}}
\end{align}
and assuming that
\begin{align}\label{CONDTIONLAMBDA1contr}
	m_{3}<b.
\end{align}
We notice that, since all the remaining terms in \eqref{zeromassCondGlu} are smaller, \eqref{RHSalpha1lambda1} also give \eqref{alpha1tilde} if
\begin{align}\label{CONDITINOALPHA1contr}
	m_{1}<b-\frac{1}{2}.
\end{align}
To obtain \eqref{xi1tilde} we need to use \eqref{centMasszeroCondGLUE}. The first difference with \cite{DdPDMW} is the force we used to solved the inner theory in mode 1, thanks to \eqref{EstimateFFORGLUING} we get
\begin{align*}
	|\mathcal{F}[\phi_{3}]|\le C \|\vec{\phi}\|_{X} \frac{1}{\tau^{\nu+1}(\ln \tau)^{m-1}}\frac{1}{\hat{M}^{\varsigma}}(|W_{1}(y)|+|W_{2}(y)|)
\end{align*}
and then we need
\begin{align}\label{CONDTIONGLUINXI20}
   \nu+1+\frac{\sigma}{2}>\rho+\frac{1}{2} \iff \rho<\nu+\frac{1+\varsigma}{2}.
\end{align} 
A term that gave an important contribution in \cite{DdPDMW} was
\begin{align*}
	\int_{\R^2}\lambda^{2}\varphi_{\lambda}\phi_{i}y_{j}\tilde{\chi}dy, \ \ \ \ \ \text{where here } \phi_{i}=\phi_{1}+\phi_{2}+\phi_{3}.
\end{align*}
We know that $\phi_{1}$ and $\phi_{2}$ are radial functions and then
\begin{align*}
	|\int_{\R^2}\lambda^{2}\varphi_{\lambda}(\phi_{1}+\phi_{2})y_{j}\tilde{\chi}dy|\le  \frac{|\xi_{1}|}{\lambda} \frac{e^{-\sqrt{2|\ln(T-t)|}} e^{-(\nu-\frac{1}{2})\sqrt{2|\ln(T-t)|}}}{|\ln(T-t)|^{\frac{1}{2}(\nu-\frac{1}{2})+\frac{q+1}{4}}}\le \|\vec{\phi}\|_{X}\frac{e^{-(\nu+\rho)\sqrt{2|\ln(T-t)|}}}{|\ln(T-t)|^{\frac{1}{2}(\nu-\frac{1}{2})+\frac{q+1}{4}}}
\end{align*}
and obviously $\nu+\rho>\rho+\frac{1}{2}$ since $\nu>1$.
The remaining term is then
\begin{align*}
	|\int_{\R^2}\lambda^{2}\varphi_{\lambda}\phi_{3}y_{j}\tilde{\chi}dy|\le\|\vec{\phi}\|_{X} \frac{e^{-\sqrt{2|\ln(T-t)|}}}{|\ln(T-t)|^{\nu/2}}e^{-\nu\sqrt{2|\ln(T-t)|}}
\end{align*}
and then we need 
\begin{align}\label{CONDITIONGLUIGN$XI}
	\nu+1>\rho+\frac{1}{2} \iff \rho<\nu+\frac{1}{2}.
\end{align}
Now we can estimate the terms involving the outer solution, for instance we see
\begin{align*}
     |\int_{\R^2}\lambda^{2}\varphi^{o}Uy_{j}\tilde{\chi}dy| \le \|\vec{\phi}\|_{X}\frac{e^{-(a+1})\sqrt{2|\ln(T-t)|}}{|\ln(T-t)|^{b}}
\end{align*}
and then it is enough to require
\begin{align}\label{CONDTIONTOGEOUTERXI}
    a+1>\rho+\frac{1}{2} \iff \rho<a+\frac{1}{2}=\nu+1
\end{align}
where we used \eqref{paramAFixedGlu}.
At this point we want to fix the value of the parameters $\nu$, $q$, $\sigma$, $\varsigma$, $\epsilon$, $a$, $b$, $\varrho$, $m_{1}$, $m_{2}$, $m_{3}$, $\mu$. \newline 
We observe that under the same restrictions on the parameters similar considerations give also Lipschitz estimates for the operator $\mathcal{A}$ in $X$.

We already fixed $a=\nu+\frac{1}{2}$. We observe that $m_{2}$ is free, then we can take any $m_{2}\in \mathbb{R}$. 
Now we fix $\mu$, $\sigma$ small such that \eqref{ParamGLU1sec} holds. And $\nu>1$ but small such that \eqref{ParamGLU1}, \eqref{CONDITIONPARMATERGLUIN3}, \eqref{nuUPPERGLu} hold. We can take any $q\in(0,1)$ and  $b$ in the interval
\begin{align*}
    \frac{\nu}{2}-\frac{q}{4}<b<\frac{\nu}{2}+\frac{q}{4}
\end{align*}
(from this inequality it is clear how important is the extra smallness we obtained in Propositions \ref{PropMode0Mass0ANTICIP}, \ref{PropMode0MassSecMom0ANTICIP}). We can take $m_{3}$ and $m_{1}$ such that \eqref{CONDTIONLAMBDA1contr}, \eqref{CONDITINOALPHA1contr} hold.
Now it is enough to take $\varsigma>0$ and $\rho>\frac{3}{2}$ such that \eqref{EstiParEstiLambdaCONDITIONSGLU}, \eqref{CONDTIONGLUINXI20}, \eqref{CONDTIONTOGEOUTERXI} hold. Finally we can take $\epsilon>0$ small such that \eqref{preliminarCONDITIONTODEFINOPER} holds. \newline

Once we constructed the solution the following expansion for the rate of the blow-up is well justified
\begin{align}\label{expansionlambda}
	\lambda(t)=\lambda^{\star}(t)(1+o(1))
\end{align}
where $\lambda^{\star}(t)$ is given by \eqref{approximatlambda}. Now we recall that $\varepsilon(T)$ is a parameter whose role is explained in Section \ref{MASSPHISECTION} and that we can freely choose in the interval
\begin{align*}
	T\ll \varepsilon(T)\ll 1.
\end{align*}
We get the following expansion for the mass of $\varphi_{\lambda}$.
\begin{corollary}\label{Corollary EXPANSION MASS}
	Let $\varphi_{\lambda}$ be defined in \eqref{eqphilambda}, and let $\lambda$ satisfy \eqref{expansionlambda}, one has
	\begin{align}\label{EXPANSIONMASS}
		\int_{\R^2}\varphi_{\lambda}(x,T)dx=2\sqrt{2}\pi e^{-(\gamma+2)}e^{-\sqrt{2|\ln\varepsilon(T)|}}\sqrt{2|\ln\varepsilon(T)|}(1+O(\frac{1}{\sqrt{|\ln \varepsilon(T)|}})).
	\end{align}
\end{corollary}
\begin{proof}
	The proof is an immediate consequence of Corollary \ref{ExpansionMass}.
\end{proof}
Now, since
\begin{align*}
	u(x,t)=\frac{1}{\lambda^{2}}(\alpha U_{0}(\frac{x-\xi}{\lambda})+\phi_{0}^{i}(\frac{x-\xi}{\lambda})+\phi_{1}(\frac{x-\xi}{\lambda})+\phi_{2}(\frac{x-\xi}{\lambda})+\phi_{3}(\frac{x-\xi}{\lambda}))\chi +\varphi_{\lambda}+\varphi^{o},
\end{align*}
where $\lambda=\lambda_{0}+\lambda_{1}$ satisfies \eqref{EquationForLambda1} with $\lambda_{0}$ as in Proposition \ref{prop-lambda0}, and because of the norm we used to define $X$ we get
\begin{align}\label{ExpansionMASSformula}
	\lim_{t \to T} \int_{\R^2} u(x,t)dx=8\pi+2\sqrt{2}\pi e^{-(\gamma+2)}e^{-\sqrt{2|\ln\varepsilon(T)|}}\sqrt{2|\ln\varepsilon(T)|}(1+O(\frac{1}{\sqrt{|\ln \varepsilon(T)|}})).
\end{align}
By properly choosing $\varepsilon(T)$ in  \eqref{ExpansionMASSformula} we can finally prescribe the value of the mass. \newline

The extension to the case $k>1$ is elementary and briefly discussed in Section \ref{MultiSpikes}.
\end{proof}
\subsection{The last adjustment of the rate of the blow-up}\label{lambda1Section}
The purpose of this section is to prove Proposition \ref{lambda1Lemma}.
\begin{proof}[Proof of Proposition \ref{lambda1Lemma}]
	The proof is an elementary modification of Section \ref{MASSPHISECTION}. We can obviously assume that $\lambda_{1}(T)=0$. \newline By Corollary \ref{ExpansionMass}, we know 
	\begin{align*}
		\int_{\R^2}\varphi_{\lambda}(\cdot,t)dx-\int_{\R^2}\varphi_{\lambda}(\cdot,T)dx-16\pi\frac{\beta}{\delta}\frac{\lambda^{2}(t)}{T-t}=&4\pi \int_{-\varepsilon(T)}^{T}\frac{\lambda\dot{\lambda}(s)}{t-s}ds-4\pi \int_{-\varepsilon(T)}^{t-\lambda^{2}}\frac{\lambda\dot{\lambda}(s)}{t-s}+\\
		&+4\pi(\gamma+1-\ln4)p(t)+O(\frac{|\lambda\dot{\lambda}(t)|}{|\ln(T-t)|^{1/4}})+\\
		&+[\text{lower order terms}].
	\end{align*}
Proceedings as in Section \ref{LINERAIZEDMASS}, if we write $\lambda\dot{\lambda}(t)=\lambda_{0}\dot{\lambda}_{0}+(\lambda_{0}\lambda_{1})'+\lambda_{1}\dot{\lambda}_{1}$ we see that \eqref{EquationForLambda1} is equivalent to  
\begin{align}\label{NonLocalEquationForLambda1}
	\int_{t}^{T}\frac{(\lambda_{0}\lambda_{1})'(s)}{T-s}ds-\sqrt{2|\ln(T-t)|}(\lambda_{0}\lambda_{1})'(t)+\frac{f(t)}{16\pi}=O(\frac{|\lambda_{0}\dot{\lambda}_{1}|}{|\ln(T-t)|^{1/4}})+[\text{lower order terms}].
\end{align}
Neglecting terms that require control of $(\lambda_{0}\lambda_{1})''(s)$, by Lemma \ref{T0est} and a fixed point argument find a solution that satisfies
\begin{align*}
	&|(\lambda_{0}\lambda_{1})'|\le C \|f\|_{1,\kappa,m}\frac{e^{-(1+\kappa)\sqrt{2|\ln(T-t)|}}}{|\ln(T-t)|^{m+1/2}}\\
	& \implies  |\lambda_{1}|\le\|f\|_{1,\kappa,m} \sqrt{T-t}\frac{e^{-(\frac{1}{2}+\kappa)\sqrt{2|\ln(T-t)|}}}{|\ln(T-t)|^{m+1/2}}, \ \ \ |\dot{\lambda}_{1}|\le\|f\|_{1,\kappa,m} \frac{e^{-(\frac{1}{2}+\kappa)\sqrt{2|\ln(T-t)|}}}{\sqrt{T-t}|\ln(T-t)|^{m+1/2}}
\end{align*}
where in the above implication we use the assumption $\lambda_{1}(T)=0$, the expansion for $\lambda_{0}$, $\dot{\lambda}_{0}$ we obtained in Proposition \ref{prop-lambda0} and we directly integrated.
As in Lemma \ref{SolutionNotRegular}, thanks to the control we have of $f(t)$ we can also obtain
\begin{align}\label{ControlDoubleDerivatives}
	&|(\lambda_{0}\lambda_{1})''|\le C \|f\|_{1,\kappa,m}\frac{e^{-(1+\kappa)\sqrt{2|\ln(T-t)|}}}{(T-t)|\ln(T-t)|^{m+1/2}} \\
	&\implies |\ddot{\lambda}_{1}|=|\frac{1}{\lambda_0}(-2\dot{\lambda}_{0}\dot{\lambda}_1-\ddot{\lambda}_{0}\lambda_{1}+(\lambda_{0}\lambda_{1})'')|\le C \|f\|_{1,\kappa,m} \frac{e^ {-(\frac{1}{2}+\kappa)\sqrt{2|\ln(T-t)|}}}{(T-t)^{3/2}|\ln(T-t)|^{m+1/2}}
\end{align}
Among the lower order terms we neglected in \eqref{NonLocalEquationForLambda1}, thanks to \eqref{ControlDoubleDerivatives} we can estimate the most important one
\begin{align*}
	|\int_{t-(T-t)e^{-\sigma\sqrt{2|\ln(T-t)|}}}^{t-\lambda_{0}^{2}}\frac{(\lambda_{0}\lambda_{1})'(t)-(\lambda_{0}\lambda_{1})'(s)}{t-s}ds|\le\|f\|_{1,k,m}\frac{e^{-(1+\kappa+\sigma)\sqrt{2|\ln(T-t)|}}}{|\ln(T-t)|^{m+1/2}}
\end{align*}
for any $0<\sigma<\frac{1}{2}$.\newline
Observing that we do not have control of the third derivatives of $\lambda_{1}$ (since we cannot control the third derivative of $f(t)$), we observe
\begin{align*}
	&|\frac{d}{dt}\int_{t-(T-t)e^{-\sigma\sqrt{2|\ln(T-t)|}}}^{t-\lambda_{0}^{2}}\frac{(\lambda_{0}\lambda_{1})'(t)-(\lambda_{0}\lambda_{1})'(s)}{t-s}ds|\le\\
	&\le|\frac{d}{dt}\int_{t-(T-t)e^{-\sigma\sqrt{2|\ln(T-t)|}}}^{t-\lambda_{0}^{2}}\frac{(\lambda_{0}\lambda_{1})'(t)}{t-s}ds|+|\frac{d}{dt}\int_{t-(T-t)e^{-\sigma \sqrt{2|\ln(T-t)|}}}^{t-\lambda_{0}^{2}}\frac{(\lambda_{0}\lambda_{1})'(s)}{t-s}ds|\le\\
	&\le\|f\|_{1,\kappa,m}\frac{e^{-(1+k)\sqrt{2|\ln(T-t)|}}}{(T-t)|\ln(T-t)|^{m}}.
\end{align*}
\end{proof}

\subsection{The construction for $k>1$}\label{MultiSpikes}
Let us consider for instance the case $k=2$. \newline
We denote
\begin{align}
	\chi_{j}(x,t)=\chi_{0}(\frac{|x-\xi_{j}(t)|}{\sqrt{\delta(T-t)}}), \ \ \ j=1,2.
\end{align} 
We recall that
\begin{align*}
	\xi_{1}(t)\to q_{1}\in \mathbb{R}^{2}, \ \ \ \xi_{2}(t)\to q_{2}\in \mathbb{R}^{2}.
\end{align*}
The idea is to consider as a first ansatz
\begin{align*}
	u_{1}(x,t)=&u_{1}^{(1)}(x,t)+u_{1}^{(2)}(x,t) .
\end{align*}
where
\begin{align*}
	u_{1}^{(j)}(x,t)=& \frac{\alpha_{j}(t)}{\lambda_{j}(t)^{2}}U_{0}(\frac{x-\xi_{j}(t)}{\lambda_{1}(t)})\chi_{j}+\varphi_{\lambda_{j}}(|x-q_{j}|,t), \ \ \ j=1,2.
\end{align*}
We proceed as in Section \ref{InnerOuterSection}. We write
\begin{align*}
	\Phi(x,t)=\sum_{j=1,2}\big[ \frac{1}{\lambda_{j}^{2}}\phi^{i}_{(j)}(x,t)\chi_{j}+\varphi^{o}_{(j)}(x,t)\big]
\end{align*}
and we impose 
\begin{align*}
	S(u_{1}+\Phi)=0.
\end{align*}
We can linearize the operator $S$ around $u_{1}^{(1)}$ obtaining
\begin{align*}
	S(u_{1}+\Phi)=&S(u_{1}^{(1)})-\partial_{t}\big(\frac{1}{\lambda^{2}_{1}}\phi^{i}_{(1)}\chi_{1}\big)-\partial_{t}\varphi^{o}_{(1)}+\mathcal{L}_{u_{1}^{(1)}}\big[\frac{1}{\lambda^{2}_{1}}\phi^{i}_{(1)}\chi_{1}\big]+\mathcal{L}_{u_{1}^{(1)}}[\varphi^{o}_{(1)}]+\\
	&-\partial_{t}\big(u_{1}^{(2)}+\frac{1}{\lambda^{2}_{2}}\phi^{i}_{(2)}\chi_{2}+\varphi^{o}_{(2)}\big)+\mathcal{L}_{u_{1}^{(1)}}\big[u_{1}^{(2)}+\frac{1}{\lambda^{2}_{2}}\phi^{i}_{(2)}\chi_{2}+\varphi^{o}_{(2)}\big]-\\
	&-\nabla \cdot ( (u_{1}^{(2)}+\Phi) \nabla (-\Delta)^{-1}(u_{1}^{(2)}+\Phi))
\end{align*}
where we recall that
\begin{align*}
	\mathcal{L}_{u}[\phi]=\Delta \phi-\nabla \cdot(\phi \nabla v)-\nabla \cdot(u\nabla(-\Delta)^{-1}\phi), \ \ v=(-\Delta)^{-1}u.
\end{align*}
Now we observe that
\begin{align*}
	-\nabla \cdot((u_{1}^{(2)}+\Phi)\nabla(-\Delta)^{-1}(u_{1}^{(2)}+\Phi))=\mathcal{L}_{u_{1}^{(2)}}[\Phi]-\Delta \Phi-\nabla \cdot(\Phi \nabla (-\Delta)^{-1}\Phi).
\end{align*}
If $i\ne j$, we can also write
\begin{align*}
	\mathcal{L}_{u_{1}^{(j)}}[u_{(l)}^{i}+\frac{1}{\lambda_{i}^{2}}\phi_{(l)}^{i}\chi_{i}+\varphi^{o}_{i}]=\Delta(u_{(l)}^{i}+\frac{1}{\lambda_{i}^{2}}\phi_{i}\chi_{i}+\varphi^{o}_{i})+\mathcal{E}_{j,i}.
\end{align*}
Then we have
\begin{align*}
	S(u_{1}+\Phi)=&-\partial_{t}\big(\frac{1}{\lambda^{2}_{1}}\phi^{i}_{(1)}\chi_{1}\big)-\partial_{t}\varphi^{o}_{(1)}+\mathcal{L}_{u_{1}^{(1)}}\big[\frac{1}{\lambda^{2}_{1}}\phi^{i}_{(1)}\chi_{1}\big]+\mathcal{L}_{u_{1}^{(1)}}[\varphi^{o}_{(1)}]+\\
	&-\partial_{t}\big(\frac{1}{\lambda^{2}_{2}}\phi^{i}_{(2)}\chi_{2}\big)+\partial_{t}\varphi^{o}_{(2)}+\mathcal{L}_{u_{1}^{(2)}}\big[\frac{1}{\lambda_{2}^{2}}\phi^{i}_{(2)}\chi_{2}\big]+\mathcal{L}_{u_{1}^{(2)}}\big[\varphi^{o}_{(2)}\big]+\\
	&+\big(S(u_{1}^{(1)})+S(u_{2}^{(2)})+\mathcal{E}_{1,2}+\mathcal{E}_{2,1}-\nabla \cdot ( \Phi \nabla (-\Delta)^{-1}\Phi)\big)\sum_{i=1,2}\chi_{i}+\\
	&+\big(S(u_{1}^{(1)})+S(u_{2}^{(2)})+\mathcal{E}_{1,2}+\mathcal{E}_{2,1}-\nabla \cdot ( \Phi \nabla (-\Delta)^{-1}\Phi)\big)(1-\sum_{i=1,2}\chi_{i}).
\end{align*}
It is clear that the resulting system will be made of two inner equations for $\phi_{j}^{i}$ and two outer equations for $\varphi^{o}_{j}$. Notice that the contribution of the mixed terms is negligible since there will be a factor concentrated far away and that in the interested region is simply bounded.


\section{Mass of $\varphi_{\lambda}$}\label{MASSPHISECTION}
In this section we will prove Proposition \ref{prop-lambda0}. A crucial step will be to find the expansion \eqref{ExpansionMassphilambda} for the mass of $\varphi_{\lambda}$. We anticipate that in this section the introduction of $\varepsilon(T)$ and then choice of the time interval will be completely clarified.
First, let us decompose $\varphi_{\lambda}$ as 
\begin{align}
	\varphi_{\lambda}=\varphi_{\lambda}^{(1)}+\varphi_{\lambda}^{(2)}
\end{align}
where 
\begin{align}
	&\begin{cases}\label{2dphilambez3}
		\partial_{t}\varphi_{\lambda}^{(1)}=\Delta_{6}\varphi_{\lambda}^{(1)}+\frac{\dot{\lambda}}{\lambda^{3}}Z_{0}\chi \ \ \ \ \ \text{in } \mathbb{R}^{2}\times(-\varepsilon(T),T),\\[5pt]
		\varphi_{\lambda}^{(1)}(\cdot,-\varepsilon(T))=0 \ \ \ \ \text{in }\mathbb{R}^{2},
	\end{cases}  \\ \label{2dphilambez4}
	&\begin{cases}
		\partial_{t}\varphi_{\lambda}^{(2)}=\Delta_{6}\varphi_{\lambda}^{(2)}-\frac{1}{\lambda^{2}}\frac{1}{2(T-t)}U_{0}\nabla_{w}\chi \cdot w + \tilde{E} \ \ \ \ \ \text{in } \mathbb{R}^{2}\times(-\varepsilon(T),T),\\[5pt]
		\varphi_{\lambda}^{(2)}(\cdot,-\varepsilon(T))=0 \ \ \ \ \text{in }\mathbb{R}^{2}.
	\end{cases}
\end{align} 
Let us generalize the preliminary estimates \eqref{EstiLambda}, \eqref{EstiPar}
\begin{align}\label{LambAss1}
	&c_{1}\sqrt{T-t}e^{-a\sqrt{|\ln(T-t)|}}\le |\lambda(t)|\le c_{2}\sqrt{T-t}e^{-a\sqrt{|\ln(T-t)|}} \text{for some } c_{1},c_{2}>0
\end{align}		
\begin{align}	\label{LambAss2}
	|\lambda\dot{\lambda}(t)||\le C e^{-2a\sqrt{|\ln(T-t)|}},
\end{align}
\begin{align}\label{RegLambAss}
	|(T-t)\sqrt{|\ln(T-t)|}|\frac{d}{dt}(\lambda\dot{\lambda})(t)|\le C e^{-2a\sqrt{|\ln(T-t)|}}
\end{align}
for two positive constants $c_{1}$, $c_{2}$ independent of $T$ and for some $a>0$. In Lemma \ref{SECMASSprofile} we will show $a$ must be equal to $\frac{\sqrt{2}}{2}$ but in this section we prefer to keep this more general notation. As we mentioned previously, will also need to introduce $\varepsilon(T)$, a constant that will satisfy
\begin{align}\label{RestrictionEPSILON}
	T\ll \varepsilon(T)\ll 1
\end{align}
or, more precisely, \eqref{EpsiTandT}.
To prove Proposition \ref{prop-lambda0}, we are interesting in finding an approximate solution of 
	\begin{align}\label{EqPhiLamHOPE}
	4 \int_{\R^2} \varphi_{\lambda_0} dx
	- 64\pi \gamma \frac{\lambda^{2}}{T-t}
	& =4\int_{\mathbb{R}^{2}}\varphi_{\lambda_0}(x,T).
\end{align}
In the following we will denote the unknown of \eqref{EqPhiLamHOPE} as
\begin{align}\label{p}
	p(t):=\lambda\dot{\lambda}(t).
\end{align}
The basic idea is to find an expansion for $p(t)$ that solves \eqref{EqPhiLamHOPE} up to a sufficiently small error. We will see that to control this error we will need to control also $\dot{p}(t)$. In order to finally achieve the desired decay, the approximate solution of \eqref{EqPhiLamHOPE} will be given by the sum of three terms. The first term, $p_{\star}$, will be explicit and will give that
\begin{align*}
	\lambda(t)\approx2e^{-\frac{\gamma+2}{2}}\sqrt{T-t}e^{-\sqrt{\frac{|\ln(T-t)|}{2}}} 
\end{align*}
where $\gamma=0.5772...$ is the Euler-Mascheroni constant.
The second term of the expansion, $p_{1}$, and its first derivative will be found by a fixed point argument. We will repeat the same argument also for the last term, $p_{2}$, that will erase the main error we introduced with $p_{1}$. We notice that in order to control the first and the second derivatives of $p_{2}$ (and then to control the final error as in Proposition \ref{prop-lambda0}) it will be necessary to control also the second and the third derivatives of $p_{1}$. \newline
Let us start by proving a Lemma that we will need in order to make explicit the dependence on the Euler-Mascheroni constant $\gamma$.
\begin{lemma}\label{SECMASSexpansionMU}
	Let $Z_{0}(y)=\frac{16-16y^{2}}{(1+y^{2})^{3}}$, then for any $a>1$ we have
	\begin{align}\label{MASSSECformulaEulerMasInt}
		\int_{0}^{\infty}e^{-\frac{u^{2}}{4}}Z_{0}(ua)udu=-\frac{2}{a^{4}}[Ei(-\frac{1}{4a^{2}})+1]+O(\frac{\ln a}{a^{6}})
	\end{align}
	where $Ei$ is the \emph{exponential integral function}, $Ei(x)=-\int_{-x}^{\infty}\frac{e^{-t}}{t}dt$.
\end{lemma}
\begin{proof}
	Since $Z_{0}$ has zero mass we can write
	\begin{align}\label{FirstDecomLEMMA61}
		\int_{0}^{\infty}e^{-\frac{u^{2}}{4}}Z_{0}(ua)udu=\int_{0}^{\infty}(e^{-\frac{u^{2}}{4}}-1)Z_{0}(ua)udu=\int_{1/a}^{\infty}(...)du+\int_{0}^{1/a}(...)du.
	\end{align}
  Let us start expanding the second integral in the right-hand side of \eqref{FirstDecomLEMMA61}. We see
  \begin{align*}
  	\int_{0}^{1/a}(e^{-\frac{u^{2}}{4}}-1)Z_{0}(au)udu=\frac{1}{a^{2}}\int_{0}^{1}(e^{-\frac{z^{2}}{4a^{2}}}-1)Z_{0}(z)zdz=-\frac{1}{4a^{4}}\int_{0}^{1}z^{3}Z_{0}(z)dz+O(\frac{1}{a^{6}}).
  \end{align*}
  We can directly compute the integral in the left-hand side. Indeed, after some elementary computations, we get
  \begin{align}\label{MASSSECOprimitiveint}
  	&\int z^{3}Z_{0}(z)dz=-8[\frac{2+3z^{2}}{(1+z^{2})^{2}}+\ln(1+z^{2})]+\text{constant}\\
  	&\implies \int_{0}^{1} z^{3}Z_{0}(z)dz=-4(-\frac{3}{2}+2\ln 2)\nonumber 
  \end{align}
  and then
  \begin{align*}
  	\int_{0}^{1/a}(e^{-\frac{u^{2}}{4}}-1)Z_{0}(ua)udu =-\frac{2}{a^{4}}[\frac{3}{4}-\ln2]+O(\frac{1}{a^{6}}).
  \end{align*}
	Now we want to study the first integral in the right-hand side of \eqref{FirstDecomLEMMA61}. We can decompose it as
	\begin{align*}
		\int_{1/a}^{\infty}(e^{-\frac{u^{2}}{4}}-1)Z_{0}(ua)udu=&-\frac{16}{a^{4}}\int_{1/a}^{\infty}(e^{-\frac{u^{2}}{4}}-1)\frac{1}{u^{3}}du+\int_{1/a}^{\infty}(e^{-\frac{u^{2}}{4}}-1)(Z_{0}(ua)+\frac{16}{a^{4}u^{4}})udu=\\
		=&\text{[I]}+\text{[II]}
	\end{align*}
	If we consider $\text{[I]}$, integrating by parts and changing the variable we observe
	\begin{align*}
		-\frac{16}{a^{4}}\int_{1/a}^{\infty}(e^{-\frac{u^{2}}{4}}-1)\frac{1}{u^{3}}du=-\frac{2}{a^{4}}[Ei(-\frac{1}{4a^{2}})-1]+O(\frac{1}{a^{6}}).
	\end{align*}
	Now we write $\text{[II]}=\int_{1/a}^{1}(...)ds+\int_{1}^{\infty}(...)ds=\text{[II]}^{(l)}+\text{[II]}^{(ii)}$. The integral $\text{[II]}^{(ii)}$ can be estimated easily
	\begin{align*}
		|\int_{1}^{\infty}(e^{-\frac{u^{2}}{4}}-1)(Z_{0}(ua)+\frac{16}{a^{4}u^{4}})udu|\le \frac{1}{a^{6}}\int_{1}^{\infty}|e^{-\frac
			{u^{2}}{4}}-1|\frac{1}{u^{5}}du\le C\frac{1}{a^{6}}.
	\end{align*} 
	for $\text{[II]}^{(l)}$ instead we observe
	\begin{align*}
		&\int_{1/a}^{1}(e^{-\frac{u^{2}}{4}}-1)(Z_{0}(ua)+\frac{16}{a^{4}u^{4}})udu=-\frac{1}{4}\int_{1/a}^{1}u^{3}(Z_{0}(ua)+\frac{16}{a^{4}u^{4}})du+O(\frac{\ln a}{a^{6}}).
	\end{align*}
	Recalling \eqref{MASSSECOprimitiveint} we can compute
	\begin{align*}
		\int_{1/a}^{1}u^{3}Z_{0}(ua)du&=\frac{1}{a^{4}}\int_{1}^{a}z^{3}Z_{0}(z)dz=\frac{8}{a^{4}}[\frac{5}{4}+\ln 2 -2\ln a +O(\frac{1}{a^{2}})]
	\end{align*}
	and
	\begin{align*}
		\int_{1/a}^{1}u^{3}\frac{16}{a^{4}u^{4}}du=\frac{16}{a^{4}}\int_{1/a}^{1}\frac{du}{u}=\frac{16}{a^{4}}[-\ln(1/a)]=\frac{8}{a^{4}}2\ln a.
	\end{align*}
	This implies 
	\begin{align*}
		\int_{1/a}^{\infty}(e^{-\frac{u^{2}}{4}}-1)(Z_{0}(ua)+\frac{16}{a^{4}u^{4}})udu=-\frac{2}{a^{4}}[\frac{5}{4}+\ln2]+O(\frac{\ln a}{a^{6}})
	\end{align*}
	and finally
	\begin{align*}
		\int_{0}^{\infty} e^{-\frac{u^{2}}{4}}Z_{0}(ua)udu =-\frac{2}{a^{4}}[\frac{3}{4}-\ln 2]-\frac{2}{a^{4}}[Ei(-\frac{1}{4a^{2}})-1]-\frac{2}{a^{4}}[\frac{5}{4}+\ln2]+O(\frac{\ln a}{a^{6}}).
	\end{align*}
\end{proof}
We recall a well known expansion for the exponential integral function when the argument is a small negative real number that we will use later and that can be found in \cite{AS} (p. 229, 5.1.11 )
\begin{align}\label{EXPONENTIALINTEGRALFUNCTIONEXP}
	Ei(-x)=\gamma +\ln x +O(x) \ \  \text{ if } x>0 \text{ is small}
\end{align}
where $\gamma=0.57721...$ is the Euler-Mascheroni constant.
Thanks to Lemma \ref{SECMASSexpansionMU} we are ready to prove the expansion of the mass of the solution of \eqref{2dphilambez3}. Notice that we will find an expansion when $t\ge-\frac{T}{2}$. We are not directly considering $(0,T)$ since we are going to cut some remainders (see for example how we will construct \eqref{FinalEq}). In this way we can take the cut-off function to be $1$ in $(0,T)$ and obtained the desired result in that interval. We stress that \eqref{RestrictionEPSILON} will be used only to prove Lemma \ref{MASSOUTERphiLAMBDA}, for the other results it would be sufficient to assume less restrictive conditions on $\varepsilon(T)$. In the following Lemmas we start expanding the mass of $\varphi_{\lambda}$ and we will isolate the constants since in Section \ref{proofThm1}, Corollary \ref{Corollary EXPANSION MASS}, we will be also interested in the value of the mass at time $T$. \newline
In the following we will make an extensive use of the following norm
\begin{align}\label{normP}
	\|p\|_{\gamma,m}:=\sup_{t\in[-\varepsilon(T),T]}|\ln(T-t)|^{m}e^{2a\gamma \sqrt{|\ln(T-t)|}}|p(t)|,
\end{align}
where $\gamma\ge1$ and $m\in \mathbb{R}$. 
\begin{lemma}\label{Mass10}
	Let $T$, $\varepsilon(T)$ be such that \eqref{RestrictionEPSILON} holds.
	Let $\lambda(t)$ satisfy \eqref{LambAss1}, \eqref{LambAss2}. Fix $\gamma\ge 1$ and $m\in \mathbb{R}$. For any $p(t)=\lambda\dot{\lambda}$ such that $\|p\|_{\gamma,m}<\infty$, for any $l>\frac{1}{2}$ and $\forall t \ge -\frac{T}{2}$we have
	\begin{align}\label{EquationLemma62}
		\int_{\R^2}\varphi^{(1)}_{\lambda}(x,t) dx=&M_{1}[\varepsilon(T)]+\nonumber\\
		&+32\pi\frac{1}{\delta} \int_{-f_{0}(T)}^{t}\frac{p(s)}{T-s}\int_{0}^{\infty}[1-e^{-\frac{z^{2}}{4}\frac{\delta(T-s)}{t-s}}(1+\frac{z^{2}}{4}\frac{\delta(T-s)}{t-s})]\frac{1}{z^{3}}(1-\chi_{0}(z))dzds-\nonumber\\
		&-4\pi \int_{-f_{0}(T)}^{t-\lambda^{2}}\frac{p(s)}{t-s}ds+4\pi (\gamma +1-\ln 4)p(t)+ \mathcal{B}[p](t)+\mathcal{E}_{0}[p](t)
	\end{align}
	where $M_{1}[\varepsilon(T)]$ is a constant satisfying $|M_{1}[\varepsilon(T)]|\le C e^{-4a\sqrt{|\ln \varepsilon(T)|}}\sqrt{|\ln(\varepsilon(T)|}$, $\gamma$ is the Euler-Mascheroni constant, $\mathcal{B}$ is an operator such that
	\begin{align*}
	|\mathcal{B}[p]| \le C \|p\|_{\gamma,m} \frac{e^{-2a\gamma\sqrt{|\ln(T-t)|}}}{|\ln(T-t)|^{m+l-\frac{1}{2}}}
	\end{align*}
    and such that, if $p_{1}$, $p_{2}$ satisfy $\|p_{i}\|_{\gamma,m}<\infty$,
 \begin{align}\label{Bdifference}
 	|\mathcal{B}[p_{1}]-\mathcal{B}[p_{2}]| \le C \|p_{1}-p_{2}\|_{\gamma,m} \frac{e^{-2a\gamma\sqrt{|\ln(T-t)|}}}{|\ln(T-t)|^{m+l-\frac{1}{2}}}.
 \end{align}
	If $\lambda$ satisfies also \eqref{RegLambAss},
	\begin{align*}
		|\mathcal{E}_{0}[p]|\le C e^{-4a\sqrt{|\ln(T-t)|}}|\ln(T-t)|^{1+l}.
	\end{align*}
\end{lemma}
\begin{proof}
	Let us just denote $\varphi$ instead of $\varphi_{\lambda}^{(1)}$. By Duhamel's formula for any $x\in\R^{6}$ we have
	\begin{align*}
		\varphi[p,\lambda](x,t)=\frac{1}{(4\pi)^{3}}\int_{-f_{0}(T)}^{t}\frac{p(s)}{\lambda^{4}(s)}\frac{1}{(t-s)^{3}}\int_{\R^6}e^{-\frac{|x-y|^{2}}{4(t-s)}}Z_{0}(\frac{y}{\lambda(s)})\chi(\frac{y}{\sqrt{\delta(T-s)}})dyds.
	\end{align*}
	By Lemma 7.2 in \cite{DdPDMW} we also know 
	\begin{align*}
		\frac{1}{(4\pi)^{3}}\int_{\R^6} e^{-\frac{|z|^{2}}{4}}\frac{1}{|w-z|^{4}}dz=\frac{1}{|w|^{4}}[1-e^{-\frac{|w|^{2}}{4}}(1+\frac{|w|^{2}}{4})], \ \ \ w\in \R^{6}.
	\end{align*}
	After observing that $\varphi$ is radial and  $\int_{\R^2}\varphi(x,t)dx=2\pi \int_{0}^{\infty} \varphi(\rho,t)\rho d\rho$ and\newline $\int_{\R^6}\varphi (x,t)\frac{1}{|x|^{4}}dx=\int_{0}^{\infty}\varphi(\rho,t)\rho^{-4}\frac{2\pi^{6/2}}{\Gamma(6/2)}\rho^{6-1}d\rho$, we get
	\begin{align*}
		\int_{\R^2}\varphi[p,\lambda]dx&=2\pi \int_{-\varepsilon(T)}^{t}\frac{p(s)}{\lambda^{4}(s)}\int_{0}^{\infty}[1-e^{-\frac{r^{2}}{4(t-s)}}(1+\frac{r^{2}}{4(t-s)})]Z_{0}(\frac{r}{\lambda(s)})\chi_{0}(\frac{r}{\sqrt{\delta(T-s)}})rdrds.
	\end{align*}
    Let us start by splitting the integral
	\begin{align*}
		\frac{1}{2\pi}\int_{\R^2}\varphi(x,t)dx=&\int_{-\varepsilon(T)}^{t}\frac{p(s)}{\lambda^{4}(s)}\int_{0}^{\infty}[1-e^{-\frac{r^{2}}{4(t-s)}}(1+\frac{r^{2}}{4(t-s)})]Z_{0}(\frac{r}{\lambda(s)})rdrds+\\
		&+\int_{-\varepsilon(T)}^{t}\frac{p(s)}{\lambda^{4}(s)}\int_{0}^{\infty}[1-e^{-\frac{r^{2}}{4(t-s)}}(1+\frac{r^{2}}{4(t-s)})]Z_{0}(\frac{r}{\lambda(s)})(\chi_{0}(\frac{r}{\sqrt{\delta(T-s)}})-1)rdrds=\\
		=&M^{\varphi}+M^{\varphi}_{\chi}.
	\end{align*}
	\underline{The mass $M^{\varphi}$:} here the hypothesis $t\ge-\frac{T}{2}$ is helpful, indeed we want to give meaning to the following decomposition
	\begin{align*}
		M^{\varphi}=\int_{-\varepsilon(T)}^{t}(...)ds=\int_{-\varepsilon(T)}^{t-(T-t)}(...)ds+\int_{t-(T-t)}^{t}(...)ds=M^{\varphi}_{(1)}+M^{\varphi}_{(2)}.
	\end{align*}
	We notice that if we assumed $t\ge-\varepsilon(T)$  at time $t=-\varepsilon(T)$ we would see $(t-(T-t))|_{t=-\varepsilon(T)}=-2\varepsilon(T)-T$ and our problem is not defined when $t<-\varepsilon(T)$. The assumption $T\ll \varepsilon(T)$ helps us to avoid this technical issue. A direct computation gives
	\begin{align*}
		M_{(1)}^{\varphi}=&-2\int_{-\varepsilon(T)}^{t-(T-t)}\frac{p(s)}{t-s}ds+ \\
		&+\int_{-\varepsilon(T)}^{t-(T-t)}\frac{p(s)}{\lambda^{2}(s)}\int_{0}^{\infty}[1-e^{-\frac{z^{2}}{4}\frac{\lambda^{2}(s)}{t-s}}(1+\frac{z^{2}}{4}\frac{\lambda^{2}(s)}{t-s})](Z_{0}(z)+\frac{16}{z^{4}})zdzds.
	\end{align*}
	To expand the last integral and in many cases later on we will make use of the following function
	\begin{align}\label{SECMASS function f }
		f_{z}(x)=1-e^{-\frac{z^{2}}{4}x}(1+\frac{z^{2}}{4}x) \implies f'_{z}(x)=\frac{z^{4}}{16}x e^{-\frac{z^{2}}{4}x}.
	\end{align}
    Now we write
	\begin{flalign*}
		M_{(1)}^{\varphi}=&-2\int_{-\varepsilon(T)}^{t-(T-t)}\frac{p(s)}{t-s}ds+\int_{-\varepsilon(T)}^{t-(T-t)}\frac{p(s)}{\lambda^{2}(s)}\int_{0}^{\infty}f_{z}(\frac{\lambda^{2}(s)}{T-s})(Z_{0}(z)+\frac{16}{z^{4}})zdzds+\\
		&+\int_{-\varepsilon(T)}^{t-(T-t)}\frac{p(s)}{\lambda^{2}(s)}\int_{0}^{\infty}(f_{z}(\frac{\lambda^{2}(s)}{t-s})-f_{z}(\frac{\lambda^{2}(s)}{T-s}))(Z_{0}(z)+\frac{16}{z^{4}})zdzds.
	\end{flalign*}
	By \eqref{SECMASS function f }, we have
	\begin{align*} |f_{z}(\frac{\lambda^{2}(s)}{t-s})-f_{z}(\frac{\lambda^{2}(s)}{T-s})| \le C z^{4}\frac{\lambda^{2}(s)}{t-s}e^{-\frac{z^{2}}{4}\frac{\lambda^{2}(s)}{T-s}}\frac{\lambda^{2}(s)}{(t-s)(T-s)}(T-t).
	\end{align*}
	Moreover $s<t-(T-t)\iff 1<\frac{T-s}{t-s}<2$ and then we have
	\begin{align*}
		|f_{z}(\frac{\lambda^{2}(s)}{t-s})-f_{z}(\frac{\lambda^{2}(s)}{T-s}))|\le C (T-t)z^{4}\frac{\lambda^{4}(s)}{(T-s)^{3}}e^{-\frac{z^{2}}{4}\frac{\lambda^{2}(s)}{T-s}}.
	\end{align*}
	We can observe 
	\begin{align*}
		|&\int_{-\varepsilon(T)}^{t-(T-t)}\frac{p(s)}{\lambda^{2}(s)}\int_{0}^{\infty}(f_{z}(\frac{\lambda^{2}(s)}{t-s})-f_{z}(\frac{\lambda^{2}(s)}{T-s}))(Z_{0}(z)+\frac{16}{z^{4}})zdzds|\le \\
		&\le C (T-t) \int_{-\varepsilon(T)}^{t-(T-t)}|p(s)|\frac{\lambda^{2}(s)}{(T-s)^{3}}\int_{0}^{\infty} e^{-\frac{z^{2}}{4}\frac{\lambda^{2}(s)}{T-s}}|Z_{0}(z)+\frac{16}{z^{4}}|z^{5}dzds.
	\end{align*}
	Observing that $|Z_{0}(z)+\frac{16}{z^{4}}|=O(\frac{1}{z^{6}})$ if $z>1$ and $|Z_{0}(z)+\frac{16}{z^{4}}|=O(\frac{1}{z^{4}})$ if $z\le1$ we have
	\begin{align*}
		\int_{0}^{\infty} e^{-\frac{z^{2}}{4}\frac{\lambda^{2}(s)}{T-s}}|Z_{0}(z)+\frac{16}{z^{4}}|z^{5}dz\le& C(\int_{0}^{1}zdz+\int_{1}^{\infty}\frac{1}{z}e^{-\frac{z^{2}}{4}\frac{\lambda^{2}(s)}{T-s}}dz)\le C[1+\int_{\frac{|\lambda(s)|}{\sqrt{T-s}}}^{\infty}\frac{1}{u}e^{-\frac{u^{2}}{4}}du]\le\\
		\le& C|\ln(\frac{|\lambda(s)|}{\sqrt{T-s}})|\le C\sqrt{|\ln(T-s)|}
	\end{align*}
	where we used that, thanks to \eqref{LambAss1}, we have $c_{1} e^{-a\sqrt{\ln(T-s)|}}\le \frac{|\lambda(s)|}{\sqrt{T-s}}\le c_{2} e^{-a\sqrt{\ln(T-s)|}}$.
	But then by \eqref{LambAss2} we have
	\begin{align*}
		|\int_{-\varepsilon(T)}^{t-(T-t)}\frac{p(s)}{\lambda^{2}(s)}\int_{0}^{\infty}&(f_{z}(\frac{\lambda^{2}(s)}{t-s})-f_{z}(\frac{\lambda^{2}(s)}{T-s}))(Z_{0}(z)+\frac{16}{z^{4}})zdzds|\le \\
		\le& C  (T-t)\sqrt{|\ln(T-t)|}\int_{-\varepsilon(T)}^{t-(T-t)}\frac{e^{-4a\sqrt{|\ln(T-s)|}}}{(T-s)^{2}}ds\le e^{-4a\sqrt{|\ln(T-t)|}}\sqrt{|\ln(T-t)|}
	\end{align*}
	where in the last inequality we used that $\frac{d}{ds}(\frac{e^{-4a\sqrt{|\ln(T-t)|}}}{T-s}|)\approx \frac{e^{-4a\sqrt{|\ln(T-s)|}}}{(T-s)^{2}}$ and that if $x>0$ is small the function $\frac{e^{-4a\sqrt{|\ln(x)|}}}{x}$ is decreasing (here we are using $\varepsilon(T)\ll1$).
	We can include this term in the operator $\mathcal{E}_{0}$ of \eqref{EquationLemma62} (notice that the hypothesis \eqref{RegLambAss} has not been used yet) and write that if $t\ge-\frac{T}{2}$
	\begin{align*}
		M^{\varphi}_{(1)}=&-2\int_{-\varepsilon(T)}^{t-(T-t)}\frac{p(s)}{t-s}ds+\\
		&+\int_{-\varepsilon(T)}^{t-(T-t)}\frac{p(s)}{\lambda^{2}(s)}\int_{0}^{\infty}[1-e^{-\frac{z^{2}}{4}\frac{\lambda^{2}(s)}{T-s}}(1+\frac{z^{2}}{4}\frac{\lambda^{2}(s)}{T-s})](Z_{0}(z)+\frac{16}{z^{4}})zdzds+\mathcal{E}_{0}[p,\lambda](t).
	\end{align*}
	Now we want to expand the second term. We see
	\begin{align*}
		\int_{-\varepsilon(T)}^{t-(T-t)}&\frac{p(s)}{\lambda^{2}(s)}\int_{0}^{\infty}[1-e^{-\frac{z^{2}}{4}\frac{\lambda^{2}(s)}{T-s}}(1+\frac{z^{2}}{4}\frac{\lambda^{2}(s)}{T-s})](Z_{0}(z)+\frac{16}{z^{4}})zdzds=\\
		=&\int_{-\varepsilon(T)}^{T}(...)ds - \int_{t-(T-t)}^{T}\frac{p(s)}{\lambda^{2}(s)}\int_{0}^{\infty}[1-e^{-\frac{z^{2}}{4}\frac{\lambda^{2}(s)}{T-s}}(1+\frac{z^{2}}{4}\frac{\lambda^{2}(s)}{T-s})](Z_{0}(z)+\frac{16}{z^{4}})zdzds.
	\end{align*}
    Observing that $|1-e^{-\frac{u^{2}}{4}}(1+\frac{u^{2}}{4})|\le C u^{4}$ if $u\le1$, that $ |1-e^{-\frac{u^{2}}{4}}(1+\frac{u^{2}}{4})|=O(1)$, that  $|Z_{0}(y)+\frac{16}{y^{4}}|\le C \frac{1}{y^{4}}$ if $y\le1$ and that $|Z_{0}(y)+\frac{16}{y^{4}}|\le C \frac{1}{y^{6}}$ if $y>1$ we can write
	\begin{align*}
		&|\int_{t-(T-t)}^{T}\frac{p(s)}{\lambda^{2}(s)}\int_{0}^{\infty}[1-e^{-\frac{z^{2}}{4}\frac{\lambda^{2}(s)}{T-s}}(1+\frac{z^{2}}{4}\frac{\lambda^{2}(s)}{T-s})](Z_{0}(z)+\frac{16}{z^{4}})zdzds|=\\
		&\le C \int_{t-(T-t)}^{T}\frac{|p(s)|}{\lambda^{4}(s)}(T-s)[\frac{\lambda^{4}(s)}{(T-s)^{2}}\int_{0}^{\frac{|\lambda(s)|}{\sqrt{T-s}}}udu+\frac{\lambda^{6}(s)}{(T-s)^{3}}\int_{\frac{|\lambda(s)|}{\sqrt{T-s}}}^{1}\frac{du}{u}+\frac{\lambda^{6}(s)}{(T-s)^{3}}\int_{1}^{\infty}\frac{du}{u^{5}}]\le \\
		&\le C \int_{t-(T-t)}^{T}\frac{e^{-4a\sqrt{|\ln(T-s)|}}}{T-s}\sqrt{|\ln(T-s)|}ds\le C e^{-4a\sqrt{|\ln(T-t)|}} |\ln(T-t)|
	\end{align*}
	where we used that $\frac{d}{ds}(e^{-4a\sqrt{|\ln(T-t)|}}|\ln(T-s)|)\approx \frac{e^{-4a\sqrt{|\ln(T-s)|}}}{T-s}\sqrt{|\ln(T-s)|}$. We can include this term in $\mathcal{E}_{0}$.
	The same considerations give us that
	\begin{align*}
	   |\int_{-\varepsilon(T)}^{T} \frac{p(s)}{\lambda^{2}(s)}\int_{0}^{\infty}[1-e^{-\frac{z^{2}}{4}\frac{\lambda^{2}(s)}{T-s}}(1+\frac{z^{2}}{4}\frac{\lambda^{2}(s)}{T-s})](Z_{0}(z)+\frac{16}{z^{4}})zdzds|\le C e^{-4a\sqrt{|\ln \varepsilon(T)|}}\sqrt{|\ln \varepsilon(T)}|
	\end{align*}
and then this constant can clearly be included in $M_{1}[\varepsilon(T)]$ of \eqref{EquationLemma62}. \newline
Then if $t\ge-\frac{T}{2}$ we proved
\begin{align}\label{MASSSECON phi1}
	M_{(1)}^{\varphi}=&M_{1}[\varepsilon(T)]-2 \int_{-\varepsilon(T)}^{t-(T-t)}\frac{p(s)}{t-s}ds + \mathcal{E}_{0}[p,\lambda](t).
\end{align}
Now we want to expand with $M_{(2)}^{\varphi}$. Here we will need the parameter $l>1/2$, indeed we write
	\begin{align}\label{ExM2phi}
		M_{(2)}^{\varphi}=&\int_{t-(T-t)}^{t}(...)ds=\int_{t-(T-t)}^{t-|\ln(T-t)|^{l}\lambda^{2}(t)}(...)ds + \int_{t-|\ln(T-t)|^{l}\lambda^{2}(t)}^{t}(...)ds=\nonumber\\
     =&-2\int_{t-(T-t)}^{t-|\ln(T-t)|^{l}\lambda^{2}(t)}\frac{p(s)}{t-s}ds +\nonumber\\
	&+ \int_{t-(T-t)}^{t-|\ln(T-t)|^{l}\lambda^{2}(t)}\frac{p(s)}{\lambda^{4}(s)}\int_{0}^{\infty}[1-e^{-\frac{r^{2}}{4(t-s)}}(1+\frac{r^{2}}{4(t-s)})](Z_{0}(\frac{r}{\lambda(s)})+16\frac{\lambda^{4}(s)}{r^{4}})rdrds+\nonumber\\
		&+\int_{t-|\ln(T-t)|^{l}\lambda^{2}(t)}^{t}(...)ds
	\end{align}
notice that here $t-|\ln(T-t)|^{l}\lambda^{2}(t)\ge t-(T-t)$ since $T-t \ge |\ln(T-t)|^{l}\lambda^{2}(t)$ if $t\in(-\frac{T}{2},T)$. If we consider the last term in \eqref{ExM2phi} we can write
	\begin{align}\label{LastTermMasPhi2}
		&\int_{t-|\ln(T-t)|^{l}\lambda^{2}(t)}^{t}\frac{p(s)}{\lambda^{4}(s)}\int_{0}^{\infty}[1-e^{-\frac{r^{2}}{4(t-s)}}(1+\frac{r^{2}}{4(t-s)})]Z_{0}(\frac{r}{\lambda(s)})rdrds=\nonumber\\
        &=\frac{p(t)}{\lambda^{4}(t)}\int_{t-|\ln(T-t)|^{l}\lambda^{2}(t)}^{t}\int_{0}^{\infty}[1-e^{-\frac{r^{2}}{4(t-s)}}(1+\frac{r^{2}}{4(t-s)})]Z_{0}(\frac{r}{\lambda(t)})rdrds+\nonumber\\
		&\hspace{0.2cm}+p(t)\int_{t-|\ln(T-t)|^{l}\lambda^{2}(t)}^{t}\int_{0}^{\infty}[1-e^{-\frac{r^{2}}{4(t-s)}}(1+\frac{r^{2}}{4(t-s)})](\frac{1}{\lambda^{4}(s)}Z_{0}(\frac{r}{\lambda(s)})-\frac{1}{\lambda^{4}(t)}Z_{0}(\frac{r}{\lambda(t)}))rdrds+\nonumber\\
		&\hspace{0.2cm}+\int_{t-|\ln(T-t)|^{l}\lambda^{2}(t)}^{t}\frac{p(s)-p(t)}{\lambda^{4}(s)}\int_{0}^{\infty}[1-e^{-\frac{r^{2}}{4(t-s)}}(1+\frac{r^{2}}{4(t-s)})]Z_{0}(\frac{r}{\lambda(s)})rdrds.
	\end{align}
	We start our expansion from the second term of \eqref{LastTermMasPhi2}. First we observe
	\begin{align*}
		\frac{1}{\lambda^{4}(s)}Z_{0}(\frac{r}{\lambda(s)})-\frac{1}{\lambda^{4}(t)}Z_{0}(\frac{r}{\lambda(t)})=\frac{1}{\lambda^{4}(t)}(\frac{\lambda^{4}(t)}{\lambda^{4}(s)}-1)Z_{0}(\frac{r}{\lambda(s)})+\frac{1}{\lambda^{4}(t)}(Z_{0}(\frac{r}{\lambda(s)})-Z_{0}(\frac{r}{\lambda(t)})).
	\end{align*}
   Now, since $\lambda$ satisfies \eqref{LambAss1} and \eqref{LambAss2} we immediately see that if $t-|\ln(T-t)|^{l}\lambda^{2}(t)<s<t$ 
	\begin{align*}
    |\frac{1}{\lambda^{4}(t)}Z_{0}(\frac{r}{\lambda(s)})(\frac{\lambda^{4}(t)}{\lambda^{4}(s)}-1)|\le C \frac{1}{r^{4}}\frac{t-s}{T-s}.
	\end{align*}
    Since by a direct computation we also have $|\partial_{s} Z_{0}(\frac{r}{\lambda(s)})| \le C \frac{1}{(\frac{r^{2}}{\lambda^{2}(s)})^{2}}\frac{|\lambda\dot{\lambda}(s)|}{\lambda^{2}(s)}\le C \frac{\lambda^{2}(s)}{r^{4}}|\lambda\dot{\lambda}(s)|$ we also get that if $t-|\ln(T-t)|^{l}\lambda^{2}(t)<s<t$
	\begin{align*}
		&|\frac{1}{\lambda^{4}(t)}(Z_{0}(\frac{r}{\lambda(s)})-Z_{0}(\frac{r}{\lambda(t)}))| \le \frac{\lambda^{4}}{r^{4}}\frac{t-s}{T-s}
	\end{align*}
 We proved that if $t-|\ln(T-t)|^{l}\lambda^{2}(t)\le s \le t$ we have
	\begin{align*}
		|\frac{1}{\lambda^{4}(s)}Z_{0}(\frac{r}{\lambda(s)})-\frac{1}{\lambda^{4}(t)}Z_{0}(\frac{r}{\lambda(t)})|\le C \frac{1}{r^{4}}\frac{t-s}{T-s}.
	\end{align*}
	Finally we can prove 
	\begin{align*}
		&|p(t)\int_{t-|\ln(T-t)|^{l}\lambda^{2}(t)}^{t}\int_{0}^{\infty}[1-e^{-\frac{r^{2}}{4(t-s)}}(1+\frac{r^{2}}{4(t-s)})][\frac{1}{\lambda^{4}(s)}Z_{0}(\frac{r}{\lambda(s)})-\frac{1}{\lambda^{4}(t)}Z_{0}(\frac{r}{\lambda(t)})]rdrds|\le\\
		&\le C |p(t)|\int_{t-|\ln(T-t)|^{l}\lambda^{2}(t)}^{t}\frac{t-s}{T-s}\int_{0}^{\infty} |1-e^{-\frac{r^{2}}{4(t-s)}}(1+\frac{r^{2}}{4(t-s)})|\frac{1}{r^{3}}drds\le \\
		&\le C|p(t)|\int_{t-|\ln(T-t)|^{l}\lambda^{2}(t)}^{t}\frac{ds}{T-s}\le C|p(t)| |\ln(1+|\ln(T-t)|^{l}\frac{\lambda^{2}(t)}{T-t})|\le C e^{-4a\sqrt{|\ln(T-t)|}}|\ln(T-t)|^{l}
	\end{align*}
and also this term can be included in $\mathcal{E}_{0}$.\newline
	We want to estimate the last term of \eqref{LastTermMasPhi2}. Thanks to \eqref{RegLambAss} we see
	\begin{align*}
		&|\int_{t-|\ln(T-t)|^{l}\lambda^{2}}^{t}\frac{p(s)-p(t)}{\lambda^{4}(s)}\int_{0}^{\infty}[1-e^{-\frac{r^{2}}{4(t-s)}}(1+\frac{r^{2}}{4(t-s)})]Z_{0}(\frac{r}{\lambda(s)})rdrds|=\\
        &\le C \int_{t-|\ln(T-t)|^{l}\lambda^{2}(t)}^{t}\frac{|p(t)-p(s)|}{t-s}ds\le C e^{-4a\sqrt{|\ln(T-t)|}}|\ln(T-t)|^{l-1/2}
	\end{align*}
     and then also this term can be included in $\mathcal{E}_{0}$.
	Then so far we have
	\begin{align}\label{M2phi2SecForm}
		M_{(2)}^{\varphi}=&-2\int_{t-(T-t)}^{t-|\ln(T-t)|^{l}\lambda^{2}(t)}\frac{p(s)}{t-s}ds +\nonumber\\
		&+ \int_{t-(T-t)}^{t-|\ln(T-t)|^{l}\lambda^{2}(t)}\frac{p(s)}{\lambda^{4}(s)}\int_{0}^{\infty}[1-e^{-\frac{r^{2}}{4(t-s)}}(1+\frac{r^{2}}{4(t-s)})](Z_{0}(\frac{r}{\lambda(s)})+16\frac{\lambda^{4}(s)}{r^{4}})rdrds+\nonumber\\
		&-\frac{p(t)}{\lambda^{4}(t)}\int_{t-|\ln(T-t)|^{l}\lambda^{2}(t)}^{t}\int_{0}^{\infty} e^{-\frac{r^{2}}{4(t-s)}}(1+\frac{r^{2}}{4(t-s)})Z_{0}(\frac{r}{\lambda(t)})rdrds+\mathcal{E}_{0}[p,\lambda](t).
	\end{align}
	Now we want to investigate the second integral of \eqref{M2phi2SecForm}. If $t-(T-t)\le s \le t-|\ln(T-t)|^{l}\lambda^{2}(t)$ we have
	\begin{align*}
		&|\int_{0}^{\infty}[1-e^{-\frac{z^{2}}{4}}(1+\frac{z^{2}}{4})](Z_{0}(\frac{z\sqrt{t-s}}{\lambda(s)})+16\frac{\lambda^{4}(s)}{z^{4}(t-s)^{2}})zdz|\le  C\frac{\lambda^{6}(s)}{(t-s)^{3}}|\ln(\frac{\lambda^{2}(s)}{t-s})|.
	\end{align*}
	If $t-(T-t)\le s \le t-|\ln(T-t)|^{l}\lambda^{2}(t)$ we have
	\begin{align*}
		C\frac{\lambda^{2}(t-|\ln(T-t)|^{l}\lambda^{2})}{T-t}\le \frac{\lambda^{2}(s)}{t-s}\le \frac{\lambda^{2}(s)}{\lambda^{2}(t)}\frac{1}{|\ln(T-t)|^{l}}\le C \frac{1}{|\ln(T-t)|^{l}}.
	\end{align*}
	But then
	\begin{align*}
		&|\int_{t-(T-t)}^{t-\lambda^{2}(t)|\ln(T-t)|^{l}}\frac{p(s)}{\lambda^{4}(s)}(t-s)\int_{0}^{\infty}[1-e^{-\frac{z^{2}}{4}}(1+\frac{z^{2}}{4})](Z_{0}(\frac{z\sqrt{t-s}}{\lambda(s)})+16\frac{\lambda^{4}(s)}{(t-s)^{2}}\frac{1}{z^{4}})zdzds|\le \\
		&\le C \sqrt{|\ln(T-t)|}\int_{t-(T-t)}^{t-\lambda^{2}|\ln(T-t)|^{l}}\frac{|p(s)|}{(t-s)^{2}}\lambda^{2}(s)ds\le\\
      &\le C \|p\|_{\gamma,m}\sqrt{|\ln(T-t)|}\int_{t-(T-t)}^{t-\lambda^{2}|\ln(T-t)|^{l}}\frac{T-s}{(t-s)^{2}}\frac{e^{-2a(1+\gamma)\sqrt{|\ln(T-s)|}}}{|\ln(T-s)|^{m}}ds\le C \|p\|_{\gamma,m} \frac{e^{-2a\gamma\sqrt{|\ln(T-t)|}}}{|\ln(T-t)|^{m+l-1/2}}.
	\end{align*}
	We can include this term in $\mathcal{B}$ and write
	\begin{align}\label{M2phiFinalAlmost}
		M_{(2)}^{\varphi}=&-2\int_{t-(T-t)}^{t-|\ln(T-t)|^{l}\lambda^{2}(t)}\frac{p(s)}{t-s}ds-\nonumber\\
		&-\frac{p(t)}{\lambda^{4}(t)}\int_{t-|\ln(T-t)|^{l}\lambda^{2}(t)}^{t}\int_{0}^{\infty} e^{-\frac{r^{2}}{4(t-s)}}(1+\frac{r^{2}}{4(t-s)})Z_{0}(\frac{r}{\lambda(t)})rdrds+\mathcal{B}[p,\lambda](t)+\mathcal{E}_{0}[p,\lambda](t).
	\end{align}
	Now we explicitly compute the second integral of \eqref{M2phiFinalAlmost}. First we observe
	\begin{align*}
		\int_{t-\lambda^{2}|\ln(T-t)|^{l}}^{t}e^{-\frac{r^{2}}{4(t-s)}}(1+\frac{r^{2}}{4(t-s)})ds&=\lambda^{2}|\ln(T-t)|^{l}e^{-\frac{r^{2}}{4\lambda^{2}|\ln(T-t)|^{l}}}.
	\end{align*}
Then by Fubini-Tonelli theorem, Lemma \ref{SECMASSexpansionMU} and the expansion \eqref{EXPONENTIALINTEGRALFUNCTIONEXP} we see
	\begin{align*}
		-\frac{p(t)}{\lambda^{4}(t)}&\int_{t-\lambda^{2}(t)|\ln(T-t)|^{l}}^{t}\int_{0}^{\infty} e^{-\frac{r^{2}}{4(t-s)}}(1+\frac{r^{2}}{4(t-s)})Z_{0}(\frac{r}{\lambda(t)})rdrds=\\
		&=-\frac{p(t)|\ln(T-t)|^{l}}{\lambda^{2}(t)}\int_{0}^{\infty}e^{-\frac{r^{2}}{4\lambda^{2}|\ln(T-t)|^{l}}}Z_{0}(\frac{r}{\lambda(t)})rdr=\\
		&=2(\gamma+1-\ln4)p(t)+2\ln(\frac{1}{|\ln(T-t)|^{l}})p(t)+O(\frac{\ln|\ln(T-t)|}{|\ln(T-t)|^{l}})p(t).
	\end{align*}
	Since $|O(\frac{\ln|\ln(T-t)|}{|\ln(T-t)|^{l}})p(t)|\le C \|p\|_{\gamma,m}\frac{e^{-2a\gamma\sqrt{|\ln(T-t)|}}}{|\ln(T-t)|^{m+l}}\ln|\ln(T-t)|$ we can include the remainder in the operator $\mathcal{B}$ and obtain
	\begin{align*}
		M_{(2)}^{\varphi}=&-2\int_{t-(T-t)}^{t-|\ln(T-t)|^{l}\lambda^{2}(t)}\frac{p(s)}{t-s}ds+2(\mu+1-\ln4)p(t)+2\ln(\frac{1}{|\ln(T-t)|^{l}})p(t)+\\
		&+\mathcal{B}[p,\lambda](t)+\mathcal{E}_{0}[p,\lambda](t).
	\end{align*}
	The last observation involving $M_{(2)}^{\varphi}$ is that 
	\begin{align*}
		-2p(t)\int_{t-\lambda^{2}|\ln(T-t)|^{l}}^{t-\lambda^{2}}\frac{ds}{t-s}=&2p(t)\ln(\frac{1}{|\ln(T-t)|^{l}})=\\
		=&-2\int_{t-\lambda^{2}|\ln(T-t)|^{l}}^{t-\lambda^{2}}\frac{p(s)}{t-s}ds-2\int_{t-\lambda^{2}|\ln(T-t)|^{l}}^{t-\lambda^{2}}\frac{p(t)-p(s)}{t-s}ds
	\end{align*}
	by assuming \eqref{RegLambAss} we see that the last term can be included in $\mathcal{E}_{0}$.
	We finally proved the following expansion for $M_{(2)}^{\varphi}$ 
	\begin{align}\label{MASSSECON Mphi2}
		M_{(2)}^{\varphi}=&-2\int_{t-(T-t)}^{t-\lambda^{2}(t)}\frac{p(s)}{t-s}ds+2(\gamma+1-\ln4)p(t)+\mathcal{B}[p,\lambda](t)+\mathcal{E}_{0}[p,\lambda](t).
	\end{align}
	Then \eqref{MASSSECON phi1} and \eqref{MASSSECON Mphi2} give
	\begin{align*}
		M^{\varphi}=M_{(1)}^{\varphi}+M_{(2)}^{\varphi}=&M_{1}[\varepsilon(T)]-2\int_{-\varepsilon(T)}^{t-\lambda^{2}}\frac{p(s)}{t-s}ds+2(\gamma+1-\ln4)p(t)+\mathcal{B}[p,\lambda](t)+\mathcal{E}_{0}[p,\lambda](t).
	\end{align*}
\underline{The mass $M_{\chi}^{\varphi}$:} we start observing
	\begin{align}\label{FIrstObsMassChi}
		M_{\chi}^{\varphi}=&\int_{-\varepsilon(T)}^{t}\frac{p(s)}{\lambda^{4}(s)}\int_{0}^{\infty}[1-e^{-\frac{r^{2}}{4(t-s)}}(1+\frac{r^{2}}{4(t-s)})]Z_{0}(\frac{r}{\lambda(s)})(\chi_{0}(\frac{r}{\sqrt{\delta(T-s)}})-1)rdrds=\nonumber\\
	=&-16\int_{-\varepsilon(T)}^{t}\frac{p(s)}{T-s}\int_{0}^{\infty}[1-e^{-\frac{u^{2}}{4}\frac{T-s}{t-s}}(1+\frac{z^{2}}{4}\frac{T-s}{t-s})]\frac{1}{z^{3}}(\chi_{0}(\frac{z}{\sqrt{\delta}})-1)zdzds+\nonumber\\
		&+\int_{-\varepsilon(T)}^{t}\frac{p(s)}{\lambda^{4}(s)}(T-s)\int_{0}^{\infty}[1-e^{-\frac{z^{2}}{4}\frac{(T-s)}{t-s}}(1+\frac{z^{2}}{4}\frac{T-s}{t-s})]\cdot \nonumber\\
		&\hspace{4cm}\cdot (Z_{0}(\frac{z\sqrt{T-s}}{\lambda(s)})+16\frac{\lambda^{4}(s)}{z^{4}(T-s)^{2}})(\chi_{0}(\frac{z}{\sqrt{\delta}})-1)zdzds.
	\end{align}
	We need to study the last term of \eqref{FIrstObsMassChi}. We can split this integral as $\int_{-\varepsilon(T)}^{t-(T-t)}(...)ds+\int_{t-(T-t)}^{t}(...)ds$ and immediately infer that the second integral can be included in $\mathcal{E}_{0}$.
	To study $\int_{-\varepsilon(T)}^{t-(T-t)}(...)ds$ recalling \eqref{SECMASS function f } we observe that $f_{z}(\frac{T-s}{t-s})=f_{z}(\frac{T-t}{t-s}+1)=1-e^{-\frac{z^{2}}{4}(1+\frac{T-t}{t-s})}(1+\frac{z^{2}}{4}(1+\frac{T-t}{t-s}))$. We want to expand $f_{z}(1+x)$ when $0<x<1$. Then we know that
	\begin{align*}
		f_{z}'(x+1)=\frac{z^{4}}{16}(1+x)e^{-\frac{z^{2}}{4}(1+x)} \implies |f_{z}(x+1)-f_{z}(1)|\le C z^{4}e^{-\frac{z^{2}}{4}}x.
	\end{align*}
	But then we have
	\begin{align*}
		&|\int_{-\varepsilon(T)}^{t-(T-t)}\frac{p(s)}{\lambda^{4}(s)}(T-s)\int_{0}^{\infty}[f_{z}(\frac{T-s}{t-s})-f_{z}(1)]\big(Z_{0}(\frac{z\sqrt{T-s}}{\lambda(s)})+16\frac{\lambda^{4}(s)}{z^{4}(T-s)^{2}}\big)(\chi_{0}(\frac{z}{\sqrt{\delta}})-1)zdzds|\le\\
		&\le(T-t)\int_{-\varepsilon(T)}^{t-(T-t)}\frac{|p(s)|}{\lambda^{4}(s)}\frac{T-s}{t-s}\int_{0}^{\infty}z^{4}e^{-\frac{z^{2}}{4}}|Z_{0}(\frac{z\sqrt{T-s}}{\lambda(s)})+16\frac{\lambda^{4}(s)}{z^{4}(T-s)^{2}}|(1-\chi_{0}(\frac{z}{\sqrt{\delta}}))zdzds\le \\
		&\le C(T-t) \int_{-\varepsilon(T)}^{t-(T-t)}\frac{|p(s)|}{(T-s)^{2}}\frac{\lambda^{2}(s)}{T-s}ds\le C (T-t)\int_{-\varepsilon(T)}^{t-(T-t)}\frac{e^{-4a\sqrt{|\ln(T-t)|}}}{(T-s)^{2}}ds\le C e^{-4a\sqrt{|\ln(T-t)|}}
	\end{align*}
	where we used again that $\frac{e^{-4a\sqrt{|\ln(x)|}}}{x}$ is decreasing as $x>0$ is small (here we used that $\varepsilon(T)\ll1$). So far we proved
	\begin{align}\label{Mchiphisecondbe}
		M_{\chi}^{\varphi}=&-16\int_{-\varepsilon(T)}^{t}\frac{p(s)}{(T-s)}\int_{0}^{\infty}[1-e^{-\frac{z^{2}}{4}\frac{T-s}{t-s}}(1+\frac{z^{2}}{4}\frac{T-s}{t-s})]\frac{1}{z^{3}}(\chi_{0}(\frac{z}{\sqrt{\delta}})-1)zdzds+\nonumber\\
		&+\int_{-\varepsilon(T)}^{t-(T-t)}\frac{p(s)}{\lambda^{4}(s)}	(T-s)\int_{0}^{\infty}[1-e^{-\frac{z^{2}}{4}}(1+\frac{z^{2}}{4})](Z_{0}(\frac{z\sqrt{T-s}}{\lambda(s)})+\nonumber\\
		&\hspace{6cm}+16\frac{\lambda^{4}(s)}{z^{4}(T-s)^{2}})(\chi_{0}(\frac{z}{\sqrt{\delta}})-1)zdzds+\mathcal{E}_{0}[p,\lambda](t).
	\end{align}
	We estimate the second integral in \eqref{Mchiphisecondbe} and we see that
	\begin{align*}
		&|\int_{t-(T-t)}^{T}\frac{p(s)}{\lambda^{4}(s)}(T-s)\int_{0}^{\infty}[1-e^{-\frac{z^{2}}{4}}(1+\frac{z^{2}}{4})](Z_{0}(\frac{z\sqrt{T-s}}{\lambda(s)})+16\frac{\lambda^{4}(s)}{z^{4}(T-s)^{2}})(\chi_{0}(\frac{z}{\sqrt{\delta}})-1)zdzds|\le\\
		&\le C \int_{t-(T-t)}^{T}\frac{|p(s)|}{T-s}\frac{\lambda^{2}(s)}{T-s}\int_{0}^{\infty}\frac{1}{z^{5}}(1-\chi_{0}(\frac{z}{\sqrt{\delta}}))dzds\le\\
		&\le C\int_{t-(T-t)}^{T}\frac{e^{-4a\sqrt{|\ln(T-s)|}}}{T-s}ds\le C e^{-4a\sqrt{|\ln(T-t)|}}\sqrt{|\ln(T-t)|}.
	\end{align*}
 and that analogously the remaining constant can be included in $M_{1}[\varepsilon(T)]$. We get then
	\begin{align*}
		M_{\chi}^{\varphi}=&M_{1}[\varepsilon(T)]+16\int_{-\varepsilon(T)}^{t}\frac{p(s)}{T-s}\int_{0}^{\infty}[1-e^{-\frac{u^{2}}{4}\frac{T-s}{t-s}}(1+\frac{u^{2}}{4}\frac{T-s}{t-s})]\frac{1}{u^{3}}(1-\chi_{0}(\frac{u}{\sqrt{\delta}}))duds + \mathcal{E}_{0}[p,\lambda](t).
	\end{align*}
 and finally
	\begin{align*}
		\int_{\R^2}\varphi[p,\lambda]dx=&2\pi M^{\varphi} +2\pi M_{\chi}^{\varphi}=\\
		=&M_{1}[\varepsilon(T)]+32\pi \int_{-\varepsilon(T)}^{t}\frac{p(s)}{T-s}\int_{0}^{\infty} [1-e^{-\frac{u^{2}}{4}\frac{T-s}{t-s}}(1+\frac{u^{2}}{4}\frac{T-s}{t-s})]\frac{1}{u^{3}}(1-\chi_{0}(\frac{u}{\sqrt{\delta}}))duds-\\
		&-4\pi \int_{-\varepsilon(T)}^{t-\lambda^{2}}\frac{p(s)}{t-s}ds+2(\gamma+1-\ln4)p(t)+\mathcal{B}[p,\lambda](t)+\mathcal{E}_{0}[p,\lambda](t).
	\end{align*}
The estimate \eqref{Bdifference} can be obtained similarly and we omit the details.
\end{proof}
Before introducing the next Lemma we want to remark that to prove it we used that $T\ll \varepsilon(T)$ or, more precisely, that
\begin{align}\label{EpsiTandT}
	\frac{e^{-4a\sqrt{|\ln T|}}}{T}\sqrt{|\ln T|}\ge \frac{e^{-2a\sqrt{|\ln \varepsilon(T)|}}}{\varepsilon(T)}
\end{align}
In fact, the inequality \eqref{EpsiTandT} will allow us to obtain the estimate \eqref{MASSSEC WHereEpsSqu} at the end of the proof.
\begin{lemma}\label{MASSOUTERphiLAMBDA}
	Let $T$, $\varepsilon(T)$ satisfy \eqref{RestrictionEPSILON}. Assume that $\lambda$ satisfies \eqref{LambAss1}, \eqref{LambAss2}. Then for any $t\ge-\frac{T}{2}$ we have
	\begin{align*}
		\int_{\R^2}\varphi_{\lambda}^{(2)}dx=&M_{2}[\varepsilon(T)]+16\pi\frac{\beta}{\delta}\frac{\lambda^{2}(t)}{T-t}-\\
		-&32\pi\frac{1}{\delta}\int_{-\varepsilon(T)}^{t}\frac{p(s)}{T-s}\int_{0}^{\infty}[1-e^{-\frac{z^{2}}{4}\frac{\delta(T-s)}{t-s}}(1+\frac{z^{2}}{4}\frac{\delta(T-s)}{t-s})]\frac{1}{z^{3}}(1-\chi_{0})dzds+\mathcal{E}_{0}[p,\lambda](t)
	\end{align*}
	where $M_{2}[\varepsilon(T)]$ is a constant that satisfies $|M_{2}[\varepsilon(T)]|\le e^{-2a\sqrt{|\ln(\varepsilon(T))|}}$, $\beta$ has been defined in \eqref{ExpansionRHSeqlambda} and
	\begin{align*}
		|\mathcal{E}_{0}[p,\lambda]|\le C e^{-4a\sqrt{|\ln(T-t)|}}\sqrt{|\ln(T-t)|}.
	\end{align*}
\end{lemma}
\begin{proof}
	If we call $w=\frac{x-\xi}{\sqrt{\delta(T-t)}}$ and recalling \eqref{decompV}, we observe that when $\frac{|x-\xi|}{\lambda(t)}\gg1$ we have the following expansion
	\begin{align*}
		&-\frac{1}{\lambda^{2}}\frac{1}{2(T-t)}U_{0}\nabla_{w}\chi\cdot w+\frac{2}{\lambda^{3}\sqrt{\delta(T-t)}}\nabla_{y}U_{0}\cdot \nabla_{w}\chi + \frac{1}{\lambda^{2}\delta(T-t)}U_{0}\Delta_{w}\chi -\\
		&\hspace{1cm}-\frac{1}{\lambda^{3}\sqrt{\delta(T-t)}}U_{0}\nabla_{w}\chi \cdot \nabla_{y} v_{0}=8\frac{\lambda^{2}}{\delta^{3}(T-t)^{3}}\frac{1}{|w|^{4}}[\chi_{0}''-\delta\frac{|w|}{2}\chi_{0}'-\frac{3}{|w|}\chi_{0}']+g(|w|,t)=\\
		&\hspace{5.6cm}=h(|w|,t)+g(|w|,t)
	\end{align*}
	where $|g(|w|,t)|\le C \mathds{1}_{1\le |w|\le 2}\frac{\lambda^{4}(t)}{(T-t)^{4}}$.
    By Duhamel's formula we can write
	\begin{align*}
		&2\pi \int_{-\varepsilon(T)}^{t}\int_{0}^{\infty}[1-e^{-\frac{r^{2}}{4(t-s)}}(1+\frac{r^{2}}{4(t-s)})]g(\frac{r}{\sqrt{\delta(T-s)}},s)rdrds=2\pi \int_{-\varepsilon(T)}^{t-(T-t)}(...)ds+2\pi \int_{t-(T-t)}^{t}(...)ds.
	\end{align*}
	The second integral can be estimated easily, we have
	\begin{align*}
		&|\int_{t-(T-t)}^{t}(T-s)\int_{0}^{\infty}[1-e^{-\frac{z^{2}}{4}\frac{T-s}{t-s}}(1+\frac{z^{2}}{4}\frac{T-s}{t-s})]g(\frac{z}{\sqrt{\delta}},s)zdzds|\le C\int_{t-(T-t)}^{t}\frac{\lambda^{4}(s)}{(T-s)^{3}}ds\le \\
		&\le C \int_{t-(T-t)}^{t}\frac{e^{-4a\sqrt{|\ln(T-s)|}}}{T-s}\le C e^{-4a\sqrt{|\ln(T-t)|}}\sqrt{|\ln(T-t)|}.
	\end{align*}
    Then so far we proved
	\begin{align}\label{MassphiLambdaInter}
		\int_{\R^2}\varphi_{\lambda}^{(2)}dx=&2\pi \int_{-\varepsilon(T)}^{t}\int_{0}^{\infty}[1-e^{-\frac{r^{2}}{4(t-s)}}(1+\frac{r^{2}}{4(t-s)})]h(\frac{r}{\sqrt{\delta(T-s)}},s)rdrds+\nonumber\\
		+&2\pi\int_{-\varepsilon(T)}^{t-(T-t)}\int_{0}^{\infty}[1-e^{-\frac{r^{2}}{4(t-s)}}(1+\frac{r^{2}}{4(t-s)})]g(\frac{r}{\sqrt{\delta(T-s)}},s)rdrds+\mathcal{E}_{0}[p,\lambda](t).
	\end{align}
To study the remaining integrals in \eqref{MassphiLambdaInter}, we recall \eqref{SECMASS function f } and we observe that $f_{z}(\frac{T-s}{t-s})=f_{z}(\frac{T-t}{t-s}+1)=1-e^{-\frac{z^{2}}{4}(1+\frac{T-t}{t-s})}(1+\frac{z^{2}}{4}(1+\frac{T-t}{t-s}))$ and that if $0<x<1$
\begin{align*}
	f_{z}'(x+1)=\frac{z^{4}}{16}(1+x)e^{-\frac{z^{2}}{4}(1+x)} \implies |f_{z}(x+1)-f_{z}(x)|\le C z^{4}e^{-\frac{z^{2}}{4}}x.
\end{align*}
Then
	\begin{align*}
		&|2\pi \int_{-\varepsilon(T)}^{t-(T-t)}(T-s)\int_{0}^{\infty}[f_{z}(\frac{T-s}{t-s})-f_{z}(1)]g(\frac{z}{\sqrt{\delta}},s)zdzds|\le\\
		&\hspace{1cm}\le C (T-t)\int_{-\varepsilon(T)}^{t-(T-t)} \frac{T-s}{t-s}\int_{1}^{2}z^{4}e^{-\frac{z^{2}}{4}}|g(\frac{z}{\sqrt{\delta}},s)|zdzds\le C e^{-4a\sqrt{|\ln(T-t)|}}
	\end{align*}
    and it can be included in $\mathcal{E}_{0}$. Since we have
	\begin{align*}
		&|\int_{t-(T-t)}^{T}(T-s)\int_{1}^{2}[1-e^{-\frac{z^{2}}{4}}(1+\frac{z^{2}}{4})]g(\frac{z}{\sqrt{\delta}},s)zdzds|\le C e^{-4a\sqrt{|\ln(T-t)|}}\sqrt{|\ln(T-t)|}.
	\end{align*}
    this can also be included $\mathcal{E}_{0}$ and  analogously $
    	2\pi\int_{-\varepsilon(T)}^{T}(T-s)\int_{1}^{2}[1-e^{-\frac{z^{2}}{4}}(1+\frac{z^{2}}{4})]g(\frac{z}{\sqrt{\delta}},s)zdzds$
can be included in $M_{2}[\varepsilon(T)]$. So far we have 
	\begin{align}\label{MassPhi2Intermediate2}
		\int_{\R^2}\varphi_{\lambda}^{(2)}dx=&M_{2}[\varepsilon(T)]+2\pi \int_{-\varepsilon(T)}^{t}\int_{0}^{\infty}[1-e^{-\frac{r^{2}}{4(t-s)}}(1+\frac{r^{2}}{4(t-s)})]h(\frac{r}{\sqrt{\delta(T-s)}},s)rdrds+\nonumber\\
		&+\mathcal{E}_{0}[p,\lambda](t).
	\end{align}
	Now we explicitly compute the second integral in \eqref{MassPhi2Intermediate2}. First we observe
	\begin{align}\label{IntegrationByPartsInSpace}
		&\delta(T-s)\int_{0}^{\infty}[1-e^{-\frac{z^{2}}{4}\frac{\delta(T-s)}{t-s}}(1+\frac{z^{2}}{4}\frac{\delta(T-s)}{t-s})]h(z,s)zdz=8\frac{\lambda^{2}(s)}{\delta^{2}(T-s)^{2}}\int_{0}^{\infty}\frac{1}{z^{3}}(\chi_{0}''-\delta\frac{z}{2}\chi_{0}'-\frac{3}{z}\chi_{0}')dz-\nonumber\\
		&\hspace{2cm}-8\frac{\lambda^{2}(s)}{\delta^{2}(T-s)^{2}}\int_{0}^{\infty}e^{-\frac{z^{2}}{4}\frac{\delta(T-s)}{t-s}}(1+\frac{z^{2}}{4}\frac{\delta(T-s)}{t-s})\frac{1}{z^{3}}(\chi_{0}''-\delta\frac{z}{2}\chi_{0}'-\frac{3}{z}\chi_{0}')dz.
	\end{align}
    Integrating by parts the first integral in \eqref{IntegrationByPartsInSpace} we get
\begin{align*}
		\int_{0}^{\infty}\frac{1}{z^{3}}(\chi_{0}''-\delta\frac{z}{2}\chi_{0}'-\frac{3}{z}\chi_{0}')dz=\delta\int_{0}^{\infty}\frac{z^{-2}}{-2}(\chi_{0}-1)'dz+\int_{0}^{\infty}(\frac{\chi_{0}'}{z^{3}})'dz=\delta\int_{0}^{\infty}(1-\chi_{0})\frac{dz}{z^{3}}=\delta\beta
	\end{align*}
where we recall that the constant $\beta$ has already been introduced in \eqref{ExpansionRHSeqlambda}. We study the second integral in \eqref{IntegrationByPartsInSpace} observing that
	\begin{align*}
		&-\int_{0}^{\infty}\frac{1}{z^{3}}e^{-\frac{z^{2}}{4}\frac{\delta(T-s)}{t-s}}(1+\frac{z^{2}}{4}\frac{\delta(T-s)}{t-s})\chi_{0}''(z)dz=\\
		&\hspace{1cm}=\frac{1}{16}(\frac{\delta(T-s)}{t-s})^{3}\int_{0}^{\infty} e^{-\frac{z^{2}}{4}\frac{\delta(T-s)}{t-s}}z(1-\chi_{0})dz-3 \int_{0}^{\infty}\chi_{0}' e^{-\frac{z^{2}}{4}\frac{\delta(T-s)}{t-s}}(1+\frac{z^{2}}{4}\frac{\delta(T-s)}{t-s})\frac{1}{z^{4}}dz
	\end{align*}
    and that
	\begin{align*}
		&\frac{\delta}{2}\int_{0}^{\infty} \frac{1}{z^{2}}e^{-\frac{z^{2}}{4}\frac{\delta(T-s)}{t-s}}(1+\frac{z^{2}}{4}\frac{\delta(T-s)}{t-s})\chi_{0}' dz=-\delta\int_{0}^{\infty}(1-\chi_{0})\frac{1}{z^{3}}e^{-\frac{z^{2}}{4}\frac{\delta(T-s)}{t-s}}dz\\
		&\hspace{1.5cm}-\frac{\delta^{2}}{4}\frac{T-s}{t-s}\int_{0}^{\infty}(1-\chi_{0})\frac{1}{z}e^{-\frac{z^{2}}{4}\frac{\delta(T-s)}{t-s}}dz-\frac{\delta^{3}}{16}(\frac{T-s}{t-s})^{2}\int_{0}^{\infty}(1-\chi_{0})e^{-\frac{z^{2}}{4}\frac{\delta(T-s)}{t-s}}zdz.
	\end{align*}
	Then, the remaining integral in \eqref{IntegrationByPartsInSpace} can be written as
	\begin{align*}
		&-\int_{0}^{\infty}e^{-\frac{z^{2}}{4}\frac{\delta(T-s)}{t-s}}(1+\frac{z^{2}}{4}\frac{\delta(T-s)}{t-s})\frac{1}{z^{3}}(\chi_{0}''-\delta\frac{z}{2}\chi_{0}'-\frac{3}{z}\chi_{0}')dz=\\
		&\hspace{1cm}=-\delta\int_{0}^{\infty}(1-\chi_{0})\frac{1}{z^{3}}e^{-\frac{z^{2}}{4}\frac{\delta(T-s)}{t-s}}(1+\frac{z^{2}}{4}\frac{\delta(T-s)}{t-s})dz+\\
		&\hspace{1.4cm}+\frac{1}{16}[(\frac{\delta(T-s)}{t-s})^{3}-\delta(\frac{\delta(T-s)}{t-s})^{2}]\int_{0}^{\infty} e^{-\frac{z^{2}}{4}\frac{\delta(T-s)}{t-s}}z(1-\chi_{0})dz.
	\end{align*}
    and consequently we have
    \begin{align*}
		\int_{\R^2}\varphi_{\lambda}^{(2)}dx=&M_{2}[\varepsilon(T)]+16\pi\frac{\gamma}{\delta} \int_{-\varepsilon(T)}^{t}\frac{\lambda^{2}(s)}{(T-s)^{2}}ds-\\
		&-16\pi\frac{1}{\delta} \int_{-\varepsilon(T)}^{t}\frac{\lambda^{2}(s)}{(T-s)^{2}}\int_{0}^{\infty}e^{-\frac{z^{2}}{4}\frac{\delta(T-s)}{t-s}}(1-\chi_{0})\frac{1}{z^{3}}(1+\frac{z^{2}}{4}\frac{\delta(T-s)}{t-s})dzds+\\
		&+\pi\frac{1}{\delta^{2}} \int_{-\varepsilon(T)}^{t}\frac{\lambda^{2}(s)}{(T-s)^{2}}[(\frac{\delta(T-s)}{t-s})^{3}-\delta(\frac{\delta(T-s)}{t-s})^{2}]\int_{0}^{\infty}e^{-\frac{z^{2}}{4}\frac{\delta(T-s)}{t-s}}z(1-\chi_{0})dzds+\\
		&+\mathcal{E}_{0}[p,\lambda].
	\end{align*}
	In order to get the desired result we just need to integrate by parts in time. We start recalling that $p(t)=\lambda\dot{\lambda}(t)$ and then
	\begin{align*}
		16\pi\frac{\gamma}{\delta}\int_{-\varepsilon(T)}^{t}\frac{\lambda^{2}(s)}{(T-s)^{2}}=M_{2}[\varepsilon(T)]+16\pi \frac{\gamma}{\delta} \frac{\lambda^{2}(t)}{T-t}-32\pi\frac{\gamma}{\delta} \int_{-\varepsilon(T)}^{t}\frac{p(s)}{T-s}ds.
	\end{align*}
	Analogously
\begin{align*}
		\hspace{-1cm}-&16\pi\frac{1}{\delta} \int_{-\varepsilon(T)}^{t}\frac{\lambda^{2}(s)}{(T-s)^{2}}\int_{0}^{\infty} e^{-\frac{z^{2}}{4}\frac{\delta(T-s)}{t-s}}(1-\chi_{0})\frac{1}{z^{3}}(1+\frac{z^{2}}{4}\frac{\delta(T-s)}{t-s})dzds=
\end{align*}
 \begin{align*}
		&=16\pi\frac{1}{\delta} \frac{\lambda^{2}(-\varepsilon(T))}{T+\varepsilon(T)}\int_{0}^{\infty} e^{-\frac{z^{2}}{4}\frac{\delta(T+\varepsilon(T))}{t+\varepsilon(T)}}(1-\chi_{0})\frac{1}{z}(1+\frac{z^{2}}{4}\frac{\delta(T+\varepsilon(T))}{t+\varepsilon(T)})dz+\\
		&\hspace{1cm}+32\pi\frac{1}{\delta} \int_{-\varepsilon(T)}^{t}\frac{p(s)}{T-s}\int_{0}^{\infty}e^{-\frac{z^{2}}{4}\frac{\delta(T-s)}{t-s}}(1-\chi_{0})\frac{1}{z^{3}}(1+\frac{z^{2}}{4}\frac{\delta(T-s)}{t-s})dzds-\\
		&\hspace{1cm}-\pi\delta \int_{-\varepsilon(T)}^{t}\frac{\lambda^{2}(s)}{(T-s)^{2}}[(\frac{T-s}{t-s})^{3}-(\frac{T-s}{t-s})^{2}]\int_{0}^{\infty}(1-\chi_{0})z e^{-\frac{z^{2}}{4}\frac{T-s}{t-s}}dzds.
	\end{align*}
	We just proved
	\begin{align}\label{MassPhi2ExpansionQuasiFinita}
		\int_{\R^2}\varphi_{\lambda}^{(2)}dx=&M_{2}[\varepsilon(T)]+16\pi \frac{\gamma}{\delta} \frac{\lambda^{2}(t)}{T-t}+\nonumber\\
		&+16\pi\frac{1}{\delta} \frac{\lambda^{2}(-\varepsilon(T))}{T+\varepsilon(T)}\int_{0}^{\infty} e^{-\frac{z^{2}}{4}\frac{T+\varepsilon(T)}{t+\varepsilon(T)}}(1-\chi_{0}(\frac{z}{\sqrt{\delta}}))\frac{1}{z}(1+\frac{z^{2}}{4}\frac{T+\varepsilon(T)}{t+\varepsilon(T)})dz-\nonumber\\
		&-32\pi\frac{1}{\delta} \int_{-\varepsilon(T)}^{t}\frac{p(s)}{T-s}\int_{0}^{\infty}[1-e^{-\frac{z^{2}}{4}\frac{\delta(T-s)}{t-s}}(1+\frac{z^{2}}{4}\frac{\delta(T-s)}{t-s})]\frac{1}{z^{3}}(1-\chi_{0})dzds+\mathcal{E}_{0}[p,\lambda].
	\end{align}
	To study the second term in the right-hand side of \eqref{MassPhi2ExpansionQuasiFinita} we want use again \eqref{SECMASS function f }. Observing that we are interested in $|f_{z}(\frac{T+\varepsilon(T)}{t+\varepsilon(T)})-f_{z}(1)|$ and then expanding $f_{z}(1+(\frac{T+\varepsilon(T)}{t+\varepsilon(T)}-1))=f_{z}(1+x)$ when $x$ is small if $t\ge -\frac{T}{2}$ we have
	\begin{align*}
		0\le \frac{T-t}{t+\varepsilon(T)}=\frac{T+\varepsilon(T)}{t+\varepsilon(T)}-1 \le\frac{3}{2} \frac{T}{t+\varepsilon(T)}\le \frac{3}{2} \frac{T}{-\frac{T}{2}+\varepsilon(T)}\le 3 \frac{T}{\varepsilon(T)}\ll1.
	\end{align*}
	We know that $|f_{z}(1+x)-f_{z}(1)|\le C z^{4}e^{-\frac{z^{2}}{4}}x$, see \eqref{SECMASS function f }. This implies that
	\begin{align*}
		&|\frac{16\pi}{\delta} \frac{\lambda^{2}(-\varepsilon(T))}{T+\varepsilon(T)}\int_{0}^{\infty}\frac{1-\chi_{0}(\frac{z}{\sqrt{\delta}})}{z}[f_{z}(\frac{T+\varepsilon(T)}{t+\varepsilon(T)})-f_{z}(1)]dz|\le \\
		&\le C \frac{\lambda^{2}(-\varepsilon(T))}{T+\varepsilon(T)}\frac{T-t}{t+\varepsilon(T)}\int_{0}^{\infty}(1-\chi_{0})z^{3}e^{-\frac{z^{2}}{4}}dz.
	\end{align*}
	We conclude the proof observing this term can be included in $\mathcal{E}_{0}$, in fact if $t\ge-\frac{T}{2}$ and if $\varepsilon(T)\gg T$ (or, more precisely, if \eqref{EpsiTandT} holds) we get
\begin{align}\label{MASSSEC WHereEpsSqu}
		|\frac{\lambda^{2}(-\varepsilon(T))}{T+\varepsilon(T)}\frac{T-t}{t+\varepsilon(T)}|\le\ e^{-4a\sqrt{|\ln(T-t)|}}\sqrt{|\ln(T-t)|}.
	\end{align}
\end{proof}
The following Corollary is a trivial consequence of Lemmas \ref{Mass10}, \ref{MASSOUTERphiLAMBDA}.
\begin{corollary}\label{ExpansionMass}
	Let $T$, $\varepsilon(T)$ be such that \eqref{RestrictionEPSILON}. Assume that $\lambda$ satisfies \eqref{LambAss1}, \eqref{LambAss2}. Then, for all $t\ge -\frac{T}{2}$ and for all $l>1/2$, we have
	\begin{align*}
		\int_{\R^2}\varphi_{\lambda}^{(1)}(\cdot,t)dx+\int_{\R^2}\varphi_{\lambda}^{(2)}(\cdot,t)dx-16\pi\frac{\gamma}{\delta}\frac{\lambda^{2}(t)}{T-t}=&M_{\lambda}^{\varphi}[\varepsilon(T)]-4\pi \int_{-f_{0}(T)}^{t-\lambda^{2}}\frac{p(s)}{t-s}+4\pi(\mu+1-\ln4)p(t)+\\
		&+\mathcal{B}[p,\lambda]+\mathcal{E}_{0}[p,\lambda]
	\end{align*}
	where $M_{\lambda}^{\varphi}[\varepsilon(T)]$ is a constant satisfying $|M_{\lambda}^{\varphi}[\varepsilon(T)]|\le Ce^{-2a\sqrt{|\ln\varepsilon(T)|}}$, $\mathcal{B}$ is an operator such that
	\begin{align}\label{B}
		|\mathcal{B}[p,\lambda]| \le C \|p\|_{\gamma,m} \frac{e^{-2a\gamma\sqrt{|\ln(T-t)|}}}{|\ln(T-t)|^{m+l-\frac{1}{2}}}
	\end{align}
    and such that if $p_{1}$, $p_{2}$ satisfy $\|p_{i}\|_{\gamma,m}<\infty$ we also have
\begin{align*}
	|\mathcal{B}[p_{1}]-\mathcal{B}[p_{2}]| \le C \|p_{1}-p_{2}\|_{\gamma,m} \frac{e^{-2a\gamma\sqrt{|\ln(T-t)|}}}{|\ln(T-t)|^{m+l-\frac{1}{2}}}.
\end{align*}
	Moreover if $\lambda$ satisfies also \eqref{RegLambAss}
	\begin{align*}
		|\mathcal{E}_{0}[p,\lambda]|\le C e^{-4a\sqrt{|\ln(T-t)|}}|\ln(T-t)|^{1+l}.
	\end{align*}
\end{corollary}
\subsection{An approximate solution}
In this section we will find an explicit approximate solution of the nonlocal equation \eqref{EqPhiLamHOPE}.
\begin{lemma}\label{SECMASSprofile}
	Let $T$, $\varepsilon(T)$ satisfy \eqref{RestrictionEPSILON}. Let $\lambda_{\star}(t)$ satisfy $\lambda_{\star}(t)=\sqrt{-\int_{t}^{T}p_{\star}(s)ds}$ with $p_{\star}(t)=-\frac{c_{\star}}{2}e^{-2a\sqrt{|\ln(T-t)|}}$, $c_{\star}=4e^{-(\mu+2)}$ and $a=\frac{\sqrt{2}}{2}$. For any $t\ge-\frac{T}{2}$ we have
	\begin{align*}
		|\int_{\R^2}\varphi_{\lambda_{\star}}^{(1)}(\cdot,t)dx+\int_{\R^2}\varphi_{\lambda_{\star}}^{(2)}(\cdot,t)dx-\int_{\R^2}\varphi_{\lambda_{\star}}^{(1)}(\cdot,T)dx+\int_{\R^2}\varphi_{\lambda_{\star}}^{(2)}(\cdot,T)-16\pi \gamma \frac{\lambda_{\star}^{2}}{T-t}|\le C \frac{e^{-2a\sqrt{|\ln(T-t)|}}}{|\ln(T-t)|^{1/4}}.
	\end{align*}
\end{lemma}
\begin{proof}
	By Corollary \ref{ExpansionMass}, if $l>3/4$ it is enough to prove
	\begin{align*}
		-\int_{-\varepsilon(T)}^{t-\lambda_{\star}^{2}}\frac{p_{\star}(s)}{t-s}ds+(\gamma+1-\ln4)p_{\star}(t)=-\int_{-\varepsilon(T)}^{T}\frac{p_{\star}(s)}{T-s}ds+O(\frac{e^{-2a\sqrt{|\ln(T-t)|}}}{|\ln(T-t)|^{1/4}}).
	\end{align*}
	We introduce a function $0<N(t)\to \infty$ as $t\to T$ that will be fixed later. We have
	\begin{align}\label{ExpansionNonLocalFirst}
		-\int_{-\varepsilon(T)}^{t-\lambda_{\star}^{2}}\frac{p_{\star}(s)}{t-s}ds&=-\int_{-\varepsilon(T)}^{T}\frac{p_{\star}(s)}{T-s}ds+\int_{t}^{T}\frac{p_{\star}(s)}{T-s}ds+\int_{t-N(t)(T-t)}^{t}\frac{p_{\star}(s)}{T-s}ds-\nonumber\\
		&\hspace{1cm}-(T-t)\int_{-\varepsilon(T)}^{t-N(t)(T-t)}\frac{p_{\star}(s)}{(t-s)(T-s)}ds-\int_{t-N(t)(T-t)}^{t-\lambda^{2}_{\star}}\frac{p_{\star}(s)}{t-s}ds.
	\end{align}
	We want to start estimating the fourth integral in \eqref{ExpansionNonLocalFirst}, but first we need the following trivial observations
	\begin{align*}
		&s<t-N(t)(T-t)=t-(T-s)N(t)-(s-t)N(t)=t-(T-s)N(t)+(t-s)N(t) \\
		&\iff 1\le \frac{T-s}{t-s}\le \frac{1+N(t)}{N(t)}\le 2
	\end{align*}
	since $N(t)\to \infty$. Moreover if we assume that $\frac{|\ln(1+N)|}{\sqrt{|\ln(T-t)|}}\to 0$, we also have $e^{-2a\sqrt{|\ln(T-t+N(t)(T-t))|}} \approx e^{-2a\sqrt{|\ln(T-t)|}}$. Since we know that $\varepsilon(T)\ll1$ and that the function $\frac{e^{-2a\sqrt{|\ln x|}}}{x}$ is decreasing when $x$ is small, if $t\ge-\frac{T}{2}$ we observe
	\begin{align*}
		&|(T-t)\int_{-\varepsilon(T)}^{t-N(t)(T-t)}\frac{p_{\star}(s)}{(t-s)(T-s)}ds|\le C \frac{ e^{-2a\sqrt{|\ln(T-t)|}}}{N(t)}.
	\end{align*}
	Now we can expand the last integral of \eqref{ExpansionNonLocalFirst}
	\begin{align*}
		-\int_{t-N(t)(T-t)}^{t-\lambda^{2}_{\star}}\frac{p_{\star}(s)}{t-s}ds=&-p_{\star}(t)\int_{t-N(t)(T-t)}^{t-\lambda_{\star}^{2}}\frac{ds}{t-s}-\int_{t-N(t)(T-t)}^{t-\lambda_{\star}^{2}}\frac{p_{\star}(s)-p_{\star}(t)}{t-s}ds =\\
		=&p_{\star}(t)\ln(\frac{\lambda_{\star}^{2}(t)}{T-t})-p_{\star}(t)\ln N(t)+O(\frac{e^{-2a\sqrt{|\ln(T-t)|}}}{\sqrt{|\ln(T-t)|}}N(t)).
	\end{align*}
    It is then natural to fix $N(t)=|\ln(T-t)|^{1/4}$. Now we observe that the third integral in \eqref{ExpansionNonLocalFirst} can be rewritten as
	\begin{align*}
		\int_{t-N(t)(T-t)}^{t}\frac{p_{\star}(s)}{T-s}ds=&p_{\star}(t)\int_{t-N(t)(T-t)}^{t}\frac{ds}{T-s}+\int_{t-N(t)(T-t)}^{t}\frac{p_{\star}(s)-p_{\star}(t)}{T-s}ds=\\
		=&p_{\star}(t)\ln(1+N(t))+O(\frac{e^{-2a\sqrt{|\ln(T-t)|}}}{|\ln(T-t)|^{1/4}}).
	\end{align*}
	We just proved that
	\begin{align}\label{FinalEquationAlmostAlmost}
		&-\int_{-\varepsilon(T)}^{t-\lambda_{\star}^{2}}\frac{p_{\star}(s)}{t-s}ds+(\gamma+1-\ln4)p_{\star}(t)+\int_{-\varepsilon(T)}^{T}\frac{p_{\star}(s)}{T-s}ds=\nonumber\\
		&\hspace{1cm}=\int_{t}^{T}\frac{p_{\star}(s)}{T-s}ds+p_{\star}(t)\ln(\frac{\lambda_{\star}^{2}(t)}{T-t})+(\gamma+1-\ln4)p_{\star}(t)+O(\frac{e^{-2a\sqrt{|\ln(T-t)|}}}{|\ln(T-t)|^{1/4}}).
	\end{align}
	For the final step we only need to expand $\lambda_{\star}^{2}(t)=-2\int_{t}^{T}p_{\star}(s)ds$. An elementary computation gives
	\begin{align*}
		\lambda_{\star}^{2}(t)=&-2\int_{t}^{T}p_{\star}(s)ds=c_{\star}\int_{t}^{T}e^{-2a\sqrt{|\ln(T-s)|}}ds=c_{\star}(T-t)e^{-2a\sqrt{|\ln(T-t)|}}(1+O(\frac{1}{|\ln(T-t)|^{1/4}}))
	\end{align*}
	that implies
	\begin{align*}
		\ln(\frac{\lambda_{\star}^{2}(t)}{T-t})=\ln c_{\star}-2a\sqrt{|\ln(T-t)|}+O(\frac{1}{|\ln(T-t)|^{1/4}}).
	\end{align*}
	We can rewrite \eqref{FinalEquationAlmostAlmost} and finally write the following equation for $p_{\star}(t)$:
	\begin{align*}
		&-\int_{-\varepsilon(T)}^{t-\lambda_{\star}^{2}}\frac{p_{\star}(s)}{t-s}ds+(\gamma+1-\ln4)p_{\star}(t)+\int_{-\varepsilon(T)}^{T}\frac{p_{\star}(s)}{T-s}ds=\\
		&=\int_{t}^{T}\frac{p_{\star}(s)}{T-s}ds-2a\sqrt{|\ln(T-t)|}p_{\star}(t)+(\ln(\frac{c_{\star}}{4})+\gamma+1)p_{\star}(t)+O(\frac{e^{-2a\sqrt{|\ln(T-t)|}}}{|\ln(T-t)|^{1/4}}).
	\end{align*}
	A direct computation shows
	\begin{align*}
		&\int_{t}^{T}\frac{p_{\star}(s)}{T-s}ds=\frac{1}{a}\sqrt{|\ln(T-t)|}p_{\star}(t)+\frac{1}{2a^{2}}p_{\star}(t)
	\end{align*}
and then the proof is a trivial consequence of the resulting identity:
	\begin{align*}
		&-\int_{-\varepsilon(T)}^{t-\lambda_{\star}^{2}}\frac{p_{\star}(s)}{t-s}ds+(\gamma+1-\ln4)p_{\star}(t)+\int_{-\varepsilon(T)}^{T}\frac{p_{\star}(s)}{T-s}ds=\\
		&\hspace{1cm}=(\frac{1}{a}-2a)\sqrt{|\ln(T-t)|}p_{\star}(t)+(\ln(\frac{c_{\star}}{4})+\gamma+1+\frac{1}{2a^{2}})p_{\star}(t)+O(\frac{e^{-2a\sqrt{|\ln(T-t)|}}}{|\ln(T-t)|^{1/4}}).
	\end{align*}
\end{proof}
\subsection{The linearized equation}\label{LINERAIZEDMASS}
By Corollary \ref{ExpansionMass}, assuming that $\lambda\dot{\lambda}=p(t)$ satisfies \eqref{LambAss1}, \eqref{LambAss2}, \eqref{RegLambAss} in order to prove
\begin{align}\label{DesiredEquation}
	\int_{\R^2}\varphi_{\lambda}dx-16\pi\gamma\frac{\lambda^{2}}{T-t}=\int_{\R^2}\varphi_{\lambda}(T)dx+O(e^{-2a(1-\rho)\sqrt{|\ln(T-t)|}}) \ \ \ \forall t\in(0,T)
\end{align}
for any $\rho>0$ (that is equivalent to Proposition \ref{prop-lambda0}), we only have to prove that if $t\in(0,T)$ we have
\begin{align}\label{EquationNOnloca}
	-\int_{-\varepsilon(T)}^{t-\lambda^{2}}\frac{p(s)}{t-s}+(\gamma+1-\ln4)p(t)=-\int_{-\varepsilon(T)}^{T}\frac{p(s)}{T-s}ds+\mathcal{B}_{0}[p](t)+O(e^{-2a(1-\rho)\sqrt{|\ln(T-t)|}})
\end{align}
where we defined
\begin{align}\label{B0}
 \mathcal{B}_{0}[p]=-\frac{\mathcal{B}[p]}{2\pi}
\end{align}
with $\mathcal{B}$ as in Corollary \ref{ExpansionMass}. Then we introduce a correction $p_{1}$, that obviously has to be intended smaller than $p_{\star}$ and we write
\begin{align*}
	p(t)=p_{\star}(t)+p_{1}(t)
\end{align*} where $p_{\star}(t)=-\frac{c_{\star}}{2}e^{-2a\sqrt{|\ln(T-t)|}}$ as in Lemma \ref{SECMASSprofile}. In the following we will define
\begin{align}\label{DefintionLAmbda}
	\lambda^{2}:=-2\int_{t}^{T}(p_{\star}+p_{1})ds
\end{align} and we will always assume that it satisfies \eqref{LambAss1}. The idea of this linearization is to recover an ordinary differential equation in $p_{1}(t)$ by neglecting terms that will be small once we get some control on $\dot{p}_{1}(t)$. Notice that controlling $\dot{p}_{1}$ we can control $\dot{p}$ and then, proceeding backwards, we can obtain \eqref{DesiredEquation}. As we mentioned previously with $p_{1}$ we will not achieve the desired size of the error and another iteration will be required. \newline We can start observing
\begin{align*}
	\int_{-\varepsilon(T)}^{t-\lambda^{2}}\frac{p(s)}{t-s}ds&=\int_{-\varepsilon(T)}^{t-\lambda_{\star}^{2}}\frac{p_{\star}(s)}{t-s}ds+\int_{-\varepsilon(T)}^{t-\lambda_{\star}^{2}}\frac{p_{1}(s)}{t-s}ds-\ln(\frac{\lambda^{2}}{\lambda_{\star}^{2}})p_{\star}(t)+\mathcal{B}_{1}[p_{1}](t)+O(\frac{e^{-4a\sqrt{|\ln(T-t)|}}}{\sqrt{|\ln(T-t)|}})
\end{align*}
where we called $\mathcal{B}_{1}[p_{1}]=\int_{t-\lambda_{\star}^{2}}^{t-\lambda^{2}}\frac{p_{1}(s)}{t-s}ds$. Recalling \eqref{DefintionLAmbda} we can also write
\begin{align}\label{ExpansionLogNonloc}
	&\ln(\frac{\lambda^{2}}{\lambda_{\star}^{2}})=\ln(\frac{-2\int_{t}^{T}p_{\star}(s)ds-2\int_{t}^{T}p_{1}(s)ds}{-2\int_{t}^{T}p_{\star}(s)ds})=\nonumber\\
	&=\frac{1}{T-t}\frac{\int_{t}^{T}p_{1}(s)ds}{p_{\star}(t)}+[\ln(1+\frac{\int_{t}^{T}p_{1}(s)ds}{\int_{t}^{T}p_{\star}(s)ds})-\frac{\int_{t}^{T}p_{1}(s)ds}{\int_{t}^{T}p_{\star}(s)ds}]+(\frac{\int_{t}^{T}p_{1}(s)ds}{\int_{t}^{T}p_{\star}(s)ds}-\frac{1}{T-t}\frac{\int_{t}^{T}p_{1}(s)ds}{p_{\star}(t)}).
\end{align}
If we include the last two terms of \eqref{ExpansionLogNonloc} in the operator $\mathcal{B}_{1}$, we have that \eqref{EquationNOnloca} is equivalent to 
\begin{align}\label{NewEq}
	-&\int_{-\varepsilon(T)}^{t-\lambda_{\star}^{2}}\frac{p_{1}(s)}{t-s}ds+(\gamma+1-\ln4)p_{1}(t)+\frac{1}{T-t}\int_{t}^{T}p_{1}(s)ds+\int_{-\varepsilon(T)}^{T}\frac{p_{1}(s)}{T-s}ds= \notag\\
	&\hspace{1cm}=f_{\star}(t)+\mathcal{B}_{0}[p_{\star}+p_{1}]-\mathcal{B}_{0}[p_{\star}](t)+\mathcal{B}_{1}[p_{1}]+O(e^{-2a(1-\rho)\sqrt{|\ln(T-t)|}})
\end{align}
where thanks to Lemma \ref{SECMASSprofile} if $t\ge-\frac{T}{2}$ we have
\begin{align*}
	f_{\star}(t)=\frac{1}{4\pi}[4\pi \int_{-\varepsilon(T)}^{t-\lambda_{\star}^{2}}\frac{p_{\star}(s)}{t-s}ds-4\pi(\gamma+1-\ln 4 )p_{\star}(t)-4\pi \int_{-\varepsilon(T)}^{T}\frac{p_{\star}(s)}{T-s}ds-\mathcal{B}[p_{\star}]]=O(\frac{e^{-2a\sqrt{|\ln(T-t)|}}}{|\ln(T-t)|^{1/4}}),
\end{align*}
and, if $\mathcal{B}$ is the operator in \eqref{B}, we denoted
\begin{align*}
	\mathcal{B}_{0}[p]=&-\frac{\mathcal{B}[p]}{2\pi},\nonumber\\
	\mathcal{B}_{1}[p_{1}]=&\int_{t-\lambda_{\star}^{2}}^{t-\lambda^{2}}\frac{p_{1}(s)}{t-s}ds+p_{\star}(t)[\ln(1+\frac{\int_{t}^{T}p_{1}(s)ds}{\int_{t}^{T}p_{\star}(s)ds})-\frac{\int_{t}^{T}p_{1}(s)ds}{\int_{t}^{T}p_{\star}(s)ds}]+\\
	&+p_{\star}(t)(\frac{\int_{t}^{T}p_{1}(s)ds}{\int_{t}^{T}p_{\star}(s)ds}-\frac{1}{T-t}\frac{\int_{t}^{T}p_{1}(s)ds}{p_{\star}(t)}).
\end{align*}
Now we want to expand the first integral in the left-hand side of \eqref{NewEq}. 
We take $N(t)=|\ln(T-t)|^{1/4}$ and we introduce a parameter $0<\sigma<1$. As we will see in Lemma \ref{EstA} the choice of this parameter will be crucial in establishing that this map gives a contraction. For any $0<\sigma<1$, we can write
\begin{align*}
	-\int_{-\varepsilon(T)}^{t-\lambda_{\star}^{2}}\frac{p_{1}(s)}{t-s}ds=&-\int_{-\varepsilon(T)}^{t-N(t)(T-t)}(...)ds-\int_{t-N(t)(T-t)}^{t-(T-t)e^{-2a\sigma\sqrt{|\ln(T-t)|}}}(...)ds-\\
	&-\int_{t-(T-t)e^{-2a\sigma\sqrt{|\ln(T-t)|}}}^{t-\lambda_{\star}^{2}}(...)ds
\end{align*}
\begin{align*}
	=&-\int_{-\varepsilon(T)}^{T}\frac{p_{1}(s)}{T-s}ds+\int_{t}^{T}\frac{p_{1}(s)}{T-s}ds-2a\sqrt{|\ln(T-t)|}p_{1}(t)+\ln c_{\star}p_{1}(t)+\\
	&+\int_{t-N(t)(T-t)}^{t}\frac{p_{1}(s)-p_{1}(t)}{T-s}ds+\int_{t-N(t)(T-t)}^{t-(T-t)e^{-2a\sigma\sqrt{|\ln(T-t)|}}}\frac{p_{1}(t)-p_{1}(s)}{t-s}ds+\\
	&+\ln(1+\frac{1}{N(t)})p_{1}(t)-(T-t)\int_{-\varepsilon(T)}^{t-N(t)(T-t)}\frac{p_{1}(s)}{(t-s)(T-s)}ds+\\
	&+\ln(\frac{\lambda_{\star}^{2}}{c_{\star}(T-t)e^{-2a\sqrt{|\ln(T-t)|}}})p_{1}(t)+\int_{t-(T-t)e^{-2a\sigma\sqrt{|\ln(T-t)|}}}^{t-\lambda_{\star}^{2}}\frac{p_{1}(t)-p_{1}(s)}{t-s}ds.
\end{align*} 
By writing
\begin{align*}
	\frac{1}{T-t}\int_{t}^{T}p_{1}(s)ds=p_{1}(t)-\frac{1}{T-t}\int_{t}^{T}(p_{1}(t)-p_{1}(s))ds
\end{align*}
observing that $\gamma+2+\ln(\frac{c_{\star}}{4})=0$ we have that \eqref{NewEq}, and then \eqref{EquationNOnloca}, is equivalent to finding a $p_{1}$, satisfying \eqref{LambAss1}, \eqref{LambAss2}, \eqref{RegLambAss}, defined on $(-\varepsilon(T),T)$ such that for any $t\ge0$ we have
\begin{align}\label{realfinalEqt}
	&\int_{t}^{T}\frac{p_{1}(s)}{T-s}ds-2a\sqrt{|\ln(T-t)|}p_{1}(t)=f_{\star}(t)+\mathcal{A}_{\sigma}[p_{1}]+\mathcal{B}_{0}[p_{\star}+p_{1}]-\mathcal{B}_{0}[p_{\star}]+\mathcal{B}_{1}[p_{1}] +\nonumber\\
	&\hspace{3cm}-\int_{t-(T-t)e^{-2a\sigma\sqrt{|\ln(T-t)|}}}^{t-\lambda_{\star}^{2}}\frac{p_{1}(t)-p_{1}(s)}{t-s}ds +O(e^{-(2-\rho)\sqrt{|\ln(T-t)|}})
\end{align}
where we defined
\begin{align}\label{Asigma}
	\mathcal{A}_{\sigma}[p_{1}]=&-\int_{t-N(t)(T-t)}^{t}\frac{p_{1}(s)-p_{1}(t)}{T-s}ds-\int_{t-N(t)(T-t)}^{t-(T-t)e^{-2a\sigma\sqrt{|\ln(T-t)|}}}\frac{p_{1}(t)-p_{1}(s)}{t-s}ds+\nonumber\\
	&+\frac{1}{T-t}\int_{t}^{T}(p_{1}(t)-p_{1}(s))ds,
\end{align}
\begin{align}\label{B0}
	\mathcal{B}_{0}[p]=-\frac{\mathcal{B}[p]}{2\pi}
\end{align}
with $\mathcal{B}$ defined as in Corollary \ref{ExpansionMass} and
\begin{align}\label{B1}
	\mathcal{B}_{1}[p_{1}]=&\int_{t-\lambda_{\star}^{2}}^{t-\lambda^{2}}\frac{p_{1}(s)}{t-s}ds+p_{\star}(t)[\ln(1+\frac{\int_{t}^{T}p_{1}(s)ds}{\int_{t}^{T}p_{\star}(s)ds})-\frac{\int_{t}^{T}p_{1}(s)ds}{\int_{t}^{T}p_{\star}(s)ds}]+p_{\star}(\frac{\int_{t}^{T}p_{1}(s)ds}{\int_{t}^{T}p_{\star}(s)ds}-\frac{1}{T-t}\frac{\int_{t}^{T}p_{1}(s)ds}{p_{\star}(t)})\nonumber\\
	&-\ln(1+\frac{1}{N(t)})p_{1}(t)+(T-t)\int_{-\varepsilon(T)}^{t-N(t)(T-t)}\frac{p_{1}(s)}{(t-s)(T-s)}ds- \nonumber\\
	&-\ln(\frac{\lambda_{\star}^{2}}{c_{\star}(T-t)e^{-2a\sqrt{|\ln(T-t)|}}})p_{1}(t).
\end{align}
Now we want to find a $p_ {1}$ defined in $(-\varepsilon(T),T)$ that solves a modified version of \eqref{realfinalEqt}, namely
\begin{align}\label{FinalEq}
	\int_{t}^{T}\frac{p_{1}(s)}{T-s}ds-2a\sqrt{|\ln(T-t)|}p_{1}(t)=[f_{\star}(t)+\mathcal{A}_{\sigma}[p_{1}]+\mathcal{B}_{0}[p_{\star}+p_{1}]-\mathcal{B}_{0}[p_{\star}]+\mathcal{B}_{1}[p_{1}]]\eta (\frac{t}{T})
\end{align}
where we introduced the cutoff
\begin{align}\label{eta}
	\eta(s)=\begin{cases}
		0 \ \ \ s\le -\frac{1}{2}\\
		1 \ \ \ s\ge 0.
	\end{cases}
\end{align}
\noindent Notice that estimates in Corollary \ref{ExpansionMass} and Lemma \ref{SECMASSprofile} hold in the interval $(-\frac{T}{2},T)$. It is important to underline that solving \eqref{FinalEq} for $p_{1}$ since the equation was \eqref{realfinalEqt} we are ignoring the fundamental term
\begin{align*}
	\int_{t-(T-t)e^{-2a\sigma\sqrt{|\ln(T-t)|}}}^{t-\lambda_{\star}^{2}}\frac{p_{1}(t)-p_{1}(s)}{t-s}ds
\end{align*}
and the operator $\mathcal{E}_{0}$ of Corollary \ref{ExpansionMass} that are small only if you assume some extra regularity. For that reason we will also find some estimates for $\dot{p}_{1}$. We remark that the presence of the cut-off in \eqref{FinalEq} will give us a slightly worse control of $\dot{p}_{1}$ that one could expect by looking at the estimate for $p_{1}$. This discussion is postponed to the proof of Lemma \ref{SolutionNotRegular}, where we we will show that\newpage
\begin{align}\label{sigmaRemainder}
	\int_{t-(T-t)e^{-2a\sigma\sqrt{|\ln(T-t)|}}}^{t-\lambda_{\star}^{2}}\frac{|p_{1}(t)-p_{1}(s)|}{t-s}ds+|\mathcal{E}_{0}[p](t)|=O(e^{-2a(1+\sigma)\sqrt{|\ln(T-t)|}}).
\end{align}
We notice that \eqref{sigmaRemainder} shows the crucial role that the parameter $\sigma$ will play. Our goal will be to take $\sigma$ as large as possible. Unfortunately to get a solution we can only take $\sigma<\frac{1}{2}$ that will not immediately give the proof of Proposition \ref{prop-lambda0}. We will show in Section \ref{manageerrorSec} that a last iteration will erase this error and will finally give Proposition \ref{prop-lambda0}.
\subsection{Solution to the linearized equation}\label{SolutionLinEq}
In order to solve equation \eqref{FinalEq} the first step is to introduce the resolvent operator that gives a solution for the following nonlocal differential equation
\begin{align}\label{ODE}
	\int_{t}^{T}\frac{g(s)}{T-s}ds-2a\sqrt{|\ln(T-t)|}g(t)=f(t)
\end{align}
where $f(t)=O(\frac{e^{-2a\gamma\sqrt{|\ln(T-t)|}}}{|\ln(T-t)|^{m}})$. Equation \eqref{ODE} can be easily solved introducing the function
\begin{align*}
	G(t)=\int_{t}^{T}\frac{g(s)}{T-s}ds.
\end{align*}
We get
\begin{align*}
	&G(t)+2a\sqrt{|\ln(T-t)|}(T-t)G'(t)=f(t)\iff \frac{d}{dt}(e^{A(t)}G(t))=\frac{e^{A(t)}f(t)}{2a\sqrt{|\ln(T-t)|}(T-t)}
\end{align*}
where $A(t)=\frac{1}{a}\sqrt{|\ln(T-t)|}$.
Recalling that thanks to Lemma \ref{SECMASSprofile} we have $a=\frac{\sqrt{2}}{2}$, after some trivial computation we can introduce the \emph{resolvent operator}
\begin{align}\label{resolvant}
	T_{0}^{(\gamma)}[f](t):=&\frac{e^{-2a\sqrt{|\ln(T-t)|}}}{2a\sqrt{|\ln(T-t)|}}\int_{-\varepsilon(T)}^{t}\frac{e^{2a\sqrt{|\ln(T-s)}}f(s)}{2a\sqrt{|\ln(T-s)|}(T-s)}ds-\frac{f(t)}{2a\sqrt{|\ln(T-t)|}}+\nonumber\\ 
	&+c_{\gamma} \frac{e^ {-2a\sqrt{|\ln(T-t)|}}}{2a\sqrt{|\ln(T-t)|}}.
\end{align}
where we take
\begin{align*}
	c_{\gamma}=\begin{cases}
		0 \ \ \ \ \text{if }\gamma=1,\\[3pt]
		-\frac{1}{2a}\int_{\varepsilon(T)}^{T}\frac{e^{2a\sqrt{|\ln(T-s)|}}}{\sqrt{|\ln(T-s)|}}\frac{f(s)}{T-s}ds \ \ \ \ \text{if } \gamma>1.
	\end{cases}
\end{align*}
Recalling the definition of the norm \eqref{normP}, by construction, for any $f(t)$ such that $\|f\|_{\gamma,m}<\infty$ the operator $T_{0}^{\gamma}[f]$  gives a solution of \eqref{ODE}. The next Lemma wants to underline an important difference with \cite{DdPW} that we will comment later.
\begin{lemma}\label{T0est}
	Let $T$, $\varepsilon(T)$ sufficiently small. For any $f\in C[(-\varepsilon(T),T); \mathbb{R}]$ such that $\|f\|_{\gamma,m}<\infty$ we have
	\begin{align}
		&\|T_{0}^{(\gamma)}[f]\|_{\gamma,m}\le C \|f\|_{\gamma,m} \ \ \ \text{if }\gamma=1\\
		&\|T_{0}^{(\gamma)}[f]\|_{\gamma,m+1/2}\le \frac{C}{\gamma-1}\|f\|_{\gamma,m} \ \ \ \ \text{if } \gamma>1
	\end{align}
\end{lemma}
\begin{proof}[Proof of Lemma \ref{T0est}] 
	If $\gamma=1$ we have
	\begin{align*}
		|T_{0}^{(\gamma)}[f]|&\le \|f\|_{\gamma,m}\frac{e^{-2a\sqrt{|\ln(T-t)|}}}{2a\sqrt{|\ln(T-t)|}}\int_{-\varepsilon(T)}^{t}\frac{1}{2a|\ln(T-s)|^{m+1/2}}\frac{ds}{T-s}+\|f\|_{\gamma,m}\frac{e^{-2a\sqrt{|\ln(T-t)|}}}{|\ln(T-t)|^{m}\sqrt{|\ln(T-t)|}}\\
		&\le C \|f\|_{\gamma,m} \frac{e^{-2a\sqrt{|\ln(T-t)|}}}{|\ln(T-t)|^{m}}.
	\end{align*}
	If $\gamma>1$ we get
	\begin{align*}
		|T_{0}^{(\gamma)}[f]|&\le \|f\|_{\gamma,m}\big[\frac{e^{-2a\sqrt{|\ln(T-t)|}}}{2a\sqrt{|\ln(T-t)|}}|\int_{t}^{T}\frac{e^{-2a(\gamma-1)\sqrt{|\ln(T-s)|}}}{2a|\ln(T-s)|^{m+1/2}}\frac{ds}{T-s}+ \frac{e^{-2a\gamma\sqrt{|\ln(T-t)|}}}{|\ln(T-t)|^{m+1/2}}\big] \le \\
		&\le \|f\|_{\gamma,m} \frac{1}{\gamma-1}\frac{1}{\sqrt{|\ln(T-t)|}} \frac{e^{-2a\gamma\sqrt{|\ln(T-t)|}}}{|\ln(T-t)|^{m}}.
	\end{align*}
\end{proof}
\noindent The 
dichotomy presented in Lemma \ref{T0est} is a new challenge with respect to \cite{DdPW}. When $\gamma=1$, indeed, the resolvent operator does not naturally produce an extra smallness that that has been used in \cite{DdPW} to find a solution by means of a fixed point argument. To overcome this difficulty the key ingredient will be to observe that that the resolvent operator \eqref{resolvant} is made by two terms and at main order it is differentiable. To underline this fact we introduce the following notation
\begin{align}\label{splitResolvant}
	T_{0}^{1}[f](t)=T_{0,\text{reg}}^{1}[f]-\frac{f(t)}{2a\sqrt{|\ln(T-t)|}}
\end{align} 
where 
\begin{align*}
	T_{0,\text{reg}}^{1}[f]=\frac{e^{-2a\sqrt{|\ln(T-t)|}}}{2a\sqrt{|\ln(T-t)|}}\int_{-\varepsilon(T)}^{t}\frac{e^{2a\sqrt{|\ln(T-s)}}f(s)}{2a\sqrt{|\ln(T-s)|}(T-s)}ds.
\end{align*}
We can rephrase once again our problem: we want to find a function $h$ such that
\begin{align}\label{RealFinalEq}
	h(t)=[f_{\star}(t)+\mathcal{A}_{\sigma}[T_{0}[h]]+\mathcal{B}_{0}[p_{\star}+T_{0}[h]]-\mathcal{B}_{0}[p_{\star}]+\mathcal{B}_{1}[T_{0}[h]]]\eta (\frac{t}{T}) \ \ t\in[-\varepsilon(T),T].
\end{align}
In fact, it is immediate to observe that if we had the solution of \eqref{RealFinalEq}, then $T_{0}[h]$ would solve \eqref{FinalEq}. In the proof of the following two Lemmas we will focus on some more significative terms in $\mathcal{A}_{\sigma}$ and $\mathcal{B}_{1}$, namely
\begin{align}\label{Asigmasimpli}
	\mathcal{A}_{\sigma}[p]=\int_{t-N(t)(T-t)}^{t}\frac{p(t)-p(s)}{T-s}-\int_{t-N(t)(T-t)}^{t-(T-t)e^{-2a\sigma\sqrt{|\ln(T-t)|}}}\frac{p(t)-p(s)}{t-s}ds+(...)
\end{align}
\begin{align}\label{B1simpli}
	\mathcal{B}_{1}[p]=-(T-t)\int_{-\varepsilon(T)}^{t-N(t)(T-t)}\frac{p(s)}{(t-s)(T-s)}ds+p_{\star}(t)[\ln(1+\frac{\int_{t}^{T}p(s)ds}{\int_{t}^{T}p_{\star}(s)ds})-\frac{\int_{t}^{T}p(s)ds}{\int_{t}^{T}p_{\star}(s)ds}]+(...)
\end{align}
and we observe that the remaining part of the operators $\mathcal{A}_{\sigma}$ and $\mathcal{B}_{1}$ can be treated analogously.
The next Lemmas \ref{EstB} and \ref{EstA} will be the last two missing ingredients to get desired improvement of $p_{\star}$. In particular Lemma \ref{EstA} will be extremely important since it will tell us that in order to get a solution by a fixed point argument a bound for the parameter $\sigma$ must be satisfied. This will have an evident effect on the size of the final main error \eqref{sigmaRemainder}. In the following results we also need to give a bound for the norm of the solution $h$. We will write $\|h\|_{\gamma,m}\le M$. We anticipate that later we will take $M$ large (but independent of $T$ and $\varepsilon(T)$).
\begin{lemma}\label{EstB}
	Let $h_{1},h_{2}\in C[(-\varepsilon(T),T);\mathbb{R}]$ such that $\|h_{i}\|_{\gamma,m}\le M$ for some $M>0$. Then, for any $t\in(-\frac{T}{2},T)$ we have
	\begin{align*}
		|\mathcal{B}_{1}[T_{0}[h_{1}]-\mathcal{B}_{1}[T_{0}[h_{2}]]]|&\le C\|h_{1}-h_{2}\|_{\gamma,m} \frac{e^{-2a\gamma\sqrt{|\ln(T-t)|}}}{|\ln(T-t)|^{m}}\cdot \\
		&\cdot\big( \frac{1}{|\ln(T+\varepsilon(T))|^{1/4}}+\frac{M/|\ln(T+\varepsilon(T))|^{m}}{1-M/|\ln(T+\varepsilon(T))|^{m}}\big)\begin{cases}
			1 \ \ \ \text{if }\gamma=1\\
			\frac{1}{\sqrt{|\ln(T-t)|}} \ \ \text{if } \gamma>1.
		\end{cases}
	\end{align*}
\end{lemma}
\begin{proof}
	We will focus on the $\gamma=1$ case since, after recalling Lemma \ref{T0est}, the proof when $\gamma>1$ is identical (we observe that it is not necessary to keep track of the dependence on the exponent $\gamma$ that can be simply absorbed by the constant $C$). The first term in \eqref{B1simpli} is linear, then we can simply observe that if $t\ge-\frac{T}{2}$ we have
	\begin{align*}
		|(T-t)\int_{-\varepsilon(T)}^{t-N(t)(T-t)}\frac{T_{0}[h](s)}{(t-s)(T-s)}ds|\le C \|h\|_{\gamma,m} \frac{e^{-2a\gamma\sqrt{|\ln(T-t)}}}{|\ln(T-t)|^{m}|\ln(T-t)|^{1/4}}.
	\end{align*}
	Recalling that $\|h_{i}\|\le M$, independent of $\varepsilon(T)$, $T$, we can easily find an estimate for the remaining term in $\mathcal{B}_{1}$.
\end{proof}
\begin{lemma}\label{EstA}
	Let $h_{1},h_{2}\in \mathcal{C}([-\varepsilon(T),T];\mathbb{R})$ such that $\|h_{i}\|\le M$ for some $M>0$. Then, for any $t\in(-\frac{T}{2},T)$ there exists an $\omega_{T}$ such that
	\begin{align*}
		&|\mathcal{A}_{\sigma}[T_{0}[h_{1}]]-\mathcal{A}_{\sigma}[T_{0}[h_{2}]]| \le \|h_{1}-h_{2}\|_{\gamma,m}\frac{e^{-2a\gamma\sqrt{|\ln(T-t)|}}}{|\ln(T-t)|^{m}} \cdot \\
		&\hspace{5cm}\cdot\big[\frac{C}{|\ln(T+\varepsilon(T))|^{1/4}}+(2\sigma+\omega_{T})\big]\begin{cases}
			1 \ \ \ \text{if } \gamma =1 \\
			\frac{1}{\gamma-1} \frac{1}{\sqrt{|\ln(T-t)|}} \ \ \ \text{if } \gamma>1.
		\end{cases} 
	\end{align*}
Moreover the constant $\omega_{T}$ can be made arbitrarly small by taking $T,\varepsilon(T)$ sufficiently small.
\end{lemma}
\begin{proof}
	As in the proof of Lemma \ref{EstB} we will focus on the case $\gamma=1$. Here it will be crucial the distinction we made in \eqref{splitResolvant}. Indeed we can easily observe that for any $t,s\in(-\varepsilon(T),T)$ we have
	\begin{align*}
		&|T_{0,\text{reg}}^{\gamma}[h](t)-T_{0,\text{reg}}^{\gamma}[h](s)|\lesssim \|h\|_{\gamma,m} \frac{e^{-2a\sqrt{|\ln(T-t)|}}}{\sqrt{|\ln(T-t)|}(T-t)}\frac{t-s}{|\ln(T-t)|^{m}}.
	\end{align*}
	then if we consider the second term of \eqref{Asigmasimpli} we get for some $\omega_{T}$
	\begin{align*}
		&|\int_{t-N(t)(T-t)}^{t-(T-t)e^{-2a\sigma\sqrt{|\ln(T-t)|}}}\frac{T_{0}[h](t)-T_{0}[h](s)}{t-s}ds| \le C \|h\|_{\gamma,m} \frac{e^{-2a\sqrt{|\ln(T-t)|}}}{\sqrt{|\ln(T-t)|}|\ln(T-t)|^{m}}\frac{N(t)}{T-t}(T-t)+\\
		&\hspace{2cm}+\frac{|h(t)|}{2a\sqrt{|\ln(T-t)|}}\int_{t-N(t)(T-t)}^{t-(T-t)e^{-2a\sigma\sqrt{|\ln(T-t)|}}}\frac{ds}{t-s}+\\
		&\hspace{2cm}+\int_{t-N(t)(T-t)}^{t-(T-t)e^{-2a\sigma\sqrt{|\ln(T-t)|}}}\frac{|h(s)|}{2a\sqrt{|\ln(T-s)|}(t-s)}ds\le \\
		&\hspace{1.5cm}\le C\|h\|_{\gamma,m}\frac{e^{-2a\gamma\sqrt{|\ln(T-t)|}}}{|\ln(T-t)|^{1/4}|\ln(T-t)|^{m}}+(2\sigma+\omega_{T})\|h\|_{\gamma,m}\frac{e^{-2a\gamma\sqrt{|\ln(T-t)|}}}{|\ln(T-t)|^{m}}.
	\end{align*}
    where the reason why we need to introduce this $\omega_{T}$ is due to the approximation $\frac{e^{-2a\sqrt{|\ln(T-t+(T-t)N(t))|}}}{\sqrt{|\ln(T-t+(T-t)N(t))}|}\approx \frac{e^{-2a\sqrt{|\ln(T-t)|}}}{\sqrt{|\ln(T-t)|}}$. We observe that $\omega_{T}$ can be chosen arbitrarily small if $T,\varepsilon(T)$ are small.\newline
	The estimates for the remaining term in \eqref{Asigmasimpli} is elementary after observing that
	\begin{align*}
		\frac{e^{-2a\gamma \sqrt{|\ln(T-t)|}}}{|\ln(T-t)|^{m}\sqrt{|\ln(T-t)|}}\int_{t-N(t)(T-t)}^{t}\frac{ds}{T-s}\approx \frac{e^{-2a\gamma\sqrt{|\ln(T-t)|}}}{|\ln(T-t)|^{m}}\frac{\ln(1+N(t))}{\sqrt{|\ln(T-t)|}}.
	\end{align*}
\end{proof}
For any $0<\sigma<1/2$, the next Lemma will give us a solution for \eqref{FinalEq} and we will be able to control $p(t)$, $\dot{p}(t)$, $\ddot{p}$ and $\dddot{p}$ (as we previously anticipated we need to control the third derivative of this first iteration to control the second derivative of $p_{2}$ and finally having some control of the second derivative of the error).
\begin{align*}
	|p_{1}(t)|\le C \frac{e^{-2a\sqrt{|\ln(T-t)|}}}{|\ln(T-t)|^{1/4}}, \ \ \ \  |\dot{p}_{1}(t)|\le \frac{e^{-2a\sqrt{|\ln(T-t)|}}}{|\ln(T-t)|^{1/4}(T-t)}, \ \ \ |\ddot{p}_{1}(t)|\le\frac{e^{-2a\sqrt{|\ln(T-t)|}}}{(T-t)^{2}|\ln(T-t)|^{1/4}}.
\end{align*}
For this reason, after recalling the norm \eqref{normP} it is helpful to introduce
\begin{align}\label{normForPder}
	\|\dot{p}\|^{*}_{\gamma,m}:=\sup_{t\in[-\varepsilon(T),T]}(T-t)|\ln(T-t)|^{m}e^{2a\gamma \sqrt{|\ln(T-t)|}}|\dot{p}(t)|,
\end{align} 
\begin{align}\label{normForPSecondder}
	\|\ddot{p}\|^{\star\star}_{\gamma,m}:=\sup_{t\in[-\varepsilon(T),T]}|(T-t)^{2}\ln(T-t)|^{m}e^{2a\gamma \sqrt{|\ln(T-t)|}}|\ddot{p}(t)|,
\end{align}
\begin{align}\label{normForPThirdder}
	\|\dddot{p}\|^{\star\star\star}_{\gamma,m}:=\sup_{t\in[-\varepsilon(T),T]}|(T-t)^{3}\ln(T-t)|^{m}e^{2a\gamma \sqrt{|\ln(T-t)|}}|\dddot{p}(t)|.
\end{align}
\begin{lemma}\label{SolutionNotRegular}
	For any $0<\sigma<\frac{1}{2}$ if we let $T$, $\varepsilon(T)$ sufficiently small, there exist a unique $h\in C^{3}[(-\varepsilon(T),T);\mathbb{R}]$ such that $\|h\|_{1,1/4}, \|h\|^{\star}_{1,1/4}$, $\|h\|^{\star\star}_{1,1/4}<\infty$ $\|h\|^{\star\star\star}_{1,1/4}<\infty$ and that solves \eqref{RealFinalEq}.
\end{lemma}
\begin{proof}
	Thanks to Lemma \ref{T0est}, \ref{EstB}, \ref{EstA} and Corollary \ref{ExpansionMass} (fixing $l=3/4$), for any $h$ such that $\|h\|_{1,1/4}<M$, we can define the operator
	\begin{align}\label{v}
		\mathcal{R}[h](t):=[f_{\star}(t)+\mathcal{A}_{\sigma}[T_{0}[h]]+\mathcal{B}_{0}[p_{\star}+T_{0}[h]]-\mathcal{B}_{0}[p_{\star}]+\mathcal{B}_{1}[T_{0}[h]]]\eta (\frac{t}{T}) \ \ t\in[-\varepsilon(T),T],
	\end{align}
	we have
	\begin{align*}
		|\mathcal{R}[h](t)|\le \frac{e^{-2a\sqrt{|\ln(T-t)|}}}{|\ln(T-t)|^{1/4}}[&C_{f_{\star}}+C\|h\|_{1,\frac{1}{4}}\big(\frac{1}{|\ln(T+\varepsilon(T))|^{1/4}}+\frac{\frac{M}{|\ln(T+\varepsilon(T)|^{1/4})}}{1-\frac{M}{|\ln(T+\varepsilon(T)|^{1/4})}}\big)+\\
		&+C\|h\|_{1,\frac{1}{4}}\big(\frac{C}{|\ln(T+\varepsilon(T))|^{1/4}}+(2\sigma+\omega_{T})\big)+C\|h\|_{1,\frac{1}{4}}\frac{1}{|\ln(T+\varepsilon(T))|^{1/4}}]\le\\
		&\hspace{-2.8cm}\le \frac{e^{-2a\sqrt{|\ln(T-t)|}}}{|\ln(T-t)|^{1/4}}[C_{f_{\star}}+C\|h\|_{1,\frac{1}{4}}\big(\frac{1}{|\ln(T+\varepsilon(T))|^{1/4}}+\frac{\frac{M}{|\ln(T+\varepsilon(T)|^{1/4})}}{1-\frac{M}{|\ln(T+\varepsilon(T)|^{1/4})}}\big)+(2\sigma+\omega_{T})\|h\|_{1,\frac{1}{4}}]
	\end{align*}
	where $C_{f_{\star}}=\|f_{\star}\|_{1,\frac{1}{4}}$ and $C$ is some universal constant. Similarly if $h_{1},h_{2}$ are such that $\|h_{i}\|<M$ we have
	\begin{align*}
		|\mathcal{R}[h_{1}](t)-\mathcal{R}[h_{2}](t)|\le \frac{e^{-2a\sqrt{|\ln(T-t)|}}}{|\ln(T-t)|^{1/4}}[C\|h_{1}-&h_{2}\|_{1,\frac{1}{4}}\big(\frac{1}{|\ln(T+\varepsilon(T))|^{1/4}}+\frac{\frac{M}{|\ln(T+\varepsilon(T)|^{1/4})}}{1-\frac{M}{|\ln(T+\varepsilon(T)|^{1/4})}}\big)+\\
		&+(2\sigma+\omega_{T})\|h_{1}-h_{2}\|_{1,\frac{1}{4}}].
	\end{align*}
	Then we clearly see that for any $\sigma<1/2$, by picking $M$ sufficiently large, but independent of $T$ and $\varepsilon(T)$, such that $C_{f_{\star}}<M(1-2\sigma)$ in order for $\mathcal{R}$ to have a fixed point it is enough to take $T$ and $\varepsilon(T)$ sufficiently small. \newline 
	Now we want to control $\|h\|^{\star}_{1,1/4}$. We formally differentiate \eqref{v} and we have
	\begin{align}\label{derivativeDecR}
		\frac{d}{dt}\mathcal{R}[h](t)=\frac{1}{T}\eta'(\frac{t}{T})[...]+\eta(\frac{t}{T})f'_{\star}(t)+\eta(\frac{t}{T})\frac{d}{dt}[...].
	\end{align}
Now, we claim that the first term in the right-hand side of \eqref{derivativeDecR} satisfy
\begin{align*}
	|\frac{1}{T}\eta'(\frac{t}{T})[...]|\le \frac{e^{-2a\sqrt{|\ln(T-t)|}}}{|\ln(T-t)|^{1/4}}.
\end{align*}
Indeed for example we see
	\begin{align*}
	&|\eta '(\frac{t}{T})\frac{1}{T}A_{\sigma}[p_{1}](t)|\le C \eta'(\frac{t}{T}) \|p_{1}\|_{1,1/4} \frac{e^{-2a\sqrt{|\ln T|}}}{|\ln T|^{1/4} T}.
\end{align*}
We can directly compute $f_{\star}'(t)$ from Lemma \ref{SECMASSprofile} and then the second term in \eqref{derivativeDecR} satisfies
\begin{align*}
	|\eta(\frac{t}{T})f'_{\star}(t)|\le C \eta(\frac{t}{T}) \frac{e^{-2a\sqrt{|\ln(T-t)|}}}{|\ln(T-t)|^{3/4}}.
\end{align*}
The remaining term in \eqref{derivativeDecR} is an operator in $\dot{h}$ that gives a contraction if $0<\sigma<\frac{1}{2}$ as can be seen by differentiating the estimates in Corollary \ref{ExpansionMass} and Lemmas \ref{T0est}, \ref{EstB}, \ref{EstA}. \newline
Differentiating twice \eqref{v} and repeating the same argument we get also that $\|h\|_{1,1/4}^{\star\star}<\infty$. Differentiating three times instead we get $\|h\|_{1,1/4}^{\star\star\star}<\infty$.
	\end{proof}
As a result of Lemma \ref{SolutionNotRegular}, if $\bar{\lambda}(t)=\sqrt{-2\int_{t}^{T}(p_{\star}(s)+T_{0}[h](s))ds}$ for $t\in(0,T)$ we have
\begin{align*}
	\int_{\mathbb{R}^{2}}\varphi_{\bar{\lambda}}(\cdot,t)dx-16\pi \gamma \frac{\bar{\lambda}^{2}(t)}{T-t}=&\int_{\mathbb{R}^{2}}\varphi_{\bar{\lambda}}(\cdot,T)dx+O(|\mathcal{E}_{0}(t)|)+\\
	&+O(\int_{t-(T-t)e^{-2a\sigma\sqrt{|\ln(T-t)|}}}^{t-\bar{\lambda}^{2}}\frac{|T_{0}[h](t)-T_{0}[h](s)|}{t-s}ds)=\\
	=&\int_{\mathbb{R}^{2}}\varphi_{\bar{\lambda}}(\cdot,T)dx+O(\frac{e^{-2a(1+\sigma)\sqrt{|\ln(T-t)|}}}{|\ln(T-t)|^{1/4}}).
\end{align*}
But since $0<\sigma<\frac{1}{2}$ we still do not have the proof of Proposition \ref{prop-lambda0} that will be given in the next Section. In the next Section we will also need the control of the second and the third derivatives we obtained in Lemma \ref{SolutionNotRegular}.
\subsection{The proof of Proposition \ref{prop-lambda0}}\label{manageerrorSec}
The previous section gave us the function $p_{1}=T_{0}[h_{1}]$. This approximation is sufficiently regular to observe that $\bar{\lambda}(t)=\sqrt{-2\int_{t}^{T}(p_{\star}(s)+T_{0}[h_{1}](s))ds}$ satisfies
\begin{align}\label{badDecayinNONl}
	\int_{\mathbb{R}^{2}}\varphi_{\bar{\lambda}}(\cdot,t)dx-16\pi \gamma \frac{\bar{\lambda}^{2}(t)}{T-t}=&\int_{\mathbb{R}^{2}}\varphi_{\bar{\lambda}}(\cdot,T)dx+O(e^{-2a(1+\sigma)\sqrt{|\ln(T-t)|}}).
\end{align}
for any $0<\sigma<\frac{1}{2}$. It is clear that this error is not sufficiently fast decaying to obtain Proposition \ref{prop-lambda0}. The idea is to find a second improvement of $p(t)$ and write
\begin{align*}
 p(t)=p_{\star}(t)+p_{1}(t)+p_{2}(t)
\end{align*}
where $p_{2}(t)$ is going to be fixed to erase
\begin{align*}
	\mathcal{R}[p_{1}]:=\int_{t-(T-t)e^{-2a\sigma\sqrt{|\ln(T-t)|}}}^{t-\lambda_{\star}^{2}}\frac{p_{1}(t)-p_{1}(s)}{t-s}ds
\end{align*}
that is responsible for the bad decay in \eqref{badDecayinNONl}. But then we are introducing another error, $\mathcal{R}[p_{2}]$. We notice that to estimate $\mathcal{R}[p_{2}]$, $\frac{d}{dt}\mathcal{R}[p_{2}]$ we need some estimates for $\dot{p}_{2}$ and consequently we need an estimate for the first and second derivatives of the right-hand side of the equation for $p_{2}$, namely $\eta(\frac{t}{T})\mathcal{R}[p_{1}]$. This is why in Lemma \ref{SolutionNotRegular} we controlled also the second and third derivatives of $p_{1}$. Finally we remark that to control $\frac{d^{2}}{dt^{2}}\mathcal{R}[p_{2}]$ we can just use the control of the $\ddot{p}_{2}$ even if we are loosing some decay. \newline
 We can introduce the proof Proposition \ref{prop-lambda0}.
\begin{proof}[Proof of Proposition \ref{prop-lambda0}] 
	As presented in the introduction of this Section the idea is to write
	\begin{align}\label{finalExpansionP}
		p(t)=p_{\star}(t)+p_{1}(t)+p_{2}(t)
	\end{align}  
   for a $p_{2}$ that erases the main error introduced by $p_{1}$. By keeping in mind \eqref{realfinalEqt} and that $p_{1}$ solves \eqref{FinalEq} it is immediate to check that we want to find $p_{2}$ that if $t\in(-\varepsilon(T),T)$ satisfies
   \begin{align}\label{FinalEqUncP2}
   	\int_{t}^{T}\frac{p_{2}(s)}{T-s}ds-2a\sqrt{|\ln(T-t)|}p_{2}(t)=&\eta(\frac{t}{T})\big[\int_{t-(T-t)e^{-2a\sigma\sqrt{|\ln(T-t)|}}}^{t-\lambda^{2}_{\star}(t)}\frac{p_{1}(t)-p_{1}(s)}{t-s}ds+\nonumber\\
   	&+\mathcal{A}_{\sigma}[p_{1}+p_{2}]-\mathcal{A}_{\sigma}[p_{1}]+\mathcal{B}_{0}[p_{\star}+p_{1}+p_{2}]-\mathcal{B}_{0}[p_{\star}+p_{1}]+\nonumber\\
   	&+ \mathcal{B}_{1}[p_{1}+p_{2}]-\mathcal{B}_{1}[p_{1}]\big].
   \end{align}
where now thanks to Lemma \ref{SolutionNotRegular} we also have that for any $t\in(-\frac{T}{2},T)$
\begin{align}\label{RHSp2}
	|\int_{t-(T-t)e^{-2a\sigma\sqrt{|\ln(T-t)|}}}^{t-\lambda^{2}_{\star}(t)}\frac{p_{1}(t)-p_{1}(s)}{t-s}ds|\le \frac{e^{-2a(1+\sigma)\sqrt{|\ln(T-t)|}}}{|\ln(T-t)|^{\frac{1}{4}}},
\end{align}
\begin{align}\label{DerRHSp2}
	|\frac{d}{dt}\big(	\int_{t-(T-t)e^{-2a\sigma\sqrt{|\ln(T-t)|}}}^{t-\lambda^{2}_{\star}(t)}\frac{p_{1}(t)-p_{1}(s)}{t-s}ds\big)|\le \frac{e^{-2a(1+\sigma)\sqrt{|\ln(T-t)|}}}{|\ln(T-t)|^{\frac{1}{4}}(T-t)},
\end{align}
\begin{align}\label{DerRHSp23}
	|\frac{d^{2}}{dt^{2}}\big(	\int_{t-(T-t)e^{-2a\sigma\sqrt{|\ln(T-t)|}}}^{t-\lambda^{2}_{\star}(t)}\frac{p_{1}(t)-p_{1}(s)}{t-s}ds\big)|\le \frac{e^{-2a(1+\sigma)\sqrt{|\ln(T-t)|}}}{|\ln(T-t)|^{\frac{1}{4}}(T-t)^{2}}.
\end{align}
We can solve \eqref{FinalEqUncP2} for $p_{2}$ in $t\in(-\varepsilon(T),T)$ and for any $0<\sigma<\frac{1}{2}$ analogously to the proof of Lemma \ref{SolutionNotRegular}. Observing that now $\gamma>1$, thanks to Lemma \ref{T0est} we get 
\begin{align*}
	|p_{2}(t)|\le C \frac{e^{-2a(1+\sigma)\sqrt{|\ln(T-t)|}}}{|\ln(T-t)|^{3/4}}, \ \ \ |\dot{p}_{2}(t)|\le C \frac{e^{-2a(1+\sigma)\sqrt{|\ln(T-t)|}}}{|\ln(T-t)|^{3/4}(T-t)}, \ \ |\ddot{p}_{2}(t)|\le C \frac{e^{-2a(1+\sigma)\sqrt{|\ln(T-t)|}}}{|\ln(T-t)|^{3/4}(T-t)^{2}} .
\end{align*} 
Finally the main order term of the final error satisfies
   \begin{align}\label{sigmaRemainderFINALFINAL}
   	|\int_{t-(T-t)e^{-2a\sigma\sqrt{|\ln(T-t)|}}}^{t-\lambda_{\star}^{2}}\frac{p_{2}(t)-p_{2}(s)}{t-s}ds|\le C \frac{e^{-2a(1+2\sigma)\sqrt{|\ln(T-t)|}}}{|\ln(T-t)|^{3/4}},
   \end{align}
   \begin{align}\label{sigmaRemainderFINALFINALDer}
	|\frac{d}{dt}\int_{t-(T-t)e^{-2a\sigma\sqrt{|\ln(T-t)|}}}^{t-\lambda_{\star}^{2}}\frac{p_{2}(t)-p_{2}(s)}{t-s}ds|\le C \frac{e^{-2a(1+2\sigma)\sqrt{|\ln(T-t)|}}}{|\ln(T-t)|^{3/4}(T-t)},
\end{align}
where $\sigma$ can be chosen arbitrarily close to $\frac{1}{2}$. \newline
Finally we want to estimate the derivative of the error in \eqref{EstimatesErrorproplambda0}. We observe that to get that estimate it is not necessary to control $\ddot{p}_{2}$, in fact we can simply observe
\begin{align*}
	|\frac{d^{2}}{dt^{2}}\int_{t-(T-t)e^{-2a\bar{\sigma}\sqrt{|\ln(T-t)|}}}^{t-\lambda_{\star}^{2}}\frac{p_{2}(t)-p_{2}(s)}{t-s}ds|=&|\frac{d^{2}}{dt^{2}}\big[2a(1-\sigma)\sqrt{|\ln(T-t)|}p_{2}(t)+...\big]|\le \\
	\le&C \frac{e^{-2a(1+\sigma)\sqrt{|\ln(T-t)|}}}{|\ln(T-t)|^{1/4}(T-t)}.
\end{align*}
\end{proof}

\section{Inner linear theory}\label{InnerTheoryIntro}
In this introductory section we consider the problem
\begin{align}\label{typicalCaseInn}
	\begin{cases}
		\lambda^{2}\partial_{t}\phi=L[\phi]+\tilde{B}[\phi]+h(y,t) \ \ \ \ \text{in }\mathbb{R}^{2}\times (0,T)\\
		\phi(\cdot,0)=0 \ \ \text{in }\mathbb{R}^{2}
	\end{cases}
\end{align}
where
\begin{align}\label{InnTheorLInvD}
	&L[\phi]=\nabla \cdot \big[U\nabla\big(\frac{\phi}{U}-(-\Delta)^{-1}\phi\big)\big], \ \ \ \ (-\Delta)^{-1}\phi(y,t)=\frac{1}{2\pi}\int_{\R^2}\ln\frac{1}{|y-z|}\phi(z,t)dz
\end{align} 
and
\begin{align}\label{BandCutOFFINN}
	&\tilde{B}[\phi]=\lambda\dot{\lambda}(k\phi\hat{\chi}+y\cdot\nabla(\phi\hat{\chi}))-\lambda\dot{\lambda}\big[\int_{\R^2}(k\phi\hat{\chi}+y\cdot \nabla(\phi\hat{\chi}))dy\big]W_{0}(y), \ \ \ \hat{\chi}=\chi_{0}(\frac{|y|\lambda}{4\sqrt{\delta(T-t)}}).
\end{align}
We notice that the many of the considerations we are going to make are very general and can be adapted to all \eqref{innereqfinalFINAL}, \eqref{innereqfinalFINAL2} and \eqref{innereqfinalFINALMODE1}. \newline
If we assume that $h$ is radial and satisfies the zero mass condition, by the conservation of mass, we observe that the mass of $\phi$ in \eqref{typicalCaseInn} must be zero and then, when $k=2$, this is the same equation that appears in \eqref{innereqfinalFINAL2}. We anticipate that investigating this more general equation depending on a parameter $k$ will help enormously when we will study equation \eqref{innereqfinalFINAL}  (the study of \eqref{innereqfinalFINALMODE1} will be essentially easier but independent of the study of the other equations).\newline
In what follows we will more simply use the notation $\chi$ instead of $\hat{\chi}$, even if we are always referring to \eqref{chihat}.
It is natural to change the time variable into 
\begin{align}\label{tauVAR}
	\tau:=\tau_{0}+\int_{0}^{t}\frac{1}{\lambda^{2}(s)}ds\implies\tau\to \infty \text{ as }t\to T.
\end{align}
where $\tau_{0}$ will be assumed large and satisfying $\tau_{0}\le C e^{\sqrt{2|\ln T|}}\sqrt{|\ln T|}$.
Since $\lambda(t)=\lambda_{0}(t)+\lambda_{1}(t)$, by Proposition \ref{prop-lambda0} and \eqref{normforlambda1}, we can easily show
\begin{align}\label{tauExp}
	\tau=\tau(t)= \frac{e^{\mu+2}}{2\sqrt{2}}e^{\sqrt{2|\ln(T-t|)}}\sqrt{|\ln(T-t)|}(1+o(1)), \ \ \ \ln \tau=\sqrt{2|\ln(T-t)|}(1+o(1))
\end{align}
and 
\begin{align}\label{lambdadotlambdaExp}
	\lambda\dot{\lambda}=-2e^{-(\mu+2)}e^{-\sqrt{2|\ln(T-t)|}}(1+o(1))=-\frac{1}{2}\frac{\ln \tau}{\tau}(1+o(1)).
\end{align}
In many occasions it will be helpful to keep the expansions \eqref{tauExp}, \eqref{lambdadotlambdaExp} in mind. 
\newline We remark that expansion \eqref{lambdadotlambdaExp} makes this problem more difficult to solve with respect to \cite{DdPW}. Indeed if we would avoid the use of the cut-off in the operator $B$, after the analogous change of the time variable, we would see
\begin{align}
	B[\phi]=\lambda\dot{\lambda}(2\phi+y\cdot \nabla \phi)\approx\begin{cases}
		-\frac{1}{2\tau \ln \tau}(2\phi+y\cdot \nabla \phi) \ \ \ \ \text{ in infinite time }\\[5pt]
		-\frac{\ln \tau}{2\tau}(2\phi+y\cdot \nabla \phi) \ \ \ \ \ \text{in finite time}.
	\end{cases}
\end{align}
But in order to obtain the decay in space $\approx\frac{1}{|y|^{4}}$ of the inner solution (as we will do for example in Lemma \ref{barriersUg}) they had to construct a barrier and it was crucial to observe that
\begin{align}\label{CrucialBarrier}
	|B[\phi]|\ll \frac{1}{1+|y|^{2}}(|\phi|+|y\cdot \nabla \phi|) \ \ \ \ \text{if }|y|\le \sqrt{\tau}.
\end{align}
We notice that introducing the cut-off $\chi=\chi_{0}(\frac{|y|\lambda}{\sqrt{\delta(T-t)}})$ we get
\begin{align*}
	|B[\chi \phi]|\le C \frac{\ln \tau}{\tau}(|\phi\chi|+|y\cdot \nabla (\phi\chi)|)\le \frac{\delta}{1+|y|^{2}}(|\phi|+|y\cdot \nabla \phi|) \ \ \ \text{if }|y|\le \sqrt{\tau}.
\end{align*}
By taking $\delta$ small we are artificially reproducing  \eqref{CrucialBarrier}. We finally notice that if we tried to remove the hypothesis \eqref{CrucialBarrier} we would need to introduce an intermediate region where our inner solution would decay only like $\approx \frac{1}{|y|^{2}}$ (that is in the kernel of the operator $B$). This would have the effect that some of the remainders produced in the energy estimates would be too large to be absorbed. \newline
We can then rewrite \eqref{typicalCaseInn} and obtain
\begin{align}\label{typicalCaseInnTAU}
	\begin{cases}
		\partial_{\tau}\phi=L[\phi]+\tilde{B}[\phi]+h(y,t) \ \ \ \ \text{in }\mathbb{R}^{2}\times (\tau_{0},\infty)\\
		\phi(\cdot,\tau_{0})=0 \ \ \text{in }\mathbb{R}^{2}.
	\end{cases}
\end{align}
We will assume that for all $\tau_{1}>0$ there is $C_{1}$ such that
\begin{align*}
	|h(y,\tau)|\le \frac{C_{1}}{1+|y|^{6}} \ \ \ \ \text{for all }(y,\tau)\in \mathbb{R}^{2}\times(\tau_{0},\tau_{1}).
\end{align*}
By a solution $\phi(y,\tau)$ of \eqref{typicalCaseInnTAU} we understand a function $\phi(y,\tau)$, of class $C^{1}$ in $y$ and of class $C^{0}$ in $t$, such that for any $\tau_{1}>\tau_{0}$ there exists a $C_{1}>0$ with
\begin{align}\label{phiinnerSmall}
	|\phi(y,\tau)|+(1+|y|)|\nabla_{y}\phi(y,\tau)|\le \frac{C_{1}}{1+|y|^{6}} \ \ \ \text{for all } (y,\tau)\in \mathbb{R}^{2}\times(\tau_{0},\tau_{1})
\end{align}
and satisfies the integral equation
\begin{align*}
	\phi(y,\tau)=\int_{\tau_{0}}^{\tau}\int_{\R^2}G(y-z,\tau-\sigma)[-\nabla\phi\cdot \nabla \Gamma_{0}-\nabla U \cdot \nabla(-\Delta)^{-1}\phi+2U\phi+\tilde{B}[\phi]+h](z,\sigma)dzds
\end{align*}
where $(-\Delta)^{-1}$ has been defined in \eqref{InnTheorLInvD} and $G(y,\tau)$ is the two-dimensional heat kernel $G(y,\tau)=\frac{1}{4\pi \tau}e^{-\frac{|y|^{2}}{4\tau}}$. An elementary computation shows that
\begin{align*}
	|\phi(y)|\le \frac{C}{1+|y|^{6}} \implies |\nabla(-\Delta)^{-1}\phi(y)|\le \frac{C}{1+|y|}\|(1+|y|^{6})\phi\|_{L^{\infty}(\mathbb{R}^{2})}.
\end{align*}
By this observation existence and uniqueness of a solution of \eqref{typicalCaseInnTAU} satisfying \eqref{phiinnerSmall} is standard: for a short time $\tau_{1}>\tau_{0}$ this is established by a contraction mapping argument and then a linear continuation procedure applies. \newline
Now we recall the norms \eqref{RHSinnernorm}, \eqref{SolutionINNERNORM} we introduced in Section \ref{proofThm1}: for some $\nu$, $p$, $\epsilon$ and $m\in\mathbb{R}$ we define
\begin{align*}
	\|h\|_{0;\nu,m,p,\epsilon}=\inf \left\{ K \text{ s.t. } |h(y,\tau)|\le \frac{K}{\tau^{\nu}(\ln \tau)^{m}}\frac{1}{(1+|y|)^{p}}\begin{cases}
		1 \ \ \ \ |y|\le \sqrt{\tau}\\
		\frac{\tau^{\epsilon/2}}{|y|^{\epsilon}} \ \ \ |y|\ge \sqrt{\tau}
	\end{cases}\right\},
\end{align*}
\begin{align*}
	\|\phi\|_{1;\nu,m,p,\epsilon}=\inf \left\{K \text{ s.t. } 
	|\phi(y,t)|+(1+|y|)|\nabla_{y}\phi(y,t)|\le \frac{K}{\tau^{\nu}(\ln \tau)^{m}}\frac{1}{(1+|y|)^{p}}\begin{cases}
		1 \ \ \ \ |y|\le \sqrt{\tau}\\
		\frac{\tau^{\epsilon/2}}{|y|^{\epsilon}} \ \ \ \ |y|\ge\sqrt{\tau}
	\end{cases}\right\}.
\end{align*}
It is easy observe that fast decay of the right-hand side does not imply fast decay of the solution unless certain orthogonality conditions hold. For example, let us consider the following simplified problem
\begin{align}\label{SimplifiedInnerEq}
	\begin{cases}
		\partial_{ \tau}\phi=L[\phi]+h(y,\tau) \ \ \ \ \text{in }\mathbb{R}^{2}\times(\tau_{0},\infty)\\
		\phi(\cdot,\tau_{0})=0 \ \ \text{in }\mathbb{R}^{2}.
	\end{cases}
\end{align} 
Let us multiply \eqref{SimplifiedInnerEq} by $|y|^{2}$, observing that $\int_{\R^2}L[\phi]|y|^{2}dy=0$ we need
\begin{align*}
	\int_{\tau_{0}}^{\infty}\int_{\R^2}h(y,\tau)|y|^{2}dyd\tau=0.
\end{align*}
But we know that a zero second moment condition cannot be achieved (see equation \eqref{innereqfinalFINAL2}).
As we will show in the following sections we can solve this problem introducing an initial condition and a parameter $c_{1}$, namely
\begin{align}\label{typicalCaseInnTAUIC}
	\begin{cases}
		\partial_{\tau}\phi=L[\phi]+\tilde{B}[\phi]+h(y,t) \ \ \ \ \text{in }\mathbb{R}^{2}\times (\tau_{0},\infty)\\
		\phi(\cdot,\tau_{0})=c_{1}\tilde{Z}_{0} \ \ \text{in }\mathbb{R}^{2}
	\end{cases}
\end{align}
where $\tilde{Z}_{0}$ is defined as
\begin{align}\label{InitialCond}
	\tilde{Z}_{0}(\rho)=(Z_{0}(\rho)-m_{Z_{0}}U)\chi_{0}(\frac{\rho}{\sqrt{\tau_{0}}})=(Z_{0}(\rho)-m_{Z_{0}}U)\bar{\chi}
\end{align}
where $m_{Z_{0}}$ is such that 
\begin{align}\label{mz0Est}
	\int_{\R^2}\tilde{Z}_{0}=0 \implies |m_{\tilde{Z}_{0}}|=O(\frac{1}{\tau_{0}}).
\end{align}
Notice that this initial condition is a perturbation of $Z_{0}$ that is in the kernel of the operator $L$. This will play a crucial role in the proof of Lemma \ref{Lemma101}. Moreover we have to cut the initial condition since the second moment of the solution must be bounded to compensate the missing zero second moment condition. In fact this hypothesis will be fundamental in the construction of the barriers in the proof of Lemma \ref{barriersUg}. As can be seen for example in \eqref{Hbarrier} our barriers prevent us from considering solutions decaying like $\frac{1}{\rho^{4}}$ without loosing decay in time. We stress that the choice of the cut-off $\bar{\chi}$ will be important in the proof of Proposition \ref{Prop101}. \newline
Now we state the results we will prove in the following sections. \newline 
Notice that we will not explicitly mention the estimate for the gradient of the solutions but they can can be achieved by scaling and standard parabolic estimates. 
\begin{proposition}\label{PropMode0Mass0}
Let $\lambda=\lambda_{0}+\lambda_{1}$ with $\lambda_{0}$ given by Proposition \ref{prop-lambda0} and let us assume \eqref{ExtraAssumptionSECMASSSECMOM}. Let $\sigma\in(0,1)$, $\epsilon>0$ with $\sigma+\epsilon<2$, $1<\nu<\frac{7}{4}$, $m\in \mathbb{R}$. Let $q\in(0,1)$. For $\tau_{0}$ sufficiently large and for all radially symmetric $h=h(|y|,\tau)$ with $\|h\|_{\nu,m,6+\sigma,\epsilon}<\infty$ and such that
	\begin{align*}
		\int_{\mathbb{R}^{2}}h(y,t)dy=0, \ \ \ \text{for all }(\tau_{0},\infty)
	\end{align*}
	there exists $c_{1}\in \mathbb{R}$ and a solution $\phi(y,\tau)=\mathcal{T}_{p}^{i,2}[h]$ of 
	\begin{align*}
		\begin{cases}
			\partial_{\tau}\phi = L[\phi]+B[\phi \chi]+h \\
			\phi(\cdot, \tau_{0})=c_{1}\tilde{Z}_{0}
		\end{cases}
	\end{align*}
	that define a linear operators of $h$ and satisfy the estimate
	\begin{align*}
		\|\phi\|_{1;\nu-1,m+q+1,4,2+\sigma+\epsilon}\le \frac{C}{(\ln \tau_{0})^{1-q}}\|h\|_{0;\nu,m,6+\sigma,\epsilon}.
	\end{align*}
	\begin{align*}
		|c_{1}|\le C \frac{1}{\tau_{0}^{\nu-1}(\ln \tau_{0})^{m+2}}\|h\|_{\nu,m,6+\sigma,\epsilon}.
	\end{align*} 
	
\end{proposition}
The proof of Proposition \ref{PropMode0Mass0} will be given in Section \ref{ProofProp71}. \newline
The next result will consider a right-hand side with zero mass and zero second moment. The idea is to transform the problem applying $L^{-1}[\phi]$ to the equation and then use Proposition \ref{PropMode0Mass0}. As an effect of this transformation we will have to manage with some errors. Some of them will be treated as lower order terms, others will be included in a force that we will denote $\mathcal{E}[\phi]$. This force will have good decay in the inner region, but some estimates deteriorate in the outer region. Since $\mathcal{E}[\phi]$ will appear in the right-hand side it is important to observe that it will be multiplied by \eqref{chitilde} and that this force is radial with zero mass and zero second moment. \newline
We also observe that because of the estimates we get for $\mathcal{E}[\phi]$ in the outer region the solution will not decay as fast as in Proposition \ref{PropMode0Mass0} in the outer region. This loss will not spoil our construction. \newline
In the following proposition we will use the following notation
\begin{align}\label{FunctionM}
	\chi=\chi_{0}(\frac{|y|\lambda}{\sqrt{\delta(T-t)}})=\chi_{0}(\frac{|y|}{M(\tau)}).
\end{align}
\begin{proposition}\label{PropMode0MassSecMom0}
	Let $\lambda=\lambda_{0}+\lambda_{1}$ with $\lambda_{0}$ given by Proposition \ref{prop-lambda0} and let us assume \eqref{ExtraAssumptionSECMASSSECMOM}. Let $\sigma\in(0,1)$, $\varepsilon>0$, $\sigma+\varepsilon<2$, $1<\nu<\min(1+\frac{\varepsilon}{2},3-\frac{\sigma}{2},\frac{3}{2})$, $m\in \mathbb{R}$. Let $0<q<1$. There exists a number $C>0$ such for $\tau_{0}$ sufficiently large and for all radially symmetric $h=h(|y|,\tau)$ with $\|h\|_{0;\nu,m,6+\sigma,\epsilon}<\infty$ and
	\begin{align*}
		\int_{\mathbb{R}^{2}}h(y,\tau)dy=0, \ \ \ \text{for all }(\tau_{0},\infty), \ \ \ \int_{\mathbb{R}^{2}}h(y,\tau)|y|^{2}dy=0 \ \ \ \ \text{for all}(\tau_{0},\infty)
	\end{align*}
	there exists $c_{1}\in \mathbb{R}$, an operator $\mathcal{E}$ and a solution $\phi(y,\tau)=\mathcal{T}_{p}^{i,1}[h]$ of 
	\begin{align*}
		\begin{cases}
			\partial_{\tau}\phi = L[\phi]+B[\phi]+\mathcal{E}[\phi]+h \\
			\phi(\cdot, \tau_{0})=c_{1}L[\tilde{Z}_{0}]
		\end{cases}
	\end{align*}
	that defines a linear operator of $h$ and satisfies the estimate
	\begin{align*}
		\|\phi\|_{1;\nu-\frac{1}{2},m+\frac{q+1}{2},4,2}\le C\|h\|_{0;\nu,m,6+\sigma,\epsilon}.
	\end{align*}
	Moreover we know that $\mathcal{E}[\phi]$ is radial, it satisfies $\int_{\R^2}\mathcal{E}[\phi](y)dy=0$, $ \int_{\R^2}\mathcal{E}(y)[\phi]|y|^{2}dy=0$ and
\begin{align*}
		\ \ |\mathcal{E}[\phi](y)|dy\le C \|h\|_{0;\nu,m,6+\sigma,\epsilon}\frac{1}{\tau^{\nu}\ln^{m+q}\tau} \begin{cases}
			\frac{1}{M^{2}(\tau)}\frac{1}{(1+\rho)^{6}} \ \ \ \ |y|\le M(\tau) \\
			\frac{1}{1+|y|^{6}} \ \ \ |y|\ge M(\tau).
		\end{cases}
\end{align*}
	Moreover $c_{1}$ is a linear operator of $h$ and
	\begin{align*}
		|c_{1}|\le C \frac{1}{\tau_{0}^{\nu-1}(\ln \tau_{0})^{m+2}}\|h\|_{0,\nu,m,6+\sigma,\epsilon}.
	\end{align*} 	
\end{proposition}
The proof of Proposition \ref{PropMode0MassSecMom0} will be given in Section \ref{Section11}. We remark that we introduced the cut-off \eqref{BandCutOFFINN}, since our estimate of the force will be sufficiently good only when $|x-\xi|\le 2\sqrt{\delta(T-t)}$. But this is sufficient since the force will be multiplied by the cut-off \eqref{chitilde}. \newline
Now we propose a result in mode 1 that will solve \eqref{innereqfinalFINALMODE1}. Similarly to the previous result we want to add a force that now allows us to cut $B$ but preserving the conservation of the center of mass. We recall the definitions of the functions \eqref{W1j} we already used to project the right-hand side of the inner equations. We will need the decomposition \eqref{DecomPRAD}.
\begin{proposition}\label{PropMode1}
	Let $\lambda=\lambda_{0}+\lambda_{1}$ with $\lambda_{0}$ given by Proposition \ref{prop-lambda0} and let us assume \eqref{ExtraAssumptionSECMASSSECMOM}. Let $0<\sigma<1$, $0<\epsilon<2$, $1<\nu<\min(1+\frac{\varepsilon}{2},\frac{3}{2}-\frac{\sigma}{2})$ and $m\in \mathbb{R}$. There exists a number $C>0$ such for $\tau_{0}$ sufficiently large and for all $h=h(y,\tau)$ such that $[h]_{\text{rad}}=0$ with $\|h\|_{\nu,m,5+\sigma,\epsilon}<\infty$ and
	\begin{align*}
		\int_{\mathbb{R}^{2}}h(y,\tau)y_{j}dy=0,  \ \ j=1,2\ \ \ \text{for all }(\tau_{0},\infty),
	\end{align*}
	there exists an operator $\mathcal{F}[\phi]$ and a solution $\phi(y,\tau)=\mathcal{T}_{p}^{i,3}[h]$ of 
	\begin{align*}
		\begin{cases}
			\partial_{\tau}\phi = L[\phi]+B[\phi \chi]+\mathcal{F}[\phi]+h \\
			\phi(\cdot, \tau_{0})=0
		\end{cases}
	\end{align*}
	that defines a linear operator of $h$ and satisfies the estimate
	\begin{align*}
		\|\phi\|_{1;\nu,m,3+\sigma,2+\sigma+\epsilon}\le C\|h\|_{0;\nu,m,5+\sigma,\epsilon}.
	\end{align*}
    Moreover $\mathcal{F}[\phi]$ does not have radial mode and 
    \begin{align*}
    	|\mathcal{F}[\phi]|\le \frac{1}{\tau^{\nu+1}(\ln \tau)^{m-1}}\frac{1}{M^{\sigma}}(|W_{1}(y)|+|W_{2}(y)|).
    \end{align*}
\end{proposition}
The proof of Proposition \ref{PropMode1} will be given in Section \ref{Section12}.

\section[]{Preliminaries for the linear theory}\label{PreliminariesInnTh}
In this section we will report some important results proved in Section 9 of \cite{DdPDMW}. In order to obtain good estimates for the solution it is necessary to study the quadratic form
\begin{align}\label{quadraticform}
	\phi \mapsto \int_{\R^2}g\phi, \ \ \ \ g=\frac{\phi}{U}-(-\Delta)^{-1}\phi.
\end{align}
Indeed, we see that $L[\phi]=\nabla \cdot(U\nabla g)$ and then it is natural to test equation \eqref{SimplifiedInnerEq} against $g$ and obtain
\begin{align*}
	\int_{\R^2}L[\phi]g=\int_{\R^2}\nabla \cdot(U\nabla g)g=-\int_{\R^2}U|\nabla g|^{2}.
\end{align*}
From the time derivative instead we get $\int_{\R^2}(\partial_{ \tau}\phi)g $.\newline
Let us take a function $\phi:\mathbb{R}^{2}\mapsto \mathbb{R}$ such that for some $\sigma>0$
\begin{align}\label{CondDecMass}
	|\phi(y)|\le \frac{1}{(1+|y|)^{2+\sigma}}, \ \ \ \ \int_{\R^2}\phi dy=0.
\end{align}
Recalling \eqref{quadraticform} it is natural to introduce the notation
\begin{align*}
	\psi=(-\Delta)^{-1}\phi
\end{align*}
so that
\begin{align*}
	-\Delta \psi-U\psi=Ug \ \ \ \ \text{in }\mathbb{R}^{2}.
\end{align*}
The operator $\Delta\psi+U(y)\psi$ is classical. It corresponds to linearizing the Liouville equation
\begin{align*}
	\Delta v+e^{v}=0 \ \ \ \text{in }\mathbb{R}^{2},
\end{align*}
around the solution $\Gamma_{0}=\log U$. The bounded kernel of this linearization is spanned by generators of rigid motions, namely dilation and translations of the equations, which are precisely the functions $z_{0}$, $z_{1}$, $z_{2}$ defined by
\begin{align}\label{kernel Liouville}
	\begin{cases}
		z_{0}(y)=\nabla \Gamma_{0}(y)\cdot y+2 \\
		z_{j}(y)=\partial_{y_{j}}\Gamma_{0}(y), \ \ j=1,2.
	\end{cases}
\end{align}
We introduce a normalized version of $g$, that we call $g^{\perp}$, defined by 
\begin{align*}
	g^{\perp}=g+a,
\end{align*}
where $a\in \mathbb{R}$ is chosen so that 
\begin{align}\label{OrtLiouv}
\int_{\R^2}g^{\perp}Udy=0.
\end{align}
It will be convenient to work with functions $\phi^{\perp}$, $\psi^{\perp}$ associated to $g^{\perp}$. Indeed, thanks to \eqref{OrtLiouv}, we can take $\psi^{\perp}$ such that
\begin{align}\label{EqPsiPerp}
	-\Delta \psi^{\perp}-U\psi^{\perp}=Ug^{\perp}, \ \ \ \psi^{\perp}(y)\to0 \ \ \ \text{as }|y|\to \infty.
\end{align}
We can introduce the function $\psi_{0}=1+\frac{1}{2}z_{0}$, where $z_{0}$ is defined in \eqref{kernel Liouville}. We see that
\begin{align*}
  -\Delta \psi_{0}-U\psi_{0}=-U, \ \ \ \psi_{0}\to 0 \text{ as }|y|\to \infty.
\end{align*}
Then we can write
\begin{align*}
	\psi^{\perp}=\psi-a\big(1+\frac{z_{0}}{2}\big)=\psi-a\psi_{0}.
\end{align*}
Now if we define 
\begin{align*}
	\phi^{\perp}=U(g^{\perp}+\psi^{\perp}),
\end{align*}
we get
\begin{align*}
	\phi=\phi^{\perp}+\frac{a}{2}Uz_{0}, \ \ \ -\Delta\psi^{\perp}=\phi^{\perp}, \ \ \ \int_{\R^2}\phi^{\perp}=0.
\end{align*}
We notice that $\phi-\phi^{\perp}=\frac{a}{2}Z_{0}$ is in the kernel of the operator $L$.\newline Now we can introduce four lemmas whose proof can be found in Section 9 of \cite{DdPDMW}. 
\begin{lemma}\label{Lemma93}
	If $\phi:\mathbb{R}^{2}\mapsto \mathbb{R}$ satisfies \eqref{CondDecMass}, then there exist positive constants $c_{1}$, $c_{2}$ such that
	\begin{align}\label{QuadrForIne}
		c_{1}\int_{\R^2}U(g^{\perp})^{2}\le \int_{\R^2}\phi g^{\perp}\le c_{2}\int_{\R^2}U(g^{\perp})^{2}.
	\end{align}
\end{lemma}
\begin{lemma}\label{Lemma95}
	Let $\phi:\mathbb{R}^{2}\mapsto \mathbb{R}$ be radially symmetric and satisfy \eqref{CondDecMass}, then there exists positive constants $c$, $c_{1}$, $c_{2}$ such that
	\begin{align}
		c_{1}\int_{\R^2}U(g^{\perp})^{2}\le \int_{\R^2}U^{-1}(\phi^{\perp})^{2}\le c_{2}\int_{\R^2}U(g^{\perp})^{2},
	\end{align}
\begin{align}
	\int_{\R^2}U(\psi^{\perp})^{2}\le c\int_{\R^2}U(g^{\perp})^{2}.
\end{align}
\end{lemma}
\begin{lemma}
	If $\phi:\mathbb{R}^{2}\mapsto \mathbb{R}$ satisfies \eqref{CondDecMass} and is differentiable with respect to $\tau$ and $\phi_{\tau}$ satisfies 
	\begin{align*}
		|\phi_{\tau}(y,\tau)|\le \frac{1}{(1+|y|)^{2+\sigma}}.
	\end{align*}
 Then
 \begin{align*}
 	\int_{\R^2}\phi_{\tau} g^{\perp}=\frac{1}{2}\partial_{\tau}\int_{\R^2}\phi g^{\perp}.
 \end{align*}
\end{lemma}
\begin{lemma}\label{Lemma106}
	Let $B_{R}(0)\subset \mathbb{R}^{2}$ be the open ball centered at $0$ of radius $R$. There exists $C>0$ such that, for any $R>0$ large and any $g\in H^{1}(B_{R})$ with 
	\begin{align*}
		\int_{B_{R}}gU=0
	\end{align*}
then we have
\begin{align*}
	\frac{C}{R^{2}}\int_{B_{R}}g^{2}U\le \int_{B_{R}}|\nabla g|^{2}U.
\end{align*}
\end{lemma}

\section{Linear theory: a decomposition}\label{ProofProp71}
In this section we want prove Proposition \ref{PropMode0Mass0}. We consider
\begin{align}\label{typicalCaseInnTAUICSEC10}
	\begin{cases}
		\partial_{\tau}\phi=L[\phi]+\tilde{B}[\phi]+h(y,t) \ \ \ \ \text{in }\mathbb{R}^{2}\times (\tau_{0},\infty)\\
		\phi(\cdot,\tau_{0})=c_{1}\tilde{Z}_{0} \ \ \text{in }\mathbb{R}^{2}
	\end{cases}
\end{align}
where
\begin{align*}
	&L[\phi]=\nabla \cdot \big[U\nabla\big(\frac{\phi}{U}-(-\Delta)^{-1}\phi\big)\big], \ \ \ \ (-\Delta)^{-1}\phi(y,t)=\frac{1}{2\pi}\int_{\R^2}\ln\frac{1}{|y-z|}\phi(z,t)dz
\end{align*} 
and
\begin{align}\label{tildeB}
	&\tilde{B}[\phi]=\lambda\dot{\lambda}(k\phi\chi+y\cdot\nabla(\phi\chi))-\lambda\dot{\lambda}\big[\int_{\R^2}(k\phi\chi+y\cdot \nabla(\phi\chi))dy\big]W_{0}(y), \ \ \ \chi=\chi_{0}(\frac{|y|\lambda}{\sqrt{\delta(T-t)}}).
\end{align}
We remark that in this section we will explain also the role of the parameter $\delta$ in the cut-off. \newline We assume
\begin{align}\label{MassRHS}
     \int_{\R^2}h(y,\tau)dy=0 \ \ \ \ \text{for all }\tau>\tau_{0}.
\end{align}
Since by \eqref{InitialCond} we know that
\begin{align}\label{MassInCond}
	\int_{\R^2}\phi(\cdot, \tau_{0})dy=0
\end{align}
from the equation we have
\begin{align*}
	\int_{\R^2}\phi(y,\tau)dy=0 \ \ \ \ \text{for all }\tau>\tau_{0}.
\end{align*}
We briefly recall the decomposition of $\phi$ introduced in Section \ref{PreliminariesInnTh}. Give $\phi:\mathbb{R}^{2}\to \mathbb{R}$ with sufficient decay and mass zero, we let $g=\frac{\phi}{U}-(-\Delta)^{-1}\phi$, and define $a$ so that $\int_{\R^2}(g+a)Udy=0$. Then define $g^{\perp}=g+a$, $\psi^{\perp}=\psi-a(1+\frac{z_{0}}{2})$, and
\begin{align}\label{DecomPositionPhiPhiperp}
	\phi^{\perp}=\phi-\frac{a}{2}Z_{0}.
\end{align}
We notice that $a$ can be directly computed, indeed since $Ug^{\perp}=Ug+aU$, $\int_{\R^2}Ug^{\perp}=0$ and $Ug=-U(-\Delta)^{1}\phi+\phi$ we have
\begin{align}\label{ComputingA}
	a=-\frac{1}{8\pi}\int_{\R^2}Ug=\frac{1}{8\pi}\int_{\R^2}U(-\Delta)^{-1}\phi=\frac{1}{8\pi}\int_{\R^2}\Gamma_{0}\phi.
\end{align}
Analogously to \cite{DdPDMW}, to describe the solution $\phi$ of \eqref{typicalCaseInnTAUICSEC10} it will be useful to consider the problem
\begin{align}\label{EqZB}
	\begin{cases}
		\partial_{ \tau}Z_{B}=L[Z_{B}]+\tilde{B}[Z_{B}], \ \ \ \text{in }\mathbb{R}^{2}\times(\tau_{0},\infty),\\
		Z_{B}(\cdot,\tau_{0})=\tilde{Z}_{0} \ \ \ \text{in }\mathbb{R}^{2}
	\end{cases}
\end{align}
where $\tilde{Z}_{0}$ is function we introduced in \eqref{InitialCond}. In the following results we will more simply denote $\|h\|_{\star\star}$ the norm \eqref{RHSinnernorm}.\newline
We finally observe that to prove Proposition \ref{PropMode0Mass0} in this section we will also prove the following result that will be fundamental in the the next sections since it will give us the necessary control of the decomposition \eqref{DecomPositionPhiPhiperp}.
\begin{proposition}\label{Prop101}
	Let $\lambda=\lambda_{0}+\lambda_{1}$ with $\lambda_{0}$ given by Proposition \ref{prop-lambda0} and let us assume \eqref{ExtraAssumptionSECMASSSECMOM}. Let $\sigma\in(0,1)$, $\epsilon>0$ with $\sigma+\epsilon<2$ and $1<\nu<\frac{7}{4}$. Then there is $C>0$ such that for any $\tau_{0}$ sufficiently large the following holds. Suppose that $\|h\|_{\star\star}<\infty$ is radially symmetric and satisfies the zero mass condition \eqref{MassRHS}. Then there exists $c_{1}$ such that the solution $\phi=\phi^{\perp}+\frac{a}{2}Z_{0}$ of \eqref{typicalCaseInnTAUICSEC10} satisfies
	\begin{align*}
	   |a(\tau)|\le C \frac{f(\tau)R^{2}(\tau)}{(\ln \tau_{0})^{1-q}}\|h\|_{\star\star}, \ \ \ \ \ \
		|\phi^{\perp}|\le \|h\|_{\star\star}f(\tau)R(\tau) \frac{1}{1+|y|^{2}},
    \end{align*}
where $R(\tau)>0$, $f(\tau)>0$ are defined by 
\begin{align}\label{DefintionfR2}
	R^{2}(\tau)=\frac{\tau}{(\ln \tau)^{1+q}}, \ q\in(0,1) \ \ \ f(\tau)=\frac{1}{\tau^{\nu}(\ln \tau)^{m}}.
\end{align}
Moreover $c_{1}$ is a linear operator of $h$ and satisfies
\begin{align*}
	|c_{1}|\le C \frac{f(\tau_{0})R^{2}(\tau_{0})}{\ln \tau_{0}^{1-q}}\|h\|_{\star\star}.
\end{align*}
If $1<\nu<\frac{3}{2}$ we additionally have
\begin{align*}
		|\phi^{\perp}|\le \|h\|_{\star\star}f(\tau)R(\tau) \frac{\min(1,\frac{\tau}{|y|^{2}})}{1+|y|^{2}}.
\end{align*}
\end{proposition}
To prove Proposition \ref{PropMode0Mass0} and Proposition \ref{Prop101} we will need the following lemmas where we will use the notation
\begin{align}\label{omegatau}
	\omega(\tau)=\big(\int_{\R^2\backslash B_{R(\tau)}(0)}Ug^{2}(\tau)\big)^{1/2}.
\end{align}
We will prove the following results only at the end of this Section.
\begin{lemma}\label{Lemma101}
	There is $C$ such that for $\tau_{0}$ large the following holds. Suppose that $\|h\|_{\star\star}<\infty$ is radially symmetric and satisfies the zero mass condition \eqref{MassRHS} and consider \eqref{typicalCaseInnTAUICSEC10}. Let $\phi^{\perp}$ and $a$ be the decomposition \eqref{DecomPositionPhiPhiperp}. Suppose that for some $c_{1}\in \mathbb{R}$ there is $\tau_{1}>\tau_{0}$ such that 
	\begin{align*}
		a(\tau_{1})=0.
	\end{align*}
    Then 
    \begin{align}
    	&|a(\tau)|\le C \frac{f(\tau)R^{2}(\tau)}{(\ln \tau_{0})^{1-q}}\|h\|_{\star\star}, \ \ \ \tau\in[\tau_{0},\tau_{1}]\\
    	&|\omega(\tau)|\le C \frac{f(\tau)R(\tau)}{(\ln \tau_{0})^{1-q}}\|h\|_{\star\star}, \ \ \ \ \tau\in[\tau_{0},\tau_{1}]\\
    	&|c_{1}|\le C \frac{f(\tau_{0})R^{2}(\tau_{0})}{(\ln \tau_{0})^{1-q}}\|h\|_{\star\star}
    \end{align}
    where the constant $C$ is independent of $\tau_{1}$ and $c_{1}$.
\end{lemma}
\begin{lemma}\label{Lemma102}
	There is $C$ such that for $\tau_{0}$ large the following holds. Suppose that $\|h\|_{\star\star}<\infty$ is radially symmetric and satisfies the zero mass condition \eqref{MassRHS} and consider \eqref{typicalCaseInnTAUICSEC10}. Let $\phi^{\perp}$ and $a$ be the decomposition \eqref{DecomPositionPhiPhiperp}. Suppose that for some $c_{1}\in \mathbb{R}$, 
	\begin{align*}
		\frac{a}{f R^{2}}\in L^{\infty}(\tau_{0},\infty).
	\end{align*}
     Then 
     \begin{align}
     	&|a(\tau)|\le C \frac{f(\tau)R^{2}(\tau)}{(\ln \tau_{0})^{1-q}}\|h\|_{\star\star}, \ \ \ \tau>\tau_{0}\label{ainf}\\
     	&|\omega(\tau)|\le C \frac{f(\tau)R(\tau)}{(\ln \tau_{0})^{1-q}}\|h\|_{\star\star}, \ \ \ \ \tau>\tau_{0}\label{omegainf}\\
     	&|c_{1}|\le C \frac{f(\tau_{0})R^{2}(\tau_{0})}{(\ln \tau_{0})^{1-q}}\|h\|_{\star\star}\label{c1inf}.
     \end{align}
\end{lemma}
\begin{lemma}\label{Lemma103}
	Let $Z_{B}$ be the solution to \eqref{EqZB} and write it as $Z_{B}=Z_{B}^{\perp}+\frac{a_{Z}}{2}Z_{0}$ according to the decomposition \eqref{DecomPositionPhiPhiperp}. Then $a_{Z}(\tau)\neq0$ for all $\tau\ge \tau_{0}$.
\end{lemma}
\begin{lemma}\label{Lemma104}
	There is $C$ such that for $\tau_{0}$ large the following holds. Suppose that $\|h\|_{\star\star}<\infty$ is radially symmetric and satisfies the zero mass condition \eqref{MassRHS}. Then there is a unique $c_{1}\in \mathbb{R}$ such that the solution $\phi=\phi^{\perp}+\frac{a}{2}Z_{0}$ of \eqref{typicalCaseInnTAUICSEC10} (as in \eqref{DecomPositionPhiPhiperp}) satisfies \eqref{ainf},\eqref{omegainf} and \eqref{c1inf}.
\end{lemma}	
To give the proofs of Lemmas \ref{Lemma101}, \ref{Lemma102}, \ref{Lemma103} and \ref{Lemma104} we will need a series of intermediate results.
In the following we will use the notation \eqref{FunctionM}:
\begin{align*}
	\chi=\chi_{0}(\frac{|y|\lambda}{\sqrt{\delta(T-t)}})=\chi_{0}(\frac{|y|}{M(\tau)}).
\end{align*}
We observe that thanks to \eqref{lambdadotlambdaExp} we can easily see that $M^{2}(\tau)\approx c \frac{\delta \tau}{\ln \tau}$ for some universal constant $c>0$ .
\begin{lemma}\label{Lemma105}
	Let $R^{2}(\tau)=\frac{\tau}{(\ln \tau)^{1+q}}$ with $q\in (0,1)$ and $M(\tau)$ be like \eqref{FunctionM}, we have
	\begin{align*}
		\left| \int_{\mathbb{R}^{2}}\tilde{B}[\phi]g^{\perp} \right|\le& C\big[\frac{\ln \tau}{\tau}\int_{\R^2} U (g^{\perp})^{2}+|a(\tau)|\frac{\ln \tau}{\tau}\|\nabla g^{\perp} U^{1/2}\|_{L^{2}}+a^{2}(\tau)\frac{\ln \tau}{\tau}\frac{1}{M^{2}(\tau)}+\frac{a^{2}(\tau)}{R^{4}}+\\
		&+\frac{\ln \tau}{\tau}\int _{B_{M}^{c}}|\phi g| +\frac{\ln \tau}{\tau} \int_{B_{M}^{c}} Ug^{2}+(R^{2}\frac{\ln \tau}{\tau})^{2}(\int_{B_{M}^{c}}|\phi|)^{2}\big].
	\end{align*}
\end{lemma}
\begin{proof}
	It turns out to be useful to decompose the operator $\tilde{B}$ in \eqref{tildeB} as $\tilde{B}[\phi]=\tilde{B}_{0}[\phi]-W_{0}(y)\lambda\dot{\lambda}\int (k\phi \chi+y\cdot{\nabla}(\phi \chi))dy$. First we observe
	\begin{align}\label{B0gperp}
		\int_{\mathbb{R}^{2}} \tilde{B}_{0}[\phi]g^{\perp}=\lambda\dot{\lambda}\big[k\int_{\mathbb{R}^{2}} \phi g^{\perp} -k \int_{\mathbb{R}^{2}}\phi g^{\perp}(1-\chi) +\int_{\mathbb{R}^{2}} y\cdot \nabla \phi g^{\perp} \chi+ \int_{\mathbb{R}^{2}} y\cdot \nabla \chi \phi g^{\perp} \big].
	\end{align}
	By \eqref{lambdadotlambdaExp} we see
	\begin{align*}
		|-k\int_{\mathbb{R}^{2}} \phi g^{\perp} (1-\chi)+\int_{\mathbb{R}^{2}} y\cdot \nabla \chi \phi g^{\perp}| &\le  C\int_{\mathbb{R}^{2}} |\phi g^{\perp}|(1-\chi) \le |a(\tau)|\int_{\mathbb{R}^{2}} |\phi|(1-\chi)+ \int_{\mathbb{R}^{2}} |\phi g|(1-\chi)\le \\
		&\le |a(\tau)|\int_{B_{M}^{c}}|\phi|+\int_{B_{M}^{c}}|\phi g|\\
		\implies |\lambda\dot{\lambda}[-k\int_{\mathbb{R}^{2}} \phi g^{\perp} (1-\chi)+&\int_{\mathbb{R}^{2}} y\cdot \nabla \chi \phi g^{\perp}]|\le C\frac{a^{2}}{R^{4}}+(R^{2}\frac{\ln \tau}{\tau})^{2}(\int_{B_{M}^{c}}|\phi| )^{2}+C\frac{\ln \tau}{\tau} \int_{B_{M}^{c}}|\phi g|.
	\end{align*}
	By Lemma \ref{Lemma93} we also have $|k\lambda \dot{\lambda}\int_{\mathbb{R}^{2}} \phi g^{\perp}|\le C \frac{\ln \tau}{\tau} \int_{\mathbb{R}^{2}} U (g^{\perp})^ {2}$. Consequently the only term in \eqref{B0gperp} we still need to estimate is $\lambda\dot{\lambda} \int_{\mathbb{R}^{2}} y\cdot \nabla \phi g^{\perp} \chi $. First we use \eqref{DecomPositionPhiPhiperp} and we write
	\begin{align}\label{MissinTermLemma105}
		\lambda \dot{\lambda}\int_{\mathbb{R}^{2}} y\cdot \nabla \phi g^{\perp} \chi =\lambda\dot{\lambda}[\int_{\mathbb{R}^{2}} y\cdot \nabla \phi^{\perp} g^{\perp} \chi + \frac{a}{2} \int_{\mathbb{R}^{2}} y \cdot \nabla Z_{0}g^{\perp}\chi].
	\end{align}
	By observing that $y\cdot \nabla Z_{0}=\nabla \cdot(yZ_{0}-2\nabla z_{0})-4Z_{0}$ and recalling that $\int_{\R^2}Z_{0}g^{\perp}=0$, proceeding similarly to the proof of Lemma 10.5 in \cite{DdPDMW}, we see
	\begin{align*}
		|\lambda \dot{\lambda} \frac{a}{2} \int_{\mathbb{R}^{2}} y \cdot \nabla Z_{0} g^{\perp} \chi|&= |\lambda\dot{\lambda}\frac{a}{2} \int_{\mathbb{R}^{2}} y \cdot \nabla Z_{0} g^{\perp}-\lambda\dot{\lambda} \frac{a}{2} \int_{\mathbb{R}^{2}} y \cdot \nabla Z_{0} g^{\perp}(1-\chi)|\le \\
		&\le C|a|\frac{\ln \tau}{\tau} \|\nabla g^{\perp} U^{1/2}\|_{L^{2}}+C|a|\frac{\ln \tau}{\tau} \int_{B_{M}^{c}}U|g^{\perp}|\le \\
		&\le C|a|\frac{\ln \tau}{\tau}\|\nabla g^{\perp} U^{1/2}\|_{L^{2}}+Ca^{2}\frac{\ln \tau}{\tau} \frac{1}{M^{2}}+C\frac{\ln \tau}{\tau} \int_{B_{M}^{c}} U g^{2}.
	\end{align*}
	By the decomposition \eqref{EqPsiPerp}, we can rewrite the first term in \eqref{MissinTermLemma105} as
	\begin{align}\label{Lemma105Exp}
		\lambda \dot{\lambda} \int_{\mathbb{R}^{2}} y \cdot \nabla \phi^{\perp}g^{\perp} \chi = \lambda\dot{\lambda} \int_{\mathbb{R}^{2}} y \cdot \nabla (Ug^{\perp})g^{\perp} \chi +\lambda\dot{\lambda} \int_{\mathbb{R}^{2}} y \cdot \nabla (U \psi^{\perp})g^{\perp}\chi.
	\end{align}
	To estimate the first term in \eqref{Lemma105Exp} we can simply observe
	\begin{align*}
		\int_{\mathbb{R}^{2}} y \cdot \nabla (Ug^{\perp})g^{\perp}\chi =& \int_{\mathbb{R}^{2}} y \cdot \nabla U (g^{\perp})^{2}\chi + \int_{\mathbb{R}^{2}} U y \cdot \nabla g^{\perp} g^{\perp}= \int_{\mathbb{R}^{2}} y \cdot \nabla U (g^{\perp})^{2} \chi + \frac{1}{2} \int_{\mathbb{R}^{2}} U y\cdot \nabla (g^{\perp})^{2}\chi=\\
		=&\frac{1}{2} \int_{\mathbb{R}^{2}} y \cdot \nabla U ( g^{\perp})^{2} \chi -\int_{\mathbb{R}^{2}} U(g^{\perp})^{2}\chi -\frac{1}{2}\int_{\mathbb{R}^{2}} U (g^{\perp})^{2}y\cdot \nabla \chi \le C \int_{\mathbb{R}^{2}} U(g^{\perp})^{2}
	\end{align*}
	and for the second term of \eqref{Lemma105Exp} we see
	\begin{align}\label{Lemma105Exp2}
		\int_{\mathbb{R}^{2}} y \cdot \nabla (U\psi ^{\perp})g^{\perp} \chi=\int_{\mathbb{R}^{2}} y \cdot \nabla U \psi ^{\perp} g^{\perp} \chi + \int_{\mathbb{R}^{2}} U y \cdot \nabla \psi ^{\perp} g^{\perp}\chi.
	\end{align}
	To estimate these terms first we observe that by radial symmetry we can write for istance
	\begin{align*}
		\int_{\mathbb{R}^{2}} U y \cdot \nabla \psi ^{\perp} g^{\perp}\chi = 2\pi \int_{0}^{\infty}U(\rho)(\psi^{\perp})'(\rho)g^{\perp}(\rho)\rho^{2}d\rho.
	\end{align*}
    Now, recalling that $\psi^{\perp}$ solves $-\Delta \psi^{\perp}-U\psi^{\perp}=Ug^{\perp}$ in $\mathbb{R}^{2}$ and $\psi^{\perp}(\rho,\tau)\to0$ as $\rho\to\infty$, by variation of parameters formula, since $\int_{\R^2}Ug^{\perp}z_{0}dy=0$ we get
	\begin{align}
		&\psi ^{\perp}(\rho)= z_{0}(\rho) \int_{\rho}^{\infty} U g^{\perp} \bar{z}_{0}rdr+\bar{z}_{0}(\rho) \int_{0}^{\rho} U g^{\perp} z_{0}rdr,\label{psiperpglob}\\
		&(\psi ^{\perp})'(\rho)= z_{0}'(\rho) \int_{\rho}^{\infty} U g^{\perp} \bar{z}_{0}rdr+\bar{z}_{0}'(\rho) \int_{0}^{\rho} U g^{\perp} z_{0}rdr \label{psiperpprimeglob}
	\end{align}
where $\bar{z}_{0}$ is a second linear independent function in the kernel of $\Delta+U$ satisfying $|\bar{z}_{0}(\rho)|\le C (|\ln \rho|+1)$. When $\rho >1$, we can also write $\psi^{\perp}$ as
	\begin{align}\label{psiperpinf}
		\psi ^{\perp}(\rho)= z_{0}(\rho)\int_{\rho}^{\infty} \frac{1}{z_{0}^{2}(r)r}\int_{r}^{\infty}Ug^{\perp}z_{0}sdsdr.
	\end{align}
Then if $0<\rho\le1$ by \eqref{psiperpglob} we have 
	\begin{align*}
		|\psi ^{\perp}|\le& C \int_{\rho}^{\infty} U |g^{\perp}| |\bar{z}_{0}|rdr + C (|\ln \rho|+1) \int_{0}^{\rho} U |g^{\perp}| rdr \le \\
		\le& C ( \int_{\rho}^{\infty} U(g^{\perp})^{2}rdr)^{1/2} (\int_{\rho}^{\infty} U |\bar{z}_{0}|^{2}rdr)^{1/2}+C (|\ln\rho|+1) (\int_{0}^{\rho} U (g^{\perp})^{2}rdr)^{1/2} (\int_{0}^{\rho} Urdr)^{1/2}\le\\
		\le& C \|U^{1/2}g^{\perp}\|_{L^{2}}
	\end{align*}
	If $\rho >1$ by \eqref{psiperpinf}
	\begin{align*}
		|\psi ^{\perp}| \le \int_{\rho}^{\infty} \frac{1}{r}\int_{r}^{\infty} U |g^{\perp}| s dsdr \le \| U^{1/2} g^{\perp}\|_{L^{2}} \int_{\rho}^{\infty} \frac{1}{r}(\int_{r}^{\infty} Us ds )^{1/2}dr\le \frac{\|U^{1/2}g^{\perp}\|_{L^{2}}}{1+\rho}.
	\end{align*}
	Proceeding analogously with $(\psi^{\perp})'$ we finally get
	\begin{align}\label{psiperp}
		|\psi ^{\perp}|\le C \frac{\|U^{1/2}g^{\perp}\|_{L^{2}}}{1+\rho}, \ \ \ \ |(\psi^{\perp})'| \le C \frac{\|U^{1/2} g^{\perp}\|_{L^{2}}}{1+\rho ^{2}}.
	\end{align}
	Recalling \eqref{Lemma105Exp2}, we get 
	\begin{align*}
		|\lambda\dot{\lambda} \int_{\mathbb{R}^{2}} y \cdot \nabla ( U\psi^{\perp}) g^{\perp}\chi |\le C \frac{\ln \tau}{\tau} \int_{\mathbb{R}^{2}} U (g^{\perp})^{2}.
	\end{align*}
	At this point we want to estimate $\int_{\mathbb{R}^{2}}(\tilde{B}[\phi]-\tilde{B}_{0}[\phi])g^{\perp}=-\lambda\dot{\lambda} \int _{\mathbb{R}^{2}}( k\phi \chi+y\cdot \nabla( \phi \chi))dy \int_{\mathbb{R}^{2}} W_{0}g^{\perp}$. We see
	\begin{align*}
		\int_{\mathbb{R}^{2}} (k\phi \chi+ y\cdot \nabla (\phi \chi))=(k-2)\int_{\mathbb{R}^{2}} \phi \chi \le |k-2|\int_{B_{M}^{c}}|\phi|
	\end{align*}
	and then
	\begin{align*}
		&\left| -\lambda\dot{\lambda} \int_{\mathbb{R}^{2}} ( k\phi \chi + y \cdot \nabla(\phi \chi))dy \int_{\mathbb{R}^{2}} W_{0}g^{\perp} \right|\le C \frac{\ln \tau}{\tau} \int_{B_{M}^{c}}|\phi| (\int_{\mathbb{R}^{2}} U (g^{\perp})^{2})^{1/2}\le \\
		&\le C \frac{\ln \tau}{\tau} \int_{\mathbb{R}^{2}} U (g^{\perp})^{2}+ C \frac{\ln \tau}{\tau} ( \int_{B_{M}^{c}}|\phi|)^{2}.
	\end{align*}
    Observing that $\frac{\ln \tau}{\tau}\le \big(R^{2}\frac{\ln \tau}{\tau}\big)^{2}$ we get the desired conclusion.
\end{proof}
\begin{lemma}\label{ineqphigperp}
	Let $R(\tau)^{2}=\frac{\tau}{(\ln \tau)^{1+q}}$ with $q\in (0,1)$, $M(\tau)$ as in \eqref{FunctionM} and $f$ given by \eqref{DefintionfR2}. Then there is $c>0$ and $C>0$ such that for $\tau_{0}$ sufficiently large
	\begin{align*}
		\partial \tau \int_{\mathbb{R}^{2}} \phi g^{\perp} +\frac{c}{R^{2}}\int_{\mathbb{R}^{2}} \phi g^{\perp} \le C\|h\|_{\star\star}^{2} f^{2}(\tau)+ C \frac{a^{2}}{R^{4}}+C\frac{\omega _{0}^{2}}{R^{2}}+C \frac{\omega_{1}^{2}}{R^{2}}
	\end{align*}
	where $\omega_{0}=(\int_{B_{R}^{c}} U g^{2} )^{1/2}$ and 
	\begin{align*}
		\omega_{1}=R[a^{2}\frac{\ln \tau}{\tau}\frac{1}{M^{2}}+\frac{\ln \tau}{\tau} \int_{B_{M}^{c}}|\phi g| + \frac{\ln \tau}{\tau} \int_{B_{M}^{c}} U g^{2}+(R^{2}\frac{\ln \tau}{\tau})^{2}(\int_{B_{M}^{c}}|\phi|)^{2}]^{1/2}.
	\end{align*}
\end{lemma}
\begin{proof}
	After multiplying \eqref{typicalCaseInnTAUICSEC10} by $g^{\perp}$, integrating and using Lemma \ref{Lemma95} we have
	\begin{align}\label{EqMultgP}
		\frac{1}{2} \partial \tau \int_{\mathbb{R}^{2}} \phi g^{\perp}+ \int_{\mathbb{R}^{2}} U |\nabla g^{\perp}|^{2}= \int_{\mathbb{R}^{2}} \tilde{B}[\phi]g^{\perp}+\int_{\mathbb{R}^{2}} h g^{\perp}.
	\end{align}
	Let $H=(-\Delta)^{-1}h$, since $h$ is radial and $\int_{\R^2}h=0$ we have
	\begin{align*}
		|\nabla H(\rho,\tau)|\le C \|h\|_{\star\star}\frac{f(\tau)}{(1+\rho)^{5+\sigma}} \implies |\int_{\mathbb{R}^{2}} h g^{\perp}| \le \frac{1}{2}\int_{\mathbb{R}^{2}} U |\nabla g^{\perp}|^{2}+ C \|h\|_{\star \star}^{2} f^{2}(\tau).
	\end{align*}
    Then from \eqref{EqMultgP} we can obtain
	\begin{align}\label{EqMultgPRHS}
		\frac{1}{2} \partial \tau \int_{\mathbb{R}^{2}} \phi g^{\perp}+ \frac{1}{2} \int_{\mathbb{R}^{2}} U |\nabla g^{\perp}|^{2}\le  |\int_{\mathbb{R}^{2}} \tilde{B}[\phi]g^{\perp}|+C\|h\|_{\star\star}^{2}f^{2}(\tau).
	\end{align}
    By using the inequality in Lemma \ref{Lemma106} we see that that for some $c>0$
    \begin{align*}
    	\frac{c}{R^{2}}\int_{B_{R}}(g^{\perp}-\bar{g}_{R}^{\perp})^{2}U\le \int_{\R^2}U|\nabla g^{\perp}|^{2}
    \end{align*}
    where $\bar{g}^{\perp}_{R}=\frac{1}{\int_{B_{R}}U}\int_{B_{R}}g^{\perp}U$. Some elementary computations that we omit but that can be found in the proof of Lemma 10.6 in \cite{DdPDMW} give that for some new $c>0$ we have
    \begin{align*}
    	\frac{c}{R^{2}}\int_{\R^2}(g^{\perp})^{2}U\le \int_{\R^2}U|\nabla g^{\perp}|^{2}+C \frac{a^{2}}{R^{4}}+C\frac{\omega_{0}^{2}}{R^{2}}, \ \ \ \omega_{0}=(\int_{B_{R}^{c}} U g^{2} )^{1/2}.
    \end{align*}
If we use this with \eqref{EqMultgPRHS} we get (for a new $c>0$) 
	\begin{align*}
		\frac{1}{2}\partial_{ \tau }\int_{\mathbb{R}^{2}} \phi g^{\perp}+ \frac{c}{R^{2}}\int_{\mathbb{R}^{2}} U (g^{\perp})^{2}+\frac{1}{4} \int_{\mathbb{R}^{2}} U |\nabla g^{\perp}|^{2}\le |\int_{\mathbb{R}^{2}} \tilde{B}[\phi] g^{\perp}| + C \|h\|_{\star\star}^{2}f^{2}(\tau)+ C \frac{a^{2}}{R^{4}}+ C \frac{\omega_{0}^{2}}{R^{2}}.
	\end{align*}
	By Lemma \ref{Lemma105} we have
	\begin{align*}
		|\int_{\mathbb{R}^{2}} \tilde{B}[\phi] g^{\perp}|\le C [\frac{\ln \tau}{\tau} \int_{\mathbb{R}^{2}} U (g^{\perp})^{2}+|a|\frac{\ln \tau}{\tau} \|\nabla g^{\perp} U^{1/2}\|_{L^{2}}+\frac{a^{2}}{R^{4}}+\frac{\omega_{1}^{2}}{R^{2}}]
	\end{align*}
	and $\omega_{1}^{2}=R^{2}[a^{2}\frac{\ln \tau}{\tau} \frac{1}{M^{2}}+\frac{\ln \tau}{\tau}\int_{B_{M}^{c}}|\phi g|+ \frac{\ln \tau}{\tau}\int_{B_{M}^{c}}Ug^{2}+(R^{2}\frac{\ln \tau}{\tau})^{2}(\int_{B_{M}^{c}}|\phi|)^{2}]$. Then, recalling that $R^{2}(\tau)=\frac{\tau}{(\ln \tau)^{1+q}}$, with $q\in(0,1)$, we get that if $\tau_{0}$ is sufficiently large
	\begin{align*}
		|\int_{\mathbb{R}^{2}} \tilde{B}[\phi]g^{\perp}| \le C [\frac{\ln \tau}{\tau} \int_{\mathbb{R}^{2}} U (g^{\perp})^{2}+\frac{a^{2}}{R^{4}}+R^{4}\frac{\ln^{2}\tau}{\tau^{2}}\|\nabla g^{\perp} U^{1/2}\|_{L^{2}}^{2}+\frac{\omega_{1}^{2}}{R^{2}}]
	\end{align*}
	implies
	\begin{align*}
		\partial_{ \tau} \int_{\mathbb{R}^{2}} \phi g^{\perp} + \frac{1}{R^{2}}\int_{\mathbb{R}^{2}} U (g^{\perp})^{2}\le C\|h\|_{\star\star}^{2}f^{2}(\tau)+C\frac{a^{2}}{R^{4}}+C \frac{\omega_{0}^{2}}{R^{2}}+ C \frac{\omega_{1}^{2}}{R^{2}}.
	\end{align*}
\end{proof}
\begin{lemma}\label{Lemma107}
	Let $\nu>0$ and $f(\tau)$, $R(\tau)$ like \eqref{DefintionfR2}. Assume that $\|g U^{1/2}\|_{L^{2}}\le K_{1} f(t)R^{2}(t)$, then if $\tau_{1}\ge \tau_{0}$ for any $y\in\mathbb{R}^{2}$ we have
	\begin{align}\label{PrelimG0Spat}
		|U(y)g(y,\tau)|\le C ( K_{1}+ \frac{\|h\|_{\star\star}}{R^{2}(\tau_{0})}+\frac{|c_{1}|}{f(\tau_{0})R^{2}(\tau_{0})})\frac{f(\tau)R^{2}(\tau)}{(1+|y|)^{2}}, \ \ \ \tau\in[\tau_{0},\tau_{1}].
	\end{align}
\end{lemma}
\begin{proof}
	In the following we will use the notation 
	\begin{align*}
		f_{1}(\tau)=f(\tau)R^{2}(\tau).
	\end{align*}
	We will also write $g_{0}=Ug=\phi-U(-\Delta)^{-1}\phi$. We have
	\begin{align}\label{Eqg0}
		\partial_{ \tau} g_{0}=&\nabla \cdot [U \nabla ( \frac{g_{0}}{U})]- U (-\Delta)^{-1}[\nabla \cdot ( U \nabla g)]+h-U(-\Delta)^{-1}h+\nonumber\\
		&+\tilde{B}[g_{0}]+\tilde{B}[U\psi [g_{0}]]-U(-\Delta)^{-1}(\tilde{B}[g_{0}+U\psi[g_{0}]]).
	\end{align}
     We start observing that $\psi$ has an analogous integral representation as $\psi^{\perp}$ in \eqref{psiperpglob}, \eqref{psiperpinf}. It is immediate then to get
    \begin{align}\label{psiest}
    	|\psi| \le C K_{1} f_{1}(\tau) \frac{1}{1+ \rho}, \ \ \ |\psi ' |\le C K_{1} f_{1}(\tau) \frac{1}{1+\rho^{2}}.
    \end{align}
	We can start considering the nonlocal component of $\tilde{B}[g_{0}]$. By the definition of $\tilde{B}$ we can write
	\begin{align}\label{DecomBtilde}
		\tilde{B}[\phi]= \lambda\dot{\lambda}(k\phi \chi + y\cdot \nabla(\phi \chi))-\lambda \dot{\lambda} W_{0}(y)\int_{\mathbb{R}^{2}} (k\phi \chi+y\cdot \nabla(\phi \chi)).
	\end{align}
	Then since $\int_{\mathbb{R}^{2}} k \phi \chi + \int_{\mathbb{R}^{2}} y \cdot \nabla(\phi \chi)= (k-2)\int_{\mathbb{R}^{2}} \phi \chi \le |k-2| \int_{B_{M}^{c}}|\phi|\le C \int_{B_{M}^{c}}(U|\psi|+|g_{0}|)$ and thanks to \eqref{psiest}  we have
	\begin{align*}
	  \tilde{B}[g_{0}]=&\lambda\dot{\lambda}(kg_{0}\chi+y\cdot \nabla(g_{0}\chi))+O(\frac{\ln \tau}{\tau}|W_{0}(y)| (\int_{B_{M}^{c}}(|g_{0}|+U|\psi|)) )=\\
		=&\lambda\dot{\lambda}(k g_{0}\chi+y\cdot \nabla(g_{0}\chi))+O(K_{1}\frac{\ln \tau}{\tau} \frac{f_{1}(\tau)}{M(\tau)}|W_{0}(y)|).
	\end{align*}
    We can also consider $\tilde{B}[U\psi]$:
	\begin{align*}
		\tilde{B}[U\psi]&= \lambda\dot{\lambda}(kU\psi\chi + y \cdot \nabla (U\psi \chi))+\lambda\dot{\lambda}(k-2)(\int_{\mathbb{R}^{2}} U\psi \chi ) W_{0}(y)=\\
		&=\lambda\dot{\lambda}(kU\psi \chi + U \psi y\cdot \nabla\chi + \chi y\cdot \nabla (U\psi)) +\lambda\dot{\lambda} (k-2)(\int_{\mathbb{R}^{2}} U \psi \chi)W_{0}(y)
	\end{align*}
	and then thanks to \eqref{psiest} we have
	\begin{align*}
		|\tilde{B}[U\psi]|\le C \frac{\ln \tau}{\tau} \frac{K_{1}f_{1}(\tau)}{1+|y|^{5}}+ C \frac{\ln \tau}{\tau} K_{1}f_{1}(\tau) |W_{0}(y)|.
	\end{align*}
	The last term involving the operator $\tilde{B}$ in \eqref{Eqg0} is $U(-\Delta)^{-1} ( \tilde{B}[g_{0}+U\psi[g_{0}]])$.
	Since $\int_{\mathbb{R}^{2}} \tilde{B}[g_{0}+U\psi[g_{0}]]=0$ and all the functions involved are radially symmetric we know that 
	\begin{align}\label{InversofBtilde}
		(-\Delta)^{-1}(\tilde{B}[g_{0}+U\psi[g_{0}]])=-\int_{\rho}^{\infty}\frac{1}{r}\int_{r}^{\infty} \tilde{B}[g_{0}+U\psi[g_{0}]](s,\tau)sdsdr.
	\end{align}
	In order to estimate \eqref{InversofBtilde}, first we recall \eqref{DecomBtilde} and we observe
	\begin{align*}
		\int_{\mathbb{R}^{2}} (k\phi\chi+y\cdot(\phi \chi))dy=&(k-2)\int_{\mathbb{R}^{2}} \phi \chi =-(k-2)\int_{\mathbb{R}^{2}} \phi(1-\chi)\le \\
		\le& C\big(\int_{B_{M}^{c}}|g_{0}|+\int_{B_{M}^{c}}U|\psi|\big)\le
		\le C K_{1}\frac{f_{1}(\tau)}{M(\tau)}.
	\end{align*}
Then we have
		\begin{align*}
			|\lambda\dot{\lambda}\int_{\rho}^{\infty} \frac{1}{r}\int_{r}^{\infty} [\int_{\mathbb{R}^{2}} (k\phi \chi+y\cdot \nabla(\phi\chi))dy]W_{0}(s)sdsdr|\le C K_{1}\frac{\ln \tau}{\tau} \frac{f_{1}(\tau)}{M(\tau)}\frac{1}{1+|y|^{5}}
		\end{align*}
		(actually one can get any spatial decay since $W_{0}(y)$ is a compactly supported function). Now we want to estimate the remaining term in \eqref{InversofBtilde}. We see
	\begin{align*}
		&|-\lambda\dot{\lambda} \int_{\rho}^{\infty} \frac{1}{r}\int_{r}^{\infty} (k\phi\chi+s(\phi\chi)')sdsdr|\le|(k-2)\lambda\dot{\lambda}\int_{\rho}^{2M}\frac{1}{r}\int_{r}^{\infty}\phi \chi s dsdr|+ |\lambda\dot{\lambda}\int_{\rho}^{2M}r\phi \chi dr|.
	\end{align*}
	We recall that $\phi=Ug+U\psi$ and we see that
	\begin{align*}
		&|-\lambda\dot{\lambda} \int_{\rho}^{\infty} \frac{1}{r}\int_{r}^{\infty} (kUg\chi+s(Ug\chi)')sdsdr|\le C \frac{\ln \tau}{\tau} \int_{\rho}^{2M}\frac{1}{r}\int_{r}^{\infty}U|g|\chi sdsdr+ C \frac{\ln \tau}{\tau} \int_{\rho}^{2M} U|g|\chi rdrds\le \\
		&\le C \frac{\ln \tau}{\tau}K_{1}f_{1}(\tau)\int_{\rho}^{2M} \frac{1}{r}(\int_{r}^{\infty}Usds)^{1/2}dr+C\frac{\ln \tau}{\tau} K_{1}f_{1}(\tau)(\int_{\rho}^{2M}Urdr)^{1/2}\le C \frac{\ln \tau}{\tau} K_{1}f_{1}(\tau)\frac{1}{1+\rho}.
	\end{align*}
	When we consider $\psi$, because of \eqref{psiest} we have
	\begin{align*}
		|-\lambda\dot{\lambda} \int_{\rho}^{\infty} \frac{1}{r}\int_{r}^{\infty} (kU\psi\chi+s(U\psi\chi)')sdsdr|\le C \frac{\ln \tau}{\tau} K_{1}f_{1}(\tau)\frac{1}{1+|y|^{3}}.
	\end{align*}
	We just proved that
	\begin{align*}
		|U(-\Delta)^{-1}(\tilde{B}[g_{0}+U\psi[g_{0}]])|\le C \frac{\ln \tau}{\tau} K_{1}f_{1}(\tau)\frac{1}{1+|y|^{5}}.
	\end{align*}
	Now, similarly to the proof of Lemma 10.7 in \cite{DdPDMW}, we get
	\begin{align}\label{FinalEqg0Est}
		\partial_{ \tau}g_{0}=&\nabla \cdot [U\nabla(\frac{g_{0}}{U})]+\lambda\dot{\lambda}(kg_{0}\chi+y\cdot \nabla(g_{0}\chi))+\nonumber\\
		&+h-U(-\Delta)^ {-1}h-U(-\Delta)^{-1}(\nabla \cdot ( U\nabla g))+O(\frac{\ln \tau}{\tau}K_{1}f_{1}(\tau)\frac{1}{1+|y|^{5}}) .
	\end{align}
	and then we can proceed as in \cite{DdPDMW} (estimating identically the remaining terms and using standard parabolic estimates) to get the desired result.
\end{proof}
The next lemma will improve the spatial decay in \eqref{PrelimG0Spat}. Before we need to make an elementary observation: let us consider the equation
\begin{align*}
  \begin{cases} \partial_{\tau}H=\Delta_{\mathbb{R}^{6}}H+h \ \ \ \text{in }(\tau_{0},\infty)\times \mathbb{R}^{6}\\
   H(\tau_{0},\cdot)=0
   \end{cases}
\end{align*}
where $\Delta_{\mathbb{R}^{6}}$ is the Laplace operator in $\mathbb{R}^{6}$ and let us assume that 
\begin{align*}
	|h(y,\tau)|\le \frac{1}{\tau^{1+\gamma}\ln^{\mu} \tau}\frac{1}{(1+\frac{|y|}{\sqrt{\tau}})^{b}}
\end{align*}
for some $\gamma,b\in \mathbb{R}$, $\mu\in \mathbb{R}$. If $\gamma<3$, $\gamma<\frac{b}{2}$ and for any $\mu \in \mathbb{R}$ there is a barrier satisfying
\begin{align}\label{Hbarrier}
	C_{1}\frac{1}{\tau^{\gamma}}\frac{1}{\ln ^{\mu}\tau}\frac{1}{(1+\frac{|y|}{\sqrt{\tau}})^{b}}\le H(y,\tau)\le C_{2}\frac{1}{\tau^{\gamma}}\frac{1}{\ln^{\mu}\tau}\frac{1}{(1+\frac{|y|}{\sqrt{\tau}})^{b}}
\end{align}
for some positive constants $C_{1}$, $C_{2}$.
The interested reader can find a proof in \cite{DdPDMW} where they used this type of barriers to prove Lemma 10.8. \newline
We notice that in the following Lemma we will use the assumptions $1<\nu<\frac{7}{4}$ and $\delta\ll1$.
\begin{lemma}\label{barriersUg}
	Let $f$, $R$ like \eqref{DefintionfR2}. Let $\sigma\in(0,1)$, $1<\nu<\frac{7}{4}$, and $h$ such that $\|h\|_{\star\star}<\infty$. Let $\delta$ in \eqref{FunctionM} sufficiently small and let us assume that $\|g U^{1/2}\|_{L^{2}}\le K_{1} f(t)R^{2}(t)$ for $\tau \in[\tau_{0},\tau_{1}]$. 
    We have
	\begin{align}\label{UgInEQBAr}
		|Ug|\le C (K_{1}+\frac{\|h\|_{\star\star}}{R^{2}(\tau_{0})}+\frac{|c_{1}|}{f(\tau_{0})R^{2}(\tau_{0})})\frac{f(\tau)R^{2}(\tau)}{1+|y|^{4}}.
	\end{align}
\end{lemma}
\begin{proof}
	 Recalling \eqref{DefintionfR2} we can assume that for any $\tau\in[\tau_{0},\tau_{1}]$
	 \begin{align*}
	 	f(\tau)R^{2}(\tau)\le C \frac{\ln^{\mu}\tau}{\tau^{\nu-1}} \ \ \ \text{for some }\mu\in \mathbb{R}.
	 \end{align*}
	Let $g_{0}=Ug$, we need to expand the equation \eqref{Eqg0} when $\rho=|y|\in ( R_{0},\infty)$ and $R_{0}\gg1$
	\begin{align}\label{ExpansionOperatorG0}
		&\partial_{ \tau} g_{0}-[\Delta g_{0}-\nabla \Gamma _{0}\cdot \nabla g + \lambda\dot{\lambda}(kg_{0}\chi + y \cdot \nabla (g_{0}\chi))+2Ug_{0}]=\nonumber\\
		&\hspace{1cm}=\partial_{ \tau}g_{0}-[\partial_{\rho}^{2}g_{0}+\frac{1}{\rho}\partial_{\rho}g_{0}+\frac{4\rho}{1+\rho^{2}}\partial_{\rho}g_{0}+\lambda\dot{\lambda}(kg_{0}\chi+\rho\partial_{\rho}(g_{0}\chi))+\frac{16}{(1+\rho^{2})^{2}}g_{0}]=\nonumber\\
		&\hspace{1cm}=\partial_{ \tau}g_{0}-[\partial_{\rho}^{2}g_{0}+\frac{5}{\rho}\partial_{\rho}g_{0}+\tilde{B}_{0}[g_{0}]+O(\frac{1}{\rho^{2}})\frac{1}{\rho}\partial_{\rho}g_{0}+O(\frac{1}{\rho^{4}})g_{0}   ]
	\end{align}
where we used the notation $\tilde{B}[\phi]=\tilde{B}_{0}[\phi]-W_{0}(y)\lambda\dot{\lambda}\int_{\mathbb{R}^{2}} (k\phi \chi+y\cdot{\nabla}(\phi \chi))dy$ and we observed that $W_{0}(y)$ is compactly supported.
	Then $g_{0}$ satisfies
	\begin{align}\label{FinalEqG0Real}
		\partial_{ \tau}g_{0}=(\Delta_{6,\rho}+...)g_{0}+\tilde{B}_{0}[g_{0}]+\tilde{h}
	\end{align} 
    (where in the parenthesis we are omitting the lower order terms in the expansion \eqref{ExpansionOperatorG0}). As can be seen from equation \eqref{FinalEqg0Est} and estimating the remaining terms thanks to \eqref{PrelimG0Spat} (similarly to how they did in Lemma 10.8 of \cite{DdPDMW}), for the right-hand side we know
	\begin{align}\label{InitialRHSg0}
		|\tilde{h}|\le C (K_{1}+\frac{\|h\|_{\star\star}}{R^{2}(\tau_{0})}+\frac{|c_{1}|}{f(\tau_{0})R^{2}(\tau_{0})})f(\tau)R^{2}(\tau)(\frac{1}{1+\rho^{6}}+\frac{\ln \tau}{\tau}\frac{1}{1+\rho^{5}}), \ \ \ \ \tau\in[\tau_{0},\tau_{1}].
	\end{align}
In order to simplify the notation in the following we write 
\begin{align}\label{f1}
	f_{1}(\tau):=f(\tau)R^{2}(\tau).
\end{align}
	\emph{\underline{First barrier:}} let $0<\theta<1$ that will be fixed only at the end of the proof. \newline We take $-\Delta_{6} \tilde{g}_{1}= \frac{1}{1+\rho^{6-\theta}}$ in $\mathbb{R}^{6}$ such that $\frac{c_{1}}{1+\rho^{4-\theta}}\le\tilde{g}_{1}\le \frac{c_{2}}{1+\rho^{4-\theta}}$ for some positive constants $c_{1}, c_{2}$. We see
	\begin{align}\label{EllipTICFIrstBAR}
		&-(\partial_{\rho}^{2}\tilde{g}_{1}+\frac{5}{\rho}\partial_{\rho}\tilde{g}_{1}+\tilde{B}_{0}[\tilde{g}_{1}])\ge \frac{c}{1+\rho^{6-\theta}}+O(M^{2}(\tau)\frac{\ln \tau}{\tau}\frac{1}{1+\rho^{6-\theta}}),
	\end{align}
    where we used \eqref{FunctionM}.
	We claim that by choosing properly $\varepsilon$, $\delta$ small and $C_{H_{1}}$ large constants  the first barrier will look like
	\begin{align*}
		\bar{g}_{1}(\rho,\tau)=f_{1}(\tau)\big[\tilde{g}_{1}\chi_{0}(\frac{\rho}{\varepsilon\sqrt{\tau}})+C_{H_{1}}H_{1}(\frac{\rho}{\sqrt{\tau}},\tau)\big]=f_{1}(\tau)\big[\tilde{g}_{1}\bar{\chi}+C_{H_{1}}H_{1}(\frac{\rho}{\sqrt{\tau}},\tau)\big]
	\end{align*}
	where $H_{1}(\frac{\rho}{\sqrt{\tau}},\tau)$ is some function of the type \eqref{Hbarrier}. We observe that since $\chi=\chi_{0}(\frac{\rho}{M(\tau)})$ and $\bar{\chi}=\chi_{0}(\frac{\rho}{\varepsilon\sqrt{\tau}})$ we see
	\begin{align*}
		\tilde{B}_{0}[\phi\bar{\chi}]=\lambda\dot{\lambda}(k\phi\chi\bar{\chi}+y\cdot\nabla(\phi\bar{\chi}\chi))=\bar{\chi}\tilde{B}_{0}[\phi]+\lambda\dot{\lambda}\phi\chi y\cdot\nabla\bar{\chi}=\bar{\chi}\tilde{B}_{0}[\phi].
	\end{align*}
 Let $\mathds{1}$ be the indicator function, if $\varepsilon>0$ is sufficiently small and recalling \eqref{EllipTICFIrstBAR} we have
	\begin{align}\label{epsilonsmallfirstBarr}
		\big[\partial_{ \tau}-(\Delta_{6}+\tilde{B}_{0})\big](f_{1}(\tau)\tilde{g}_{1}\bar{\chi})=&\bar{\chi}\big[\partial_{ \tau}-(\Delta_{6}+\tilde{B}_{0})\big](f_{1}(\tau)\tilde{g}_{1})+O(\frac{f_{1}(\tau)}{\varepsilon^{6-\theta}\tau^{3-\theta/2}}\mathds{1}_{ \{1\le \frac{\rho}{\varepsilon\sqrt{\tau}}\le 2\}})\ge \nonumber \\
		\ge &\bar{\chi}\big[\frac{c}{2}\frac{f_{1}(\tau)}{1+\rho^{6-\theta}}+O(M^{2}(\tau)\frac{\ln \tau}{\tau}\frac{f_{1}(\tau)}{1+\rho^{6-\theta}})\big]+O(\frac{f_{1}(\tau)}{\varepsilon^{6-\theta}\tau^{3-\theta/2}}\mathds{1}_{ \{1\le \frac{\rho}{\varepsilon\sqrt{\tau}}\le 2\}}).
	\end{align}
	If we take $H_{1}$ as in \eqref{Hbarrier} with $\gamma=\nu+1-\frac{\theta}{2}$ and $b=5$ (this spatial decay clearly comes from \eqref{InitialRHSg0}). We have the conditions
	\begin{align}\label{FIRSTBARRthetaCOND}
		\nu+1-\frac{\theta}{2}<3, \ \ \ \nu+1-\frac{\theta}{2}<\frac{5}{2} \iff \nu<\frac{3}{2}+\frac{\theta}{2}.
	\end{align}
    Now we see
	\begin{align*}
		|\tilde{B}_{0}[C_{H_{1}}H_{1}(\frac{\rho}{\sqrt{\tau}},\tau)]|=O(\chi_{0}(\frac{\rho}{2M(\tau)})\delta \frac{f_{1}(\tau)}{\tau^{2-\theta/2}}\frac{C_{H_{1}}}{1+\rho^{2}})=O(\chi_{0}(\frac{\rho}{2M(\tau)})C_{H_{1}}\frac{f_{1}(\tau)}{(\ln \tau)^{2-\theta/2}}\frac{\delta^{3-\theta/2}}{1+\rho^{6-\theta}}).
	\end{align*}
Then $\bar{g}_{1}(\rho,\tau)=f_{1}(\tau)\tilde{g}_{1}\chi_{0}(\frac{\rho}{\varepsilon\sqrt{\tau}})+C_{H_{1}}f_{1}(\tau)H_{1}(\frac{\rho}{\sqrt{\tau}},\tau)$ satisfies
	\begin{align*}
		\partial_{ \tau}\bar{g}_{1}-[\Delta_{6}\bar{g}_{1}+\tilde{B}_{0}[\bar{g}_{1}]\big]&\ge f_{1}(\tau)(\frac{1}{1+\rho^{6}}+\frac{\ln \tau}{\tau}\frac{1}{1+\rho^{5}})
	\end{align*}
	where we recall that we chose $\varepsilon$ small to get \eqref{epsilonsmallfirstBarr} and that we fixed $\delta$ sufficiently small to control $M^{2}(\tau)\frac{\ln \tau}{\tau}$ and $C_{H}$ (depending on $\varepsilon$), $\tau_{0}$ sufficiently large (in fact when we compute $\partial_{ \tau}(f_{1}(\tau)\tilde{g}_{1})$ we get rid of some terms that are smaller only when $\tau_{0}$ is large). This is a barrier, indeed we have
	\begin{align*}
		&\big(\partial_{ \tau}-[\Delta-\nabla \Gamma_{0}\cdot \nabla + \tilde{B}_{0}]\big)(N\bar{g}_{1})\ge |\tilde{h}| \ \ \ \ \text{in } B_{R_{0}}^{c}\times (\tau_{0},\tau_{1})\\
		&N\bar{g}_{1}\ge |g_{0}| \ \ \ \text{on }\rho=R_{0}, \ \tau\in(\tau_{0},\tau_{1})\\
		&N\bar{g}_{1}(\tau_{0})\ge |c_{1}Ug_{\tilde{Z}_{0}}| \ \ \ \text{in }\mathbb{R}^{2}.
	\end{align*}
	Notice that in the second inequality we fixed $N=R_{0}^{2-\theta}(K_{1}+\frac{\|h\|_{\star\star}}{R^{2}(\tau_{0})}+\frac{|c_{1}|}{f(\tau_{0})R^{2}(\tau_{0})})$. In the last one we used \eqref{InitialCond} and then that $g_{0}|_{\tau_{0}}=c_{1}Ug_{\tilde{Z}_{0}}$ satisfies
	\begin{align*}
		c_{1}Ug_{\tilde{Z}_{0}}=c_{1}U\big(\frac{\tilde{Z}_{0}}{U}-(-\Delta)^{-1}(\tilde{Z}_{0})\big)=c_{1}U\big(\frac{\tilde{Z}_{0}}{U}+\int_{\rho}^{\infty}\frac{1}{r}\int_{r}^{\infty}\tilde{Z}_{0}(s)sdsdr\big)
	\end{align*}
	that decays like $1/\rho^{4}$ and is supported where $|\rho|\le 2\sqrt{\tau_{0}}$, and we observed
\begin{align*}
		\bar{g}_{1}(\tau_{0})\ge C f(\tau_{0})R^{2}(\tau_{0})\frac{1}{1+\rho^{4}} \ \ \ \text{if } \rho<\sqrt{ \tau_{0}}.
	\end{align*}
    We just proved that
	\begin{align}\label{firstbar}
		|g_{0}|\le C(K_{1}+\frac{\|h\|_{\star\star}}{R^{2}(\tau_{0})}+\frac{|c_{1}|}{f(\tau_{0})R^{2}(\tau_{0})})\frac{f(t)R^{2}(\tau)}{1+\rho^{4-\theta}}, \ \ \ \ \tau \in (\tau_{0}, \tau_{1}).
	\end{align}
	\emph{\underline{Second barrier}:} the construction of the second barrier is more involved since it requires to consider also an intermediate region. As we will show in what follows, thanks to the cut-off \eqref{FunctionM} the operator $\tilde{B}_{0}$ in this intermediate region, as well as in the self-similar region, will give only lower order terms. \newline 
	The first observation is that, thanks to \eqref{firstbar} and recalling \eqref{f1}, now we have
	\begin{align}\label{RHSSECBARR}
		|\tilde{h}|\le C(K_{1}+\frac{\|h\|_{\star\star}}{R^{2}(\tau_{0})}+\frac{|c_{1}|}{f(\tau_{0})R^{2}(\tau_{0})})f_{1}(\tau)(\frac{1}{(1+\rho)^{6+\sigma}}+\frac{\ln \tau}{\tau}\frac{1}{(1+\rho)^{6-\theta}}).
	\end{align}
	Indeed, recalling that $R_{0}\gg1$, since $W_{0}$ is compactly supported we have
	\begin{align*}
		\tilde{h}=&-U(-\Delta)^{-1}[\nabla \cdot (U\nabla g)]+\tilde{B}_{0}[U\psi[g_{0}]]-U(-\Delta)^{-1}(\tilde{B}[g_{0}+U\psi[g_{0}]])+h-U(-\Delta)^{-1}h.
	\end{align*}
	An elementary computation shows that
	\begin{align*}
		|U(-\Delta)^{-1}[\nabla \cdot (U\nabla g)]|\le \frac{C}{1+\rho^{4}}\int_{\rho}^{\infty}|g_{0}| |\Gamma_{0}'(s)|ds\le C (K_{1}+\frac{\|h\|_{\star\star}}{R^{2}(\tau_{0})}+\frac{|c_{1}|}{f(\tau_{0})R^{2}(\tau_{0})})\frac{f_{1}(\tau)}{1+\rho^{8-\theta}}.
	\end{align*}
	Moreover, by recalling that analogously to \eqref{psiperpinf} we have $\psi(\rho)=z_{0}(\rho)\int_{\rho}^{\infty} \frac{1}{z_{0}^{2}(r)r}\int_{r}^{\infty} g_{0}(s)z_{0}(s)sdsdr$, we know
	\begin{align*}
       &|\psi|\le C(K_{1}+\frac{\|h\|_{\star\star}}{R^{2}(\tau_{0})}+\frac{|c_{1}|}{f(\tau_{0})R^{2}(\tau_{0})})\frac{f_{1}(\tau)}{1+\rho^{2-\theta}}\\
       &  \implies |\tilde{B}_{0}[U\psi]|\le C(K_{1}+\frac{\|h\|_{\star\star}}{R^{2}(\tau_{0})}+\frac{|c_{1}|}{f(\tau_{0})R^{2}(\tau_{0})})\frac{\ln \tau}{\tau} \frac{f_{1}(\tau)}{1+\rho^{6-\theta}}.
	\end{align*}
	Since $\phi=Ug+U\psi$, we also have $|\phi|\le C(K_{1}+\frac{\|h\|_{\star\star}}{R^{2}(\tau_{0})}+\frac{|c_{1}|}{f(\tau_{0})R^{2}(\tau_{0})})\frac{f_{1}(\tau)}{1+\rho^{4-\theta}}$. Then if $R_{0}\gg1$ we have by \eqref{InversofBtilde} that
	\begin{align*}
		|(-\Delta)^{-1}(\tilde{B}[g_{0}+U\psi[g_{0}]])|\le |\int_{\rho}^{\infty} \frac{1}{r}\int_{r}^{\infty}\tilde{B}_{0}[\phi]sdsdr|\le C \frac{\ln \tau}{\tau}(K_{1}+\frac{\|h\|_{\star\star}}{R^{2}(\tau_{0})}+\frac{|c_{1}|}{f(\tau_{0})R^{2}(\tau_{0})})\frac{f_{1}(\tau)}{1+\rho^{2-\theta}}.
	\end{align*}
	Finally we obviously have $|h-U(-\Delta)^{-1}h| \le C \|h\|_{\star\star}\frac{f(\tau)}{(1+\rho)^{6+\sigma}}$. \newline
	In the next observation we want to adapt the right-hand side of \eqref{RHSSECBARR} to the barrier we are going to construct. We observe that the error $\tilde{h}$ satisfies
	\begin{align}\label{EstimateTildeH}
		|\tilde{h}|\le& C(K_{1}+\frac{\|h\|_{\star\star}}{R^{2}(\tau_{0})}+\frac{|c_{1}|}{f(\tau_{0})R^{2}(\tau_{0})})f_{1}(\tau)\big(\frac{1}{1+\rho^ {6+\sigma}}\chi_{0}(\frac{\rho}{\sqrt{\tau}}\sqrt{\ln \tau})+ \frac{\ln^{3} \tau}{\tau^{3}}\frac{1}{1+(\frac{\rho}{\sqrt{\tau}}\sqrt{\ln \tau})^{6+\sigma}}\chi_{0}(\frac{\rho}{\sqrt{\tau}})+\nonumber\\
		&\hspace{4.8cm}+\frac{1}{\tau^{3}}\frac{1}{1+(\frac{\rho}{\sqrt{\tau}})^{6-\theta}}\big).
	\end{align}
    As we will show later the intermediate region in \eqref{EstimateTildeH} will help us to include the operator $\tilde{B}_{0}$. \newline
	As we did for the first barrier we could take $-\Delta_{6} \tilde{g}_{2}=\frac{1}{1+\rho^{6+\sigma}} \implies \frac{c_{1}}{1+\rho^{4}}\le \tilde{g}_{2}\le \frac{c_{2}}{1+\rho^{4}}$ but this is \emph{not} a good starting point since
	\begin{align*}
		-(\partial_{\rho}^{2}\tilde{g}_{2}+\frac{5}{\rho}\partial_{\rho}\tilde{g}_{2}+\tilde{B}_{0}[g_{2
		}])\ge \frac{c}{1+\rho^{6+\sigma}}+O(M^{2}(\tau)\frac{\ln \tau}{\tau}\frac{1}{1+\rho^{6}})
	\end{align*}
	and, recalling \eqref{FunctionM}, we observe that the second term does not have sufficiently fast spatial decay. For this reason for some constant $\bar{C}_{2}$ that has to be intended large and that will be fixed later we take
	\begin{align}\label{gtilde2}
		\tilde{g}_{2}=(-\Delta_{6})^{-1}(\frac{1}{1+\rho^{6+\sigma}})+\bar{C}_{2}\frac{\ln \tau}{\tau} (-\Delta_{6})^{-1}(\frac{1}{1+\rho^{4}})
	\end{align}
(notice that $\frac{c_{1}}{1+\rho^{2}}\le (-\Delta_{6})^{-1}(\frac{1}{1+\rho^{4}})\le \frac{c_{2}}{1+\rho^{2}}$ for two positive constants $c_{1},c_{2}$).
    We see that $\tilde{g}_{2}$ satisfies
	\begin{align*}
		-(\Delta_{6}+\tilde{B}_{0})\tilde{g}_{2}\ge \frac{1}{1+\rho^{6+\sigma}}+O(\frac{\ln \tau}{\tau} \frac{1}{1+\rho^{4}}\chi_{0}(\frac{\rho}{2M(\tau)}))+\frac{\ln \tau}{\tau} \frac{\bar{C}_{2}}{1+\rho^{4}}+O(\bar{C}_{2}\big(\frac{\ln \tau}{\tau}\big)^{2} \frac{1}{1+\rho^{2}}\chi_{0}(\frac{\rho}{2M(\tau)})).
	\end{align*}
     We notice that, thank to the expansion \eqref{lambdadotlambdaExp} and because of \eqref{FunctionM}, we can take $m$ small but such that 
    \begin{align}\label{mCUTOFFSECBARR}
    	\sqrt{\frac{m\tau}{\ln \tau}} \gg M(\tau).
    \end{align}
    The construction of the second barrier differs from the first one since now we take $\bar{\chi}=\chi_{0}(\frac{\rho}{\sqrt{m\tau}}\sqrt{\ln\tau })$. We observe that thanks to \eqref{mCUTOFFSECBARR} we still have
	\begin{align*}
		\tilde{B}_{0}[\phi\bar{\chi}]=\lambda\dot{\lambda}(k\phi\chi\bar{\chi}+y\cdot\nabla(\phi\bar{\chi}\chi))=\bar{\chi}\tilde{B}_{0}[\phi]+\lambda\dot{\lambda}\phi\chi y\cdot\nabla\bar{\chi}=\bar{\chi}\tilde{B}_{0}[\phi].
	\end{align*}
    We notice that taking this cut-off we have $|\tilde{g}_{2}\bar{\chi}|\le \frac{1}{1+\rho^{4}}$ and that if $\rho \approx \sqrt{\frac{\tau}{\ln \tau}}$ we see $\tilde{g}_{2}\approx \big(\frac{\ln \tau}{\tau}\big)^{2}$ (taking the same $\tilde{\chi}$ of the first barrier we would have a loss of decay when $\rho \approx \sqrt{\tau}$, namely $\tilde{g}_{2}\approx \frac{\ln \tau}{\tau^{2}}$ and not just $\tilde{g}_{2}\approx \frac{1}{\tau^{2}}$).
	Then if \eqref{mCUTOFFSECBARR} holds for a small $m$ we have
	\begin{align*}
		&-(\Delta_{6}+\tilde{B}_{0})[\tilde{g}_{2}\chi_{0}(\frac{\rho}{\sqrt{m\tau}}\sqrt{\ln \tau})]\ge\\
		&\hspace{1cm}\ge\big [\frac{1}{1+\rho^{6+\sigma}}+\frac{\ln \tau}{\tau} \frac{\bar{C}_{2}}{1+\rho^{4}}+O[(1+\delta \bar{C}_{2})\frac{\ln \tau}{\tau}\frac{1}{1+\rho^{4}}\chi_{0}(\frac{\rho}{2M(\tau)})]\big]\chi_{0}(\frac{\rho}{\sqrt{m\tau}}\sqrt{\ln\tau })+\\
		&\hspace{2cm}+O(|\tilde{g}_{2}||\Delta_{6}\bar{\chi}|+...)=\\
		&\hspace{1cm}=[...]\bar{\chi}+O[(\frac{\ln \tau}{m\tau})^{3}\bar{C}_{2}\mathds{1}_{ \{1\le \frac{\rho}{\sqrt{m\tau}}\sqrt{\ln \tau}\le 2 \}}].
	\end{align*}
	Now we observe that if $\tau_{0}$ is sufficiently large
	\begin{align*}
		|\partial_{ \tau}[f_{1}(\tau)\tilde{g}_{2}\bar{\chi}]|\le& \frac{1}{\tau}f_{1}(\tau)|\tilde{g}_{2}| \bar{\chi}+\frac{1}{\tau} f_{1}(\tau)|\tilde{g}_{2}| |\chi_{0}'(\frac{\rho}{\sqrt{m\tau}}\sqrt{\ln \tau})| |\frac{\rho}{\sqrt{m\tau}}\sqrt{\ln \tau}|\le \\
		\le& C \frac{f_{1}(\tau)}{\tau}\big[\frac{1}{1+\rho^{4}}+\bar{C}_{2}\frac{\ln \tau}{\tau}\frac{1}{1+\rho^{2}}\big]\bar{\chi}+O(\bar{C}_{2}\frac{f_{1}(\tau)}{\tau}(\frac{\ln \tau}{m\tau})^{2}\mathds{1}_{ \{1\le \frac{\rho}{\sqrt{m\tau}}\sqrt{\ln \tau}\le 2 \}}).
	\end{align*}
	We get
	\begin{align}\label{FirstPieceSECBARR}
		(\partial_{ \tau}-\Delta_{6}-\tilde{B}_{0})(f_{1}(\tau)\tilde{g}_{2}\bar{\chi})\ge&\frac{f_{1}(\tau)}{1+\rho^{6+\sigma}}\bar{\chi}+f_{1}(\tau)\frac{\ln \tau}{\tau}\frac{\bar{C}_{2}}{1+\rho^{4}}\bar{\chi}+\nonumber\\
		&+O\big[(1+\delta \bar{C}_{2})\frac{\ln \tau}{\tau}f_{1}(\tau)\frac{1}{1+\rho^{4}}\chi_{0}(\frac{\rho}{2M(\tau)})\big]+O\big[\frac{1+\bar{C}_{2}m}{\tau}\frac{f_{1}(\tau)}{1+\rho^{4}}\bar{\chi}\big]+\nonumber\\
		&+O\big[(\frac{\ln \tau}{m\tau})^{3}\bar{C}_{2}f_{1}(\tau)\mathds{1}_{\{1\le \frac{\rho}{\sqrt{m\tau}}\sqrt{\ln \tau}\le 2\}}\big].
	\end{align}
    By looking at \eqref{FirstPieceSECBARR} it is clear why later we will take $\bar{C}_{2}$, $\tau_{0}$ sufficiently large and $\delta$ sufficiently small.
	The second part of the barrier is based on the construction of a function $g_{3}$ such that
	\begin{align*}
		(\partial_{ \tau}-\Delta_{6}-\tilde{B}_{0})(g_{3})\ge (\frac{\ln \tau}{\tau})^{3}\frac{f_{1}(\tau)}{(1+(\frac{\rho}{\sqrt{\tau}}\sqrt{\ln \tau})^{6+\sigma})}+(...), \ \ \ \ \rho \le \sqrt{\gamma \tau}
	\end{align*} 
	with $\gamma$ a small parameter to be fixed and where in the right-hand side we will have some small errors that will be controlled by $\bar{C}_{2}$ and by the last piece of the barrier. Notice that the decay in space of the right-hand side is consistent with \eqref{EstimateTildeH}. Similarly to the previous piece of the barrier, writing $\zeta=\frac{\rho}{\sqrt{\tau}}\sqrt{\ln \tau}$ and denoting $\Delta_{6,\zeta}=\frac{\tau}{\ln \tau}\Delta_{6,\rho}$ the rescaled operator, we can take
	\begin{align}\label{gtilde3}
		\tilde{g}_{3}= (-\Delta_{6,\zeta})^{-1}(\frac{1}{(1+(\frac{\rho}{\sqrt{\tau}}\sqrt{\ln \tau})^{6+\sigma})})+\frac{\bar{C}_{3}}{\ln\tau}(-\Delta_{6,\zeta})^{-1}\big(\frac{1}{1+(\frac{\rho}{\sqrt{\tau}}\sqrt{\ln \tau})^{4}}\big).
	\end{align}
	We have
	\begin{align*}
		-\Delta_{6}(f_{1}(\tau)(\frac{\ln \tau}{\tau})^{2}\tilde{g}_{3}\chi_{0}(\frac{\rho}{\sqrt{\gamma\tau}}))=&f_{1}(\tau)(\frac{\ln \tau}{\tau})^{3}[\frac{1}{1+ \zeta^{6+\sigma}}+\frac{\bar{C}_{3}}{\ln \tau}\frac{1}{1+\zeta^{4}}]\chi_{0}(\frac{\rho}{\sqrt{\gamma \tau}})+\\
		&+O(f_{1}(\tau)\frac{\bar{C}_{3}}{\gamma^{3}\tau^{3}}\mathds{1}_{\{1\le \frac{\rho}{\sqrt{\gamma \tau}}\le 2\}}).
	\end{align*}
	Moreover since $\tilde{B}_{0}[\phi\bar{\chi}]=\lambda\dot{\lambda}(k\phi\chi\chi_{0}(\frac{\rho}{\sqrt{\gamma \tau}})+y\cdot\nabla(\phi\chi\chi_{0}(\frac{\rho}{\sqrt{\gamma \tau}})))=\tilde{B}_{0}[\phi]$ we have
	\begin{align*}
		|\tilde{B}_{0}[f_{1}(\tau)(\frac{\ln \tau}{\tau})^{2}\tilde{g}_{3}\chi_{0}(\frac{\rho}{\sqrt{\gamma \tau}})]|\le C f_{1}(\tau)(\frac{\ln \tau}{\tau})^{3}\bar{C}_{3}\chi_{0}(\frac{\rho}{2M(\tau)}) \le C f_{1}(\tau) \bar{C}_{3}\frac{\ln \tau}{\tau}\delta^{2} \frac{1}{1+\rho^{4}}\chi_{0}(\frac{\rho}{2M(\tau)}).
	\end{align*}
	Since for $\tau_{0}$ large we additionally have that
	\begin{align*}
		|\partial_{ \tau}(f_{1}(\tau)(\frac{\ln \tau}{\tau})^{2}\tilde{g}_{3}\chi_{0}(\frac{\rho}{\sqrt{\gamma\tau}}))|\le &C \frac{1}{\tau} f_{1}(\tau)(\frac{\ln \tau}{\tau})^{2}\frac{1+\gamma \bar{C}_{3}}{1+\zeta^{4}}\chi_{0}(\frac{\rho}{\sqrt{\gamma \tau}})+O[\frac{f_{1}(\tau)}{(\gamma \tau)^{2}\tau}(1+\gamma\bar{C}_{3})\mathds{1}_{\{1\le \frac{\rho}{\sqrt{\gamma \tau}}\le 2\}}]
	\end{align*}
    we get
	\begin{align}\label{g3INeSECBARRIER}
		(\partial_{ \tau}-\Delta_{6}-\tilde{B}_{0})(f_{1}(\tau)(\frac{\ln \tau}{\tau})^{2}\tilde{g}_{3}\chi_{0}(\frac{\rho}{\sqrt{\gamma \tau}}))=&f_{1}(\tau)(\frac{\ln \tau}{\tau})^{3}[\frac{1}{1+\zeta^{6}}+\frac{\bar{C}_{3}}{\ln \tau}\frac{1}{1+\zeta^{4}}]\chi_{0}(\frac{\rho}{\sqrt{\gamma \tau}})+\nonumber\\
		&+O[f_{1}(\tau)(\frac{\ln \tau}{\tau})^{2}\frac{1}{\tau}\frac{1+\gamma \bar{C}_{3}}{1+\zeta^{4}}\chi_{0}(\frac{\rho}{\sqrt{\gamma\tau}})]+\nonumber\\
		&+O[\frac{1}{(\gamma \tau)^{3}}f_{1}(\tau)\bar{C}_{3}\mathds{1}_{\{1\le \frac{\rho}{\sqrt{\gamma \tau}}\le 2\}}]+\nonumber\\
		&+O[\frac{\ln \tau}{\tau}f_{1}(\tau)\frac{\delta^{2}\bar{C}_{3}}{1+\rho^{4}}\chi_{0}(\frac{\rho}{2M(\tau)})].
	\end{align}
    Since we wanted to erase the remainder in the right-hand side of \eqref{FirstPieceSECBARR} it is clear that this second piece has to be multiplied by $\frac{\bar{C}_{2}}{m^{3}}$ times some sufficiently large factor independent of all the parameters.
    By looking at \eqref{g3INeSECBARRIER} it easy to see why later we are going to choose $\bar{C}_{3}$ sufficiently large such that $\bar{C}_{3}\gtrsim1+\gamma \bar{C}_{3}$ (with the notation $\gtrsim$ we are neglecting factors independent of the parameters). Keeping in mind \eqref{FirstPieceSECBARR} it is also clear why we are taking $\bar{C}_{2}$ sufficiently large such that
    \begin{align}\label{FirstINEQC2SECBARR}
    	\bar{C}_{2}\gtrsim 1+\delta \bar{C}_{2}+\frac{m\bar{C}_{2}}{\ln \tau} + \bar{C}_{2} \bar{C}_{3}\delta^{2} \frac{1}{m^{3}}.
    \end{align}
    In order to find a $\bar{C}_{2}$ such that \eqref{FirstINEQC2SECBARR} holds it is sufficient to recall \eqref{mCUTOFFSECBARR} and that $m$, $\delta$ are small (we can assume $\frac{\delta^{2}}{m^{3}}\ll1$).
	\newline In the final step we take $H_{2}(\rho,\tau)$ like \eqref{Hbarrier} with $\gamma=\nu+1$ and $b=6-\theta$ (this spatial decay is a consequence of the decay in \eqref{EstimateTildeH}). Then we need
	\begin{align}\label{SecCondTheta}
	     \nu+1<3-\frac{\theta}{2} \iff \nu<2-\frac{\theta}{2}.
	\end{align} 
First we observe that 
	\begin{align*}
		|\tilde{B}_{0}[H_{2}(\frac{\rho}{\sqrt{\tau}})]=O(\frac{\delta}{\tau^{2}}\frac{f_{1}(\tau)}{1+\rho^{2}}\chi_{0}(\frac{\rho}{2M(\tau)}))=O(\frac{\delta^{2}}{\tau \ln \tau}\frac{f_{1}(\tau)}{1+\rho^{4}}\chi_{0}(\frac{\rho}{2M(\tau)})).
	\end{align*}
	Then we have
	\begin{align}\label{LastPieceBarrierSEC}
	(\partial_{ \tau}-\Delta_{6}-\tilde{B}_{0})(H_{2}(\rho,\tau))\ge \frac{1}{\tau^{3}}\frac{f_{1}(\tau)}{1+(\frac{\rho}{\sqrt{\tau}})^{6-\theta}}+O(\frac{\delta^{2}}{\tau \ln \tau}\frac{f_{1}(\tau)}{1+\rho^{4}}\chi_{0}(\frac{\rho}{2M(\tau)})).
	\end{align}
Finally for some constants $D$, $E$ universal, recalling \eqref{FirstPieceSECBARR}, \eqref{g3INeSECBARRIER} and \eqref{LastPieceBarrierSEC}, we get 
	\begin{align*}
		&(\partial_{ \tau}-\Delta_{6}-\tilde{B}_{0})[f_{1}(\tau)\tilde{g}_{2}\bar{\chi}+D \frac{\bar{C}_{2}}{m^{3}}f_{1}(\tau)(\frac{\ln \tau}{\tau})^{2}\tilde{g}_{3}\chi_{0}(\frac{\rho}{\sqrt{\gamma \tau}})+\frac{DE\bar{C}_{2}\bar{C}_{3}}{\gamma^{3}m^{3}}H_{2}(\rho,\tau)]\ge\\
		&\hspace{2cm} \ge \frac{f_{1}(\tau)}{1+\rho^{6+\sigma}}\bar{\chi}+\frac{1}{2}\frac{D\bar{C}_{2}}{m^{3}}f_{1}(\tau)(\frac{\ln \tau}{\tau})^{3}\frac{1}{1+\zeta^{6+\sigma}}\chi_{0}(\frac{\rho}{\sqrt{\gamma \tau}})+\frac{1}{2}\frac{DE\bar{C}_{2}\bar{C}_{3}}{\gamma^{3}m^{3}\tau^{3}}\frac{f_{1}(\tau)}{1+(\frac{\rho}{\sqrt{\tau}})^{6-\theta}}.
	\end{align*}
	where we fixed $\bar{C}_{3}$ such that $\bar{C}_{3}\gtrsim 1+\gamma \bar{C}_{3}$ where $\gamma$ is small and $\bar{C}_{2}$ such that
	\begin{align}\label{CoeffiCientIneqBarrier}
		\bar{C}_{2}\gtrsim 1+ \delta \bar{C}_{2}+\frac{m\bar{C}_{2}}{\ln \tau}+\bar{C}_{2}\bar{C}_{3}\frac{\delta^{2}}{m^{3}}+\bar{C}_{2}\bar{C}_{3}\frac{\delta^{2}}{\gamma^{3}m^{3}}\frac{1}{\ln^{2}\tau}.
	\end{align}
    In order to find a $\bar{C}_{2}$ such that \eqref{CoeffiCientIneqBarrier} holds it is sufficient to recall \eqref{mCUTOFFSECBARR}, that $\tau_{0}$ is sufficiently large and that $m$, $\delta$ are small (we can assume $\frac{\delta^{2}}{m^{3}}\ll1$). We observe that it is useful to take $m$ small since we want to say that $|\tilde{g}_{2}\bar{\chi}|\le \frac{1}{1+\rho^{4}}$.\newline
    Proceeding as for the first barrier (after the analogous considerations about the initial condition) if we call $K:=K_{1}+\frac{\|h\|_{\star\star}}{R^{2}(\tau_{0})}+\frac{|c_{1}|}{f(\tau_{0})R^{2}(\tau_{0})}$, we just proved that 
	\begin{align*}
		|g_{0}|\le& K[ f_{1}(\tau)(\frac{1}{1+\rho^{4}}+ \frac{\ln \tau}{\tau} \frac{1}{1+\rho^{2}})\chi_{0}(\frac{\rho}{\sqrt{m\tau}}\sqrt{\ln \tau})+\\
		&\hspace{1cm}+f_{1}(\tau)(\frac{\ln \tau}{\tau})^{2}(\frac{1}{1+(\frac{\rho}{\sqrt{ \tau}}\sqrt{\ln \tau})^{4}}+\frac{1}{\ln \tau}\frac{1}{1+(\frac{\rho}{\sqrt{ \tau}}\sqrt{\ln \tau})^{2}})\chi_{0}(\frac{\rho}{\sqrt{\gamma \tau}})+\\
		&\hspace{1cm}+\frac{f_{1}(\tau)}{\tau^{2}}\frac{1}{1+(\frac{\rho}{\sqrt{\tau}})^{6-\theta}}].\\
		\implies& |g_{0}|\le C(K_{1}+\frac{\|h\|_{\star\star}}{R^{2}(\tau_{0})}+\frac{|c_{1}|}{f(\tau_{0})R^{2}(\tau_{0})})\frac{1}{1+\rho^{4}}.
	\end{align*}
	Finally, recalling \eqref{FIRSTBARRthetaCOND} and \eqref{SecCondTheta}, we can simply take $\theta=1/2$.
\end{proof}
We omit the proof of the following result that can be found in \cite{DdPDMW}.
\begin{lemma}\label{Lemma109}
	Let $\phi:\mathbb{R}^{2}\to \mathbb{R}$ be radial such that $\int_{\R^2}\phi=0$ and $|\phi(y)|\le \frac{C}{(1+|y|)^{2+\sigma}}$ for some $\sigma>0$. Let $g=\frac{\phi}{U}-(-\Delta)^{-1}\phi$ and assume that $\|g\|_{L^{\infty}}<\infty$. Then
	\begin{align}\label{phiFirstEstimate}
		|\phi(y)|\le C \frac{\|g\|_{L^{\infty}}}{(1+|y|)^{4}}.
	\end{align}
\end{lemma}
\begin{proof}
	See Lemma 10.9 of \cite{DdPDMW}.
\end{proof}
Before proving Lemma \ref{Lemma101} we need some estimates for $\tilde{Z}_{0}$ defined in \eqref{InitialCond}. Thanks to decomposition \eqref{DecomPositionPhiPhiperp} we can write
\begin{align*}
	\tilde{Z}_{0}=\tilde{Z}_{0}^{\perp}+\frac{\tilde{a}(\tau_{0})}{2}Z_{0}.
\end{align*}
By \eqref{ComputingA} we have, recalling \eqref{mz0Est}
\begin{align*}
	\tilde{a}=\frac{1}{8\pi}\int_{\R^2}\Gamma_{0}\tilde{Z}_{0}=2+\frac{1}{8\pi}\big(-\int_{\R^2}Z_{0}\Gamma_{0}(1-\bar{\chi})-m_{\tilde{Z}_{0}}\int_{\R^2}U\Gamma_{0}\bar{\chi}\big)=2+O(\frac{\ln \tau_{0}}{\tau_{0}}).
\end{align*}
Hence 
\begin{align}\label{ZoperpExp}
	\tilde{Z}_{0}^{\perp}(\rho)&=\tilde{Z}_{0}(\rho)-\frac{\tilde{a}(\tau_{0})}{2}Z_{0}(\rho)=(Z_{0}(\rho)-m_{Z_{0}}U(\rho))\bar{\chi}-(1+O(\frac{\ln \tau_{0}}{\tau_{0}}))Z_{0}(\rho)=\nonumber\\
	&=-Z_{0}(\rho)(1-\bar{\chi})+O(\frac{\ln \tau_{0}}{\tau_{0}}\frac{1}{1+\rho^{4}}).
\end{align}
Moreover we also observe that $-\Delta z_{0}-Uz_{0}=0$ that means $-\Delta z_{0}=Z_{0}$ and then $z_{0}=(-\Delta)^{-1}Z_{0}-2$ (since $\lim_{\rho\to\infty}z_{0}=-2$). Consequently
\begin{align*}
	g|_{\tau=\tau_{0}}=&c_{1}\big(\frac{\tilde{Z}_{0}}{U}-(-\Delta)^{-1}(\tilde{Z}_{0})\big)=\\
	=&-2c_{1}+c_{1}z_{0}(\bar{\chi}-1)-c_{1}m_{\tilde{Z}_{0}}\chi-c_{1}(-\Delta)^{-1}(Z_{0}(\bar{\chi}-1)-m_{Z_{0}}U\bar{\chi}).
\end{align*}
We call $d:=\int_{0}^{\infty}\frac{1}{r}\int_{0}^{r}(Z_{0}(\bar{\chi}-1)-m_{Z_{0}}U\chi)sdsdr$ and we observe that since $m_{Z_{0}}$ satisfies \eqref{mz0Est}, we have
\begin{align*}
	|\int_{0}^{\rho}\frac{1}{r}\int_{0}^{r}Um_{Z_{0}}sdsdr|\le C \frac{1}{\tau_{0}}\int_{0}^{\rho}\frac{1}{r}\int_{0}^{r}\frac{s}{1+s^{4}}dsdr\le C \frac{1}{\tau_{0}} \ln(1+\rho).
\end{align*}
Then we have
\begin{align*}
	c_{1}(-\Delta)^{-1}(Z_{0}(\chi-1)-m_{Z_{0}}U\chi)=&\begin{cases}
	dc_{1}+O(\frac{1}{\tau_{0}}|c_{1}|\ln(1+\rho)) \ \ \ \text{if }\rho\le \sqrt{\tau_{0}}\\[7pt]
	dc_{1}+O\big((\ln(1+\rho)+\ln(\frac{\rho}{\sqrt{\tau_{0}}}))\frac{1}{\tau_{0}}|c_{1}|\big) \ \ \ \ \text{if }\rho\ge \sqrt{\tau_{0}}
\end{cases} 
\end{align*}
then finally we have
\begin{align}\label{gExpansion}
	g|_{\tau=\tau_{0}}=-(2+d)c_{1}-c_{1}z_{0}(1-\chi)+O(\frac{1}{\tau_{0}}|c_{1}|\ln(2+\rho)).
\end{align}
Thus, recalling that $\int_{\R^2}\phi g^{\perp}=\int_{\R^2}\phi^{\perp}g$ (since $\int_{\R^2}\phi=0$ and $\int_{\R^2}gZ_{0}=0$) and $\int_{\R^2}\tilde{Z}_{0}^{\perp}=0$, since we have \eqref{ZoperpExp} and \eqref{gExpansion} we finally observe
\begin{align}\label{GronwallID}
	\int_{\R^2}(\phi g^{\perp})|_{\tau=\tau_{0}}&=\int_{\R^2}(\phi^{\perp}g)|_{\tau=\tau_{0}}=c_{1}\int_{\R^2}\tilde{Z}^{\perp}_{0}g|_{\tau=\tau_{0}}=\nonumber\\
	&=c_{1}^{2}\int_{\R^2}Z_{0}z_{0}(1-\bar{\chi})^{2}dy+O(|c_{1}|^{2}\frac{\ln \tau_{0}}{\tau_{0}^{2}})=O(|c_{1}|^{2}\frac{1}{\tau_{0}}).
\end{align}
A \emph{crucial} observation, that will be fundamental in the proof of Lemma \ref{Lemma101}, is that, since  $R^{2}(\tau)=\frac{\tau}{\ln^{1+q}\tau}$ for $q\in(0,1)$, thanks to \eqref{GronwallID} we have
\begin{align}\label{GronwallInitialCOND}
	|\int_{\R^2}(\phi g^{\perp})|_{\tau=\tau_{0}}|\le C \frac{|c_{1}|^{2}}{R^{2}(\tau_{0})}.
\end{align}
\begin{proof}[Proof of Lemma \ref{Lemma101}]\label{ProofOFLemma101}
	Thanks to Lemma \ref{ineqphigperp}, we have
	\begin{align*}
		\partial_{ \tau} \int_{\mathbb{R}^{2}} \phi g^{\perp} + \frac{c}{R^{2}(\tau)}\int_{\mathbb{R}^{2}} \phi g^{\perp} \le C[\|h\|^{2}_{\star\star}f^{2}(\tau)+\frac{a^{2}}{R^{4}}+\frac{\omega_{0}^{2}}{R^{2}}+\frac{\omega_{1}^{2}}{R^{2}}]
	\end{align*}
	where
	\begin{align*}
		\omega_{0}=(\int_{B_{R}^{c}} U g^{2} )^{1/2}, \ \ \ \omega_{1}=R[a^{2}\frac{\ln \tau}{\tau}\frac{1}{M^{2}}+\frac{\ln \tau}{\tau} \int_{B_{M}^{c}}|\phi g| + \frac{\ln \tau}{\tau} \int_{B_{M}^{c}} U g^{2}+(R^{2}\frac{\ln \tau}{\tau})^{2}(\int_{B_{M}^{c}}|\phi|)^{2}]^{1/2}.
	\end{align*}
	In the following we will simply denote 
	\begin{align}\label{normtau0tau1}
		\|x\|:=\sup_{\tau\in [\tau_{0},\tau_{1}]}|x(\tau)|.
	\end{align}
	We have
	\begin{align*}
		\partial_{ \tau}\int_{\mathbb{R}^{2}} \phi g^{\perp}+\frac{c}{R^{2}}\int_{\mathbb{R}^{2}} \phi g^{\perp} \le C f^{2}(\tau)[\|h\|_{\star\star}+\|\frac{a}{R^{2}f}\|^{2}+\|\frac{\omega_{0}}{fR}\|^{2}+\|\frac{\omega_{1}}{fR}\|^{2}].
	\end{align*}
	By Gr\"onwall's inequality, recalling \eqref{GronwallID}, we have
	\begin{align}\label{firstineqGron}
		\int_{\mathbb{R}^{2}} \phi g^{\perp}\le C f^{2}(\tau)R^{2}(\tau)\big[ \|h\|_{\star\star}^{2}+\|\frac{a}{R^{2}f}\|^{2}+\|\frac{\omega_{0}}{fR}\|^{2}+\|\frac{\omega_{1}}{fR}\|^{2}+c_{1}^{2}D^{2}(\tau_{0})\big]
	\end{align}
	with $D(\tau_{0})=\frac{1}{f(\tau_{0})R(\tau_{0})}\sqrt{\frac{\ln \tau_{0}}{\tau_{0}}}$. By Lemma \ref{QuadrForIne}, we know $\int_{\mathbb{R}^{2}} \phi g^{\perp}\ge c \int_{\mathbb{R}^{2}} U (g^{\perp})^{2}$, moreover we can easily observe that
	\begin{align*}
		\int_{\mathbb{R}^{2}}U g^{2}\le \int_{\R^2}U(g^{\perp})^{2}+C a^{2}
	\end{align*}
	and then
	\begin{align*}
		\int_{\mathbb{R}^{2}} Ug^{2}&\le C f^{2}(\tau)R^{4}(\tau)\big[ \frac{1}{R^{2}(\tau_{0})}\|h\|_{\star\star}^{2}+\|\frac{a}{R^{2}f}\|^{2}+\frac{1}{R^{2}(\tau_{0})}\|\frac{\omega_{0}}{fR}\|^{2}+\frac{1}{R^{2}(\tau_{0})}\|\frac{\omega_{1}}{fR}\|^{2}+\frac{c_{1}^{2}}{f^{2}(\tau_{0})R^{4}(\tau_{0})}\big].
	\end{align*}
    Then
    \begin{align}\label{PrelimUg^2}
        \|gU^{1/2}\|_{L^{2}}\le C f(\tau)R^{2}(\tau)\big[\frac{\|h\|_{\star\star}}{R(\tau_{0})}+\|\frac{a}{fR^{2}}\|+\frac{1}{R(\tau_{0})}\|\frac{\omega_{0}}{fR}\|+\frac{1}{R(\tau_{0})}\|\frac{\omega_{1}}{fR}\|+\frac{|c_{1}|}{f(\tau_{0})R^{2}(\tau_{0})}\big].
    \end{align}
    Now we want to apply Lemma \ref{barriersUg}. We notice from the right-hand side of \eqref{PrelimUg^2} that even if in the inequality \eqref{UgInEQBAr} we had $\frac{\|h\|}{R^{2}(\tau_{0})}$ from now on we will only have $\frac{\|h\|}{R(\tau_{0})}$. This smallness is sufficient for our purposes. More precisely, we get
	\begin{align*}
		|Ug|&\le C \big[ \frac{\|h\|_{\star\star}}{R(\tau_{0})}+\|\frac{a}{R^{2}f}\|+\frac{1}{R(\tau_{0})}\|\frac{\omega_{0}}{fR}\|+\frac{1}{R(\tau_{0})}\|\frac{\omega_{1}}{fR}\|+\frac{|c_{1}|}{f(\tau_{0})R^{2}(\tau_{0})}\big]\frac{f(\tau)R^{2}(\tau)}{1+\rho^{4}}\le C \mu_{0}\frac{f(\tau)R^{2}(\tau)}{1+\rho^{4}},
	\end{align*}
    where we defined $\mu_{0}:= \frac{\|h\|_{\star\star}}{R(\tau_{0})}+\|\frac{a}{R^{2}f}\|+\frac{1}{R(\tau_{0})}\|\frac{\omega_{0}}{fR}\|+\frac{1}{R(\tau_{0})}\|\frac{\omega_{1}}{fR}\|+\frac{|c_{1}|}{f(\tau_{0})R^{2}(\tau_{0})}$.
	Now need to absorb the constants in $\mu_{0}$. We start with $\omega_{0}$:
	\begin{align}\label{DECAYPHIREMARK}
		\omega_{0}^{2}=\int_{B_{R}^{c}}Ug^{2}\le \mu_{0}^{2}f^{2}(\tau)R^{4}(\tau)\int_{R}^{\infty} \frac{1}{\rho^{4}}\rho d\rho \le C \mu_{0}^{2}f^{2}(\tau)R^{2}(\tau).
	\end{align}
	Then we have
	\begin{align*}
		&\|\frac{\omega_{0}}{fR}\|\le  \big[ \frac{\|h\|_{\star\star}}{R(\tau_{0})}+\|\frac{a}{R^{2}f}\|+\frac{1}{R(\tau_{0})}\|\frac{\omega_{0}}{fR}\|+\frac{1}{R(\tau_{0})}\|\frac{\omega_{1}}{fR}\|+\frac{|c_{1}|}{f(\tau_{0})R^{2}(\tau_{0})}\big]\\
		&\implies \mu_{0}\le C  \big[ \frac{\|h\|_{\star\star}}{R(\tau_{0})}+\|\frac{a}{R^{2}f}\|+\frac{1}{R(\tau_{0})}\|\frac{\omega_{1}}{fR}\|+\frac{|c_{1}|}{f(\tau_{0})R^{2}(\tau_{0})}\big]=:\mu_{1}\implies |Ug|\le C \mu_{1}\frac{f(\tau)R^{2}(\tau)}{1+\rho^{4}}.
	\end{align*}
	Now we want to control $\omega_{1}$, we have
	\begin{align*}
		\frac{\omega_{1}^{2}}{R^{2}}=a^{2}\frac{\ln \tau}{\tau}\frac{1}{M^{2}}+\frac{\ln \tau}{\tau} \int_{B_{M}^{c}}|\phi g| + \frac{\ln \tau}{\tau} \int_{B_{M}^{c}} U g^{2}+(R^{2}\frac{\ln \tau}{\tau})^{2}(\int_{B_{M}^{c}}|\phi|)^{2}.
	\end{align*}
	Since $M^{2}(\tau)\approx c\frac{\delta \tau}{\ln\tau}$ for some universal constant $c$ (because of \eqref{FunctionM} and \eqref{lambdadotlambdaExp}), we have $\frac{\ln \tau}{\tau}\frac{1}{M^{2}(\tau)}\le \frac{\ln^{2+2q}\tau}{\tau^{2}}= \frac{1}{R^{4}(\tau)}$ if $q>0$ and $\tau_{0}$ is sufficiently large (we recall that now $\delta$ is fixed small universal number). Moreover, thanks to Lemma \ref{Lemma109} we have $|\phi| \le C U|g|$ and then we can write
	\begin{align*}
		\frac{\omega_{1}^{2}}{R^{2}}\le\frac{a^{2}}{R^{4}}+\frac{\ln \tau}{\tau} \int_{B_{M}^{c}} U g^{2}+(R^{2}\frac{\ln \tau}{\tau})^{2}(\int_{B_{M}^{c}}|\phi|)^{2}.
	\end{align*}
	We observe that, by \eqref{FunctionM}, we have $R^{4}(\tau)\frac{\ln \tau}{\tau}\frac{1}{M^{2}(\tau)}\approx(\frac{\tau}{\ln^{1+q}\tau})^{2}\frac{\ln \tau}{\tau}\frac{\ln\tau}{\delta \tau}\ll 1$ if $q>0$ and $\tau_{0}$ is sufficiently large. Then
	\begin{align*}
		\frac{\ln \tau}{\tau} \int_{B_{M}^{c}}Ug^{2}\le C \mu_{1}^{2}\frac{\ln \tau}{\tau}f^{2}(\tau)R^{4}(\tau)\int_{M}^{\infty} \frac{1}{\rho^{3}} d\rho \le C \mu_{1}^{2}f^{2}(\tau)R^{4}(\tau)\frac{\ln \tau}{\tau}\frac{1}{M^{2}(\tau)}\le C \mu_{1}^{2}f^{2}(\tau)
	\end{align*}
	and thanks to Lemma \ref{Lemma109} we have
	\begin{align*}
		(R^{2}\frac{\ln \tau}{\tau})^{2}(\int_{B_{M}^{c}}|\phi|)^{2}&\le \frac{1}{\ln^{2q}\tau}(\mu_{1}f(\tau)R^{2}(\tau)\frac{1}{M^{2}(\tau)})^{2}\le C \mu_{1}^{2}f^{2}(\tau).
	\end{align*}
	We just proved
	\begin{align*}
		&\|\frac{\omega_{1}}{fR}\|\le C\big[\|\frac{a}{fR^{2}}\|+\mu_{1}\big]\le C\big[\|\frac{a}{fR^{2}}\|+\frac{\|h\|_{\star\star}}{R(\tau_{0})}+\frac{|c_{1}|}{f(\tau_{0})R^{2}(\tau_{0})}+\frac{1}{R(\tau_{0})}\|\frac{\omega_{1}}{fR}\|\big]\\
		&\implies \mu_{1}\le  C \big[\|\frac{a}{fR^{2}}\|+\frac{\|h\|_{\star\star}}{R(\tau_{0})}+\frac{|c_{1}|}{f(\tau_{0})R^{2}(\tau_{0})}\big] =:\mu_{2}.
	\end{align*}
	This means that we have
	\begin{align}
		&\int_{\mathbb{R}^{2}} Ug^{2}\le C \big(\frac{\|h\|_{\star\star}^{2}}{R(\tau_{0})}+\|\frac{a}{fR^{2}}\|^{2}+\frac{|c_{1}|^{2}}{f^{2}(\tau_{0})R^{4}(\tau_{0})}\big)f^{2}(\tau)R^{4}(\tau), \label{intUg2}\\
		&\int_{\mathbb{R}^{2}} U(g^{\perp})^{2}\le C \big(\|h\|_{\star\star}^{2}+\|\frac{a}{fR^{2}}\|^{2}+\frac{|c_{1}|^{2}}{f^{2}(\tau_{0})R^{4}(\tau_{0})}\big)f^{2}(\tau)R^{2}(\tau), \label{intUgperp2}\\
		&|Ug| \le C \big(\frac{\|h\|_{\star\star}}{R(\tau_{0})}+\|\frac{a}{fR^{2}}\|+\frac{|c_{1}|}{f(\tau_{0})R^{2}(\tau_{0})}\big)\frac{f(\tau)R^{2}(\tau)}{1+\rho^{4}} \label{estimateUg}.
	\end{align}
	We notice that to get \eqref{intUgperp2} we used Lemma \ref{QuadrForIne}, the inequality \eqref{firstineqGron} and the crucial observation \eqref{GronwallInitialCOND}. \newline
	Now we are interested in absorbing the parameter $a$. At least initially we proceed as in the proof of Lemma 10.1 in \cite{DdPDMW}. We multiply the equation by $|y|^{2}\chi_{0}(\frac{y}{R})$ and we integrate in space. We get
	\begin{align}\label{EqSecMomEstA}
		\partial_{ \tau}\int_{\mathbb{R}^{2}} \phi |y|^{2}\chi_{0}(\frac{y}{R})=\int_{\mathbb{R}^{2}} (L[\phi]+h)|y|^{2}\chi_{0}(\frac{y}{R})+\int_{\mathbb{R}^{2}} \tilde{B}[\phi]|y|^{2}\chi_{0}(\frac{y}{R})-\frac{R'(\tau)}{R(\tau)}\int_{\mathbb{R}^{2}} \phi |y|^{2}\nabla_{z} \chi_{0}(\frac{y}{R})\cdot \frac{y}{R}dy.
	\end{align}
	Since $L[\phi]=\nabla \cdot(U\nabla g)$ we have
	\begin{align}\label{SecMomL}
		\int_{\mathbb{R}^{2}} L[\phi]|y|^{2}\chi_{0}(\frac{y}{R})dy=&2\int_{\mathbb{R}^{2}} gZ_{0}\chi_{0}(\frac{y}{R})+\frac{4}{R}\int_{\mathbb{R}^{2}} gUy\cdot \nabla_{z} \chi_{0}(\frac{y}{R})+\frac{1}{R}\int_{\mathbb{R}^{2}} g|y|^{2}\nabla U \cdot \nabla_{z} \chi_{0}(\frac{y}{R})dy+\nonumber\\
		&+\frac{1}{R^{2}}\int_{\mathbb{R}^{2}} g|y|^{2}U\Delta_{z} \chi_{0}(\frac{y}{R}).
	\end{align}
	Recalling that $\int_{\R^2}gZ_{0}=0$, we see
	\begin{align*}
		&|2\int_{\mathbb{R}^{2}} gZ_{0}\chi_{0}(\frac{y}{R})|\le C \int_{B_{R}^{c}}U|g|\le C \big(\frac{\|h\|_{\star\star}}{R(\tau_{0})}+\|\frac{a}{fR^{2}}\|+\frac{|c_{1}|}{f(\tau_{0})R^{2}(\tau_{0})}\big)f(\tau).
		\end{align*}
	The other terms in \eqref{SecMomL} can be estimated similarly. We clearly also have
	\begin{align*}
		|\int_{\mathbb{R}^{2}} h |y|^{2}\chi_{0}(\frac{y}{R})|\le C\|h\|_{\star\star}f(\tau).
	\end{align*}
    We proceed estimating the right-hand side of \eqref{EqSecMomEstA}. Since $|\frac{R'(\tau)}{R(\tau)}|\le C \frac{1}{\tau}$ and by Lemma \ref{Lemma109} we see
	\begin{align*}
		|\frac{R'}{R}\int_{\mathbb{R}^{2}} \phi |y|^{2}\nabla_{z} \chi_{0}(\frac{y}{R})\cdot \frac{y}{R}|&\le C \big(\frac{\|h\|_{\star\star}}{R(\tau_{0})}+\|\frac{a}{fR^{2}}\|+\frac{|c_{1}|}{f(\tau_{0})R^{2}(\tau_{0})}\big)f(\tau).
	\end{align*}
	We just proved that
	\begin{align*}
		|\int_{\mathbb{R}^{2}} (L[\phi]+h)|y|^{2}\chi_{0}(\frac{y}{R})-\frac{R'(\tau)}{R(\tau)}\int_{\mathbb{R}^{2}} \phi |y|^{2}\nabla_{z} \chi_{0}(\frac{y}{R})\cdot \frac{y}{R}dy|\le C \big(\|h\|_{\star\star}+\|\frac{a}{fR^{2}}\|+\frac{|c_{1}|}{f(\tau_{0})R^{2}(\tau_{0})}\big)f(\tau).
	\end{align*}
	The most delicate term in \eqref{EqSecMomEstA} is $\int_{\mathbb{R}^{2}} \tilde{B}[\phi]|y|^{2}\chi_{0}(\frac{y}{R})$ where we recall that the operator $\tilde{B}$ has been introduced in \eqref{tildeB}.  Observing that $W_{0}$ has compact support and second moment equal to zero and since, by \eqref{FunctionM} and \eqref{lambdadotlambdaExp} we have $M^{2}(\tau)\approx c\frac{\delta \tau}{\ln\tau}\gg \frac{\tau}{(\ln \tau)^{1+q}}=R^{2}(\tau)$  we get
	\begin{align*}
		\int_{\mathbb{R}^{2}} \tilde{B}[\phi]|y|^{2}\chi_{0}(\frac{y}{R})=&\lambda\dot{\lambda}\int_{\mathbb{R}^{2}}(k\phi\chi+y\cdot \nabla (\phi \chi))|y|^{2}\chi_{0}(\frac{y}{R})=\\
		=&k\lambda\dot{\lambda}\int_{\mathbb{R}^{2}} \phi |y|^{2}\chi_{0}(\frac{y}{R})-\lambda\dot{\lambda} \int_{\mathbb{R}^{2}} \phi\nabla \cdot (y|y|^{2}\chi_{0}(\frac{y}{R}))=\\
		=&\lambda\dot{\lambda}(k-4)\int_{\mathbb{R}^{2}} \phi |y|^{2}\chi_{0}(\frac{y}{R})+O\big(|\lambda\dot{\lambda}|\int_{\mathbb{R}^{2}} |\phi| |y|^{2}\nabla_{z}\chi_{0}(\frac{y}{R(\tau)})\cdot \frac{y}{R} dy \big).
	\end{align*}
	The last term can be estimated observing that
	\begin{align*}
		|\lambda\dot{\lambda}|\int |\phi| |y|^{2}\nabla_{z}\chi_{0}(\frac{y}{R(\tau)})\cdot \frac{y}{R}|&\le \big(\|h\|_{\star\star}+\|\frac{a}{fR^{2}}\|+\frac{|c_{1}|}{f(\tau_{0})R^{2}(\tau_{0})}\big)f(\tau).
	\end{align*}
	We just proved that we can rewrite \eqref{EqSecMomEstA} as
	\begin{align}\label{EqSecMomFin}
		\partial_{ \tau}\int \phi |y|^{2}\chi_{0}(\frac{y}{R})-\lambda\dot{\lambda}(k-4)\int \phi |y|^{2}\chi_{0}(\frac{y}{R})=F(\tau)\le C \big(\|h\|_{\star\star}+\|\frac{a}{fR^{2}}\|+\frac{|c_{1}|}{f(\tau_{0})R^{2}(\tau_{0})}\big)f(\tau).
	\end{align}
	We introduce the notation $S(\tau):=\int \phi |y|^{2}\chi_{0}(\frac{y}{R})$ and then the equation \eqref{EqSecMomFin} can be rewritten as
     \begin{align}\label{EqSecFinFIN}
     	\partial_{ \tau}S(\tau)-\lambda\dot{\lambda}(\tau)(k-4)S(\tau)=F(\tau), \ \ \ \tau\in(\tau_{0},\tau_{1}).
     \end{align}
	 Recalling \eqref{tauVAR}, if we write $\tilde{S}(t)=S(\tau)$ and $\tilde{F}(t)=F(\tau)$, we get
	\begin{align*}
		\lambda^{2}(t)\partial_{t}\tilde{S}(t)-\lambda\dot{\lambda}(t)(k-4)\tilde{S}(t)=\tilde{F}(t)&\iff \partial_{t}\big(e^{\frac{4-k}{2}\ln \lambda^{2}}\tilde{S}(t)\big)= \lambda^{2-k}\tilde{F}(t), \ \ \ t\in(0,T_{1}).
	\end{align*}
	In the following we will need $k\le 2$. Indeed if $\tau(T_{1})=\tau_{1}$ integrating in time we get
	\begin{align*}
		&e^{\frac{4-k}{2}\ln \lambda^{2}(T_{1})}\tilde{S}(T_{1})-e^{\frac{4-k}{2}\ln \lambda^{2}(t)}\tilde{S}(t)=\int_{t}^{T_{1}}\lambda^{2-k}(s)\tilde{F}(s)ds\\
		&\implies -\tilde{S}(t)=\lambda^{k-4}(t)\int_{t}^{T_{1}}\lambda^{2-k}(s)\tilde{F}(s)ds-\frac{\lambda^{4-k}(T_{1})}{\lambda^{4-k}(t)}\tilde{S}(T_{1}).
	\end{align*}
 Since $T_{1}\ge t$ and because of \eqref{lambdadotlambdaExp} we have $|\lambda^{4-k}(T_{1})/\lambda^{4-k}(t)|\le 2$. Moreover we know that by hypothesis  $a(\tau_{1})=0$. This means that since  $\phi=\phi^{\perp}+\frac{a}{2}Z_{0}$ we have
	\begin{align}
		-\frac{a(\tau)}{2}&\int_{\mathbb{R}^{2}} Z_{0}|y|^{2}\chi_{0}(\frac{y}{R(\tau)})= \nonumber\\
		&=\frac{\lambda^{k-2}(t)}{\lambda^{2}(t)}\int_{t}^{T_{1}}\lambda^{2-k}(s)\tilde{F}(s)ds+O(\int_{\mathbb{R}^{2}} |\phi^{\perp}(\tau_{1})||y|^{2}\chi_{0}(\frac{y}{R(\tau_{1})}))+O(\int_{\mathbb{R}^{2}} |\phi^{\perp}(\tau)||y|^{2}\chi_{0}(\frac{y}{R(\tau)}))\label{Equation for A}.
	\end{align}
	By Lemma \ref{Lemma95} and \eqref{intUgperp2} we have
	\begin{align}
		|\int_{\mathbb{R}^{2}} |\phi^{\perp}(\tau)| |y|^{2}\chi_{0}(\frac{y}{R(\tau)})|&\le C R(\tau)(\int |\phi^{\perp}|^{2}\frac{1}{U})^{1/2}\le C R(\tau)(\int U(g^{\perp})^{2})^{1/2}\le\nonumber \\&\le C \big(\|h\|_{\star\star}+\|\frac{a}{fR^{2}}\|+\frac{|c_{1}|}{f(\tau_{0})R^{2}(\tau_{0})}\big)f(\tau)R^{2}(\tau)\label{inequalityphiperpsecmomtau}
	\end{align}
	and since $f(\tau)R^{2}(\tau)$ is decreasing, we also have
	\begin{align}
		|\int |\phi^{\perp}(\tau_{1})| |y|^{2}\chi_{0}(\frac{y}{R(\tau)})|&\le \big(\|h\|_{\star\star}+\|\frac{a}{fR^{2}}\|+\frac{|c_{1}|}{f(\tau_{0})R^{2}(\tau_{0})}\big)f(\tau)R^{2}(\tau)\label{inequalityphiperpsecmomtau1}.
	\end{align}
    We notice that we just used \eqref{intUgperp2} and in what follows it is crucial the extra smallness we got from the initial condition in \eqref{firstineqGron} and the inequality \eqref{GronwallInitialCOND}. \newline From \eqref{EqSecMomFin} we know that $F(\tau)\le C \big(\|h\|_{\star\star}+\|\frac{a}{fR^{2}}\|+\frac{|c_{1}|}{f(\tau_{0})R^{2}(\tau_{0})}\big)f(\tau)$. We also observe that by \eqref{lambdadotlambdaExp} we have
    \begin{align}\label{ExpansionLambda}
    	\lambda(t)\approx \bar{\lambda}(t)=2e^{-\frac{\gamma+2}{2}}e^{-\sqrt{\frac{|\ln(T-t)|}{2}}}.
    \end{align}
    Then since $k\le 2$, by \eqref{ExpansionLambda} and by \eqref{lambdadotlambdaExp}, we see
	\begin{align}
		|\frac{\lambda^{k-2}(t)}{\lambda^{2}(t)}\int_{t}^{T_{1}}\lambda^{2-k}(s)\tilde{F}(s)ds|&\le C \frac{\bar{\lambda}^{k-2}(t)}{\bar{\lambda}^{2}(t)}\big(\|h\|_{\star\star}+\|\frac{a}{fR^{2}}\|+\frac{|c_{1}|}{f(\tau_{0})R^{2}(\tau_{0})}\big)f(\tau)\bar{\lambda}^{2-k}(t)(T_{1}-t)= \nonumber\\
		&=C \frac{T-t}{\bar{\lambda^{2}}(t)}\big(\|h\|_{\star\star}+\|\frac{a}{fR^{2}}\|+\frac{|c_{1}|}{f(\tau_{0})R^{2}(\tau_{0})}\big)f(\tau)\le \nonumber\\
		&\le C\big(\|h\|_{\star\star}+\|\frac{a}{fR^{2}}\|+\frac{|c_{1}|}{f(\tau_{0})R^{2}(\tau_{0})}\big) \frac{\tau}{\ln \tau} f(\tau) \label{inequalityRHSEqSecMom}.
	\end{align}
	Now, since
	\begin{align*}
		-\frac{a}{2} \int Z_{0}(y)|y|^{2}\chi_{0}(\frac{y}{R})\approx 16\pi a(\tau)\int_{1}^{R}\frac{d\rho}{\rho}\approx 16\pi a(\tau)\log R
	\end{align*}
	equation \eqref{Equation for A} and inequalities \eqref{inequalityphiperpsecmomtau}, \eqref{inequalityphiperpsecmomtau1}, \eqref{inequalityRHSEqSecMom} give
	\begin{align}
		&|a(\tau)| \le C\big(\|h\|_{\star\star}+\|\frac{a}{fR^{2}}\|+\frac{|c_{1}|}{f(\tau_{0})R^{2}(\tau_{0})}\big) [\frac{\tau}{\ln^{2}\tau}f(\tau)+\frac{1}{\ln \tau}f(\tau)R^{2}(\tau)] \label{InequalitYA}\\
		&\implies \|\frac{a}{f R^{2}}\|\le C\big(\|h\|_{\star\star}+\|\frac{a}{fR^{2}}\|+\frac{|c_{1}|}{f(\tau_{0})R^{2}(\tau_{0})}\big)[\frac{\tau_{0}}{\ln^{2}\tau_{0}}\frac{1}{R^{2}(\tau_{0})}+\frac{1}{\ln \tau_{0}}]\nonumber\\
		&\implies \|\frac{a}{f R^{2}}\|\le \frac{C}{\log ^{1-q} \tau_{0}}\big(\|h\|_{\star\star}+\frac{|c_{1}|}{f(\tau_{0})R^{2}(\tau_{0})}\big)\label{estimatingaabs}
	\end{align}
    where in the last inequality we used $q\in(0,1)$. \newline
    As a last step we observe that $a(\tau_{0})$ and $c_{1}$ are related. Indeed since $\phi|_{\tau_{0}}=c_{1}\tilde{Z}_{0}$, \eqref{ComputingA} gives
    \begin{align}\label{ac1estim}
    	a(\tau_{0})=\frac{c_{1}}{8\pi}\int_{\R^2}\tilde{Z}_{0}\Gamma_{0} \implies |c_{1}|\le C|a(\tau_{0})|.
    \end{align}
    Then from \eqref{estimatingaabs} we have
    \begin{align}\label{EstimatingA}
    	\|\frac{a}{f R^{2}}\|\le \frac{C}{\log ^{1-q} \tau_{0}}\|h\|_{\star\star}
    \end{align}
    and from \eqref{ac1estim} and \eqref{EstimatingA} we have
    \begin{align}\label{EstimatingC1}
    	\frac{|c_{1}|}{f(\tau_{0})R^{2}(\tau_{0})}\le C \frac{C}{\log^{1-q}\tau_{0}}\|h\|_{\star\star}.
    \end{align}
     Inequalities \eqref{EstimatingA} and \eqref{EstimatingC1} gives $(\frac{\|h\|_{\star\star}}{R(\tau_{0})}+\|\frac{a}{fR^{2}}\|+\frac{|c_{1}|}{f(\tau_{0})R^{2}(\tau_{0})})\le \frac{1}{\ln^{1-q}\tau_{0}}\|h\|_{\star\star}$ and we have
	\begin{align*}
		|\phi|\le C \frac{f(\tau)R^{2}(\tau)}{(\log^{1-q} \tau_{0})}\frac{\|h\|_{\star\star}}{1+\rho^{4}}.
	\end{align*}
\end{proof}
We observe that the role of the cut-off was crucial to prove Lemma \ref{Lemma101}. Without cutting the operator $B$ (or cutting at a higher distance) this argument is not expected to work. The main issue is that we would see an intermediate region in the construction of the barriers in Lemma \ref{barriersUg}. In this intermediate region $Ug$ will decay only like $\frac{1}{\rho^{2}}$ (that is in the kernel of the operator $B$) and this would make impossible to absorb $\omega_{0}$ and $\omega_{1}$ (see for instance \eqref{DECAYPHIREMARK}).
\begin{proof}[Proof of Lemma \ref{Lemma102}]
	Let us rename the norm \eqref{normtau0tau1} we used in the proof of Lemma \ref{Lemma101} and write
	\begin{align*}
		\|x\|_{1}:=\sup_{\tau\in [\tau_{0},\tau_{1}]}|x(\tau)|.
	\end{align*}
	In this case we are not assuming $a(\tau_{1})=0$ and then inequality \eqref{InequalitYA} in the proof of Lemma \ref{Lemma101} becomes
	\begin{align}\label{ineqAtau1}
		|a(\tau)|\log \tau \le C f(\tau)R^{2}(\tau)\log^{q}\tau(\|h\|_{\star\star}+\|\frac{a}{fR^{2}}\|_{1}+\frac{|c_{1}|}{f(\tau_{0})R^{2}(\tau_{0})})+C|a(\tau_{1})| \log \tau_{1}.
	\end{align}
	Since we are assuming $\frac{a}{f R^{2}}\in L^{\infty}(\tau_{0},\infty)$, we can write $ \lim\limits_{\tau \to \infty} a(\tau)\ln \tau=0$ (here we use that the solution is of class $C^{0}$ in time). Since all the constants are independent of $\tau_{1}$ we can let $\tau_{1}\to \infty$ in \eqref{ineqAtau1} and obtain
	\begin{align*}
		|a(\tau)|\ln \tau \le C (\|h\|_{\star\star}+\|\frac{a}{fR^{2}}\|_{\infty}+\frac{|c_{1}|}{f(\tau_{0})R^{2}(\tau_{0})}) f(\tau)R^{2}(\tau)\ln^{q}\tau.
	\end{align*}
	The proof can be concluded as for the proof of Lemma \ref{Lemma101}.
\end{proof}
\begin{proof}[Proof of Lemma \ref{Lemma103}]
	We are considering 
	\begin{align}\label{SysZB}
		\begin{cases}
			\partial_{ \tau}Z_{B}=L[Z_{B}]+\tilde{B}[Z_{B}]\\
			Z_{B}(\cdot,\tau_{0})=\tilde{Z}_{0}.
		\end{cases}
	\end{align}
	Assume $\exists$ $\tau_{1}>\tau_{0}$ such that $a_{Z}(\tau_{1})=0$, by substituting $h=0$ and $c_{1}=1$ in the inequality \eqref{InequalitYA} of Lemma \ref{Lemma101}, we get
		\begin{align}\label{FirstIneqContr}
			|a|\le C (\|\frac{a}{fR^{2}}\|+\frac{1}{f(\tau_{0})R^{2}(\tau_{0})})\frac{f(\tau)R^{2}(\tau)}{\ln^{1-q}\tau_{0}} \implies \|\frac{a}{fR^{2}}\|\le C\frac{1}{\ln^{1-q}\tau_{0}} \frac{1}{f(\tau_{0})R^{2}(\tau_{0})}.
		\end{align}
	    At the same time by \eqref{ComputingA} we have
	    \begin{align}\label{SecondINeqContr}
	    	a(\tau_{0})=\frac{1}{8\pi}\int \Gamma_{0} \tilde{Z}_{0}=2+O(\frac{\ln \tau_{0}}{\tau_{0}})\gg \frac{1}{\ln^{1-q}\tau_{0}}.
	    \end{align}
		Clearly \eqref{FirstIneqContr} and \eqref{SecondINeqContr} give us a contradiction that concludes the proof.
\end{proof}

\begin{proof}[Proof of Lemma \ref{Lemma104}]\label{ProofLemma104}
	Take a sequence $\tau_{n}\to \infty$ as $n\to \infty$ and $\bar{\phi}$
	\begin{align*}
		\begin{cases}
			\partial_{ \tau}\bar{\phi}=L[\bar{\phi}]+\widetilde{B}[\bar{\phi}]+h\\
			\bar{\phi}(\cdot, \tau_{0})=0
		\end{cases}
	\end{align*}
	we call $\bar{\phi}^{\perp}$, $\bar{a}(\tau)$ the usual decomposition \eqref{DecomPositionPhiPhiperp}. We also call $Z_{B}^{\perp}$, $a_{Z}(\tau)$ the usual decomposition for the solution of \eqref{SysZB}. By Lemma \ref{Lemma103} we know that $a_{Z}(\tau)\neq 0$ $\forall \tau \ge \tau_{0}$, then there is a $c_{n}\in \mathbb{R}$ such that 
	\begin{align*}
		\bar{a}(\tau_{n})+c_{n}a_{Z}(\tau_{n})=0.
	\end{align*} 
	We define $\phi _{n}=\bar{\phi}+ c_{n}Z_{B}$. Also for this function we have the usual decomposition $\phi_{n}=\phi^{\perp}_{n}+\frac{a_{n}}{2}Z_{0}=\bar{\phi}^{\perp}+c_{n}Z_{B}^{\perp}+\frac{\bar{a}(\tau)+c_{n}a_{Z}(\tau)}{2}Z_{0}$. By \eqref{ComputingA} we see that
	\begin{align*}
		a_{n}=\frac{1}{8\pi}\int_{\R^2}\Gamma_{0}\phi_{n}=\frac{c_{n}}{8\pi}\int_{\R^2}\Gamma_{0}Z_{B}+\frac{1}{8\pi}\int_{\R^2}\Gamma_{0}\bar{\phi}=c_{n}a_{Z}+\bar{a}.
	\end{align*}
	This implies that $a_{n}(\tau_{n})\neq0$ and then that $\phi_{n}$ satisfies
		\begin{align*}
			\begin{cases}
				\partial_{ \tau}\phi_{n}=L[\phi_{n}]+\widetilde{B}[\phi_{n}]+h\\
				\phi_{n}(\cdot, \tau_{0})=c_{n}\widetilde{Z}_{0}\\
				\phi_{n}=\phi_{n}^{\perp}+\frac{a_{n}}{2}Z_{0}, \ \ \text{with }a_{n}(\tau_{1})=0.
			\end{cases}
	\end{align*} We can apply Lemma \ref{Lemma101} and obtain
	\begin{align*}
		&|a_{n}(\tau)|\le C \frac{f(\tau)R^{2}(\tau)}{\ln^{1-q} \tau_{0}}\|h\|_{\star\star} \ \ \ \tau\in[\tau_{0},\tau_{1}],\\
		&|c_{n}|\le C \frac{f(\tau_{0}) R^{2}(\tau_{0})}{\ln^{1-q}\tau_{0}}\|h\|_{\star\star},\\
		&|\phi_{n}|\le C \frac{f(\tau)R^{2}(\tau)}{\ln^{1-q}\tau_{0}}\frac{\|h\|_{\star\star}}{1+\rho^{4}}\ \ \ \tau\in[\tau_{0},\tau_{1}].
	\end{align*}
	We take the limit as $n\to \infty$. The interested reader can check the details in Lemma 10.4 of \cite{DdPDMW}.
\end{proof}
With the proof of Lemma \ref{Lemma104} it is finally clear what is the role of $c_{1}$. In fact we used it to erase $a(\tau_{1})$ so that we can solve the equation \eqref{EqSecMomEstA} for a decaying $a(\tau)$. If we had $\phi(\cdot,\tau_{0})=0$ the correspondent solution of \eqref{EqSecMomEstA} (now we could not assume $a(\tau_{1})=0$) would not be decaying (we would need to integrate from $\tau_{0}$ to $\tau$).\newline
\begin{proof}[Proof of Proposition \ref{PropMode0Mass0}]
	So far we only know
	\begin{align*}
		|\phi(\rho,\tau)|\le C \frac{\|h\|_{\star\star}}{\ln^{1-q}\tau_{0}}\frac{f(\tau)R^{2}(\tau)}{1+\rho^{4}}.
	\end{align*}
    To improve the decay at infinity we construct a barrier in $B_{R_{0}}^{c}$ similarly to how we proceeded in the proof of Lemma \ref{barriersUg}. Recalling the decomposition $\tilde{B}[\phi]=\tilde{B}_{0}[\phi]-W_{0}(y)\lambda\dot{\lambda}\int_{\mathbb{R}^{2}} (k\phi \chi+y\cdot{\nabla}(\phi \chi))dy$, if $R_{0}$ is sufficiently large we have
	\begin{align*}
		\partial_{ \tau}\phi=\Delta \phi -\nabla \Gamma_{0} \cdot\nabla \phi +\tilde{B}_{0}[\phi]+2U\phi +\tilde{h}, \ \ \ \ \tilde{h}=-\nabla U\cdot\nabla \psi +h.
	\end{align*}
     Since $\psi=(-\Delta)^{-1}\phi$ with $\int_{\mathbb{R}^{2}}\phi=0$ we get
	\begin{align*}
		|\nabla \psi|\le C \frac{\|h\|_{\star\star}}{\ln ^{1-q}\tau_{0}}\frac{f(\tau)R^{2}(\tau)}{1+\rho^{3}}, \ \ \ \ |\nabla U \cdot\nabla\psi |\le C \frac{\|h\|_{\star\star}}{\ln^{1-q}\tau_{0}}\frac{f(\tau)R^{2}(\tau)}{1+\rho^{8}}.
	\end{align*}
	Moreover we have
	\begin{align*}
		|h|\le C \|h\|_{\star\star} \frac{f(\tau)}{1+|y|^{6+\sigma}} \min (1, \frac{\tau^{\varepsilon/2}}{|y|^{\varepsilon}})\le C \frac{\|h\|_{\star\star}}{\ln^{1-q}\tau_{0}}\frac{f(\tau)R^{2}(\tau)}{1+|y|^{6+\sigma}}\min(1,\frac{\tau^{\varepsilon/2}}{|y|^{\varepsilon}})
	\end{align*}
	since $\frac{R^{2}(\tau)}{\ln^{1-q}\tau_{0}}\ge \frac{R^{2}(\tau_{0})}{\ln^{1-q}\tau_{0}}\ge 1$. Then
	\begin{align*}
		|\tilde{h}|\le C \frac{\|h\|_{\star\star}}{\ln^{1-q}\tau_{0}}\frac{f(\tau)R^{2}(\tau)}{1+|y|^{6+\sigma}}\min(1,\frac{\tau^{\varepsilon/2}}{|y|^{\varepsilon}})
	\end{align*}
	(where we notice that to control $\nabla U\cdot \nabla \psi$ we need $1/\rho^{2-\sigma}\le \tau^{\varepsilon/2}/\rho^{\varepsilon}$ where $\rho \ge \sqrt{\tau}$ and then we used the assumption $\sigma+\varepsilon<2$). Similarly to the proof of Lemma \ref{barriersUg} we are constructing a barrier like
	\begin{align*}
		\bar{\phi}(\rho,\tau)=f(\tau)R^{2}(\tau)\big[(...)\chi_{0}(\frac{\rho}{\sqrt{\gamma \tau}})+H(\rho,\tau)\big]
	\end{align*}
    where $H(\rho,\tau)$ is a barrier of the type described in \eqref{Hbarrier} and we omitted to write explicitly the barrier in the region $\rho\le 2\sqrt{\gamma\tau}$ since it is the same of the proof of Lemma \ref{barriersUg}. As we know form \eqref{Hbarrier}, $H(\rho,\tau)$ is actually a barrier only if some conditions are satisfied. In our case $\gamma=\nu+1$ (this is due to the spatial decay of the barrier in the inner region) and $b=6+\sigma+\varepsilon$ (because of the spatial decay of the right-hand side in the outer region). Consequently we need $\nu+1<3$ and $\nu+1<3+\frac{\sigma+\varepsilon}{2}$. For the initial condition the same observation we made in the proof of Lemma \ref{barriersUg} holds.
\end{proof}
\begin{proof}[Proof of Proposition \ref{Prop101}]
	From  \eqref{intUgperp2} and as a consequence of the estimates we found to prove Lemma \ref{Lemma101} we know
	 \begin{align}\label{IntegralCondProp101}
	 	\int_{\mathbb{R}^{2}} U(g^{\perp})^{2}\le C f^{2}(\tau)R^{2}(\tau)\|h\|_{\star\star}^{2}. 
	 \end{align}
	 From this inequality we claim that
	\begin{align}\label{Ugperppointwise}
		U|g^{\perp}|\le C f(\tau)R(\tau)\frac{\|h\|_{\star\star}}{1+\rho^{2}} \ \ \tau>\tau_{0}.
	\end{align}
	To prove \eqref{Ugperppointwise}, we define $g_{0}^{\perp}:=Ug^{\perp}=U(g+a)$ and, denoting $\psi=\psi[g_{0}]$, we obtain the equation
	\begin{align*}
		\partial_{ \tau}g_{0}^{\perp}=&\nabla \cdot[U\nabla(\frac{g_{0}^{\perp}}{U})]-U(-\Delta)^{-1}(\nabla \cdot(U\nabla(\frac{g_{0}^{\perp}}{U})))+h-U(-\Delta)^{-1}h+\\
		&+ \tilde{B}[g_{0}+U\psi[g_{0}]]-U(-\Delta)^{-1}(\tilde{B}[g_{0}+U\psi[g_{0}]])+a'U.
	\end{align*}
    Now since $g_{0}+U\psi[g_{0}]=\phi=\phi^{\perp}+\frac{a}{2}Z_{0}=g_{0}^{\perp}+U\psi^{\perp}[g_{0}^{\perp}]+\frac{a}{2}Z_{0}$, we have
	\begin{align}\label{equationforg0perp}
		\partial_{ \tau}g_{0}^{\perp}=&\nabla \cdot[U\nabla(\frac{g_{0}^{\perp}}{U})]-U(-\Delta)^{-1}(\nabla \cdot(U\nabla(\frac{g_{0}^{\perp}}{U})))+h-U(-\Delta)^{-1}h+ \nonumber\\
		&+ \tilde{B}[g_{0}^{\perp}]+\tilde{B}[U\psi^{\perp}[g_{0}^{\perp}]]+\frac{a}{2}\tilde{B}[Z_{0}]-U(-\Delta)^{-1}(\tilde{B}[g_{0}^{\perp}+U\psi^{\perp}[g_{0}^{\perp}]])+\nonumber\\
		&-\frac{a}{2}U(-\Delta)^{-1}\tilde{B}[Z_{0}]+a'U.
	\end{align}
	The idea is to repeat the argument we used to prove Lemma \ref{Lemma107} but then we need an estimate for $a'(\tau)$. We observe that
	\begin{align*}
		\partial_{ \tau}g_{0}=&\nabla \cdot (U\nabla \frac{g_{0}}{U})-U(-\Delta)^{-1}(\nabla \cdot (U\nabla g))+h-U(-\Delta)^{-1}h+\\
		&+\tilde{B}[g_{0}+U\psi[g_{0}]]-U(-\Delta)^{-1}(\tilde{B}[g_{0}+U\psi[g_{0}]])\\
		&\hspace{-1cm}\implies \partial_{ \tau}\int_{\mathbb{R}^{2}} g_{0}=-\int_{\mathbb{R}^{2}} U(-\Delta)^{-1}(\nabla \cdot (U\nabla g))-\int_{\mathbb{R}^{2}} U (-\Delta)^{-1}h-\int_{\mathbb{R}^{2}} U(-\Delta)^{-1}(\tilde{B}[g_{0}+U\psi[g_{0}]])
	\end{align*}
	and, since by \eqref{ComputingA} $a(\tau)=-\frac{1}{8\pi}\int_{\mathbb{R}^{2}} Ug=-\frac{1}{8\pi} \int_{\mathbb{R}^{2}} g_{0}$, we have
	\begin{align}\label{a'Id}
		a'(\tau)=\frac{1}{8\pi} \int_{\mathbb{R}^{2}} U(-\Delta)^{-1}(\nabla \cdot ( U\nabla (\frac{g_{0}}{U})))+\frac{1}{8\pi}\int U(-\Delta)^{-1}h+\frac{1}{8\pi}\int_{\mathbb{R}^{2}} U (-\Delta)^{-1}(\tilde{B}[g_{0}+U\psi[g_{0}]]).
	\end{align}
	We claim that 
	\begin{align}\label{a'estimate}
		|a'(\tau)|\le C \|h\|_{\star\star} f(\tau)R(\tau).
	\end{align}
	We can start from the first two terms in the right-hand side of \eqref{a'Id}, by \eqref{IntegralCondProp101} we have
	\begin{align*}
		&\int_{\mathbb{R}^{2}} U(-\Delta)^{-1}(\nabla \cdot (U\nabla (\frac{g_{0}}{U})))=\int_{\mathbb{R}^{2}} \Gamma_{0}\nabla \cdot (U\nabla g^{\perp})=-\int_{\mathbb{R}^{2}}\nabla U \cdot\nabla g^{\perp}=\int_{\mathbb{R}^{2}} \Delta U g^{\perp} \\
		&\implies |\int_{\mathbb{R}^{2}} U(-\Delta)^{-1}(\nabla \cdot (U\nabla \frac{g_{0}}{U}))|\le C (\int_{\mathbb{R}^{2}} U (g^{\perp})^{2})^{1/2}\le C \|h\|_{\star\star}f(\tau)R(\tau)
	\end{align*}
	and $|\int_{\mathbb{R}^{2}} U (-\Delta)^{-1}h| \le C \|h\|_{\star\star}f(\tau)\le C \|h\|_{\star\star}f(\tau)R(\tau)$. For the last term of \eqref{a'Id} we have
	\begin{align*}
		\int_{\mathbb{R}^{2}} U (-\Delta)^{-1}(\tilde{B}[\phi])=&k\lambda\dot{\lambda}\int_{\mathbb{R}^{2}}\Gamma_{0} \phi \chi-\lambda\dot{\lambda}\int_{\mathbb{R}^{2}} \nabla \cdot (y\Gamma_{0})\phi \chi-\lambda\dot{\lambda}(k-2)\int_{\mathbb{R}^{2}} \phi \chi\int_{\mathbb{R}^{2}} W_{0}\Gamma_{0}.
	\end{align*}
	By Lemma \ref{Lemma104} we know $|\phi|\le C \|h\|_{\star\star}\frac{f(\tau)R^{2}(\tau)}{1+\rho^{4}}$ and then by \eqref{lambdadotlambdaExp} we get
	\begin{align*}
		|k\lambda\dot{\lambda}\int_{\mathbb{R}^{2}} \Gamma_{0} \phi \chi-\lambda\dot{\lambda}\int_{\mathbb{R}^{2}} \nabla \cdot (y\Gamma_{0})\phi \chi|\le C \frac{\ln \tau}{\tau}\|h\|_{\star\star}f(\tau)R^{2}(\tau)\le C \|h\|_{\star\star}f(\tau)R(\tau)
	\end{align*}
	Since $\int_{\mathbb{R}^{2}} \phi=0$ we also have
	\begin{align*}
		|\lambda\dot{\lambda}(k-2)\int_{\mathbb{R}^{2}} \phi \chi \int_{\mathbb{R}^{2}} W_{0}\Gamma_{0}|\le |\lambda\dot{\lambda}(k-2)\int_{\mathbb{R}^{2}} \phi(1-\chi)\int_{\mathbb{R}^{2}} W_{0}\Gamma_{0}| \le C \|h\|_{\star\star}f(\tau)R(\tau)
	\end{align*}
that gives the claim \eqref{a'estimate}.
	By \eqref{IntegralCondProp101} and observing that
	\begin{align*}
		&|a' U|\le C \|h\|_{\star\star}\frac{f(\tau)R(\tau)}{1+|y|^{4}},\\
		&|\frac{a}{2}\tilde{B}[Z_{0}]|\le C \|h\|_{\star\star} f(\tau)R^{2}(\tau)\frac{\ln \tau}{\tau}\frac{1}{1+|y|^{4}}\le C \|h\|_{\star\star}\frac{f(\tau)R(\tau)}{1+|y|^{4}},\\
		&|\frac{a}{2}U(-\Delta)^{-1}(\tilde{B}[Z_{0}])|\le C \|h\|_{\star\star}\frac{f(\tau)R(\tau)}{1+|y|^{4}}
	\end{align*}
    we can repeat the same argument of the proof of Lemma \ref{Lemma107} (and Lemma 10.7 in \cite{DdPDMW}) to obtain
\begin{align*}
	|g_{0}^{\perp}|\le C \|h\|_{\star\star} \frac{f(\tau)R(\tau)}{1+|y|^{2}}.
\end{align*}
We can easily observe now that since if $\rho>1$ we have
\begin{align*}
	\psi^{\perp}=z_{0}\int_{\rho}^{\infty} \frac{1}{z_{0}^{2}(r)r}\int_{r}^{\infty}Ug^{\perp}z_{0}sdsdr\le \int_{\rho}^{\infty} \frac{1}{r} (\int_{r}^{\infty}U(g^{\perp})^{2}sds)^{1/2}(\int_{r}^{\infty}Usds)^{1/2}\le C\|h\|_{\star\star}\frac{f(\tau)R(\tau)}{1+\rho}
\end{align*}
and the same estimate holds if $\rho<1$ (using \eqref{psiperpglob}). Then, since $\phi^{\perp}-U\psi^{\perp}=Ug^{\perp}$, we also get
\begin{align*}
	|\phi^{\perp}|\le \|h\|_{\star\star}f(\tau)R(\tau) \frac{1}{1+|y|^{2}}.
\end{align*}
	Now we want to improve the decay of $\phi^{\perp}$ at infinity. We can construct a barrier for
	\begin{align*}
		\partial_{ \tau}g_{0}^{\perp}=\Delta g_{0}^{\perp} -\nabla g_{0}^{\perp}\cdot\nabla \Gamma_{0}+2Ug_{0}^{\perp}+\tilde{B}_{0}[g_{0}^{\perp}]+\tilde{h}_{1}
	\end{align*}
	where, thanks to \eqref{Ugperppointwise}, as in the proof of Lemma \ref{barriersUg}, we have $|\tilde{h}_{1}|\le C f(\tau)R(\tau)\frac{\|h\|_{\star\star}}{1+|y|^{4}}$. In fact, keeping in mind \eqref{equationforg0perp}, we see that the only additional terms satisfy
	\begin{align*}
		|\frac{a}{2}\tilde{B}[Z_{0}]-\frac{a}{2}U(-\Delta)^{-1}\tilde{B}[Z_{0}]+a'U|\le C \|h\|_{\star\star}\frac{f(\tau)R(\tau)}{1+|y|^{4}}.
	\end{align*}
	Our barrier (see also the proof of Lemma \ref{barriersUg}) will be
	\begin{align*}
		g^{\perp}_{0}(\rho,\tau)=f(\tau)R(\tau)\tilde{g}(\rho)\chi_{0}(\frac{\rho}{\varepsilon\sqrt{\tau}})+C_{H}H(\frac{\rho}{\sqrt{\tau}},\tau)
	\end{align*}
	where $-\Delta_{6}\tilde{g}=\frac{1}{1+\rho^{4}}$ such that $c_{1}\frac{1}{1+\rho^{2}}\le \tilde{g}\le \frac{c_{2}}{1+\rho ^{2}}$ and $H(\frac{\rho}{\sqrt{\tau}},\tau)$ is a function of the type \eqref{Hbarrier}. We take $b=4$ (because of the spatial decay of the right-hand side) and $\gamma=\nu+1/2$ (because of the spatial decay of the barrier in the inner region). Then to construct a barrier we need $\nu+\frac{1}{2}<2$ and we deduce
	\begin{align*}
		|g_{0}^{\perp}|\le C f(\tau)R(\tau)\|h\|_{\star\star}\frac{\min(1,\frac{\tau}{|y|^{2}})}{1+|y|^{2}}.
	\end{align*}
    We also remark that the initial condition if $\rho\le \sqrt{\tau_{0}} $  satisfies 
    \begin{align*}
    	Ug^{\perp}|_{\tau=\tau_{0}}&=c_{1}U\big(\frac{\tilde{Z}_{0}^{\perp}}{U}-(-\Delta)^{-1}(\tilde{Z}_{0}^{\perp})\big)=c_{1}U\big(\frac{\tilde{Z}_{0}^{\perp}}{U}+\int_{\rho}^{\infty}\frac{1}{r}\int_{r}^{\infty}\tilde{Z}_{0}^{\perp}(s)sdsdr\big)\le\\
    	&\le C |c_{1}| \frac{\ln\tau_{0}}{\tau_{0}}\frac{1}{1+\rho^{4}}\le C \frac{f(\tau_{0})R^{2}(\tau_{0})}{(\ln \tau_{0})^{1-q}}\frac{\ln \tau_{0}}{\tau_{0}}\frac{1}{1+\rho^{4}}\le C \frac{f(\tau_{0})}{1+\rho^{4}}
\end{align*} 
where we used $\int_{\R^2}\tilde{Z}_{0}^{\perp}=0$, \eqref{ZoperpExp} and \eqref{c1inf}. If $\rho \ge{\sqrt{\tau_{0}}}$, we can only say
\begin{align*}
	Ug^{\perp}|_{\tau=\tau_{0}}&\le C |c_{1}| \frac{1}{1+\rho^{4}}\le C f(\tau_{0})R^{2}(\tau_{0})\frac{1}{1+\rho^{4}}\le C\frac{f(\tau_{0})}{\ln^{2} \tau_{0}} \frac{\tau_{0}}{1+\rho^{4}}\le C f(\tau_{0})\frac{\tau_{0}}{1+\rho^{4}}.
\end{align*}
We observe that it is important to cut the initial condition where $\rho\le 2\sqrt{\tau}$ (and not for example where $\rho \le 2\sqrt{\frac{\tau}{\ln \tau}}$) because of the poor decay of $Ug^{\perp}|_{\tau=\tau_{0}}$ where the cut-off is zero. \newline
Finally we observe that if $\rho>1$ we have
	\begin{align*}
		\psi^{\perp}=z_{0}\int_{\rho}^{\infty} \frac{1}{z_{0}^{2}(r)r}\int_{r}^{\infty}Ug^{\perp}z_{0}sdsdr\le \int_{\rho}^{\infty} \frac{1}{r} (\int_{r}^{\infty}U(g^{\perp})^{2}sds)^{1/2}(\int_{r}^{\infty}Usds)^{1/2}\le C\|h\|_{\star\star}\frac{f(\tau)R(\tau)}{1+\rho}
	\end{align*}
	then, since $\phi^{\perp}-U\psi^{\perp}=Ug^{\perp}$, we get
	\begin{align*}
		|\phi^{\perp}|\le \|h\|_{\star\star}f(\tau)R(\tau) \frac{\min(1,\frac{\tau}{|y|^{2}})}{1+|y|^{2}}.
	\end{align*}
\end{proof}

\section{Linear estimate with second moment (radial)}\label{Section11}
In this section we want to prove Proposition \ref{PropMode0MassSecMom0}, namely for some initial data $\phi_{0}$ and some operator $\mathcal{E}$ we want we want to find estimates for the solution of
\begin{align}
	\begin{cases}
		\partial_{\tau}\phi = L[\phi]+B[\phi]+\mathcal{E}[\phi]+h, \ \ \ \ \text{with }h\text{ radial such that }\int_{\R^2}hdy=0,\ \ \int_{\R^2}h|y|^{2}dy=0 \\
		\phi(\cdot, \tau_{0})=\phi_{0}
	\end{cases}
\end{align}
(where, as in the previous section, we are considering \eqref{FunctionM} instead of keeping the notation $\hat{\chi}$).
We are also interested in estimating $\mathcal{E}[\phi]$ and to show that its mass and second moment are small enough to be included in the orthogonality conditions \eqref{orthogonalitycond}. \newline
In what follows we want to give a sketch of the proof and in particular we want to show what is the idea behind the construction of the force $\mathcal{E}[\phi]$.
We start considering
\begin{align}\label{prototypesecmom}
	\begin{cases}
		\partial_{ \tau}\Phi=L[\Phi]+\tilde{B}_{1}[\Phi]+\lambda\dot{\lambda}L^{-1}[\nabla \cdot (\Phi Uy)-\nabla \cdot (Z_{0}\nabla \Psi)]+L^{-1}[h]\\
		\Phi(\cdot, \tau_{0})=c_{1}\tilde{Z}_{0}\\
		\int_{\mathbb{R}^{2}} \Phi=0
	\end{cases}
\end{align}
with $\tilde{B}_{1}[\Phi]=\lambda\dot{\lambda}y\cdot \nabla(\Phi\chi)-\lambda\dot{\lambda}(\int_{\mathbb{R}^{2}} y \cdot \nabla(\Phi\chi)) W_{0}(y)$. We notice here that $L^{-1}[h]$ has zero mass (thanks to $\int_{\R^2}h=\int_{\R^2}h|y|^{2}=0$ and Lemma \ref{ellipticsol}) and that the operator $\tilde{B}_{1}$ is included in the class of operators we studied in the previous section, \eqref{tildeB} (here $k=0$). In fact, in order to solve \eqref{prototypesecmom}, we want to apply Proposition \ref{PropMode0Mass0} and to include the extra term as a small perturbation (we will prove that it has zero mass). \newline
A fundamental remark is that $L^{-1}[h]$ gains more regularity in space (we are essentially integrating in space twice $h$), this will play a crucial role in estimating $\mathcal{E}[\phi]$. \newline If we apply $L$ to \eqref{prototypesecmom} we get, calling $\phi:=L[\Phi]$, that
\begin{align*}
	\begin{cases}
		\partial_{ \tau}\phi=L[\phi]+L\circ\tilde{B}_{1}[\Phi]+\lambda\dot{\lambda}(\nabla \cdot (\Phi Uy)-\nabla\cdot(Z_{0}\nabla\Psi))+h\\
		\phi(\cdot,\tau_{0})=c_{1}\hat{Z}_{0}:=c_{1}L[\tilde{Z}_{0}].
	\end{cases}
\end{align*} 
We need to compute the \emph{commutator} $[ L,\tilde{B}_{1}]$. In \cite{DdPDMW} it has been proved that if $\Psi=(-\Delta)^{-1}\Phi$, after calling $\Lambda[\phi]:=y\cdot \nabla \phi$ one has
\begin{align}\label{InfiniteTimeCommutator}
	\Lambda \circ L[\Phi]-L\circ \Lambda[\Phi]=\nabla \cdot (\Phi Uy)-2L[\Phi]-\nabla \cdot (Z_{0}\nabla\Psi)
\end{align}
where $Z_{0}$ has been introduced in \eqref{Z0} (for more details the interested reader can check formula (11.6) in \cite{DdPDMW}). We have then
\begin{align*}
	L\circ \Lambda (\Phi \chi)=&2\phi+\Lambda(\phi)-\nabla \cdot (\Phi Uy)+\nabla \cdot (Z_{0}\nabla \Psi)+L\circ \Lambda (\Phi(\chi-1)).
\end{align*}
Then $\phi$ satisfies
\begin{align*}
	\begin{cases}
		\partial_{ \tau}\phi=L[\phi]+\lambda\dot{\lambda}(2\phi+y\cdot \nabla \phi)+h-\lambda\dot{\lambda}(\int_{\mathbb{R}^{2}} \Lambda(\Phi \chi))L[W_{0}](y)+\lambda\dot{\lambda}L\circ\Lambda(\Phi(\chi-1)) \\
		\phi(\cdot,\tau_{0})=c_{1}\hat{Z}_{0}
	\end{cases}
\end{align*}
and we can write
\begin{align}\label{EfirstDef}
	\mathcal{E}[\phi]=-\lambda\dot{\lambda}(\int_{\mathbb{R}^{2}} \Lambda(\Phi \chi))L[W_{0}](y)+\lambda\dot{\lambda}L\circ\Lambda(\Phi(\chi-1)) =\lambda\dot{\lambda}L[-W_{0}(y)(\int_{\mathbb{R}^{2}} \Lambda(\Phi \chi))+\Lambda(\Phi\chi)-\Lambda\Phi].
\end{align}
By looking at \eqref{EfirstDef} it is clear why we need to control third-order derivatives of $\Phi$ that is consistent with the control we have of the right-hand side of \eqref{prototypesecmom}. We notice that $\int_{\R^2}\Phi=0$ and then $-\int_{\R^2}(-W_{0}(\int_{\R^2}\Lambda(\Phi \chi))+\Lambda(\Phi \chi)-\Lambda \Phi)=0$. As a consequence $\mathcal{E}[\phi]$ has zero mass and zero second moment.\newline As we will notice in Lemma \ref{PHIcostruction}, $\mathcal{E}[\phi]$ will contain some additional terms. \newline We write
\begin{align}\label{L0}
  L_{0}[\phi]:=L[\phi]+\nabla \cdot(U\nabla \psi)=\Delta\phi-\nabla\cdot(\phi\nabla \Gamma_{0})=\Delta \phi -\nabla \Gamma_{0}\cdot \nabla \phi +U\phi.
\end{align}
\begin{lemma} \label{Phi1Lemma}
	Let $\sigma>0$, $\varepsilon>0$ and $1<\nu<\min(1+\frac{\varepsilon}{2},3-\frac{\sigma}{2})$.  For any $H$ such that $\|H\|_{0;\nu,m,4+\sigma,\varepsilon}<\infty$ there exist $F$ and $\Phi_{1}$ such that
	\begin{align*}
		\begin{cases}
			\partial_{ \tau} \Phi_{1}=L[\Phi_{1}]+\tilde{B}_{1}[\Phi_{1}]+H-F\\
			\Phi_{1}(\cdot,\tau_{0})=0
		\end{cases}
	\end{align*}
	with 
	\begin{align*}
		\int_{\mathbb{R}^{2}} \Phi_{1}=0, \ \ \ |\Phi_{1}|\le C\frac{\|H\|_{0;\nu,m,4+\sigma,\varepsilon}}{\tau^{\nu}(\ln \tau)^{m}}\begin{cases}
			\frac{1}{(1+\rho)^{2+\sigma}} \ \ \ \rho\le \sqrt{\tau}\\
			\frac{\tau^{1+\varepsilon/2}}{(1+\rho)^{4+\sigma+\varepsilon}} \ \ \ \rho\ge \sqrt{\tau}
		\end{cases}
	\end{align*}
    and $F=F_{1}+F_{2}$ with
	\begin{align*}
		\int_{\mathbb{R}^{2}} F_{1}=0, \ \ |F_{1}|\le C \frac{\|H\|_{0;\nu,m,4+\sigma,\varepsilon}}{\tau^{\nu}(\ln \tau)^{m}}\begin{cases}
			\frac{1}{(1+\rho)^{6+\sigma}}, \ \ \ \ \rho\le \sqrt{\tau}\\
			\frac{\tau^{\varepsilon/2}}{(1+\rho)^{6+\sigma+\varepsilon}} \ \ \ \ \rho \ge \sqrt{\tau}
		\end{cases}
	\end{align*}
    and for some universal constant $\bar{c}$ we have
    \begin{align*}
    	\int_{\mathbb{R}^{2}} F_{2}=0, \ \ |F_{2}|\le C \frac{\|H\|_{0;\nu,m,4+\sigma,\varepsilon}}{\tau^{\nu}(\ln \tau)^{m}}\begin{cases}
    			\frac{\ln \tau}{\tau}\frac{1}{M^{\sigma}(\tau)} \ \ \ |y|\le \bar{c}\\
    			0 \ \ \ \ \bar{c}\le|y|\le M(\tau)\\
    			\frac{\delta}{1+\rho^{4+\sigma}} \ \ \ \ M(\tau)\le|y|\le 2M(\tau)\\
    			0 \ \ \ \ |y|\ge 2M(\tau).
    	\end{cases}
    \end{align*}
\end{lemma}
\begin{proof}
	We recall the definition \eqref{L0} and we introduce $\tilde{H}=\tilde{H}[H]$ such that
	\begin{align*}
		L_{0}[\tilde{H}]=H,  \ \ \ |\tilde{H}|+(1+\rho)|\nabla \tilde{H}|\le C \frac{\|H\|_{0;\nu,m,4+\sigma,\varepsilon}}{\tau^{\nu}(\ln \tau)^{m}}\begin{cases}
			\frac{1}{(1+\rho)^{2+\sigma}} \ \ \ \ \rho\le \sqrt{\tau}\\
			\frac{\tau^{\varepsilon/2}}{(1+\rho)^{2+\sigma+\varepsilon}} \ \ \ \rho\ge \sqrt{\tau}.
		\end{cases}
	\end{align*}
	By using barriers (similarly to the proof of Lemma \ref{barriersUg} and the proof of Lemma 11.2 in \cite{DdPDMW}) we have that the solution of
	\begin{align}\label{tildephi}
		\begin{cases}
			\partial_{ \tau}\tilde{\Phi}_{1}=\Delta \tilde{\Phi}_{1}-\nabla \Gamma_{0} \cdot \nabla \tilde{\Phi}_{1}+\lambda\dot{\lambda}(-2\tilde{\Phi}_{1}\chi+ y \cdot \nabla (\tilde{\Phi}_{1}\chi))+\tilde{H}\\
			\tilde{\Phi}_{1}(\cdot,\tau_{0})=0
		\end{cases}
	\end{align}
	if $\nu<1+\frac{\varepsilon}{2}$ and $\nu<3-\frac{\sigma}{2}$, satisfies
	\begin{align}\label{EstimateTildePhi1}
		|\tilde{\Phi}_{1}|\le C \frac{\|H\|_{\nu,m,4+\sigma,\varepsilon}}{\tau^{\nu}(\ln \tau)^{m}}\begin{cases}
			\frac{1}{(1+\rho)^{\sigma}} \ \ \ \ \ \rho \le \sqrt{\tau}\\
			\frac{\tau^{1+\varepsilon/2}}{(1+\rho)^{2+\sigma+\varepsilon}} \ \ \ \rho \ge \sqrt{\tau}.
		\end{cases}
	\end{align}
	Recalling \eqref{L0} we can rewrite \eqref{tildephi} as
	\begin{align}\label{systemtildephi1}
		\begin{cases}
			\partial_{ \tau} \tilde{\Phi}_{1}=L_{0}[\tilde{\Phi}_{1}]+\lambda\dot{\lambda}(-2\tilde{\Phi}_{1}\chi+y\cdot \nabla (\tilde{\Phi}_{1}\chi))+\tilde{H}-U\tilde{\Phi}_{1}\\
			\tilde{\Phi}_{1}(\cdot, \tau_{0})=0.
		\end{cases}
	\end{align}
    Thanks to \eqref{EstimateTildePhi1} and by standard parabolic estimates we also have
\begin{align}\label{GradientPhi1Tilde}
	|\nabla\tilde{\Phi}_{1}|\le C \frac{\|H\|_{0;\nu,m,4+\sigma,\varepsilon}}{\tau^{\nu}(\ln \tau)^{m}}\begin{cases}
		\frac{1}{(1+\rho)^{1+\sigma}} \ \ \ \ \ \rho \le \sqrt{\tau}\\
		\frac{\tau^{1+\varepsilon/2}}{(1+\rho)^{3+\sigma+\varepsilon}} \ \ \ \rho \ge \sqrt{\tau}.
	\end{cases}
\end{align}
Moreover differentiating in $y_{j}$, $j=1,2$ the equation \eqref{tildephi}, by standard parabolic estimates and using \eqref{systemtildephi1}, \eqref{GradientPhi1Tilde} we get
\begin{align}\label{HessianPhi1Tilde}
	|D^{2}\tilde{\Phi}_{1}|\le C \frac{\|H\|_{0;\nu,m,4+\sigma,\varepsilon}}{\tau^{\nu}(\ln \tau)^{m}}\begin{cases}
		\frac{1}{(1+\rho)^{2+\sigma}} \ \ \ \ \ \rho \le \sqrt{\tau}\\
		\frac{\tau^{1+\varepsilon/2}}{(1+\rho)^{4+\sigma+\varepsilon}} \ \ \ \rho \ge \sqrt{\tau}.
	\end{cases}
\end{align}
Now we call $\Phi_{1}:=L_{0}[\tilde{\Phi}_{1}]$. By \eqref{EstimateTildePhi1}, \eqref{GradientPhi1Tilde} we observe that 
\begin{align*}
	\int_{\mathbb{R}^{2}} \Phi_{1}=0, \ \ \ |\Phi_{1}|\le C\frac{\|H\|_{0;\nu,m,4+\sigma,\varepsilon}}{\tau^{\nu}(\ln \tau)^{m}}\begin{cases}
		\frac{1}{(1+\rho)^{2+\sigma}} \ \ \ \rho\le \sqrt{\tau}\\
		\frac{\tau^{1+\varepsilon/2}}{(1+\rho)^{4+\sigma+\varepsilon}} \ \ \ \rho\ge \sqrt{\tau}.
	\end{cases}
\end{align*}
We apply $L_{0}$ to the problem \eqref{systemtildephi1} and we obtain
	\begin{align*}
		\partial_{ \tau} \Phi_{1}=L_{0}[\Phi_{1}]+\lambda\dot{\lambda}(-2L_{0}[\tilde{\Phi}_{1}\chi]+L_{0}\circ\Lambda (\tilde{\Phi}_{1}\chi))+H-L_{0}[U\tilde{\Phi}_{1}].
	\end{align*}
	We need to compute the commutator $[L_{0}, \Lambda]$. Similarly to \eqref{InfiniteTimeCommutator} a direct computation gives
	\begin{align*}
		L_{0}\circ \Lambda[\phi]=\Lambda \circ L_{0}[\phi]-\nabla \cdot (\phi Uy)+2L_{0}[\phi].
	\end{align*}
We have then
	\begin{align*}
		L_{0}\circ \Lambda(\tilde{\Phi}_{1}\chi)=&\Lambda(\Phi_{1}\chi)-\nabla ( \tilde{\Phi}_{1}\chi U y) + \Lambda\circ\{L_{0}[\tilde{\Phi}_{1}(\chi-1)]+\Phi_{1}(1-\chi)\}+2L_{0}[\tilde{\Phi}_{1}\chi].
	\end{align*}
	Then we have
	\begin{align}
		\partial_{ \tau}\Phi_{1}&=L_{0}[\Phi_{1}]+\lambda\dot{\lambda}\Lambda(\Phi_{1}\chi)+H-L_{0}[U\tilde{\Phi}_{1}]-\lambda\dot{\lambda}\nabla \cdot (\tilde{\Phi}_{1}\chi Uy)+\lambda\dot{\lambda}\Lambda\circ \{L_{0}[\tilde{\Phi}_{1}(\chi-1)]+\Phi_{1}(1-\chi)\} =\nonumber\\
		&=L[\Phi_{1}]+\lambda\dot{\lambda}(\Lambda(\Phi_{1}\chi)-\lambda\dot{\lambda}W_{0}(y)(\int_{\mathbb{R}^{2}} \Lambda(\Phi_{1}\chi)) ) +H -\nonumber\\
		&\hspace{1cm} -L_{0}[U\tilde{\Phi}_{1}]-\lambda\dot{\lambda}\nabla \cdot (\tilde{\Phi}_{1}\chi U y)+\nabla \cdot (U\nabla \Psi_{1})+ \hspace{2cm }=: -F_{1}\nonumber\\
		&\hspace{1cm}+ \lambda\dot{\lambda}(\int_{\mathbb{R}^{2}}\Lambda(\Phi_{1}\chi))W_{0}(y)+ \lambda\dot{\lambda}\Lambda\{L_{0}[\tilde{\Phi}_{1}(\chi-1)]+\Phi_{1}(1-\chi)\}= \hspace{2cm }=: -F_{2}\nonumber\\
		&=L[\Phi_{1}]+\tilde{B}_{1}[\Phi_{1}] +H -F_{1}-F_{2}. \label{DefF1}
	\end{align}
Observing that $\int_{\R^2}\Phi_{1}=0$, we can rewrite $F_{2}$ and obtain
	\begin{align*}
	F_{2}&=-2\lambda\dot{\lambda}W_{0}(t)\int_{\mathbb{R}^{2}} \Phi_{1}(1-\chi)-\lambda\dot{\lambda}\Lambda\circ \{L_{0}[\tilde{\Phi}_{1}\chi]-\Phi_{1}\chi\}.
\end{align*}
Recalling \eqref{FunctionM} and \eqref{lambdadotlambdaExp} we have for some universal constant $\bar{c}$ ($W_{0}(y)$ is compactly supported) $F_{2}$ satisfies
\begin{align*}
	\int_{\mathbb{R}^{2}} F_{2}=0, \ \ |F_{2}|\le C \frac{\|H\|_{0;\nu,m,4+\sigma,\varepsilon}}{\tau^{\nu}(\ln \tau)^{m}}\begin{cases}
			\frac{\ln \tau}{\tau}\frac{1}{M^{\sigma}(\tau)} \ \ \ |y|\le \bar{c}\\
			0 \ \ \ \ \bar{c}\le|y|\le M(\tau)\\
			\frac{\delta}{1+\rho^{4+\sigma}} \ \ \ \ M(\tau)\le|y|\le 2M(\tau)\\
			0 \ \ \ \ |y|\ge 2M(\tau).
	\end{cases} 
\end{align*}
	In $F_{1}$, thanks to \eqref{EstimateTildePhi1}, \eqref{GradientPhi1Tilde}, \eqref{HessianPhi1Tilde} and recalling $\int_{\R^2}\Phi_{1}=0$, we can easily estimate $L_{0}[U\tilde{\Phi}_{1}]$ and $\nabla \cdot (U\nabla \Psi_{1})$. Finally
	\begin{align*}
		|\lambda\dot{\lambda}\nabla \cdot (\tilde{\Phi}_{1}U\chi y)| &\le C\frac{\|H\|_{0;\nu,m,4+\sigma,\varepsilon}}{\tau^{\nu}(\ln \tau)^{m}}\frac{\ln \tau}{\tau} \begin{cases}
			\frac{1}{(1+|y|)^{4+\sigma}} \ \ \ \rho \le 2M(\tau) \\\
			0 \ \ \ \ \rho \ge 2M(\tau)
		\end{cases} \le\\
	&\le C \frac{\|H\|_{0;\nu,m,4+\sigma,\varepsilon}}{\tau^{\nu} (\ln \tau)^{m}}\begin{cases}
			\frac{1}{(1+\rho)^{6+\sigma}} \ \ \ \ \rho \le \sqrt{\tau} \\
			\frac{\tau^{\varepsilon/2}}{(1+\rho)^{6+\sigma+\varepsilon}} \ \ \ \ \ \rho\ge \sqrt{\tau}.
		\end{cases}
	\end{align*}
\end{proof}
In the following result we want to include the extra term in \eqref{prototypesecmom} as a perturbation (we will also include one of the two components of the forces of Lemma \ref{Phi1Lemma}).
\begin{lemma}\label{PHIcostruction}
	Let $0<\sigma<1$, $\varepsilon>0$, $\sigma+\varepsilon<2$, $1<\nu<\min(1+\frac{\varepsilon}{2},3-\frac{\sigma}{2},\frac{3}{2})$. Let $0<q<1$. For any $H$ such that $\|H\|_{0;\nu,m,4+\sigma,\varepsilon}<\infty$ there is $\Phi$ such that
	\begin{align}\label{sysPhi}
		\begin{cases}
			\partial_{ \tau}\Phi=L[\Phi]+\tilde{B}_{1}[\Phi]+\lambda\dot{\lambda}L^{-1}[\nabla \cdot (\Phi Uy)-\nabla \cdot (Z_{0}\nabla \Psi)]+H+\tilde{F}[\Phi]\\
			\Phi(\cdot, \tau_{0})=c_{1}\tilde{Z}_{0}
		\end{cases}
	\end{align}
	such that $\Phi=\Phi_{1}+\Phi_{2}$ has zero mass and
	\begin{align*}
		|\Phi_{1}|\le C \frac{\|H\|_{0;\nu,m,4+\sigma,\varepsilon}}{\tau^{\nu}(\ln \tau)^{m}}\begin{cases}
			\frac{1}{(1+\rho)^{2+\sigma}} \ \ \ \ \rho \le\sqrt{\tau} \\
			\frac{\tau^{1+\varepsilon/2}}{(1+\rho)^{4+\sigma+\varepsilon}} \ \ \ \rho \ge \sqrt{\tau}
		\end{cases}, \ \ \ \ |\Phi_{2}|\le C \frac{\|H\|_{0;\nu,m,4+\sigma,\varepsilon}}{\tau^{\nu-1} (\ln \tau)^{m+q+1}}\frac{1}{(1+|y|)^{4}}.
	\end{align*}
	Moreover we have $\Phi_{2}=\Phi_{2}^{\perp}+\frac{a}{2}Z_{0}$ with
	\begin{align*}
		&|\Phi_{2}^{\perp}|\le C \frac{\|H\|_{0;\nu,m,4+\sigma,\varepsilon}}{\tau^{\nu-1/2} ( \ln \tau)^{m+\frac{q+1}{2}}} \begin{cases}
			\frac{1}{(1+|y|)^{2}} \ \ \ \ |y|\le \sqrt{\tau} \\
			\frac{\tau}{|y|^ {4}} \ \ \ \ \ |y| \ge \sqrt{\tau}  
		\end{cases}, \ \ \ \ |a(\tau)|\le C \frac{\|H\|_{0;\nu,m,4+\sigma,\varepsilon}}{\tau^{\nu-1}(\ln \tau)^{m+q+1}}\\
		&|c_{1}| \le C \frac{\|H\|_{0;\nu,m,4+\sigma,\varepsilon}}{\tau_{0}^{\nu-1}(\ln \tau)^{m+2}}.
	\end{align*}
	Moreover $\tilde{F}[\Phi]$ for some sufficiently large and universal constant $\bar{c}$ satisfies
	\begin{align}\label{EstimateFtilde}
		\int_{\R^2}\tilde{F}=0, \ \ \ 
		|\tilde{F}|\le C \frac{\|H\|_{0;\nu,m,4+\sigma,\varepsilon}}{\tau^{\nu}(\ln \tau)^{m}}\begin{cases}
			\frac{\ln \tau}{\tau}\frac{1}{M^{\sigma}(\tau)} \ \ \ |y|\le \bar{c}\\
			0 \ \ \ \ \bar{c}\le|y|\le M(\tau)\\
		     \frac{\delta}{1+\rho^{4+\sigma}} \ \ \ \ M(\tau)\le|y|\le 2M(\tau)\\
		    0 \ \ \ \ |y|\ge 2M(\tau).
		\end{cases} 
	\end{align}
\end{lemma}
\begin{proof}
	The idea is to include $\lambda\dot{\lambda}A[\Phi]:=\lambda\dot{\lambda}L^{-1}[\nabla \cdot ( \Phi Uy)-\nabla \cdot 
	(Z_{0}\nabla \Psi)]$ as a perturbation. Recalling the two forces $F_{1}$, $F_{2}$ introduced in Lemma \ref{Phi1Lemma}, we write $\Phi=\Phi_{1}+\Phi_{2}$ where 
	\begin{align*}
		&\begin{cases}
			\partial \Phi_{1} = L[\Phi_{1}]+\tilde{B}_{1}[\Phi_{1}]+H-F_{1}[\Phi_{1}]-F_{2}[\Phi_{1}] \\
			\Phi_{1}(\cdot,\tau_{0})=0,
		\end{cases}\\
		&\begin{cases}
			\partial_{ \tau} \Phi_{2}=L[\Phi_{2}]+\tilde{B}_{1}[\Phi_{2}]+F_{1}[\Phi_{1}]\\
			\Phi_{2}(\cdot, \tau_{0})=c_{1}\tilde{Z}_{0}
		\end{cases}
	\end{align*}
	we observe that we could not include $F_{2}[\Phi_{1}]$ in the second problem since in the region $M(\tau)\le |y|\le 2M(\tau)$ we do not have sufficiently fast decay. \newline The solution $\Phi_{1}$ can be constructed as in Lemma \ref{Phi1Lemma}. We get
	\begin{align}\label{EstimatePhi1ADter}
		\int_{\mathbb{R}^{2}} \Phi_{1}=0, \ \ |\Phi_{1}|\le C \frac{\|H\|_{0;\nu,m,4+\sigma,\varepsilon}}{\tau^{\nu} (\ln \tau)^{m}} \begin{cases}
			\frac{1}{(1+\rho)^{2+\sigma}} \ \ \ \ \ \rho \le \sqrt{\tau} \\
			\frac{\tau^{1+\varepsilon/2}}{(1+\rho)^{4+\sigma+\varepsilon}} \ \ \ \ \rho \ge \sqrt{\tau}
		\end{cases}
	\end{align}
and 
	\begin{align*}
		\int_{\mathbb{R}^{2}} F_{1}[\Phi_{1}]=0, \ \ \ \ \ |F_{1}[\Phi_{1}]| \le C \frac{\|H\|_{0;\nu,m,4+\sigma,\varepsilon}}{\tau ^{\nu} (\ln \tau)^{m}} \begin{cases}
			\frac{1}{(1+\rho)^{6+\sigma+\varepsilon}} \ \ \ \ \ \ \rho \le \sqrt{\tau} \\
			\frac{\tau^{\varepsilon/2}}{(1+\rho)^{6+\sigma+\varepsilon}} \ \ \ \ \rho \ge \sqrt{\tau},
		\end{cases}
	\end{align*}
    \begin{align*}
  	\int_{\mathbb{R}^{2}} F_{2}[\Phi_{1}]=0, \ \ |F_{2}[\Phi_{1}]|\le C \frac{\|H\|_{0;\nu,m,4+\sigma,\varepsilon}}{\tau^{\nu}(\ln \tau)^{m}}\begin{cases}
  			\frac{\ln \tau}{\tau}\frac{1}{M^{\sigma}(\tau)} \ \ \ |y|\le \bar{c}\\
  			0 \ \ \ \ \bar{c}\le|y|\le M(\tau)\\
  			\frac{\delta}{1+\rho^{4+\sigma}} \ \ \ \ M(\tau)\le|y|\le 2M(\tau)\\
  			0 \ \ \ \ |y|\ge 2M(\tau).
  	\end{cases} 
  \end{align*}
	Then by Proposition \ref{Prop101} we also know that if $1<\nu<\frac{3}{2}$, $\sigma \in (0,1)$, $\epsilon>0$ we have $\Phi_{2}=\Phi_{2}^{\perp}+\frac{a}{2} Z_{0}$ with $\int_{\mathbb{R}^{2}} \Phi_{2}=\int_{\mathbb{R}^{2}} \Phi_{2}^{\perp}=0$ and
	\begin{align}\label{Phi2EstimateAfter}
	|\Phi_{2}|\le C \frac{\|H\|_{0;\nu,m,4+\sigma,\varepsilon}}{\tau^{\nu-1}(\ln \tau)^{m+q+1}}\frac{1}{(1+|y|)^{4}}, \ \ \ \  |\Phi_{2}^{\perp}|\le C \frac{\|H\|_{0;\nu,m,4+\sigma,\varepsilon}}{\tau^{\nu} (\ln \tau)^{m+\frac{q+1}{2}}} \begin{cases}
		\frac{1}{(1+|y|)^{2}} \ \ \ \ \ |y|\le \sqrt{\tau}\\
		\frac{\tau}{|y|^{4}} \ \ \ \ |y| \ge\sqrt{\tau}
	\end{cases}
	\end{align}
	with the estimates
	\begin{align*}
		|a(\tau)| \le C \frac{\|H\|_{0;\nu,m,4+\sigma,\varepsilon}}{\tau^{\nu-1}(\ln \tau)^{m+q+1}}, \ \ \ |c_{1}|\le C \frac{\|H\|_{0;\nu,m,4+\sigma,\varepsilon}}{\tau_{0}^{\nu-1}(\ln \tau)^{m+2}}
	\end{align*}
	(here it is not necessary to consider the smallness in $\tau_{0}$ we have in the estimates for $a(\tau)$ and $\Phi_{1}$).
	We just constructed a function $\Phi$ with zero mass and that solves
	\begin{align*}
		\begin{cases}
			\partial_{ \tau}\Phi=L[\Phi]+\tilde{B}_{1}[\Phi]+H-F_{2}[\Phi_{1}]\\
			\Phi(\cdot, \tau_{0})=c_{1}\tilde{Z}_{0}.
		\end{cases}
	\end{align*}
\newline
	This means that for a given $H$ radial, satisfying $\int_{\mathbb{R}^{2}} H=0$ and $\|H\|_{0;\nu,m,4+\sigma,\varepsilon}<\infty $, we have two operators
	\begin{align*}
		\Phi_{1}=T_{1}[H], \ \ \Phi_{2}=T_{2}[H], \ \ \
	\end{align*}
 where the operator $T_{2}$ gives also the estimates for $\Phi_{2}^{\perp}$, $a$ and $c_{1}$ we mentioned above.
The idea is to find fixed points $\Phi_{1}\in X_{1}$, $\Phi_{2}\in X_{2}$ (where are $X_{1}$ and $X_{2}$ are suitably defined normed spaces, see \eqref{EstimatePhi1ADter} and \eqref{Phi2EstimateAfter}), introducing the operator $\lambda\dot{\lambda}A$ as a perturbation, namely
	\begin{align*}
		&\Phi_{1}=T_{1}[H+\lambda\dot{\lambda}A[\Phi_{1}
		+\Phi_{2}]], \ \ \Phi_{2}=T_{2}[H+\lambda\dot{\lambda}A[\Phi_{1}+\Phi_{2}]].
	\end{align*}
	We need to compute $\|\lambda\dot{\lambda}A[\Phi]\|_{0;\nu,m,4+\sigma,\varepsilon}$. First we recall that , recalling $A[\Phi_{i}]=L^{-1}[\nabla \cdot (\Phi_{i}Uy)-\nabla \cdot (Z_{0}\nabla \Psi_{i})]$, when $\Psi_{i}=(-\Delta)^{-1}\Phi_{i}$, $\int_{\mathbb{R}^{2}} \Phi_{i}=0$ and $|\Phi_{i}|\le \frac{1}{(1+|y|)^{2+\gamma}}$ for some $\gamma>0$, we have
	\begin{align*}
		\int_{\mathbb{R}^{2}} (\nabla \cdot (\Phi_{i}Uy)-\nabla \cdot (Z_{0}\nabla \Psi_{i}))|y|^{2}=0
	\end{align*}
(for more details see the proof of Lemma 11.3 in \cite{DdPDMW}).
	As it has also been shown in the proof of Lemma 11.3 in \cite{DdPDMW} we can write $A[\Phi_{i}]=Ug_{i}+U\psi_{i}$ where we underline that $\psi_{i}$ is the inverse Laplacian of $A[\Phi_{i}]$ and we have
	\begin{align*}
		&g_{i}(\rho, \tau)=-\int_{\rho}^{\infty}[\Phi_{i}(s, \tau)s-\frac{Z_{0}(s)}{U(s)}\partial_{s}\Psi_{i}(s,\tau)]ds, \ \ \ \partial_{\rho}\Psi_{i}=\frac{1}{\rho}\int_{\rho}^{\infty} \Phi_{i}(s,\tau)sds,\\
		&\psi_{i}(\rho)=z_{0}(\rho)\int_{\rho}^{\infty}\frac{1}{z_{0}^{2}(r)r}\int_{r}^{\infty} Ug_{i}(s,\tau)z_{0}(s)sdsdr \ \ \text{if }\rho>1.
	\end{align*}
	Since $A$ is a linear operator, we will start with $A[\Phi_{1}]$. Recalling \eqref{EstimatePhi1ADter} we have
	\begin{align*}
		|\partial_{\rho} \Psi_{1}|\le C \frac{\|\Phi_{1}\|_{X_{1}}}{\tau^{\nu}(\ln \tau)^{m}}\begin{cases}
			\frac{1}{(1+\rho)^{1+\sigma}} \ \ \ \ \rho\le \sqrt{\tau}\\
			\frac{\tau^{1+\varepsilon/2}}{\rho^{3+\sigma+\varepsilon}} \ \ \ \rho \ge \sqrt{\tau}
		\end{cases}
	\end{align*}
    and
	\begin{align*}
		|g_{1}|\le C \frac{\|\Phi_{1}\|_{X_{1}}}{\tau^{\nu}(\ln \tau)^{m}}\begin{cases}
			\frac{1}{(1+\rho)^{\sigma}} \ \ \ \ \rho \le \sqrt{\tau}\\
			\frac{\tau^{1+\varepsilon/2}}{\rho^{2+\sigma+\varepsilon}} \ \ \ \ \rho \ge \sqrt{\tau}.
		\end{cases}
	\end{align*}
	If $\rho \ge \sqrt{\tau}$ we have
	\begin{align*}
		|\psi_{1}|\le C \frac{\|\Phi_{1}\|_{X_{1}}}{\tau^{\nu}(\ln \tau)^{m}}\frac{\tau^{1+\varepsilon/2}}{\rho^{4+\sigma+\varepsilon}}
	\end{align*}
	and if $\rho\le \sqrt{\tau}$ we have
	\begin{align*}
		|\psi_{1}|\le C \int_{\rho}^{\sqrt{\tau}} \frac{1}{r}\int_{r}^{\infty} U|g_{1}|sdsdr + C \frac{\|\Phi\|_{X_{1}}}{\tau^{\nu}(\ln \tau)^{m}} \frac{1}{\tau^{1+\sigma/2}} .
	\end{align*}
	Since $\int_{r}^{\infty} U|g_{1}| sds \le C \frac{\|\Phi_{1}\|_{X_{1}}}{\tau^{\nu}(\ln \tau)^{m}}\frac{1}{(1+r)^{2+\sigma}}$ we finally get
	\begin{align*}
		|\psi_{1}|\le C \frac{\|\Phi_{1}\|_{X_{1}}}{\tau^{\nu}(\ln \tau)^{m}}\begin{cases}
			\frac{1}{(1+\rho)^{2+\sigma}} \ \ \ \rho \le \sqrt{\tau}\\
			\frac{\tau^{1+\varepsilon/2}}{\rho^{4+\sigma+\varepsilon}}  \ \ \ \rho\ge\sqrt{\tau}.
		\end{cases}
	\end{align*}
	We just proved
	\begin{align}\label{APhi1}
		|A[\Phi_{1}]|\le C \frac{\|\Phi_{1}\|_{X_{1}}}{\tau^{\nu}(\ln \tau)^{m}}\begin{cases}
			\frac{1}{(1+\rho)^{4+\sigma}} \ \ \ \rho\le \sqrt{\tau}\\
			\frac{1}{\rho^{4+\sigma}}\frac{\tau^{1+\varepsilon/2}}{\rho^{2+\varepsilon}} \ \ \ \rho\ge \sqrt{\tau}.
		\end{cases}
	\end{align}
	Now we want to estimate $A[\Phi_{2}]$. First we recall \eqref{Phi2EstimateAfter} and we have
	\begin{align*}
		|\partial_{\rho}\Psi_{2}|\le C \frac{\|\Phi_{2}\|_{X_{2}}}{\tau^{\nu-1}(\ln \tau)^{m+q+1}} \frac{1}{1+\rho^{3}}, \ \ 
		|g_{2}|\le C \frac{\|\Phi_{2}\|_{X_{2}}}{\tau^{\nu-1}(\ln\tau)^{m+q+1}}\frac{1}{1+\rho^{2}},
	\end{align*}
\begin{align*}
		|\psi_{2}|\le C \frac{\|\Phi_{2}\|_{X_{2}}}{\tau^{\nu-1}(\ln \tau)^{m+q+1}}\frac{1}{(1+\rho)^{4}}.
	\end{align*}
	Consequently
	\begin{align}\label{APhi2}
		|A[\Phi_{2}]|\le C \frac{\|\Phi_{2}\|_{X_{2}}}{\tau^{\nu-1}(\ln \tau)^{m+q+1}}\frac{1}{1+\rho^{6}}.
	\end{align}
	By \eqref{APhi1}, since if $\rho \ge\sqrt{\tau} \iff \frac{\tau}{\rho^{2}}\le 1$, we have that
	\begin{align*}
		|\lambda\dot{\lambda}A[\Phi_{1}]|\le&C\frac{\ln \tau_{0}}{\tau_{0}}\frac{\|\Phi_{1}\|_{X_{1}}}{\tau^{\nu}(\ln \tau)^{m}} \frac{1}{(1+|y|)^{4+\sigma}}\begin{cases}
			1 \ \ \ \ \rho\le \sqrt{\tau}\\
			\frac{\tau^{\varepsilon/2}}{\rho^{\varepsilon}} \ \ \ \ \rho\ge\sqrt{\tau}.
		\end{cases}
	\end{align*}
	Then $\|\lambda\dot{\lambda}A[\Phi_{1}]\|_{0;\nu,m,4+\sigma,\varepsilon}\le C \frac{\ln \tau_{0}}{\tau_{0}}\|\Phi_{1}\|_{X_{1}}$. By \eqref{APhi2}, we have
	\begin{align*}
		|\lambda\dot{\lambda}A[\Phi_{2}]|\le C \frac{\|\Phi_{2}\|_{X_{2}}}{\tau^{\nu}(\ln \tau)^{m}}\frac{1}{(\ln \tau)^{q}} \frac{1}{(1+\rho)^{6}}.
	\end{align*}
	Since if $\rho \ge \sqrt{\tau}$ and $\sigma+\varepsilon<2$ we have $\frac{1}{\rho^{2-\sigma-\varepsilon}}\le \tau^{\varepsilon/2}$, then
	\begin{align*}
		|\lambda\dot{\lambda}A[\Phi_{2}]|\le &C\frac{1}{\ln^{q}\tau_{0}}\frac{\|\Phi_{2}\|_{X_{2}}}{\tau^{\nu}(\ln \tau)^{m}} \frac{1}{(1+|y|)^{4+\sigma}}\begin{cases}
			1 \ \ \ \ \rho\le \sqrt{\tau}\\
			\frac{\tau^{\varepsilon/2}}{\rho^{\varepsilon}} \ \ \ \ \rho\ge\sqrt{\tau}.
		\end{cases}
	\end{align*}
	We just proved that
	\begin{align*}
		\sum_{i=1,2} \|T_{i}[\lambda\dot{\lambda}A[\Phi_{1}+\Phi_{2}]]\|_{X_{i}} \le C \frac{1}{\ln^{q}\tau_{0}}(\|\Phi_{1}\|_{X_{1}}+\|\Phi_{2}\|_{X_{2}})
	\end{align*}
	then we have $\Phi_{1}$, $\Phi_{2}$, $\Phi_{2}^{\perp}$, $a$ and $c_{1}$ such that
	\begin{align*}
		\begin{cases}
			\partial_{ \tau} \Phi_{1}=&\hspace{-1cm}L[\Phi_{1}]+\tilde{B}_{1}[\Phi_{1}]++\lambda\dot{\lambda}L^{-1}[\nabla\cdot(\Phi Uy)-Z_{0}\nabla \Psi]+H-F_{1}[\Phi_{1}]-F_{2}[\Phi_{1}]\\
			\Phi_{1}(\cdot,\tau_{0})=0
		\end{cases}
	\end{align*}
	and
	\begin{align*}
		\begin{cases}
			\partial_{ \tau} \Phi_{2}=L[\Phi_{2}]+\tilde{B}_{1}[\Phi_{2}]+F_{1}[\Phi_{1}]\\
			\Phi_{2}(\cdot, \tau_{0})=c_{1}\tilde{Z}_{0}\\
			\Phi_{2}=\Phi_{2}^{\perp}+ \frac{a}{2}Z_{0}
		\end{cases}
	\end{align*}
	and consequently $\Phi=\Phi_{1}+\Phi_{2}$ is the desired solution.
\end{proof}
\begin{remark}
		We observe that $\Phi_{2}$ in \eqref{Phi2EstimateAfter} would have more decay when $|y|\ge\sqrt{\tau}$. Considering that additional decay we would slightly improve the estimate of $\mathcal{E}$ when $|y|\ge\sqrt{\tau}$ but there will always be a region where the decay $\approx\frac{1}{|y|^{6}}$ prevents us from improving the decay of $\phi$ when $|y|\ge \sqrt{\tau}$ (see also the proof of Proposition \ref{PropMode0MassSecMom0}).
\end{remark}
\begin{remark}\label{remarkNumberDer}
	We observe that since $\tilde{\Phi}_{1}$ solves \eqref{tildephi}, after formally differentiating and standard parabolic estimates, we can control fifth-order derivatives. As a consequence we can control third-order derivatives of $\Phi_{1}=L[\tilde{\Phi}_{1}]$. By the definition of $F_{1}[\Phi_{1}]$, see \eqref{DefF1}, we observe that we can also control third-order derivatives for $\Phi_{2}$.
\end{remark}

\begin{proof}[Proof of Proposition \ref{PropMode0MassSecMom0}]
	The starting point is the solution $\Phi=\Phi_{1}+\Phi_{2}$ of \eqref{sysPhi} we constructed in Lemma \ref{PHIcostruction}. We observe again that $\int_{\mathbb{R}^{2}}\Phi=0$. If we call $\phi:=L[\Phi]$ we see it solves
	\begin{align*}
		\begin{cases}
			\partial_{ \tau}\phi=L[\phi]+\lambda\dot{\lambda}(2\phi+y\cdot \nabla \phi)+h-\lambda\dot{\lambda}(\int \Lambda(\Phi \chi))L[W_{0}](y)+\lambda\dot{\lambda}L\circ\Lambda(\Phi(\chi-1)) +L[\tilde{F}]\\
			\phi(\cdot, \tau_{0})=c_{1}\hat{Z}_{0}.
		\end{cases}
	\end{align*}
    We recall that
\begin{align}
	L[\phi]&=\Delta \phi -\nabla \cdot(\phi \nabla \Gamma_{0})+U'(\rho)\frac{1}{\rho}\int_{0}^{\rho}\phi(s)sds+U\phi=\label{ExpansionL1}\\
	&=\Delta \phi -\nabla \cdot(\phi \nabla \Gamma_{0})-U'(\rho)\frac{1}{\rho}\int_{\rho}^{\infty}\phi(s)sds+U\phi\label{ExpansionL2}.
\end{align} 
    Since $\phi=L[\Phi_{1}+\Phi_{2}]=L[\Phi_{1}+\Phi_{2}^{\perp}]$, $\int_{\R^2}(\Phi_{1}+\Phi_{2}^{\perp})=0$ and they are sufficiently fast decaying in space we have $\int_{\mathbb{R}^{2}} \phi=0$, $\int_{\R^2}\phi|y|^{2}=0$ and, by \eqref{ExpansionL1}, \eqref{ExpansionL2}, we have
	\begin{align*}
		|\phi|\le C \frac{\|h\|}{\tau^{\nu-1/2}(\ln \tau)^{m+\frac{1+q}{2}}} \begin{cases}
			\frac{1}{(1+|y|)^{4}} \ \ \ \ \ |y|\le \sqrt{\tau}\\
			\frac{\tau}{|y|^{6}} \ \ \ \ \ |y|\ge \sqrt{\tau}.
		\end{cases}
	\end{align*}
    We clearly have then
    \begin{align}\label{Eoper}
    	\mathcal{E}[\phi]=L[\tilde{F}]-\lambda\dot{\lambda}(\int_{\R^2}\Lambda(\Phi \chi))L[W_{0}](y)+\lambda\dot{\lambda}L\circ \Lambda(\Phi(\chi-1)).
    \end{align}
     Notice that $-W_{0}(y)(\int_{\mathbb{R}^{2}} \Lambda(\Phi \chi))+\Lambda(\Phi(\chi-1))$ and $\tilde{F}$ have both zero mass and they are sufficiently fast decaying in space. Then we have
     \begin{align*}
     	\int_{\mathbb{R}^{2}}\mathcal{E}[\phi]dy=0, \ \ \ \int_{\R^2}\mathcal{E}[\phi]|y|^{2}dy=0.
     \end{align*}   
    Now from \eqref{EstimateFtilde} and \eqref{ExpansionL1}, \eqref{ExpansionL2} we have
    \begin{align*}
        |L[\tilde{F}]|\le C \frac{\|H\|_{0;\nu,m,4+\sigma,\varepsilon}}{\tau^{\nu}(\ln \tau)^{m}}\begin{cases}
        	\frac{\ln \tau}{\tau}\frac{1}{M^{\sigma}} \frac{1}{1+\rho^{6}}\ \ \ |y|\le M(\tau)\\
        	\frac{\delta}{1+\rho^{6+\sigma}} \ \ \ \ M(\tau)\le |y|\le 2M(\tau)\\
        	0 \ \ \ |y|\ge 2M(\tau).
        \end{cases}
    \end{align*}
	Recalling \eqref{EstimatePhi1ADter} and \eqref{Phi2EstimateAfter} for some universal constant $\bar{c}$ we obtain
    \begin{align*}
    	|-W_{0}(y)\lambda\dot{\lambda}(\int_{\R^2}\Lambda(\Phi \chi))+\lambda\dot{\lambda}\Lambda(\Phi(\chi-1))|\le C \frac{\|H\|_{0;\nu,m,4+\sigma,\varepsilon}}{\tau^{\nu}\ln^{m+q}\tau}\begin{cases}
    		\frac{1}{M^{2}(\tau)} \ \ \ \ |y|\le \bar{c}\\
    		0 \ \ \ \ \bar{c}\le |y|\le M(\tau)\\
         \frac{1}{1+\rho^{4}}\ \ \ \ \ M(\tau)\le |y|\le \sqrt{\tau}\\
    	   \frac{1}{1+\rho^{4}} \ \ \ |y|\ge \sqrt{\tau}
    	\end{cases}
    \end{align*}
 and then, using again \eqref{ExpansionL1}, \eqref{ExpansionL2}, we have
 \begin{align*}
 	L[-W_{0}(y)\lambda\dot{\lambda}\int_{\R^2}\Lambda(\Phi\chi)+\lambda\dot{\lambda}\Lambda(\Phi(\chi-1))]\le C \frac{\|H\|_{0;\nu,m,4+\sigma,\varepsilon}}{\tau^{\nu}(\ln \tau)^{m+q}}\begin{cases}
 		 \frac{1}{M^{2}(\tau)}\frac{1}{1+\rho^{6}} \ \ \ |y|\le M(\tau)\\
 		\frac{1}{1+\rho^{6}} \ \ \ \ \ M(\tau)\le|y|\le \sqrt{\tau}\\
 		\frac{1}{1+\rho^{6}} \ \ \ |y|\ge \sqrt{\tau}.
 	\end{cases}
 \end{align*}
\end{proof}

\section{Linear estimate with no radial mode}\label{Section12}
The purpose of this Section is to prove Proposition \ref{PropMode1}. As we already observed in Section \ref{ProjectionIOsystem} we can split the solution into Fourier modes. In this case we assume that $h$ \emph{has not radial mode} and we study 
\begin{align}\label{EquationMode1}
		\begin{cases}
			\partial_{\tau}\phi = L[\phi]+B[\phi \chi]+\mathcal{F}[\phi]+h, \ \ \ \ \ \int_{\R^2}h(y)y_{1}dy=0, \ \ \int_{\R^2}h(y)y_{2}dy=0,  \\
			\phi(\cdot, \tau_{0})=0
		\end{cases}
\end{align}
where, as we did in the previous two Sections, we indicate $\chi$ the cut-off \eqref{FunctionM} instead of keeping the notation $\hat{\chi}$. \newline
Now we want to explain the role of the force $\mathcal{F}[\phi]$. It is clear by looking \eqref{EquationMode1} that if we cut the operator $B$ we are loosing the conservation of the first moments. The conservation of first moments is a key tool since we want to treat $B$ as a small perturbation of the problem
\begin{align}\label{SimplifiedMode1Prob}
	\begin{cases}
		\partial_{ \tau}\phi=L[\phi]+h, \ \ \ (y,\tau)\in \mathbb{R}^{2}\times(\tau_{0},\infty)\\
		\phi(\cdot, \tau_{0})=0
	\end{cases}
\end{align}
where $h$ has not radial mode and satisfies $\int_{\R^2}h(y)y_{1}dy=\int_{\R^2}h(y)y_{2}dy=0$.
Then we take 
\begin{align}\label{ForceF}
	\mathcal{F}[\phi]=-(\int_{\R^2}B[\phi\chi]y_{1}dy)W_{1}(y)-(\int_{\R^2}B[\phi\chi]y_{2}dy)W_{2}(y)
\end{align}
where the functions $W_{1}(y)$ and $W_{2}(y)$ have been introduced in \eqref{W1j}.
To find estimates for problem \eqref{SimplifiedMode1Prob} we will simply recall the Proposition 12.1 that has been proved in \cite{DdPDMW}.
\begin{proposition}\label{Prop121}
	Let $0<\sigma<1$, $0<\epsilon<2$, $0<\nu<\min(1+\frac{\varepsilon}{2},\frac{3}{2}-\frac{\sigma}{2})$, $m\in \mathbb{R}$. Then there is a $C>0$ such that for any $\tau_{0}$ sufficiently large the following holds. Suppose that $h(y,\tau)$ has no radial mode and satisfies $\|h\|_{0;\nu,m,5+\sigma,\varepsilon}<\infty$,
	\begin{align}\label{CondtionMode1T}
		\int_{\R^2}h(y,\tau)y_{j}dy=0 \ \ \text{for all }\tau>\tau_{0}, \ \ j=1,2.
	\end{align}
Then the solution $\phi(y,\tau)$ of \eqref{SimplifiedMode1Prob} satisfies
\begin{align*}
	|\phi(y,\tau)|\le C \frac{\|h\|_{0;\nu,m,5+\sigma,\varepsilon}}{\tau^{\nu}(\ln \tau)^{m}}\begin{cases}
		\frac{1}{(1+|y|)^{3+\sigma}}, \ \ \ |y|\le \sqrt{\tau}\\
		\frac{\tau^{1+\varepsilon/2}}{|y|^{5+\sigma+\varepsilon}}, \ \ |y|\ge \sqrt{\tau}.
	\end{cases}
\end{align*}
\end{proposition}
\begin{proof}
	See the proof of Proposition 12.1 in \cite{DdPDMW}.
\end{proof}
Now we want to include the operator $B[ \phi\chi]+\mathcal{F}[\phi]$ as a perturbation.
\begin{proof}[Proof of Proposition \ref{PropMode1}]
	By Proposition \ref{Prop121} we know that there is a linear operator $T$ so that given $h$ with $\|h\|_{0;\nu,m,5+\sigma,\varepsilon}<\infty$, with no radial mode, and satisfying condition \eqref{CondtionMode1T} associates the solution $\phi$ of \eqref{SimplifiedMode1Prob} that in particular satisfies
	\begin{align*}
		\int_{\R^2}\phi y_{j}dy=0, \ \ j=1,2.
	\end{align*}
	Notice that by standard parabolic estimates we can easily observe that
	\begin{align*}
		\| |y| \nabla \phi \|_{1;\nu,m,3+\sigma,2+\varepsilon}\le \|h\|_{0;\nu,m,5+\sigma,\varepsilon}.
	\end{align*}
    Then the solution $\phi$ of \eqref{EquationMode1} can be written as
	\begin{align*}
		\phi=T[h+B[\phi\chi]+\mathcal{F}[\phi]].
	\end{align*}
    Thanks to \eqref{FunctionM} since $M^{2}\approx c \frac{\delta \tau}{\ln \tau}$ for some universal constant $c$ it is immediate to observe that
    \begin{align*}
    	\| B[\phi \chi]\|_{0;\nu,m,5+\sigma,\varepsilon}\le C \delta (\|\phi\|_{1;\nu,m,3+\sigma,2+\varepsilon}+ \| |y| \nabla \phi\|_{1;\nu,m,3+\sigma,2+\varepsilon}).
    \end{align*}
    Finally, recalling \eqref{ForceF}, since $B[\phi\chi]=\lambda\dot{\lambda}\nabla \cdot (y\phi \chi)$
    \begin{align*}
    	|\int_{\R^2}B[\phi \chi]y_{j}dy|&=|\lambda\dot{\lambda} \int_{\R^2}\nabla \cdot (y \phi\chi)y_{j}|=|\lambda\dot{\lambda}\int_{\R^2}y_{j}\phi \chi dy|=\\
    	&=|\lambda\dot{\lambda}\int_{\R^2}y_{j}\phi(1-\chi)|\le C(\|\phi\|_{1;\nu,m,3+\sigma,2+\varepsilon}+\| |y| \nabla \phi\|_{1;\nu,m,3+\sigma,2+\varepsilon})\frac{1}{\tau^{\nu+1}(\ln \tau)^{m-1}}\frac{1}{M^{\sigma}}.
    \end{align*}
   Then 
   \begin{align*}
   	\|\mathcal{F}[\phi]\|_{0;\nu,m,5+\sigma,\varepsilon}\le C \frac{\ln \tau_{0}}{\tau_{0}} \frac{1}{M^{\sigma}(\tau_{0})}(\|\phi\|_{1;\nu,m,3+\sigma,2+\varepsilon}+ \| |y| \nabla \phi\|_{1;\nu,m,3+\sigma,2+\varepsilon}).
   \end{align*}
By taking $\tau_{0}$ sufficiently large we get the desired conclusion.
\end{proof}

\section{The outer problem}\label{OuterProblemSection}
We consider the linear outer problem:
\begin{align}\label{OuterEqfinSec}
	\begin{cases}
		\partial_{t}\phi^{o}=L^{o}[\phi^{o}]+g(x,t), \ \ \ \text{in }\mathbb{R}^{2}\times(0,T)\\
		\phi^{o}(\cdot,0)=0, \ \ \text{in }\mathbb{R}^{2}.
	\end{cases}
\end{align}
where 
\begin{align*} 
	L^{o}[\phi]:=\Delta_{x}\phi-\nabla_{x}\big[\Gamma_{0}\big(\frac{x-\xi(t)}{\lambda(t)}\big)\big]\cdot \nabla_{x}\phi=\Delta_{x}\phi+4\frac{x-\xi}{\lambda^{2}+|x-\xi|^{2}}\cdot \nabla_{x}\phi.
\end{align*}
For $g:\mathbb{R}^{2}\times(0,T)\to \mathbb{R}$ we consider the norm $\|g\|_{\star\star,0}$ defined as the least $K$ such that for all $(x,t)\in \mathbb{R}^{2}\times(0,T)$
\begin{align*}
	|g(x,t)|\le K \begin{cases}
		\frac{e^{-a\sqrt{2|\ln(|x|^{2}+(T-t))|}}}{(|x|^{2}+(T-t))^{2}|\ln((T-t)+|x|^{2})|^{b} }\ \ \ \ \text{if } 0\le |x|\le \sqrt{T}\\[7pt] 
		\frac{e^{-a\sqrt{2|\ln T|}}}{T^{2}|\ln T|^{b}}e^{-\frac{|x|^{2}}{4(t+T)}} \ \ \ \ \text{if }|x|\ge \sqrt{T}.
	\end{cases}
\end{align*}
where $a>0$ and $b\in \mathbb{R}$ will be fixed in Section \ref{proofThm1}. We also define the norm $\|\phi\|_{\star,o}$ as the least $K$ such that 
\begin{align*}
	|\phi^{o}(x,t)|+(\lambda+|x|)|\nabla_{x}\phi^{o}(x,t)|\le K\begin{cases}
		\frac{e^{-a\sqrt{2|\ln(|x|^{2}+(T-t))|}}}{(|x|^{2}+(T-t))|\ln((T-t)+{|x|^{2}})|^{b} }\ \ \ \ \text{if } 0\le |x|\le \sqrt{T}\\[7pt] 
		\frac{e^{-a\sqrt{2|\ln T|}}}{T|\ln T|^{b}}e^{-\frac{|x|^{2}}{4(t+T)}} \ \ \ \ \text{if }|x|\ge \sqrt{T}.
	\end{cases}
\end{align*}
Keeping in mind these norms we state the Proposition we want to prove in this section.
\begin{proposition}\label{Prop131}
	Assume that $a>0$, $b\in \mathbb{R}$ and $\lambda$, $\xi$, satisfy \eqref{EstiPar}, \eqref{EstiLambda}. Then there is a constant $C$ so that for $T$ sufficiently small and for $\|g\|_{\star\star,o}<\infty$ there exists a solution $\phi^{o}=\mathcal{T}^{o}_{p}[g]$ of \eqref{OuterEqfinSec}, which defines a linear operator of $g$ and satisfies
	\begin{align*}
	    \|\phi^{o}\|_{\star\star,o}\le C \|g\|_{\star\star,o}.
	\end{align*}
\end{proposition}
To prove Proposition \ref{Prop131} we will need the following Lemma whose proof can be found in \cite{DdPDMW}.
\begin{lemma}\label{Lemma131}
	Let $2<\kappa<6$ and $h(r)$ satisfy
	\begin{align*}
		0<h(r)\le \frac{\lambda^{-2}}{(\frac{r}{\lambda}+1)^{\kappa}}= \frac{\lambda^{\kappa-2}}{(r+\lambda)^{\kappa}},
	\end{align*}
 where $\lambda>0$. Then there is a unique bounded radial function $\phi(r)>0$ satisfying 
 \begin{align*}
 	L^{o}[\phi]+h=0 \ \ \text{in }\mathbb{R}^{2}.
 \end{align*}
 Moreover $\phi$ satisfies
 \begin{align*}
 	|\phi(r)|+(\lambda+r)|\partial_{r}\phi(r)|\le \frac{C}{(1+\frac{r}{\lambda})^{\kappa-2}}=C \frac{\lambda^{\kappa-2}}{(r+\lambda)^{\kappa-2}}.
 \end{align*}
\end{lemma}
\begin{proof}
	See Lemma 13.1 in \cite{DdPDMW}.
\end{proof}
\begin{proof}[Proof of Proposition \ref{Prop131}]
	As a preliminary observation we see that if $\varphi(x,t)=\phi(|x|,t)=\phi(r,t)$ then
	\begin{align*}
		L^{o}[\varphi]=\partial_{r}^{2}\phi+\frac{1}{r}\partial_{r}\phi + 4 \frac{x-\xi}{\lambda^{2}+|x-\xi|^{2}}\cdot \frac{x}{r}\partial_{r}\phi .
	\end{align*}
    We notice that it is useful to avoid barriers as $\phi(|x-\xi|,t)$ since the time derivative would produce $\dot{\xi}\cdot \nabla \phi$ that as we already pointed out is singular away from the singularity. Proceeding similarly to the proof of Lemma \ref{estimatephilambda}, we can find a function $\psi_{1}(|x|,t)$ that satisfies
	\begin{align*}
		\partial_{t}\psi_{1}-\partial_{r}^{2}\psi_{1}-\frac{5}{r}\partial_{r}\psi_{1}\ge \begin{cases}
			\frac{e^{-a\sqrt{2|\ln(|x|^{2}+(T-t))|}}}{(|x|^{2}+(T-t))^{2}|\ln((T-t)+|x|^{2})|^{b} }\ \ \ \ \text{if } 0\le |x|\le \sqrt{T}\\[7pt] 
			\frac{e^{-a\sqrt{2|\ln T|}}}{T^{2}|\ln T|^{b}}e^{-\frac{|x|^{2}}{4(t+T)}} \ \ \ \ \text{if }|x|\ge \sqrt{T}
		\end{cases}
	\end{align*}
such that 
	\begin{align}\label{ps1boundOut}
	|\psi_{1}(x,t)|+(|x|+\sqrt{T-t})|\nabla \psi_{1}(x,t)|\le C \begin{cases}
		\frac{e^{-a\sqrt{2|\ln (|x|^{2})|}}}{|x|^{2}+T-t}\frac{1}{|\ln(T-t+|x|^{2})|^{b}} \ \ \ \ \ 0\le |x|\le \sqrt{T},\\[7pt]
		\frac{e^{-a\sqrt{2|\ln T|}}}{T|\ln T|^{b}}e^{-\frac{|x|^{2}}{4(t+T)}} \ \ \ \ \ |x|\ge\sqrt{T}
	\end{cases}
\end{align}
and that
\begin{align}\label{psi1gradOUT}
	|\nabla \psi_{1}(x,t)|\le C e^{-a\sqrt{2|\ln(T-t)|}}\frac{|x|}{(T-t+|x|^{2})^{2}}, \ \ \ |x|\le \sqrt{T-t}.
\end{align}
In the following we denote
\begin{align*}
	\bar{g}(x,t)=\begin{cases}
		\frac{e^{-a\sqrt{2|\ln(|x|^{2}+(T-t))|}}}{(|x|^{2}+(T-t))^{2}|\ln((T-t)+|x|^{2})|^{b} }\ \ \ \ \text{if } 0\le |x|\le \sqrt{T},\\[7pt] 
		\frac{e^{-a\sqrt{2|\ln T|}}}{T^{2}|\ln T|^{b}}e^{-\frac{|x|^{2}}{4(t+T)}} \ \ \ \ \text{if }|x|\ge \sqrt{T}.
	\end{cases}
\end{align*}
Similarly to proof of Lemma \ref{lemmaesterr} we have that
\begin{align}\label{BARRIERPSI1Est}
	\big[\partial_{t}-\partial_{r}^{2}-(\frac{1}{r}+ 4 \frac{x-\xi}{\lambda^{2}+|x-\xi|^{2}}\cdot \frac{x}{r})\partial_{r}\big]\psi_{1}\ge C \bar{g}(x,t)+O(\frac{\lambda^{2}+|\xi| |x-\xi|}{|x|(|x-\xi|^{2}+\lambda^{2})}|\partial_{r}\psi_{1}|).
\end{align}
But then thanks to \eqref{psi1gradOUT} and recalling \eqref{EstiPar}, if $|x|\le\sqrt{T-t}$ we have
\begin{align}
	\frac{\lambda^{2}+|\xi||x-\xi|}{|x|(|x-\xi|^{2}+\lambda^{2})}|\partial_{r}\psi_{1}|&\le C \frac{\lambda^{2}}{|x-\xi|^{2}+\lambda^{2}}
		\frac{e^{-a\sqrt{2|\ln (T-t)|}}}{(T-t)^{2}}\frac{1}{|\ln(T-t)|^{b}}\le \nonumber\\
		&\le C \frac{\lambda^{2}}{(|x-\xi|^{2}+\lambda^{2})^{2}}\frac{e^{-a\sqrt{2|\ln(T-t)|}}}{(T-t)}\frac{1}{|\ln(T-t)|^{b}}\le \nonumber\\
		&\le C \frac{\lambda^{2}}{(|x|^{2}+\lambda^{2})^{2}}\frac{e^{-a\sqrt{2|\ln(T-t)|}}}{(T-t)|\ln(T-t)|^{b}}.\label{EstimateREmainderPSIBAR}
\end{align}
Let us consider a function $\tilde{\psi}_{2}$ as in Lemma \ref{Lemma131} with $\kappa=4$. Then if we take $\psi_{2}=\frac{e^{-a\sqrt{2|\ln(T-t)|}}}{(T-t)|\ln(T-t)|^{b}}\tilde{\psi}_{2}$
we get
\begin{align}
     \big[\partial_{t}-\partial_{r}^{2}-(\frac{1}{r}+\frac{4r}{\lambda^{2}+r^{2}})\partial_{r}\big]\psi_{2}\ge&\frac{e^{-a\sqrt{2|\ln(T-t)|}}}{(T-t)|\ln(T-t)|^{b}}\frac{\lambda^{2}}{(r^{2}+\lambda^{2})^{2}}\big[1+\nonumber\\
     &+O(\frac{1}{T-t}|\tilde{\psi}_{2}|\frac{(r^{2}+\lambda^{2})^{2}}{\lambda^{2}})+O(\frac{|\dot{\lambda}|}{\lambda}|\tilde{\psi}_{2}'(\frac{r}{\lambda})|  |\frac{r}{\lambda}||\frac{(r^{2}+\lambda^{2})^{2}}{\lambda^{2}})\big]\ge \nonumber\\
     \ge&\frac{e^{-a\sqrt{2|\ln(T-t)|}}}{(T-t)|\ln(T-t)|^{b}}\frac{\lambda^{2}}{(r^{2}+\lambda^{2})^{2}}\big[1+O(\frac{r^{2}+\lambda^{2}}{T-t})\big].\label{psi2ESTIMATEB}
\end{align}
Now we claim that if $0<r<\sqrt{\epsilon(T-t)}$ we have
\begin{align}\label{ClaimOuter}
	\big[\partial_{t}-\partial_{r}^{2}-(\frac{1}{r}+4\frac{x-\xi}{\lambda^{2}+|x-\xi|^{2}}\cdot\frac{x}{r})\partial_{r}\big]\psi_{2}\ge\frac{e^{-a\sqrt{2|\ln(T-t)|}}}{(T-t)|\ln(T-t)|^{b}}\frac{\lambda^{2}}{(r^{2}+\lambda^{2})^{2}}.
\end{align}
Indeed we have
\begin{align*}
	\big(\frac{x-\xi}{\lambda^{2}+|x-\xi|^{2}}-\frac{x}{|x|^{2}+\lambda^{2}}\big)\cdot \frac{x}{r}=&\frac{1}{r}\big(-\frac{\lambda^{2}}{\lambda^{2}+|x-\xi|^{2}}+\frac{\lambda^{2}}{\lambda^{2}+|x|^{2}}+\frac{(x-\xi)\cdot \xi}{\lambda^{2}+|x-\xi|^{2}}\big).
\end{align*}
Thanks to \eqref{EstiPar} the last term is small when $|x|\le \sqrt{\epsilon(T-t)}$.  For the remaining two terms we observe that $
		|x-\xi|^{2}=|x|^{2}+|\xi|^{2}-2x\cdot \xi$
and then
\begin{align*}
	|-\frac{\lambda^{2}}{\lambda^{2}+|x-\xi|^{2}}+\frac{\lambda^{2}}{\lambda^{2}+|x|^{2}}|\le \frac{\lambda^{2}||x-\xi|^{2}-|x|^{2}|}{(\lambda^{2}+|x-\xi|^{2})(\lambda^{2}+|x|^{2})}\le \frac{|\xi|^{2}+2|x||\xi|}{\lambda^{2}+|x-\xi|^{2}}
\end{align*}
that is small because of \eqref{EstiPar}. This proves the claim \eqref{ClaimOuter}.\newline
Let us now consider
\begin{align}
	\tilde{\psi}=\psi_{1}+M\psi_{2}\chi_{\epsilon} \ \ \ \chi_{\epsilon}(r,t)=\chi_{0}(\frac{r}{\sqrt{\epsilon(T-t)}})
\end{align}
where $M>0$. After properly choosing $M>0$ and $\epsilon>0$ small, $\bar{\psi}$ will be our barrier. Indeed, thanks to \eqref{BARRIERPSI1Est} and \eqref{psi2ESTIMATEB} we have
\begin{align}
	\big[\partial_{t}-\partial_{r}^{2}-(\frac{1}{r}+4\frac{x-\xi}{\lambda^{2}+|x-\xi|^{2}}\cdot\frac{x}{r})\partial_{r}\big]\tilde{\psi}\ge& C \bar{g}(x,t)+ O(\frac{\lambda^{2}+|\xi||x-\xi|}{|x|(|x-\xi|^{2}+\lambda^{2})}|\partial_{r}\psi_{1}|)+\nonumber\\
	&+M\chi_{\epsilon}\frac{e^{-a\sqrt{2|\ln(T-t)|}}}{(T-t)|\ln(T-t)|^{b}}\frac{\lambda^{2}}{(r^{2}+\lambda^{2})^{2}}+R_{\epsilon}(r,t)\label{INEQUALITYGIVESPSIBAR}
\end{align}
where $R_{\epsilon}(r,t)=O\big(M(|\psi_{2}|\partial_{r}^{2}\chi_{\epsilon}+...)\big)$ is a remainder supported only when $\sqrt{\epsilon(T-t)}\le r \le 2\sqrt{\epsilon(T-t)}$. \newline
If we look at \eqref{INEQUALITYGIVESPSIBAR} and because of \eqref{EstimateREmainderPSIBAR} we see that when $r\le \sqrt{\epsilon(T-t)}$ it is enough to take $M$ sufficiently large and $\epsilon$ sufficiently small. The remaining regions we need to investigate are $\sqrt{\epsilon(T-t)}\le r \le 2\sqrt{\epsilon(T-t)}$ and $r\ge 2\sqrt{\epsilon(T-t)}$. \newline If $\sqrt{\epsilon(T-t)}\le r \le 2 \sqrt{\epsilon(T-t)}$ we see that
\begin{align*}
	O(\frac{\lambda^{2}+|\xi| |x-\xi|}{|x|(|x-\xi|^{2}+\lambda^{2})}|\partial_{r}\psi_{1}|)+|R_{\epsilon}(r,t)|=O(M \frac{\lambda^{2}}{\epsilon(T-t)}\frac{e^{-a\sqrt{2|\ln(T-t)|}}}{(T-t)|\ln(T-t)|^{b}})\ll |\bar{g}(x,t)|
\end{align*}
if $T$ is sufficiently small. \newline
In the remaining region, namely $r\ge 2\sqrt{\epsilon(T-t)}$, we observe: 
\begin{align*}
	O(\frac{\lambda^{2}+|\xi||x-\xi|}{|x|(|x-\xi|^{2}+\lambda^{2})}|\partial_{r}\psi_{1}|)=O(\frac{\lambda^{2}}{T-t}\frac{1}{|x|}|\partial_{r}\psi_{1}|)\ll \bar{g}(r,t).
\end{align*}
Thus, if $M$ is sufficiently large, $\epsilon$, $T$ are sufficiently small, \eqref{INEQUALITYGIVESPSIBAR} gives that for any $(x,t)\in \mathbb{R}^{2}\times(0,T)$
\begin{align*}
		\big[\partial_{t}-\partial_{r}^{2}-(\frac{1}{r}+4\frac{x-\xi}{\lambda^{2}+|x-\xi|^{2}}\cdot\frac{x}{r})\partial_{r}\big]\tilde{\psi}\ge& \frac{C}{2} \bar{g}(x,t)
\end{align*}
and then $\bar{\psi}$ is the desired barrier for \eqref{OuterEqfinSec}.
\end{proof}

\subsection*{Acknowledgements}
The present work is part of the PhD thesis of the first author who is funded by URSA (University Research Studentship Award - Science).
J.~D\'avila has been supported  by  a Royal Society  Wolfson Fellowship, UK.
M.~del Pino has been supported by a Royal Society Research Professorship, UK.
M. Musso has been supported by EPSRC research Grant EP/T008458/1.

\bibliographystyle{amsplain}

\begin{thebibliography}{10}
	
	\bibitem{AS} M. Abramovitz and I. A. Stegun \emph{Handbook of mathematical functions}, New York: Dover. ISBN 978-0-486-61272-0.
    
    \bibitem{Bi} P. Biler \emph{Local and global solvability of some parabolic systems modelling chemotaxis}, , Adv. Math. Sci. Appl., 8
    (1998), pp. 715–743.
    
    \bibitem{Bla}
    A. Blanchet \emph{On the parabolic-elliptic Patlak-Keller-Segel system in dimension 2 and higher}, S´eminaire Laurent Schwartz—Equations aux d´eriv´ees partielles et applications. Ann´ee 2011–2012, S´emin. ´ Equ. D´eriv. Partielles, ´
    Ecole Polytech., Palaiseau, 2013, pp. Exp. No. VIII, 26.
    
	\bibitem{BKLN} P. Biler, G. Karch, P. Laurençot, and T. Nadzieja \emph{The $8\pi$-problem for radially symmetric solutions of a chemotaxis model
	in a disc}. Topol. Methods Nonlinear Anal. 27 (2006), no. 1, 133–147.
	
	\bibitem{BKLN1} P. Biler, G. Karch, P. Laurençot, and T. Nadzieja \emph{The $8\pi$-problem for radially symmetric solutions of a chemotaxis model
		in the plane}.  Math. Methods Appl. Sci. 29 (2006), no. 13, 1563–1583. 
	
	
	
	\bibitem{BCM}
	A. Blanchet, J. A. Carrillo, and N. Masmoudi \emph{Infinite time aggregation for a critical Patlak-Keller-Segel model in $\mathbb{R}^{2}$}, Comm. Pure Appl. Math., 61(10):1449–1481, 2008.
	
	
	
	\bibitem{BDP}
	A. Blanchet, J. Dolbeault, and B. Perthame \emph{Two-dimensional Keller-Segel model: optimal critical mass and qualitative properties of the solutions}, Electron. J. Differential
	Equations, pages No. 44, 32 pp. (electronic), 2006.
	
	\bibitem{CaDo}
	J. Campos \& J. Dolbeault, \emph{A functional framework for the Keller-Segel system: logarithmic Hardy-LittlewoodSobolev and related spectral gap inequalities},  C. R. Math. Acad. Sci. Paris, 350 (2012), pp. 949–954.
	
	\bibitem{CaDo1}
	J. Campos \& J. Dolbeault, \emph{Asymptotic estimates for the parabolic-elliptic Keller-Segel model in the plane}, Comm. Partial Differential Equations, 39 (2014), pp. 806–841.
	
	\bibitem{Camp}
	J. Campos Serrano \emph{Mod`eles attractifs en astrophysique et biologie: points critiques et comportement en temps
		grand des solutions}, PhD thesis, Thèse de l’Universitè Paris Dauphine, 2012.
	
	\bibitem{CaLo}
	E. Carlen \& M. Loss \emph{Competing symmetries, the logarithmic HLS inequality and Onofri’s inequality on $\mathbb{S}^{n}$}, Geom. Funct. Anal., 2 (1992), pp. 90–104.
	
	\bibitem{CaFi}
	E. Carlen \& A. Figalli, \emph{Stability for a GNS inequality and the log-HLS inequality, with application to the
		critical mass Keller-Segel equation}, Duke Math. J., 162 (2013), pp. 579–625.
	
	\bibitem{Ch}
	P .H. Chavanis, \emph{Nonlinear mean field Fokker-Planck equations. application to the chemotaxis of biological populations}, , The European Physical Journal B, 62 (2008), pp. 179–208.
	
	\bibitem{ChSi}
	P. H. Chavanis \& C. Sire, \emph{Virial theorem and dynamical evolution of self-gravitating Brownian particles in
		an unbounded domain. I. Overdamped models},  Phys. Rev. E (3), 73 (2006), pp. 066103, 16.
	
		\bibitem{ChSi2}
	P. H. Chavanis \& C. Sire, \emph{Virial theorem and dynamical evolution of self-gravitating Brownian particles in an unbounded domain.
		II. Inertial models}, Phys. Rev. E (3), 73 (2006), pp. 066104, 13.
	
	\bibitem{CP}
	S. Childress \& J. K. Percus \emph{Nonlinear aspects of chemotaxis}, Math. Biosci., 56
	(3-4):217–237, 1981.
	
	
	
	\bibitem{CGMN1}
	C. Collot, T. Ghoul, N. Masmoudi, and V.T. Nguyen \emph{Spectral analysis
		for singularity formation of the 2D Keller-Segel system}, Annals of PDE, Vol.8, Issue 1, 2022.
	
	
	\bibitem{CGMN2}
	C. Collot, T. Ghoul, N. Masmoudi, and V.T. Nguyen \emph{Refined description and stability
		for singular solutions of the 2d Keller-Segel system}, Comm. Pure Appl. Math., Volume 75, Issue 7, pp. 1419-1516, 2021.
	\bibitem{CGMN3}
	C. Collot, T. Ghoul, N. Masmoudi, and V.T. Nguyen \emph{Collapsing-ring blowup solutions for the Keller-Segel system in three dimensions and higher}, arXiv:2112.15518, 2021.
	
	
	\bibitem{CdPM}
	C. Cort\'azar, M. del Pino, and M. Musso \emph{Green’s function and infinite-time bubbling in the critical nonlinear
		heat equation},Journal of the European Mathematical Society, Vol. 22, No. 1 pp. 283–344, 2019.	
	
	\bibitem{DdPDMW}
	J. D\'avila,  M. del Pino, Dolbeault J., M. Musso, and J. Wei, \emph{ Infinite time blow-up in
		the patlak-keller-segel system: existence and stability.} arXiv:1911.12417, 2023 .
	
	
	\bibitem{DdPMW}
	J. D\'avila,  M. del Pino, M. Musso, and J. Wei, \emph{Gluing methods for vortex dynamics in Euler flows.} Archive for Rational Mechanics and Analysis,  235, pp. 1467–1530(2020) .
	
	\bibitem{DdPPMW}
	J. D\'avila,  M. del Pino, C. Pesce, and J. Wei, \emph{Blow-up for the 3-dimensional axially symmetric harmonic map flow into 
		$\mathbb{S}^{2}$.} Discrete and Continuous Dynamical Systems,  Volume 39, Issue 12: 6913-6943(2019) .


	
	
	
	\bibitem{DdPW}
	J. D\'avila,  M. del Pino, and J. Wei, \emph{ Singularity formation for the two-dimensional harmonic map flow into $\mathbb{S}^{2}$}, Inventiones Mathematicae, Volume 219, issue 2, pp. 345-466 (2019).
	
		
	\bibitem{DLOS}
	S. Dejak, P. Lushnikov, Y. Ovchinnikov, and I. Sigal, \emph{On spectra of linearized operators for Keller–Segel
		models of chemotaxis}, , Physica D: Nonlinear Phenomena, 241 (2012), pp. 1245–1254.
	
	\bibitem{dP}
	M. del Pino \emph{Bubbling blow-up in critical parabolic problems, in Nonlocal and nonlinear diffusions and interactions:
		new methods and directions},  vol. 2186 of Lecture Notes in Math., Springer, Cham, 2017, pp. 73–116.
	
	\bibitem{dPMW}
	M. del Pino, M. Musso, and J. Wei \emph{Infinite-time blow-up for the 3-dimensional energy-critical heat equation}, Anal. PDE, 13 (2020), pp. 215–274.
	
	\bibitem{dPW}
	M. del Pino \& J. Wei \emph{Collapsing steady states of the Keller-Segel system. }, Nonlinearity 19 (2006), no. 3, 661–684.
	
	\bibitem{DiNaRa}
	J. I. Diaz, T. Nagai, and J. -M. Rakatoson, \emph{Symmetrization techniques on unbounded domains: application to
		a chemotaxis system on $\mathbb{R}^{N}$}, , J. Differential Equations, 145 (1998), pp. 156–183.
	
	
	\bibitem{DP}
	J. Dolbeault \& B. Perthame \emph{Optimal critical mass in the two-dimensional Keller-Segel
		model in $\mathbb{R}^{2}$}, C. R. Math. Acad. Sci. Paris, 339(9):611–616, 2004.	
	
	\bibitem{DS}
	J. Dolbeault \& C. Schmeiser, \emph{The two-dimensional Keller-Segel model after blow-up}, Discrete and Continuous
	Dynamical Systems, 25 (2009), pp. 109–121.
	
	\bibitem{EFM}
	G. Egaña Fern\'andez \& S. Mischler, \emph{Uniqueness and long time asymptotic for the Keller-Segel equation: the
		parabolic-elliptic case}, Arch. Ration. Mech. Anal., 220 (2016), pp. 1159–1194.
	
	
	
	\bibitem{GM}
	T. Ghoul \& N. Masmoudi, \emph{Minimal mass blowup solutions for the Patlak-Keller-Segel
		equation}, Comm. Pure Appl. Math., 71(10):1957–2015, 2018.
	
	

	
	\bibitem{HV1}
	M. A. Herrero \& J. J. L. Vel\'azquez, \emph{A blow-up mechanism for a chemotaxis model}, Annali della
	Scuola Normale Superiore di Pisa-Classe di Scienze, 24(4):633–683, 1997.
	
	\bibitem{HV2}
	M. A. Herrero \& J. J. L. Vel\'azquez, \emph{Chemotactic collapse for the Keller-Segel model}, J. Math. Biol.  35: 177Ð194 (1996)
	
			\bibitem{HV}
	M. A. Herrero \& J. J. L. Vel\'azquez, \emph{Finite-time aggregation into a single point in a reaction-diffusion system}, Nonlinearity, 10(6):
	1739–1754,1997.
	
			\bibitem{HV3}
	M. A. Herrero \& J. J. L. Vel\'azquez, \emph{ Singularity patterns in a chemotaxis model.}, Math. Ann. 306 (1996), no. 3, 583–623.
	
	\bibitem{HMV0}
	M. A. Herrero, E. Medina, and J. J. L. Vel\'azquez, \emph{Finite-time aggregation into a single point in a reaction-diffusion system}, Nonlinearity, 10(6):
	1739–1754, 1997.
	
	
	\bibitem{HMV}
	M. A. Herrero, E. Medina, and J. J. L. Vel\'azquez, \emph{Self-similar blowup for a reaction-diffusion system}, 
	J. Comput. Appl. Math., 97(1-2):99–
	119, 1998.


\bibitem{HP}
T. Hillen \& K. J. Painter, \emph{, A user’s guide to pde models for chemotaxis}, Journal of Mathematical Biology, 58
(2008), pp. 183–217.


\bibitem{H0}
D. Horstmann \emph{From 1970 until present: the Keller-Segel model in chemotaxis and its consequences. I},  Jahresber.
Deutsch. Math.-Verein., 105 (2003), pp. 103–165.

\bibitem{H1}
D. Horstmann \emph{From 1970 until present: the Keller-Segel model in chemotaxis and its consequences. II},   Jahresber.
Deutsch. Math.-Verein., 106 (2004), pp. 51–69.

		\bibitem{JL}
	W. J\"ager \& S. Luckhaus, \emph{On explosions of solutions to a system of partial differential equations modelling chemotaxis}, Trans. Amer. Math Soc., 329 (1992), pp. 819-824.
	
	\bibitem{KaSo}
	N. I. Kavallaris \& P. Souplet \emph{Grow-up rate and refined asymptotics for a two-dimensional Patlak-KellerSegel model in a disk}, , SIAM J. Math. Anal., 40 (2008/09), pp. 1852–1881.
	
	
	
	
	\bibitem{KS0}
	E. F. Keller \emph{Assessing the Keller-Segel model: how has it fared? }, Biological growth and spread (Proc. Conf., Heidelberg,
	1979), Lecture Notes in Biomath., vol. 38, pp. 379–387, Springer, Berlin-New York, 1980.
	
	\bibitem{KS}
	E. F. Keller \& L. A. Segel, \emph{Initiation of slime mold aggregation viewed as an instability}, Journal of Theoretical Biology
	26 (1970), no. 3, 399 – 415.
	
	
	\bibitem{KS2}
	E. F. Keller \& L. A. Segel
	\emph{Model for chemotaxis}, Journal of Theoretical Biology 30 (1971), no. 2, 225 – 234.
	
	\bibitem{KS1}
	E. F. Keller \& L. A. Segel \emph{Traveling bands of chemotactic bacteria: A theoretical analysis}, Journal of Theoretical Biology
	30 (1971), no. 2, 235 – 248.
	
	\bibitem{KoSu}
	H. Kozono \& Y. Sugiyama \emph{ Local existence and finite time blow-up of solutions in the 2-D Keller-Segel system.} J. Evol.
	Equ. 8 (2008), no. 2, 353–378.

	
\bibitem{LoNaYa}
J. L\'opez-G\'omez, T. Nagai, and T. Yamada \emph{The basin of attraction of the steady-states for a chemotaxis model
	in $\mathbb{R}^{2}$ with critical mass}, Arch. Ration. Mech. Anal., 207 (2013), pp. 159–184.

	
	\bibitem{LoNaYa1}
	J. L\'opez-G\'omez, T. Nagai, and T. Yamada \emph{Non-trivial $\omega$-limit sets and oscillating solutions in a chemotaxis model in $\mathbb{R}^{2}$ with critical mass}, J. Funct.
	Anal., 266 (2014), pp. 3455–3507.
	
	
	\bibitem{Mi}
	N. Mizoguchi \emph{ Refined asymptotic behavior of blowup solutions to a simplified chemotaxis system.} Communications on Pure and Applied Mathematics, n/a(n/a), 2020.





	\bibitem{Na}
	T. Nagai \emph{Blow-up of radially symmetric solutions to a chemotaxis system}. Adv. Math.
	Sci. Appl., 5(2):581–601, (1995)
	
	\bibitem{Na1}
	T. Nagai \emph{Blowup of nonradial solutions to parabolic–elliptic systems modeling chemotaxis in two-dimensional domains}, Journal of Inequalities and Applications, 2001 (2001), p. 970292.
	
	\bibitem{Na2}
	T. Nagai \emph{Convergence to self-similar solutions for a parabolic-elliptic system of drift-diffusion type in $\mathbb{R}^{2}$}, Adv.
	Differential Equations, 16 (2011), pp. 839–866.
	
	
	\bibitem{NaOg} 
	T. Nagai \& T. Ogawa, \emph{Global existence of solutions to a parabolic-elliptic system of drift-diffusion type in $\mathbb{R}^{2}$}, Funkcial. Ekvac., 59 (2016), pp. 67–112.
	
	
	\bibitem{NaYa}
	T. Nagai \& T. Yamada \emph{Boundedness of solutions to a parabolic-elliptic Keller-Segel equation in $\mathbb{R}^{2}$ with
		critical mass}, Adv. Nonlinear Stud., 18 (2018), pp. 337–360.
	
	\bibitem{Nanj}
	V. Nanjundiah, \emph{Chemotaxis, signal relaying and aggregation morphology}, Journal of Theoretical Biology, 42(1):63 – 105, 1973.
	

	
	
	\bibitem{Ni}
	K. Nishihara, \emph{A note on the stability of travelling wave solutions of Burgers’ equation}, Japan J. Appl. Math., 2(1):27–35, 1985
	
	
		\bibitem{Pat}
	C. S. Patlak, \emph{ Random walk with persistence and external bias}, Bull. Math. Biophys. 15 (1953), 311–338.
	
	
	\bibitem{RS}
	P. Rapha\"el \& R. Schweyer, \emph{On the stability of critical chemotactic aggregation}. Math. Ann., 359 (1-2): 268-377, 2014.
	
	\bibitem{SeSu}
	T. Senba \& T. Suzuki, \emph{Weak solutions to a parabolic-elliptic system of chemotaxis}, J. Funct. Anal., 191 (2002),
	pp. 17–51.

	
	\bibitem{SiCh}
	C. Sire \& P. -H. Chavanis, \emph{Thermodynamics and collapse of self-gravitating brownian particles in d dimensions}, Phys. Rev. E, 66 (2002), p. 046-133.
	
	
	\bibitem{SoWi}
	P. Souplet \& M. Winkler, \emph{Blow-up profiles for the parabolic-elliptic Keller-Segel system in dimensions $n\ge3$}, Comm. Math. Phys., 367 (2019), pp. 665–681.
	
	
	
	
	\bibitem{Ve}
	J. J. L. Vel\'azquez, \emph{Stability of some mechanisms of chemotactic aggregation}. SIAM J.
	Appl. Math., 62(5):1581–1633, 2002.
	
		\bibitem{Ve1}
	J. J. L. Vel\'azquez, \emph{Point dynamics in a singular limit of the Keller-Segel model. I. Motion of the concentration regions}. SIAM J. Appl. Math. 64 (2004), no. 4, 1198–1223. 
	
		\bibitem{Ve2}
	J. J. L. Vel\'azquez, \emph{Point dynamics in a singular limit of the Keller-Segel model. II. Formation of the concentration regions.}. SSIAM J. Appl. Math. 64 (2004), no. 4, 1224–1248. 
	
	\bibitem{Wo}
	G. Wolansky, \emph{On steady distributions of self-attracting clusters under friction and fluctuations}. Arch. Rational Mech. Anal.,119(4):355–391, 1992.
	
	
	
	
	
	
\end{thebibliography}

\end{document}